\documentclass[12pt,a4paper,twoside,openright]{report}
\usepackage{graphicx}
\usepackage{scrextend}
\usepackage{lscape}
\usepackage{enumerate}
\usepackage{amsmath}
\usepackage{amsfonts}
\usepackage{amssymb}
\usepackage{amsthm}
\usepackage{epsfig} 
\usepackage{pdfsync}
\usepackage{epstopdf}
\usepackage[mathscr]{euscript}
\usepackage{setspace}
\usepackage{a4}
\usepackage{amssymb,amsmath,amsthm,latexsym}
\usepackage{amsfonts}
\usepackage{amsfonts}
\usepackage{graphicx}
\usepackage{textcomp}
\usepackage{cite}
\usepackage{enumerate}
\usepackage[mathscr]{euscript}
\usepackage{mathtools}
\usepackage{amssymb}
\usepackage{amsmath}
\usepackage{tikz}
\usepackage{amssymb}
\usepackage{amssymb}
\usepackage{amssymb}
\usepackage{amssymb}
\usepackage{amsmath}
\usepackage{tikz}
\usepackage{hyperref}
\usepackage{enumerate}
\usepackage{mathtools}
\usepackage{amsmath}
\usepackage{tikz}
\usepackage{amssymb}
\usepackage{amsmath}
\usepackage{tikz}
\usepackage{hyperref}
\usepackage{setspace}
\usepackage{a4}
\usepackage{amssymb,amsmath,amsthm,latexsym}
\usepackage{amsfonts}
\usepackage{amsfonts}
\usepackage{graphicx}
\usepackage{textcomp}
\usepackage{cite}
\usepackage{enumerate}
\usepackage[mathscr]{euscript}
\usepackage{mathtools}
\usepackage{setspace}
\usepackage{a4}
\usepackage{amssymb,amsmath,amsthm,latexsym}
\usepackage{amsfonts}
\usepackage{amsfonts}
\usepackage{graphicx}
\usepackage{textcomp}
\usepackage{tikz-cd}
\usetikzlibrary{cd}

\usepackage{cite}
\newtheorem{theorem}{Theorem}[section]

\newtheorem{corollary}[theorem] {Corollary}
\newtheorem{definition}[theorem]{Definition}
\newtheorem{example}[theorem]{Example}
\newtheorem{lemma} [theorem]{Lemma}

\newtheorem{proposition}[theorem]{Proposition}
\newtheorem{remark}[theorem]{Remark}
\usepackage[hhmmss]{datetime}
\usepackage[belowskip=-15pt,aboveskip=5pt]{caption} 
\usepackage[authoryear]{natbib}   
\usepackage{hyperref} 
\usepackage{setspace} 
\usepackage{rotating} 
\usepackage{multirow} 
\usepackage{fancyhdr} 
\usepackage{sectsty}  
\usepackage{makeidx}
\usepackage[T1]{fontenc}  
\usepackage{array}
\usepackage{hvfloat,lipsum}
\usepackage[labelsep=space]{caption}
\usepackage{ragged2e} 
\usepackage{enumitem}
\usepackage[acronym]{glossaries} 
\usepackage{pslatex} 
\usepackage{mathtools}

\renewcommand{\baselinestretch}{1.5} 
\usepackage{mathptmx}
\usepackage[titles]{tocloft}
\chapterfont{\LARGE}
\sectionfont{\large}
\usepackage[top=1.1in, bottom=1.4in, outer=1in, inner=1.5in]{geometry}

\usepackage{makecell}
\usepackage{subcaption}
\usepackage{adjustbox}
\usepackage{bm}

\allowdisplaybreaks
\numberwithin{equation}{section}
\usepackage{pdflscape}
\usepackage{afterpage}
\usepackage{capt-of}
\usepackage[titletoc]{appendix} 
\usepackage{afterpage}
\usepackage{romannum}

\begin{document}
\thispagestyle{empty} 
\newgeometry{top=1.2in,inner=1.5in,outer=1in}
\begin{center}
	\doublespacing {\large{\bf {METRIC, SCHAUDER AND OPERATOR-VALUED
				FRAMES}}}
	\par ~
	\par ~
	\singlespacing
	\par {\large {Thesis}} \vspace{0.5cm}
	\par {\large{Submitted in partial fulfillment of the requirements for the degree of }} \vspace{1cm}
	\par {\large{DOCTOR OF PHILOSOPHY}}\vspace{0.5cm}
	\par \large{{by}}\vspace{0.5cm}
	\par \large{{{\fontsize{12}{1em}\selectfont MAHESH KRISHNA K}}}
	\vskip 2cm
\end{center}
\par ~
\begin{center}
	\includegraphics[width=2.5 in, height=2.5 in]{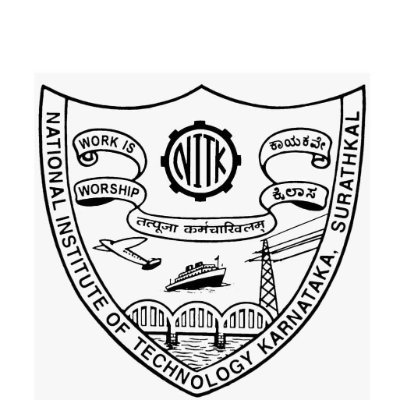}
\end{center}
\begin{center}
	\par ~
	\par{{\fontsize{12}{1em}\selectfont DEPARTMENT OF MATHEMATICAL AND COMPUTATIONAL SCIENCES}}
	\par ~\noindent{{{\fontsize{12}{1em}\selectfont NATIONAL INSTITUTE OF TECHNOLOGY KARNATAKA}}}
	\par \noindent{{{\fontsize{12}{1em}\selectfont SURATHKAL, MANGALURU - 575 025}}}
	\par {\fontsize{12}{1em}\selectfont FEBRUARY $2022$}
\end{center}
\leavevmode\newpage
\thispagestyle{empty}
\leavevmode\newpage
\thispagestyle{empty}
\def\baselinestretch{1}

\par~
\par~
\par~
\par~
\par~
\par~
\par~
\par~
\par~
\begin{center}
	{\LARGE Dedicated to the memory of}\\ 
	\vspace{1cm}
	{\LARGE \bf { {\em John von Neumann}}}\\
	\vspace{.1cm}
	({28 December 1903 - 08 February 1957})
	
\end{center}
\par~
\par~
\leavevmode\newpage
\thispagestyle{empty}
\leavevmode\newpage
\thispagestyle{empty}
\def\baselinestretch{1}
\par~
\begin{center}
	\textbf{{\fontsize{16}{1em}\selectfont DECLARATION}} \\
	\textit{By the Ph.D. Research Scholar}
\end{center}

\par I hereby declare that the research thesis entitled
\textbf{METRIC, SCHAUDER AND OPERATOR-VALUED
	FRAMES} which is being
submitted to the \textbf{National Institute of Technology Karnataka,
	Surathkal} in partial fulfillment of the requirements for the award
of the Degree of \textbf{Doctor of Philosophy} in
\textbf{Mathematical and Computational Sciences} is a
\textit{bonafide report of the research work carried out by me}. The material contained in this research thesis has not been submitted to any University or Institution for the award of any degree.

\par~
\par~
\par~
\par~
\par~
\par~

\begin{flushright}
	Place: NITK, Surathkal \hfill MAHESH KRISHNA K\\ 
	Date:  $\quad$    February 2022 \hfill 165106MA16F02\\
	Department of MACS\\
	NITK, Surathkal
\end{flushright}
\leavevmode\newpage
\thispagestyle{empty}

\leavevmode\newpage
\thispagestyle{empty}
\def\baselinestretch{0.1}
\par~
\begin{center}
	\textbf{{\fontsize{16}{1em}\selectfont CERTIFICATE}}
\end{center}

\par This is to \textit{certify} that the research thesis entitled
\textbf{METRIC, SCHAUDER AND OPERATOR-VALUED
	FRAMES} submitted by
\textbf{Mr. \; MAHESH \; KRISHNA \; K} (Register Number : 165106MA16F02) as the record of the research work carried out by him is \textit{accepted as the research thesis submission} in partial fulfillment of the requirements for the award of degree of \textbf{Doctor of Philosophy}.
\par~
\par~
\par~
\par~
\begin{flushright}
	Place: NITK, Surathkal \hfill \textbf{Dr. P. Sam Johnson}\\ 
	Date:   $\quad$    February 2022 \hfill Research Guide\\
	\hfill Department of MACS\\
	\hfill NITK Surathkal, Karnataka, India\\
	\par~
	\hspace{2cm}
\end{flushright}

\par~
\par~
\par~
\par~
\begin{flushright}
	\textbf{Chairman - DRPC}\\
	\hspace{3cm}
	
	(Signature with Date and Seal)
\end{flushright}
\leavevmode\newpage
\thispagestyle{empty}
\leavevmode\newpage
\thispagestyle{empty}

\def\baselinestretch{0.1}
\begin{center}
	{\onehalfspacing \section*{\LARGE{ACKNOWLEDGEMENTS}}}
\end{center}
I begin by remembering my guide Dr. P. Sam Johnson. It is because of him of whatever epsilon today I am.\\
\indent The kind reply by Prof. Victor Kaftal, University of Cincinnati, Ohio developed confidence in me to do research. I am thankful to Prof. Victor Kaftal very much. Aslo, I thank Prof. P. G. Casazza, Missouri University, USA, Prof. Radu Balan, University of Maryland, USA,  Prof. D. Freeman, St. Louis University, USA, Prof. Orr Moshe Shalit, Technion, Israel, Prof. Terence Tao, University of California, Los Angeles, USA and Prof. B. V. Rajarama Bhat, ISI Bangalore for their replies to my mails. I am thankful to all of them.\\
I am happy to thank my  friends Ramu G., Vinoth A.,   Kanagaraj K., Palanivel R., Niranjan P.K., Sumukha S., Saumya Y. M., Shivarama K. N., Sachin M., Mahesh Mayya, Chaitanya G. K., Manasa M., Rashmi K., Megha P. for their help in my life at NITK.\\
\indent My Research Progress Assessment Committee (RPAC) members Prof. Subhash C. Yaragal, Department of Civil Engineering and  Prof. B. R. Shankar, Department of MACS  helped me at various stages of research. I thank the RPAC members.\\
\indent I would like to thank former heads of MACS department Prof. Santhosh George, Prof. B. R. Shankar, Prof. S. S. Kamath and the present head Dr. R. Madhusudhan for their help at different times. Further, I thank all faculty members of MACS department and office staff.\\
\indent It was in MTTS  where I started thinking seriously about results.  I am  thankful to  Prof. S.  Kumaresan, Central University of Hyderabad for the initiation of MTTS which transformed Mathematics of India to a greater extent. \\
\indent The concept of Advanced Training in Mathematics Schools (ATM Schools) is a very useful initiative of National Centre for Mathematics (NCM) India which helps a lot for Ph.D. Scholars. I am fortunate to attend a dozen of these schools. Thanks to initiators and organizers of various ATM schools.\\
\indent Ph.D. course work papers of mine were given by Dr. Murugan V., Mrs. Sujatha D. Achar,  Prof. B. R. Shankar and Prof. A. H. Sequeira. I am thankful to them.\\ 
\indent When I was teaching at Centre for Post-Graduate Studies and Research, St. Philomena College at  Puttur, my classmates and colleagues   Vaishnavi C. and Prasad H. M. helped at various times. My students also helped a lot understand the subject better. It is necessary to remember all of them.\\
\indent Prof. M. S. Balasubramani, Prof. S. Parameshwara Bhatta,  Mohana K. S., Dr. Chandru Hegde taught me at Mangalore University. I particularly remember them for their teaching.\\
\indent I end the acknowledgements by remembering my father Narasimha D. K., my mother Sundari V., my sisters Shraddha K., Pavithra K.,  Pavana K., my  brother late Ganesh K. R. and my primary teacher Nataliya K.

\vspace{3cm}
\noindent Place: NITK, Surathkal \hfill MAHESH KRISHNA K
\\Date: 17 February 2022
\thispagestyle{empty}

\bigskip\medskip



\leavevmode\newpage



\pagenumbering{roman}
\addcontentsline{toc}{section}{ABSTRACT}
\def\baselinestretch{0.1}
\pagenumbering{roman}
\begin{center}
	{\onehalfspacing \section*{\LARGE{ABSTRACT}}}
\end{center}
Notion of  frames and Bessel sequences for metric spaces have  been introduced.   This notion is related  with the notion  of Lipschitz free Banach spaces. \ It is proved that   every  separable metric space admits a metric $\mathcal{M}_d$-frame. Through Lipschitz-free Banach spaces it is showed  that there is a correspondence
between  frames for metric spaces  and  frames for subsets of  Banach spaces. Several characterizations of metric frames are obtained.  Stability results  are also presented. Non linear multipliers are introduced and studied. This notion is connected with the notion of Lipschitz compact operators. Continuity properties of multipliers are discussed.

For a subclass of approximated Schauder frames for Banach spaces, characterization result is derived using standard Schauder basis for standard sequence spaces. Duals of a subclass of approximate Schauder frames are completely described. Similarity of this class is characterized and interpolation result is derived using orthogonality. A dilation result is obtained. A  new identity is derived for Banach spaces which admit a homogeneous semi-inner product.  Some stability results are obtained for this class.

A generalization of operator-valued frames for Hilbert spaces are introduced  which unifies all the known generalizations of frames for Hilbert spaces. This notion has been studied in depth by imposing factorization property of the frame operator. Its duality, similarity and orthogonality are addressed. Connections between this notion and unitary representations of groups and group-like unitary systems are derived. Paley-Wiener theorem for this class are derived.

\vspace{1cm}

\noindent {\small \textbf{Keywords: Frame, Riesz basis, Bessel sequence, Lipschitz function, multiplier, operator-valued frame, metric space. } }


\restoregeometry
\leavevmode\newpage 

\tableofcontents


\newpage
\pagenumbering{arabic}
{\onehalfspacing \chapter{INTRODUCTION}\label{chap1} }
\vspace{0.5cm}
{\onehalfspacing \section{GENERAL INTRODUCTION}
A vector in  a vector space is usually obtained as a linear  combination of elements of a basis for the vector space. Thus a vector is fully known if we know the coefficients in the representation of it using basis elements.   However, as dimension of the space increases, it is difficult to get the coefficients. Hence we look for nice spaces and certain bases which give coefficients of a given vector easily. For this purpose, Hilbert spaces and orthonormal bases become a very handy tool to obtain representation of a vector. \\
Orthonormal bases for Hilbert spaces have practical disadvantages. Since each coefficient in the expansion is very important, a small error in one of the coefficient leads to significant variation in the resultant vector and the actual vector. Thus we seek a  collection in Hilbert space which gives representation as well as a small change in coefficient need not effect much to the  original vector. This is where the theory of frames becomes important.\\
Historically it was \cite{GABOR}, who first studied representation of functions using translations and modulations of a single function (\cite{FEICHTINGERSTROHMERBOOK, FEICHTINGERSTROHMERBOOK2, GROCHENIGBOOK, FEICHTINGERLUEFWERTHER, FEICHTINGERKOZEKLUEF, SONDERGAARD}). In 1947, Sz. Nagy studied sequences which are close to orthonormal bases using Paley-Wiener type results (\cite{NAGYEXP}). Modern definition of frames was set by \cite{DUFFIN1}  in the study of sequences of type $ \{e^{i\lambda_nx}\}_{n\in \mathbb{Z}}, \lambda_n \in \mathbb{C}, x\in \left(-r, r \right), r>0$. After this work, \cite{YOUNG}   made an account of frames in his book `An introduction to nonharmonic Fourier series'.\\
Paper of Daubechies, Grossmann, and Meyer (\cite{MEYER1}) triggered the  area of frames for Hilbert spaces. Later, the paper of \cite{BENEDETTOFICKUS} influenced the development of frame theory for finite dimensional Hilbert spaces. Today theory of frames find its uses in many areas such as wireless communication (\cite{STROHMERAPP}), signal processing (\cite{MALLAT}), image processing (\cite{DONOHOELADOPTIMAL}), sampling theory (\cite{BSAMPLING}), filter banks (\cite{FICKUSMASSAR}), psycho acoustics (\cite{BALAZSPSYC}), quantum design (\cite{BODMANNHAASDESIGN}), quantum channels (\cite{HANJUSTE}), quantum optics (\cite{JAMIOIKOWSKI}), quantum measurement (\cite{ELDARFORNEYMEASUREMENT}), numerical approximation (\cite{ADCOCKHUYBRECHS}), Sigma-Delta quantization (\cite{BENEDETTOPOWELL}), coding (\cite{STROHMERHEATHCODING}) and graph theory (\cite{BODMANNPAULSEN}). For a comprehensive look on the theory of frames, we refer (\cite{CHRISTENSEN}, \cite{HANLARSON}, \cite{HANKORNELSONLARSON}, \cite{CASAZZABOOK}, \cite{WALDRONBOOK}, \cite{HEILBOOK}, \cite{OKOUDJOUBOOK}, \cite{PESENSON}).

Since many spaces appearing both in theoretical and practicals are Banach spaces which may not be Hilbert spaces, there is a need for extending the notion of frames to Banach spaces.   This was first done by
 \cite{GROCHENIG}. After the study of several function spaces (\cite{FEICHTINGERCHOOSE}),  Gr\"{o}chenig first studied the notion of an atomic decomposition for 
Banach spaces and then defined the notion of a Banach frame. 
\cite{FEICHTINGGERGROCHENIG1, FEICHTINGGERGROCHENIG2, FEICHTINGGERGROCHENIG3} in 90's   developed
a  theory of atomic decompositions and frames for a large class of function spaces such as modulation spaces (\cite{FEICHTINGERLOOKING})  and coorbit  spaces (\cite{BERGE}), via,
group representations and projective representations.\\
Abstract study of atomic decompositions and frames for Banach spaces started from the
fundamental paper (\cite{CASAZZAHANLARSONFRAMEBANACH}). Further study and variations of the  frames for Banach spaces are done in \cite{CARANDOLASSALLESCHMIDBERG}, \cite{TEREKHIN2010}, \cite{TEREKHIN2009}, \cite{TEREKHIN2004}, \cite{FORNASIERALPHA}, \cite{CASAZZACHRISTENSENSTOEVA}, \cite{STOEVA2009}, \cite{STOEVA2012}, \cite{ALDROUBISLANTED}, \cite{GROCHENIGLOCALIZATION} and so on.

{\onehalfspacing \section{ORTHONORMAL BASES, RIESZ BASES, FRAMES AND\\ BESSEL SEQUENCES FOR HILBERT SPACES}
In the study of integral equations, Hilbert studied the space  of square integrable sequences (\cite{BLANCHARD}). Later, John \cite{VONNEUMANN} formulated the notion of Hilbert spaces.
\begin{definition}(cf. \cite{LIMAYE})
A vector space $\mathcal{H}$ over $\mathbb{K}$ ($\mathbb{R}$ or $\mathbb{C}$)	is said to be a \textbf{Hilbert space} if there exists a map $\langle \cdot, \cdot \rangle:\mathcal{H} \times \mathcal{H}\to \mathbb{K}$  such that the following axioms hold. 
\begin{enumerate}[label=(\roman*)]
	\item $\langle h, h \rangle \geq 0$, $ \forall h \in \mathcal{H}$.
	\item If $h \in \mathcal{H}$ is such that $\langle h, h \rangle = 0$, then $h=0$.
	\item $\langle h, h_1 \rangle =\overline{\langle h_1, h \rangle}$, $ \forall h, h_1 \in \mathcal{H}$.
	\item $\langle \alpha h+ h_1, h_2 \rangle =\alpha \langle h, h_2 \rangle+\langle h_1, h_2 \rangle$, $ \forall h, h_1, h_2 \in \mathcal{H}$,  $\forall \alpha \in \mathbb{K}$. 
	\item $\mathcal{H}$ is complete with respect to the norm $\|h\|\coloneqq \sqrt{\langle h, h \rangle}$.
\end{enumerate}		
\end{definition}
We now mention two important  examples of Hilbert spaces.
\begin{example}
 \begin{enumerate}[label=(\roman*)](cf. \cite{LIMAYE}) 
	\item Let $n\in \mathbb{N}$. The space $\mathbb{K}^n$ equipped with the inner product 
	\begin{align*}
	\langle (a_k)_{k=1}^n, (b_k)_{k=1}^n\rangle \coloneqq \sum_{k=1}^na_k\overline{b_k}, \quad \forall(a_k)_{k=1}^n, (b_k)_{k=1}^n \in \mathbb{K}^n
	\end{align*}
	is a finite dimensional separable Hilbert space. 
	\item The space $\ell^2(\mathbb{N}) \coloneqq \{\{a_n\}_n: a_n \in \mathbb{K}, \forall n \in \mathbb{N}, \sum_{n=1}^{\infty}|a_n|^2<\infty\}$ equipped with the inner product 
	\begin{align*}
	\langle \{a_n\}_n, \{b_n\}_n\rangle \coloneqq \sum_{n=1}^{\infty}a_n\overline{b_n}, \quad \forall \{a_n\}_n, \{b_n\}_n \in  \ell^2(\mathbb{N})
	\end{align*}
	is an infinite dimensional  separable Hilbert space. The space $\ell^2(\mathbb{N})$ is known as the standard separable Hilbert space.
\end{enumerate}
\end{example}
Throughout this thesis, we assume that all of Hilbert spaces are separable.
Among all kinds of sets in a Hilbert space, orthonormal sets	are the easiest to handle, whose definition reads as follows.
\begin{definition}(cf. \cite{LIMAYE})
A	collection $\{\tau_n\}_{n}$ in  a Hilbert space $\mathcal{H}$ is called  an \textbf{orthonormal set} in  $\mathcal{H}$ if $\langle \tau_j, \tau_k\rangle =\delta_{j,k}, \forall j,k \in \mathbb{N}$.
\end{definition}
Following theorem, known as Gram-Schmidt orthonormalization (cf. \cite{LEON100}) shows that a linearly independent sequence of vectors can be converted into an orthonormal set such that at each stage of conversion the spaces spanned by the original set and transformed set are the  same.
\begin{theorem}(cf. \cite{LIMAYE}) (\textbf{Gram-Schmidt orthonormalization})
Let $\{\tau_n\}_{n}$ be a linearly independent subset of 	$\mathcal{H}$. Define $\omega_1\coloneqq \tau_1$, $\rho_1\coloneqq \omega_1/{\|\omega_1\|}$ and 
\begin{align*}
\omega_n\coloneqq \tau_n-\sum_{k=1}^{n-1}\langle \tau_n, \rho_k\rangle \rho_k, \quad \rho_n\coloneqq \frac{\omega_n}{\|\omega_n\|},\quad \forall n \geq 2.
\end{align*}
Then $\{\rho_n\}_{n}$ is  orthonormal and 
\begin{align*}
\operatorname{span}\{\rho_k\}_{k=1}^n=\operatorname{span}\{\tau_k\}_{k=1}^n,\quad \forall n \geq 1.
\end{align*}
\end{theorem}
One of the most important inequalities associated with an orthonormal sequence is the Bessel's inequality. It is a generalization of Cauchy-Schwarz inequality. 
\begin{theorem}(cf. \cite{LIMAYE}) (\textbf{Bessel's inequality})
If $\{\tau_n\}_{n}$  is  an orthonormal set in  $\mathcal{H}$, then the series $\sum_{n=1}^\infty|\langle h, \tau_n\rangle |^2$ converges for all $h \in \mathcal{H}$ and 
\begin{align*}
\sum_{n=1}^\infty|\langle h, \tau_n\rangle |^2\leq \|h\|^2, \quad \forall h \in \mathcal{H}.
\end{align*}	
\end{theorem}
Next theorem characterizes convergence of series in a Hilbert space with that of sequence of scalars.
\begin{theorem}(cf. \cite{LIMAYE}) (\textbf{Riesz-Fisher theorem})
Let $\{a_n\}_{n}$ be a sequence of scalars and $ \{\tau_n\}_{n}$ be an  orthonormal set in   $\mathcal{H}$. Then 
\begin{align*}
\sum_{n=1}^\infty a_n\tau_n \text{ converges in } \mathcal{H} \text{ if and only if } \sum_{n=1}^\infty|a_n|^2 \text{ converges in } \mathbb{R}.
\end{align*}
\end{theorem}
Next theorem shows that given any  orthonormal set and an element in a Hilbert space, the inner product of the element with the members of an  orthonormal set can be non zero at most  countably many times. 
\begin{theorem}(cf. \cite{LIMAYE})
Let $ \{\tau_n\}_{n}$ be an orthonormal set in   $\mathcal{H}$ and $h \in \mathcal{H}$. Then the set $E_h\coloneqq \{\tau_n: \langle h, \tau_n\rangle \neq 0, n \in \mathbb{N}\}$ is either finite or countable.	
\end{theorem}
A natural analogue of basis for finite dimensional vector spaces to that of infinite dimensional Hilbert spaces is the notion of Schauder basis and orthonormal basis. 
\begin{definition}(cf. \cite{CHRISTENSEN})\label{ONBDEFINITIONOLE}
A	collection $ \{\tau_n\}_{n}$ in  $\mathcal{H}$ is called
\begin{enumerate}[label=(\roman*)]
\item a  \textbf{Schauder basis} for $\mathcal{H}$ if for each  $h \in \mathcal{H}$, there exists a unique collection $\{a_n(h)\}_{n }$ of scalars such that $\sum_{n=1}^\infty a_n(h)\tau_n$ converges in  $\mathcal{H}$ and  $h=\sum_{n=1}^\infty a_n(h)\tau_n$.
\item an  \textbf{orthonormal  basis} for $\mathcal{H}$ if it is a Schauder basis for $\mathcal{H}$ and it is orthonormal.
\end{enumerate}
\end{definition}
We now give various examples of orthonormal bases for Hilbert spaces.
\begin{example}(cf. \cite{CHRISTENSEN})
\begin{enumerate}[label=(\roman*)]
\item 	Define $e_n\coloneqq\{\delta_{n,k}\}_{k}$, where $\delta_{\cdot,\cdot}$ is the Kronecker delta. Then $\{e_n\}_{n }$ is an orthonormal basis for $\ell^2(\mathbb{N})$. This is known as the standard orthonormal basis for $\ell^2(\mathbb{N})$.
\item Define 
	\begin{align*}
	\mathcal{L}^2[0,1]\coloneqq	\left\{f:[0,1]\to \mathbb{C} \text{ is measurable and } \int_{0}^{1}|f(x)|^2\,dx<\infty\right\}
\end{align*}
equipped with the inner product 
\begin{align*}
	\langle f, g \rangle \coloneqq \int_{0}^{1}f(x)\overline{g(x)}\,dx.
\end{align*}
 Let $n \in \mathbb{Z}$. Define $e_n: [0,1] \ni x \mapsto e^{2\pi inx}\in \mathbb{C}$. Then $\{e_n\}_{n=-\infty}^\infty$ is an orthonormal basis for $\mathcal{L}^2[0,1]$. 
\item  (\textbf{Gabor basis}) Define 
	\begin{align*}
	\mathcal{L}^2(\mathbb{R})\coloneqq	\left\{f:\mathbb{R}\to \mathbb{C} \text{ is measurable and } \int_{-\infty}^{\infty}|f(x)|^2\,dx<\infty\right\}
\end{align*}
equipped with the inner product 
\begin{align*}
	\langle f, g \rangle \coloneqq \int_{-\infty}^{\infty}f(x)\overline{g(x)}\,dx.
\end{align*}
Let $ \chi_{[0,1]}$ be the characteristic function on $[0,1]$. For $j, k \in \mathbb{Z}$, define $f_{j,k}(x)\coloneqq e^{2\pi i jx} \chi_{[0,1]}(x-k), \forall x \in \mathbb{R}$. Then  $\{f_{j,k}\}_{j, k \in \mathbb{Z}}$ is an orthonormal basis for $\mathcal{L}^2(\mathbb{R})$.
\item  (\textbf{Haar system}) Let $\psi$ be the Haar function defined on $\mathbb{R}$ by 
\begin{align*}
  \psi(x) \coloneqq\left\{
\begin{array}{ll}
1 & \operatorname{  if } ~0 \leq x<\frac{1}{2} \\
-1 & \operatorname{  if } ~\frac{1}{2}\leq x \leq 1 \\
0 & \operatorname{ otherwise} . 
\end{array} 
\right. 
\end{align*}
For $j, k \in \mathbb{Z}$, let $\psi_{j,k}(x)\coloneqq 2^{j/2}\psi(2^jx-k), \forall x \in \mathbb{R}$. Then  $\{\psi_{j,k}\}_{j, k \in \mathbb{Z}}$ is an orthonormal basis for $\mathcal{L}^2(\mathbb{R})$.
 \item (cf. \cite{CHRISTENSEN}) For $a, b  >0$, define  
\begin{align*}
	T_a:\mathcal{L}^2(\mathbb{R}) \ni f \mapsto T_af \in  \mathcal{L}^2(\mathbb{R}), \quad T_af:\mathbb{R} \mapsto (T_af) (x)\coloneqq f(x-a) \in \mathbb{C}
\end{align*}
and 
\begin{align*}
	E_b:\mathcal{L}^2(\mathbb{R}) \ni f \mapsto E_bf \in  \mathcal{L}^2(\mathbb{R}), \quad E_bf:\mathbb{R} \mapsto (E_bf) (x)\coloneqq e^{2 \pi i bx} f(x)\in \mathbb{C}.
\end{align*}
Let $g:\mathbb{R}\to \mathbb{C}$ be a continuous function with compact support. Then, for any $a,b>0$,   $  \{E_{mb}T_{na}g\}_{n,m \in \mathbb{Z}}$ is not an orthonormal    basis for $\mathcal{L}^2(\mathbb{R})$. 
\end{enumerate}	
\end{example}
Following theorem shows that given an orthonormal set, we can check whether it is an orthonormal basis by checking several equivalent conditions rather appealing to  Definition \ref{ONBDEFINITIONOLE}, which is difficult in many cases.
\begin{theorem}(cf. \cite{CHRISTENSEN})\label{CHARORTHONORMALINTRO}
Let $ \{\tau_n\}_{n}$ be an orthonormal set in   $\mathcal{H}$. The following are equivalent.
\begin{enumerate}[label=(\roman*)]
\item $ \{\tau_n\}_{n}$ is an orthonormal basis for  $\mathcal{H}$.
\item (\textbf{Fourier expansion}) $h= \sum_{n=1}^\infty\langle h, \tau_n\rangle \tau_n, \forall h \in \mathcal{H}$.
\item (\textbf{Parseval identity for the inner product}) $\langle h, g\rangle= \sum_{n=1}^\infty\langle h, \tau_n\rangle \langle  \tau_n, g\rangle , \forall h, g \in \mathcal{H}$.
\item (\textbf{Parseval identity for the norm}) $\|h\|^2=\sum_{n=1}^\infty|\langle h, \tau_n\rangle |^2, h \in \mathcal{H}.$
\item $\overline{\operatorname{span}} \{\tau_n\}_{n}=\mathcal{H}$.
\item If $h \in \mathcal{H}$ is such that $\langle h, \tau_n\rangle = 0, \forall n \in \mathbb{N}$, then $h=0$.
\end{enumerate}	
\end{theorem}
\begin{theorem}(cf. \cite{CHRISTENSEN})
A Hilbert 	space is separable if and only if it has a countable orthonormal basis.
\end{theorem}
Hilbert spaces have the remarkable property that every bounded linear functional from the space to the scalar field is given by the inner product with a unique element in the space.
\begin{theorem}(cf. \cite{LIMAYE}) (\textbf{Riesz representation theorem}) 
Let $f:	\mathcal{H}\ni \to \mathbb{K}$ be a bounded linear functional. Then there exists a unique $\tau_f\in \mathcal{H}$ such that 
\begin{align*}
f(h)=\langle h, \tau_f \rangle , \quad \forall h  \in \mathcal{H}\quad \text{ and } \quad \|f\|=\|\tau_f\|.
\end{align*}
\end{theorem}
Riesz representation theorem opens the door to the following definition. 
\begin{definition}(cf. \cite{LIMAYE})
	Let $T:\mathcal{H} \to \mathcal{H}_0$ be a bounded linear operator. The  unique bounded linear operator  $T^*:\mathcal{H}_0 \to \mathcal{H}$ such that 
	\begin{align*}
	\langle T h, h_0 \rangle =\langle h, T^*h_0\rangle , \quad \forall h  \in \mathcal{H}, ~\forall h_0 \in \mathcal{H}_0
	\end{align*}
is called as the \textbf{adjoint} of $T$.
\end{definition}
Hilbert spaces are studied along with various kinds of operators. These are defined as follows. 
\begin{definition}(cf. \cite{LIMAYE})
	Let $T:\mathcal{H} \to \mathcal{H}_0$ be a bounded linear operator. The operator $T$ is said to be 
	\begin{enumerate}[label=(\roman*)]
		\item \textbf{invertible} if there exists a bounded linear operator $S:\mathcal{H}_0 \to \mathcal{H}$ such that $ST=I_\mathcal{H}$ and $TS=I_{\mathcal{H}_0}$.
		\item  \textbf{isometry} if $\|Th\|=\|h\|, $ $ \forall h \in \mathcal{H}$.
		\item \textbf{unitary}  if $TT^*=I_{\mathcal{H}_0}$, $T^*T=I_\mathcal{H}$.
\end{enumerate}
\end{definition}
\begin{definition}(cf. \cite{LIMAYE})
	Let $T:\mathcal{H} \to \mathcal{H}$ be a bounded linear operator. 
	\begin{enumerate}[label=(\roman*)]
		\item Operator $T$ is said to be a \textbf{normal} operator  if $TT^*=T^*T$.
		\item Operator $T$ is said to be a \textbf{projection} if $T^2=T=T^*$.
		\item Operator $T$ is said to be a \textbf{self-adjoint} operator if $T=T^*$.
		\item Operator $T$ is said to be a \textbf{positive} operator if $T=T^*$ and $\langle Th, h \rangle \geq 0$, $ \forall h \in \mathcal{H}$.
	\end{enumerate}
\end{definition}
\begin{theorem}(cf. \cite{LIMAYE})
	Let $\mathcal{H}$ be a separable Hilbert space. 
	\begin{enumerate}[label=(\roman*)]
		\item If $\mathcal{H}$ is finite dimensional,  then $\mathcal{H}$ is isometrically  isomorphic to $\mathbb{K}^n$, for some $n$.
		\item If $\mathcal{H}$ is infinite dimensional, then    $\mathcal{H}$ is isometrically  isomorphic to $\ell^2(\mathbb{N})$.
	\end{enumerate} 
\end{theorem}
Orthonormal bases have the nice property that given a single orthonormal basis, we can generate all of them just by acting unitary operators.
\begin{theorem}(cf. \cite{CHRISTENSEN})
Let $ \{\tau_n\}_{n}$ be an orthonormal basis for  $\mathcal{H}$. Then the set of all orthonormal bases  for  $\mathcal{H}$ are precisely the families $\{U\tau_n\}_{n}$, where  $U\in \mathcal{B}(\mathcal{H})$ is unitary. 
\end{theorem}
	
Hilbert spaces have another nice property that closed subspaces decompose the original space.
\begin{theorem}(cf. \cite{LIMAYE}) (\textbf{Orthogonal complement theorem})
If $\mathcal{W}$ is a closed subspace of  	$\mathcal{H}$, then $\mathcal{H}=\mathcal{W}\oplus \mathcal{W}^\perp$, where $\mathcal{W}^\perp$ is the orthogonal complement of $\mathcal{W}$ in $\mathcal{H}$.
\end{theorem}
First level of generalization of orthonormal basis is that of Riesz basis. These are defined as follows.
\begin{definition}(cf. \cite{CHRISTENSEN})\label{RIESZBASISDEFINITION}
A	collection $\{\omega_n\}_{n}$ in  $\mathcal{H}$ is called	a \textbf{Riesz basis} for $\mathcal{H}$ if there exist an orthonormal basis  $ \{\tau_n\}_{n}$ for $\mathcal{H}$ and an invertible $T\in \mathcal{B}(\mathcal{H})$ such that $\omega_n=T\tau_n, \forall n$.
\end{definition}
As written by \cite{SIMONBOOK}, the origin of the term "Riesz basis" is unknown.
By taking $T$ as  the identity operator, we easily see that every orthonormal basis is a Riesz basis.
\begin{example}
\begin{enumerate}[label=(\roman*)]
\item Let $\{\lambda_n\}_{n=1}^\infty$ be a sequence of scalars such that there exist $a,b>0$ with $a\leq |\lambda_n|\leq b, \forall n \in \mathbb{N}$. Then $\{\lambda_ne_n\}_{n=1}^\infty$ is a Riesz basis  for $\ell^2(\mathbb{N})$, since it is image of the standard orthonormal basis $\{e_n\}_{n=1}^\infty$ under the invertible operator $T:\ell^2(\mathbb{N}) \ni \{x_n\}_{n=1}^\infty \mapsto \{\lambda_nx_n\}_{n=1}^\infty \in \ell^2(\mathbb{N})$. We note further that if $|\lambda_n|\neq1,$ for at least one $n$,  then $\{\lambda_ne_n\}_{n=1}^\infty$ is not an orthonormal  basis for $\ell^2(\mathbb{N})$.
\item (\cite{KADEC}) (\textbf{Kadec 1/4 theorem}) Let $ \{\lambda_n\}_{n\in \mathbb{Z}}$ be a sequence of reals  such that 
\begin{align*}
\sup_{n \in \mathbb{N}}\left|\lambda_n-n\right|< \frac{1}{4},\quad \forall n\in \mathbb{Z}.
\end{align*}
Define $f_n:\left(-\pi, \pi \right) \ni x \mapsto e^{i\lambda_nx} \in \mathbb{C}$, $\forall n\in \mathbb{Z}$. Then 
$  \{f_n\}_{n\in \mathbb{Z}}$ is a Riesz basis for $\mathcal{L}^2(-\pi, \pi)$.  It was shown that 1/4 is the optimal constant (\cite{LEVINSON}, cf. \cite{CHRISTENSENBULLETIN}).
\item (\cite{CASSAZAEVERYSUM}) (\textbf{Kalton-Casazza theorem})
 A linear combination of two orthonormal bases is a Riesz basis.
 \item (cf. \cite{CHRISTENSEN}) For $a, b  >0$, define  
 \begin{align*}
 	T_a:\mathcal{L}^2(\mathbb{R}) \ni f \mapsto T_af \in  \mathcal{L}^2(\mathbb{R}), \quad T_af:\mathbb{R} \mapsto (T_af) (x)\coloneqq f(x-a) \in \mathbb{C}
 \end{align*}
 and 
 \begin{align*}
 	E_b:\mathcal{L}^2(\mathbb{R}) \ni f \mapsto E_bf \in  \mathcal{L}^2(\mathbb{R}), \quad E_bf:\mathbb{R} \mapsto (E_bf) (x)\coloneqq e^{2 \pi i bx} f(x)\in \mathbb{C}.
 \end{align*}
 Let $g:\mathbb{R}\to \mathbb{C}$ be a continuous function with compact support. Then, for any $a,b>0$,   $  \{E_{mb}T_{na}g\}_{n,m \in \mathbb{Z}}$ is not a Riesz   basis for $\mathcal{L}^2(\mathbb{R})$. 
\end{enumerate}	
\end{example}
\begin{remark}
Let  $ \{\tau_n\}_{n}$ be an orthonormal basis for  $\mathcal{H}$. Since an invertible map preserves the cardinality, it follows that for each    $n\in \mathbb{N}$ , the set  $\{\tau_j\}_{j=1}^n$ can not be a Riesz basis for  $\mathcal{H}$.

\end{remark}
Like orthonormal basis, Riesz basis will also give a series representation of every element in a Hilbert space.
\begin{theorem}(cf. \cite{CHRISTENSEN}) \label{RIESZISAFRAME}
Let $ \{\tau_n\}_{n}$ be a Riesz basis for $\mathcal{H}$. 
\begin{enumerate}[label=(\roman*)]
\item There exists a unique collection $\{\omega_n\}_{n}$ in $\mathcal{H}$ such that 
\begin{align}\label{RBEXPANSION}
h=\sum_{n=1}^\infty\langle h, \omega_n\rangle \tau_n, \quad \forall h \in \mathcal{H}.
\end{align}
Moreover, $\{\omega_n\}_{n}$ is a Riesz basis for $\mathcal{H}$ and the series in Eq. (\ref{RBEXPANSION}) converges unconditionally for all $h \in \mathcal{H}$.
\item There exist $a,b>0$ such that $a \|h\|^2\leq \sum_{n=1}^\infty|\langle h, \tau_n\rangle |^2 \leq b\|h\|^2, \forall h \in \mathcal{H}.$
\end{enumerate}	
\end{theorem}
Next result says that there is a characterization of Riesz basis which is free from orthonormal basis. It also gives a tool to check whether a collection is a Riesz basis  for a given Hilbert space. To state the result, we need two definitions.
\begin{definition}(cf. \cite{CHRISTENSEN})
A sequence $ \{\tau_n\}_{n}$ in a Hilbert space $\mathcal{H}$	is said to be \textbf{complete} if $\overline{\text{span}}_{n\in \mathbb{N}}\{\tau_n\}=\mathcal{H}$.
\end{definition}
\begin{definition}(cf. \cite{CHRISTENSEN})
A sequence $ \{\omega_n\}_{n}$ in a Hilbert space $\mathcal{H}$	is said to be \textbf{biorthogonal} to a sequence $ \{\tau_n\}_{n}$ in  $\mathcal{H}$ if $\langle  \omega_n, \tau_m\rangle =\delta_{n,m}$ for all $n,m$.
\end{definition}
\begin{theorem}(cf. \cite{CHRISTENSEN, HEILBOOK, STOEVACHARAFOURIER}) \label{RIESZBASISTHM}
For a sequence $ \{\tau_n\}_{n}$ in  $\mathcal{H}$, the following are equivalent.
\begin{enumerate}[label=(\roman*)]
\item $ \{\tau_n\}_{n}$ is a Riesz basis for $\mathcal{H}$.
\item $\overline{\operatorname{span}} \{\tau_n\}_{n}=\mathcal{H}$ and there exist $a,b>0$ such that for every finite subset $\mathbb{S}$ of $\mathbb{N}$, 
\begin{align}\label{USEFULRESULT}
a\sum_{n \in \mathbb{S}}|c_n|^2\leq \left\|\sum_{n \in \mathbb{S}}c_n\tau_n\right\|^2 \leq  b \sum_{n \in \mathbb{S}}|c_n|^2, \quad\forall c_n \in \mathbb{K}.
\end{align}
\item $ \{\tau_n\}_{n}$ is  complete in    $\mathcal{H}$ and the  operator given  by the \textbf{infinite  Gram matrix} $[\langle \tau_m, \tau_n\rangle ]_{1\leq n, m <\infty}$ defined by 
\begin{align*}
\ell^2(\mathbb{N})\ni \{c_m\}_{m} \mapsto \left\{\sum_{m=1}^\infty \langle \tau_n, \tau_m \rangle c_m\right\}_n\in \ell^2(\mathbb{N})
\end{align*}
is a bounded invertible operator on $\ell^2(\mathbb{N})$.
\item $\{\tau_n\}_{n}$ is a bounded unconditional Schauder basis for $\mathcal{H}$.
\item $\{\tau_n\}_{n}$ is a Schauder basis for $\mathcal{H}$ such that $\sum_{n=1}^{\infty}c_n\tau_n$ converges in $\mathcal{H}$ if and only if $\sum_{n=1}^{\infty}|c_n|^2<\infty$.
\end{enumerate}	
\end{theorem}
\begin{remark}\label{COUNTERRIESZ}
Using Inequality \ref{USEFULRESULT}, we can show that certain collection of vectors  is  not a Riesz basis. As an illustration, let $\{e_n\}_{n=1}^\infty$ be the standard orthonormal basis for $\ell^2(\mathbb{N})$. We claim that $\{e_1\}\cup\{e_n\}_{n=1}^\infty$ is not a Riesz basis for $\ell^2(\mathbb{N})$. Suppose the claim fails, then we get an $a>0$ such that first inequality in Inequality \ref{USEFULRESULT} holds. By taking $1$ and $-1$ we see that 
\begin{align*}
a(|1|^2+|-1|^2)\leq \|1\cdot e_1+(-1)\cdot e_1\|^2=0 ~\Rightarrow a=0,
\end{align*}
which is a contradiction. Hence $\{e_1\}\cup\{e_n\}_{n=1}^\infty$ can  not be a Riesz basis for $\ell^2(\mathbb{N})$.
\end{remark}
Theorem \ref{RIESZBASISTHM} leads to the following definition. 
\begin{definition}(cf. \cite{CHRISTENSEN})
A sequence $ \{\omega_n\}_{n}$ in a Hilbert space $\mathcal{H}$	is said to be a  \textbf{Riesz sequence}  for   $\mathcal{H}$	if there exist $a,b>0$ such that for every finite subset $\mathbb{S}$ of $\mathbb{N}$, 
\begin{align*}
	a\sum_{n \in \mathbb{S}}|c_n|^2\leq \left\|\sum_{n \in \mathbb{S}}c_n\tau_n\right\|^2 \leq  b \sum_{n \in \mathbb{S}}|c_n|^2, \quad\forall c_n \in \mathbb{K}.
\end{align*}
\end{definition}
Theorem \ref{RIESZBASISTHM} says that every Riesz basis is a Riesz sequence. It is easy to see that a Riesz sequence need not be a Riesz basis. Following theorem gives another characterization of Riesz basis which is also free from orthonormal basis. It also helps to check whether a collection is a Riesz basis.
\begin{theorem}(cf. \cite{GOHBERGKREIN})\label{LORCH} (\textbf{Kothe-Lorch theorem})
A sequence 	$\{\tau_n\}_{n}$ is a Riesz basis for $\mathcal{H}$ if and only if the following three conditions hold.
\begin{enumerate}[label=(\roman*)]
\item  $\{\tau_n\}_{n}$ is an unconditional  Schauder basis for $\mathcal{H}$.
 \item  $0<\operatorname{inf}_{n \in \mathbb{N}}\|\tau_n\|\leq \operatorname{sup}\limits_{n \in \mathbb{N}}\|\tau_n\|<\infty. $
\end{enumerate}	
\end{theorem}
\begin{remark}
 Since the collection $\{e_1\}\cup\{e_n\}_{n=1}^\infty$ in Remark \ref{COUNTERRIESZ}  is not a Schauder basis for $\ell^2(\mathbb{N})$,  using Theorem \ref{LORCH}, we again conclude that it  is not a  Riesz basis. However, the collection $\{e_1\}\cup\{e_n\}_{n=1}^\infty$ satisfies conditions (ii) and (iii) in Theorem \ref{LORCH}.
\end{remark}

It is clear that whenever we perturb an orthonormal basis, we may not get an orthonormal basis. However, it is a classical theorem of Paley and Wiener which says whenever we perturb an orthonormal basis, we get a Riesz basis.
\begin{theorem}(\cite{PALEYWIENER}, cf. \citealp{YOUNG}) \label{PALEYWIENERTHEOREM} (\textbf{Paley-Wiener theorem}) 
Let $  \{\tau_n\}_{n}$ be an orthonormal basis  for  $\mathcal{H} $. If  $ \{\omega_n\}_{n}$  in $\mathcal{H} $ is such that  there exists $ 0< \alpha <1 $  and for every $m=1,2,\dots, $
\begin{align*}
\left\|\sum\limits_{n=1}^mc_n(\tau_n-\omega_n) \right\|\leq \alpha \left(\sum\limits_{n=1}^m|c_n|^2\right)^\frac{1}{2},\quad \forall c_n \in \mathbb{K}, 
\end{align*}
then $ \{\omega_n\}_{n}$ is  a Riesz basis for $\mathcal{H} $.
\end{theorem}
Next level of generalization of Riesz basis for Hilbert spaces is the notion of frame.
\begin{definition}(\cite{DUFFIN1})\label{OLE}
A collection $  \{\tau_n\}_{n}$ in  a Hilbert space $ \mathcal{H}$ is said to be a \textbf{frame} for $\mathcal{H}$ if there exist $ a, b >0$ such that
\begin{equation}\label{SEQUENTIALEQUATION1}
\text{ (\textbf{Frame inequalities}) } \quad a\|h\|^2\leq\sum_{n=1}^\infty|\langle h, \tau_n \rangle|^2 \leq b\|h\|^2  ,\quad \forall h \in \mathcal{H}.
\end{equation}
Constants $ a$ and $ b$ are called as \textbf{lower frame bound} and \textbf{upper frame bound}, respectively. Supremum (resp. infimum) of the set of all lower (resp. upper) frame bounds is called \textbf{optimal lower frame bound} (resp. \textbf{optimal upper frame bound}). If the optimal frame bounds are equal, then the frame is called as  \textbf{tight frame}. A tight frame whose optimal frame bound is one is termed as  \textbf{Parseval frame}.
\end{definition}
As recorded by \cite{CHEBIRA}, and \cite{HEIL},  the reason for using the term ``frame" is unknown.

Note that in Definition \ref{OLE} we indexed the frame by natural numbers. Since the convergence of series in Definition \ref{OLE} is unconditional, any rearrangement of a frame is again a frame. We also note that Definition \ref{OLE} can be formulated for arbitrary indexing set $\mathbb{J}$. In this case, by the convergence of the series we mean the convergence of the net obtained by the set inclusion, on the collections of all finite subsets of $\mathbb{J}$.

From (ii) in Theorem \ref{RIESZISAFRAME} we can conclude that every Riesz basis is a frame. On the other hand, every frame can be written as a finite union of Riesz sequences (which is called as \textbf{Feichtinger conjecture}  (\cite{CASAZZACHRISTENSENLINDNERVERSHYNIN})) and is known as \textbf{Marcus-Spielman-Srivastava Theorem} (cf. \cite{CASAZZATREMAIN2016, CASAZZAEDIDIN}). We now give various examples of frames.
\begin{example}\label{FRAMEEXAMPLES}
\begin{enumerate}[label=(\roman*)] 
\item Let $\{e_n\}_{n=1}^\infty$ be an orthonormal basis for $\mathcal{H}$ and  let  $m \in \mathbb{N}$. Then the collection $\{e_1, \dots, e_m\}\cup\{e_n\}_{n=1}^\infty$ is a frame for $\mathcal{H}$ with bounds 1 and 2. In fact, for all $h \in 
\mathcal{H}$, 
\begin{align*}
1 \cdot\|h\|^2=\sum_{n=1}^\infty|\langle h, e_n \rangle|^2\leq\sum_{k=1}^m|\langle h, e_k \rangle|^2 + \sum_{n=1}^\infty|\langle h, e_n \rangle|^2 \leq 2\sum_{n=1}^\infty|\langle h, e_n \rangle|^2=2\|h\|^2.  
\end{align*}
In particular, the collection in Remark \ref{COUNTERRIESZ} is a frame for $\ell^2(\mathbb{N})$. 
\item (cf. \cite{CHRISTENSEN}) (\textbf{Harmonic frame}) Let $m, n \in \mathbb{N}$, $n\leq m$ and $\omega_1,\dots, \omega_m$ be the distinct $m^\text{th}$ roots of unity. Define 
\begin{align*}
\eta_k\coloneqq\frac{1}{\sqrt{m}}(\omega_1^k, \dots, \omega_n^k), \quad  1\leq k \leq m.
\end{align*}
Then $\{\eta_k\}_{k=1}^m$ is a Parseval frame for $\mathbb{C}^n$.
\item (cf. \cite{CHEBIRA}, \cite{SHOR}) (\textbf{Mercedes-Benz frame} or \textbf{Peres-Wooters states})  $\{(0,1), (-\sqrt{3}/2, -1/2),(\sqrt{3}/2, -1/2) \}$ is a tight frame for $\mathbb{R}^2$ with bound $3/2$.
\item (cf. \cite{HANKORNELSONLARSON}) For $n\geq3$, the collection $\{(\cos(2 \pi j/n), \sin (2 \pi j/n)\}_{j=0}^{n-1}$ is a tight frame for $\mathbb{R}^2$. It has to be noted that we can not take $n=2$ because $\{(1,0), (-1,0)\}$ is not a frame for $\mathbb{R}^2$.
\item (cf. \cite{CHRISTENSEN}) (\textbf{Gabor frame} or \textbf{Weyl-Heisenberg frame}) Let $g$ be the Gaussian defined by $g:\mathbb{R}\ni x \mapsto g(x)\coloneqq e^{-x^2} \in \mathbb{R}$ and let $a,b>0$. Define $f_{n,m}:\mathbb{R}\ni x \mapsto e^{2\pi imbx}g(x-na)\in \mathbb{R},$ $\forall n, m\in \mathbb{Z}$. Then 
$  \{f_{n,m}\}_{n,m \in \mathbb{Z}}$ is a (Gabor) frame for $\mathcal{L}^2(\mathbb{R})$ if and only if $ab<1$. Moreover, if $  \{f_{n,m}\}_{n,m \in \mathbb{Z}}$ is a frame for $\mathcal{L}^2(\mathbb{R})$, then $ab=1$ if and only if $  \{f_{n,m}\}_{n,m \in \mathbb{Z}}$ is a Riesz basis for $\mathcal{L}^2(\mathbb{R})$.
\item (\cite{DUFFIN1}) Let $ \{\lambda_n\}_{n\in \mathbb{Z}}$ be a sequence of scalars  such that there are constants $d, L,\delta>0$ satisfying
\begin{align*}
\left|\lambda_n-\frac{n}{d}\right|\leq L,\quad \forall n\in \mathbb{Z}\quad  \text{ and } \quad |\lambda_n-\lambda_m|\geq \delta,\quad \forall n, m\in \mathbb{Z}, n\neq m.
\end{align*}
Let $0<r<d \pi$ and define $f_n:\left(-r, r \right) \ni x \mapsto e^{i\lambda_nx} \in \mathbb{C}$, $\forall n\in \mathbb{Z}$. Then 
$  \{f_n\}_{n\in \mathbb{Z}}$ is a frame for $\mathcal{L}^2(-r, r)$.
\item  (cf. \cite{FEICHTINGERSTROHMERBOOK2}) For $a, b  >0$, define  
\begin{align*}
	T_a:\mathcal{L}^2(\mathbb{R}) \ni f \mapsto T_af \in  \mathcal{L}^2(\mathbb{R}), \quad T_af:\mathbb{R} \mapsto (T_af) (x)\coloneqq f(x-a) \in \mathbb{C}
\end{align*}
and 
\begin{align*}
	E_b:\mathcal{L}^2(\mathbb{R}) \ni f \mapsto E_bf \in  \mathcal{L}^2(\mathbb{R}), \quad E_bf:\mathbb{R} \mapsto (E_bf) (x)\coloneqq e^{2 \pi i bx} f(x)\in \mathbb{C}.
\end{align*}
For $c>0$,  let $ \chi_{[0,c]}$ be the characteristic function on $[0,c]$. Then for $a\leq c \leq 1$, $  \{E_mT_{na} \chi_{[0,c]}\}_{n,m \in \mathbb{Z}}$ is a (Gabor)  frame for $\mathcal{L}^2(\mathbb{R})$   (this is a particular case of the celebrated $abc$-problem for Gabor systems (\cite{DAISUN})).
\item (\cite{JANSSENSTROHMER2002}) Let $T_a$ and $E_b$ be the operators in (vii). Let $g(x)\coloneqq \cosh (\pi x) =\frac{2}{e^{\pi x}+e^{-\pi x}}$, $\forall x \in \mathbb{R}$. Then,  for $ab<1$, $  \{E_{mb}T_{na}g\}_{n,m \in \mathbb{Z}}$ is a (Gabor)  frame for $\mathcal{L}^2(\mathbb{R})$.
\item  (\cite{JANSSEN2003}) Let $T_a$ and $E_b$ be the operators in (vii). Let $g(x)\coloneqq e^{-|x|}$, $\forall x \in \mathbb{R}$. Then,  for $ab<1$, $  \{E_{mb}T_{na}g\}_{n,m \in \mathbb{Z}}$ is a (Gabor)  frame for $\mathcal{L}^2(\mathbb{R})$.
\item (\cite{JANSSEN1996}) Let $T_a$ and $E_b$ be the operators in (vii). Let $g(x)\coloneqq e^{-x} \chi_{[0,\infty)}(x)$, $\forall x \in \mathbb{R}$. Then  $  \{E_{mb}T_{na}g\}_{n,m \in \mathbb{Z}}$ is a (Gabor)  frame for $\mathcal{L}^2(\mathbb{R})$ if and only if $ab\leq 1$. 
\item (\cite{CASSAZAEVERYSUM}) (\textbf{Kalton-Casazza theorem})
Every frame is a  sum of three orthonormal bases.
\item (cf. \cite{CHRISTENSEN}) (\textbf{Wavelet frame}) Let $0<b<0.0084$. Define the \textbf{Mexican hat} function 
\begin{align*}
	\psi(x)\coloneqq \frac{2}{\sqrt{3}}\pi^{\frac{-1}{4}}(1-x^2)e^{\frac{-x^2}{2}}, \quad \forall x \in \mathbb{R}.
\end{align*}
For $j,k \in \mathbb{Z}$, define 
\begin{align*}
	\psi_{j,k}(x)\coloneqq \psi(2^jx-kb), \quad \forall x \in \mathbb{R}.
\end{align*}
Then $\{2^\frac{j}{2}\psi_{j,k}\}_{j,k\in \mathbb{Z}}$ is a (wavelet)  frame for $\mathcal{L}^2(\mathbb{R})$.
\item (\cite{STROHMERHEATHCODING, BENEDETTOKOLESAR})  (\textbf{Grassmannian frame})  The $n$-equally spaced lines in $\mathbb{R}^2$, namely  $\{(\cos( \pi j/n),  \sin ( \pi j/n)\}_{j=0}^{n-1}$ is a (Grassmannian) frame for $\mathbb{R}^2$.
\item (\cite{BENEDETTOFICKUS}) (\textbf{Group frame}) Vertices of each of (five) \textbf{Platonic solids} is a tight frame for $\mathbb{R}^3$.
\item (cf. \cite{WALDRONBOOK}) (\textbf{Equiangular  frame}) For each $d\in \mathbb{N}$,  $d+1$ vertices of the regular simplex in $\mathbb{R}^d$ is an (equiangular) frame for $\mathbb{R}^d$.
\end{enumerate}	
\end{example}
We now give various examples which are  not frames.
\begin{example}
\begin{enumerate}[label=(\roman*)] 
\item (cf. \cite{CHRISTENSEN}) If $\{\tau_n\}_{n=1}^\infty$ is an orthonormal basis for $\mathcal{H}$, then $\{\tau_n+\tau_{n+1}\}_{n=1}^\infty$ is  not a frame for $\mathcal{H}$.
\item (cf. \cite{BACHMANNARICIBECKENSTEIN}) If $\{\tau_n\}_{n=1}^\infty$ is an orthonormal basis for $\mathcal{H}$, then $\{\frac{\tau_n}{n}\}_{n=1}^\infty$ is  not a frame for $\mathcal{H}$.
\item (cf. \cite{CHRISTENSEN}) If $\{\tau_n\}_{n=-\infty}^\infty$ is a Riesz   basis for $\mathcal{H}$, then $\{\tau_n+\tau_{n+1}\}_{n=-\infty}^\infty$ is  not a frame for $\mathcal{H}$.
\item (cf. \cite{CHRISTENSEN}) For $n\in \mathbb{Z}$, define $T_n:\mathcal{L}^2(\mathbb{R}) \ni f \mapsto  T_nf \in \mathcal{L}^2(\mathbb{R})$, $ T_nf:\mathbb{R} \ni x \mapsto f(x-n) \in \mathbb{C}$. Then for any $\phi \in \mathcal{L}^2(\mathbb{R})$, $  \{T_n\phi\}_{n \in \mathbb{Z}}$ is not a frame for $ \mathcal{L}^2(\mathbb{R})$.
\item   (\cite{CHRISTENSENHASANNASABRASHIDI, ALDROUBI2017})  Let $\mathcal{H} $ be an infinite dimensional Hilbert space and $T: \mathcal{H}\to \mathcal{H}$ be a bounded linear operator which is unitary or compact. Then for every $\tau \in \mathcal{H}$,  $\{T^n\tau\}_{n=0}^\infty$ is not a frame for $\mathcal{H}$ (this is a particular case of dynamical sampling (\cite{ALDROUBICABRELLIMOLTERTANG, ALDROUBICABRELLICAKMAK})).
\item (cf. \cite{FEICHTINGERSTROHMERBOOK2}) Let $T_a$ and $E_b$ be the operators in (vii) of Example \ref{FRAMEEXAMPLES}. For $c>0$,  let $ \chi_{[0,c]}$ be the characteristic function on $[0,c]$. Then for $c\leq a $ or $a>1$,  $  \{E_mT_{na} \chi_{[0,c]}\}_{n,m \in \mathbb{Z}}$ is not a   frame for $\mathcal{L}^2(\mathbb{R})$.
\end{enumerate}
\end{example}
There is a simple criterion to check for frames in finite dimensional Hilbert spaces. This reads as follows.
\begin{theorem}(cf. \cite{HANKORNELSONLARSON})\label{FINITEDIMESIONALCHARAC}
A finite set of vectors for a finite dimensional Hilbert space is a frame if and only if it spans the space.	
\end{theorem}
\begin{remark}
\begin{enumerate}[label=(\roman*)]
\item Theorem \ref{FINITEDIMESIONALCHARAC} gives a very useful algebraic criterion for checking whether a finite set of vectors is a frame for a finite dimensional space rather verifying the analytic condition  (frame inequality) which is harder in many cases.
\item Theorem \ref{FINITEDIMESIONALCHARAC} does not tell that a frame for a finite dimensional Hilbert space is finite. A finite dimensional Hilbert space can have a frame with infinitely many elements. For example, $\{\frac{1}{n}\}_{n=1}^\infty$ is a tight frame for $\mathbb{C}$ (as a vector space over itself), because
\begin{align*}
\sum_{n=1}^\infty\left|\left\langle h, \frac{1}{n} \right\rangle\right|^2=\sum_{n=1}^\infty\left|\frac{h}{n}\right|^2=\frac{\pi^2}{6}|h|^2,\quad \forall h \in \mathbb{C}.
\end{align*}
\item Suppose $\operatorname{dim}(\mathcal{H})=n$. From Theorem \ref{FINITEDIMESIONALCHARAC} we see that a spanning set having at least $n+1$ elements is a frame for $\mathcal{H}$ but not a Riesz basis for $\mathcal{H}$.
\item A spanning set need not be a frame. For instance, $\{n\}_{n=1}^\infty$ spans $\mathbb{C}$ but $\sum_{n=1}^\infty|\left\langle 1, n \right\rangle|^2$ $=\sum_{n=1}^\infty n^2=\infty.$ Hence $\{n\}_{n=1}^\infty$ is not a frame for  $\mathbb{C}$.
\end{enumerate}	
\end{remark}
Following theorem is the most  important result in the theory of frames.
\begin{theorem}(\cite{DUFFIN1}, \cite{CHRISTENSEN}, \cite{HANLARSON})\label{MOSTIMPORTANT}
Let $  \{\tau_n\}_{n}$  be a frame for $\mathcal{H}$ with bounds $a$ and $b$. Then	
\begin{enumerate}[label=(\roman*)]
\item $\overline{\operatorname{span}} \{\tau_n\}_{n}=\mathcal{H}$.
\item The map $\theta_\tau:\mathcal{H}\ni h \mapsto \theta_\tau h \coloneqq \{\langle h, \tau_n\rangle\}_{n}\in \ell^2(\mathbb{N}) $ is a well-defined bounded linear operator. Further, $\sqrt{a}\|h\|\leq \|\theta_\tau h\|\leq \sqrt{b}\|h\|, \forall h \in \mathcal{H}$. In particular, $\theta_\tau$ is injective and its range is closed.
\item The map $ S_\tau:\mathcal{H} \ni h \mapsto S_\tau h\coloneqq \sum_{n=1}^\infty\langle h, \tau_n\rangle \tau_n \in \mathcal{H}$ is a  well-defined bounded linear positive invertible operator. Further, 
\begin{align*}
a\|h\|^2\leq \langle S_\tau h, h\rangle\leq b\|h\|^2, \quad \forall h \in \mathcal{H}, \quad  \quad a\|h\|\leq \|S_\tau h\|\leq b\|h\|, \quad \forall h \in \mathcal{H}.
\end{align*}
\item (\textbf{General Fourier expansion} or \textbf{frame decomposition})
\begin{align}\label{GFE}
h=\sum_{n=1}^\infty\langle h, \tau_n\rangle S_\tau^{-1}\tau_n=\sum_{n=1}^\infty\langle h, S_\tau^{-1}\tau_n\rangle \tau_n, \quad \forall h \in \mathcal{H}.
\end{align} 
\item $\theta_\tau^* ( \{a_n\}_{n})=\sum_{n}a_n \tau_n, \forall \{a_n\}_{n} \in \ell^2(\mathbb{N})$. In particular, $\theta_\tau^*e_n=\tau_n,  \forall n \in \mathbb{N}$.
\item  $ S_\tau$ factors as $S_\tau=\theta_\tau^*\theta_\tau$.
\item $\theta_\tau^*$ is surjective.
\item $\|S_\tau^{-1}\|^{-1}$ is the optimal lower frame bound and  $\|S_\tau\|=\|\theta_\tau\|^2$ is the optimal upper frame bound.
\item $P_\tau\coloneqq\theta_\tau S_\tau^{-1}\theta_\tau^*$ is an orthogonal  projection onto $\theta_\tau(\mathcal{H}).$
\item $  \{\tau_n\}_{n}$ is Parseval if and only if $\theta_\tau$ is an isometry if and only if $\theta_\tau\theta_\tau^*$ is a projection.
\item $ \{S_\tau^{-1}\tau_n\}_{n}$  is a frame for $\mathcal{H}$ with bounds $b^{-1}$ and $a^{-1}$.
\item $ \{S_\tau^{-1/2}\tau_n\}_{n}$  is a Parseval frame for $\mathcal{H}$. 
\item (\textbf{Best approximation}) If $ h \in \mathcal{H}$ has representation  $ h=\sum_{n=1}^\infty c_n\tau_n,$ for some scalar sequence  $ \{c_n\}_{n}\in \ell^2(\mathbb{N})$,  then 
$$ \sum\limits_{n=1}^\infty |c_n|^2 =\sum\limits_{n=1}^\infty |\langle h, S_\tau^{-1}\tau_n\rangle|^2+\sum\limits_{n=1}^\infty | c_n-\langle h, S_\tau^{-1}\tau_n\rangle|^2. $$
\end{enumerate}	
\end{theorem}
Theorem \ref{MOSTIMPORTANT} says several things. First, it says that every vector in the Hilbert
space admits an expansion, called as general Fourier expansion, similar to
Fourier
expansion coming from an orthonormal basis for a Hilbert space. Second, it
says that coefficients in the expansion of a vector need not be unique. This is
particularly important in applications, since loss in the information of a
vector is less if some of the coefficients are missing. Third, given a frame,
it naturally generates other frames. Fourth, a frame gives a bounded linear
injective operator from the less known inner product on the Hilbert space $\mathcal{H}$ to
the well known standard inner product on the standard separable Hilbert space $\ell^2(\mathbb{N})$.
Frame inequality now clearly says that there is a comparison of norms between $\mathcal{H}$ and
the standard Hilbert space $\ell^2(\mathbb{N})$. Fifth, a frame  embeds $\mathcal{H}$ in $\ell^2(\mathbb{N})$
through the bounded linear  operator $\theta_\tau$. Sixth,  whenever a Hilbert 
space admits a frame it becomes an image of a surjective operator $\theta_\tau^*$ from
the $\ell^2(\mathbb{N})$ to it.
An easy observation from Theorem \ref{MOSTIMPORTANT} is that for an infinite dimensional Hilbert space, a finite collection of vectors can not  be a frame.

The operators $\theta_\tau$, $\theta_\tau^*$ and $S_\tau$ in Theorem \ref{MOSTIMPORTANT} are called as \textbf{analysis operator}, \textbf{synthesis  operator} and \textbf{frame operator}, respectively (cf. \cite{CHRISTENSEN}).

Dilation theory usually tries to extend  operator on Hilbert space to larger Hilbert space which are easier to handle as well as well-understood and study the original operator  as a slice of it (\cite{LEVYSHALIT, ARVESON, NAGY}). As long as frame theory for Hilbert spaces is considered,  following theorem is known as   Naimark-Han-Larson dilation theorem. This was proved independently by  \cite{HANLARSON} and by   \cite{KASHINKULIKOVA}.   History of  this theorem is nicely presented in the paper (\cite{CZAJA}).
 \begin{theorem}(\cite{HANLARSON, KASHINKULIKOVA}) \label{DILATIONTHEOREMHILBERTSPACE} (\textbf{Naimark-Han-Larson dilation theorem})
 A collection $  \{\tau_n\}_{n}$ in  $\mathcal{H}$ is a 
 \begin{enumerate}[label=(\roman*)]
 \item  frame for $\mathcal{H}$ if and only if the exist a Hilbert space   $\mathcal{H}_1 \supseteq \mathcal{H}$,  a Riesz basis $ \{\omega_n\}_{n}$ for  $\mathcal{H}_1 $ and a projection $P:\mathcal{H}_1 \to \mathcal{H}$ such that $\tau_n=P\omega_n, \forall n \in \mathbb{N}$.
 \item Parseval frame for $\mathcal{H}$ if and only if the exist a Hilbert space   $\mathcal{H}_1 \supseteq \mathcal{H}$,  an orthonormal  basis $ \{\omega_n\}_{n}$ for  $\mathcal{H}_1 $ and an orthogonal  projection $P:\mathcal{H}_1 \to \mathcal{H}$ such that $\tau_n=P\omega_n, \forall n \in \mathbb{N}$.
 \end{enumerate}	
 \end{theorem}
In order to construct an element of the Hilbert space using frames using Equation \ref{GFE}, we have to first determine inverse of the frame operator which is difficult in general. Thus we seek a way to approximate an element using a sequence which does not involve calculating inverse of frame operator. This is given in the following theorem. 
 \begin{proposition}(\cite{DUFFIN1}) (\textbf{Frame algorithm})
 Let $ \{\tau_n\}_{n=1}^\infty$ be a frame for  $ \mathcal{H}$	with bounds $a$ and  $b$. For  $ h \in \mathcal{H}$ define 
 $$ h_0\coloneqq0, \quad h_n\coloneqq h_{n-1}+\frac{2}{a+b}S_{\tau}(h-h_{n-1}), \quad\forall n \geq1.$$
 Then 
 $$ \|h_n-h\|\leq \left(\frac{b-a}{b+a}\right)^n\|h\|, \quad\forall n \geq1.$$
 In particular, $h_n\to h$ as $n\to \infty$.
 \end{proposition} 
Given a collection $\{\tau_n\}_n$, in general, it is difficult to find $a$ and $b$ such that the two inequalities in (\ref{SEQUENTIALEQUATION1}) hold. Therefore, it is natural to ask whether there is a characterization for frame without using frame bounds. Orthonormal bases are the simplest and easiest sequences we can handle in a Hilbert space, so one can attempt to obtain characterization using orthonormal bases. Since every separable Hilbert space is isometrically isomorphic to the standard Hilbert space $\ell^2(\mathbb{N})$ and the standard unit vectors $\{e_n\}_n$ form an orthonormal basis for $\ell^2(\mathbb{N})$, one can further ask whether frames can be characterized using $\{e_n\}_n$. This question was answered affirmatively by  \cite{HOLUB} as follows.
\begin{theorem}(\cite{HOLUB})\label{HOLUBTHEOREM} (\textbf{Holub's theorem})
	A sequence $\{\tau_n\}_n$ in  $\mathcal{H}$ is a
	frame for $\mathcal{H}$	if and only if there exists a surjective bounded linear operator $T:\ell^2(\mathbb{N}) \to \mathcal{H}$ such that $Te_n=\tau_n$, for all $n \in \mathbb{N}$.
\end{theorem}
There is a slight variation of Theorem \ref{HOLUBTHEOREM} given by  \cite{CHRISTENSEN}. 
\begin{theorem}(\cite{CHRISTENSEN})\label{OLECHA}
	Let $\{\omega_n\}_n$ be an orthonormal basis for   $\mathcal{H}$. Then  a sequence $\{\tau_n\}_n$ in  $\mathcal{H}$ is a
	frame for $\mathcal{H}$	if and only if there exists a surjective bounded linear operator $T:\mathcal{H} \to \mathcal{H}$ such that $T\omega_n=\tau_n$, for all $n \in \mathbb{N}$.	
\end{theorem}
Given a frame $\{\tau_n\}_n$  for   $\mathcal{H}$ we now consider the frame $\{S_\tau^{-1}\tau_n\}_n$. This frame satisfies Equation (\ref{GFE}). However, in general there may be other frames satisfying the Equation (\ref{GFE}) like $\{S_\tau^{-1}\tau_n\}_n$. This  leads to the notion of dual frames as stated below. 
\begin{definition}(cf. \cite{CHRISTENSEN})
	Let $\{\tau_n\}_n$ be a frame for $\mathcal{H}$. A  frame  $\{\omega_n\}_n$ for $  \mathcal{H}$ is said to be a \textbf{dual} frame for $\{\tau_n\}_n$ if 
	\begin{align*}
	h=\sum_{n=1}^\infty \langle h, \omega_n\rangle \tau_n=\sum_{n=1}^\infty
	\langle h, \tau_n\rangle \omega_n, \quad \forall h \in
	\mathcal{H}.
	\end{align*}
\end{definition}
Just like characterization of frames, given a frame,  we  seek a description of each of its dual frame. This problem was solved by   \cite{LI} in the following two lemmas and a theorem. 
\begin{lemma}(\cite{LI})\label{LILEMMA1}
	Let $\{\tau_n\}_n$ be a frame for $\mathcal{H}$ and $\{e_n\}_n$ be the standard orthonormal basis for 	$\ell^2(\mathbb{N})$. Then a frame $\{\omega_n\}_n$ is a dual frame for $\{\tau_n\}_n$ if and only if 
	\begin{align*}
	\omega_n=U e_n, \quad \forall n \in \mathbb{N},
	\end{align*}
	where $U:\ell^2(\mathbb{N})\to \mathcal{H}$ is a bounded left-inverse of $\theta_\tau$.
\end{lemma}
\begin{lemma}(\cite{LI})\label{LILEMMA2}
	Let $\{\tau_n\}_n$ be a frame for $\mathcal{H}$. Then $L:\ell^2(\mathbb{N})\to \mathcal{H}$ 	is a bounded left-inverse of $\theta_\tau$ if and only if 
	\begin{align*}
	L=S_\tau^{-1}\theta_\tau^*+V(I_{\ell^2(\mathbb{N})}-\theta_\tau S_\tau^{-1} \theta_\tau^*),
	\end{align*}
	where $V:\ell^2(\mathbb{N})\to \mathcal{H}$ is a bounded operator.
\end{lemma}
\begin{theorem}(\cite{LI})\label{LITHM}
	Let $\{\tau_n\}_n$ be a frame for $\mathcal{H}$. Then a frame $\{\omega_n\}_n$ is a dual frame for $\{\tau_n\}_n$ if and only if 
	\begin{align*}
	\omega_n=S_\tau^{-1} \tau_n+\rho_n-\sum_{k=1}^{\infty}\langle S_\tau^{-1} \tau_n, \tau_k\rangle \rho_k, \quad \forall  n \in \mathbb{N}, 
	\end{align*}
	where 	$\{\rho_n\}_n$ is a  sequence in  $\mathcal{H}$ such that there exists  $b>0$ satisfying 
	\begin{align*}
	\sum_{n=1}^{\infty}|\langle h, \rho_n\rangle |^2\leq b \|h\|^2, \quad \forall h \in \mathcal{H}.
	\end{align*}
\end{theorem}
We again consider the frame $\{S_\tau^{-1}\tau_n\}_n$. Note that this frame is obtained by the action of an invertible operator $S_\tau^{-1}$ to the original frame $\{\tau_n\}_n$. This leads to the question: what are all the frames which are obtained by operating an invertible operator to the  given frame? This naturally brings us to the following definition. 
\begin{definition}(\cite{BALAN})\label{SIMILARDEFHILBERT}
	Two frames $\{\tau_n\}_n$ and $\{\omega_n\}_n$ for $  \mathcal{H}$ are said to be \textbf{similar} or \textbf{equivalent} if there exists a bounded invertible operator $T:\mathcal{H} \to \mathcal{H}$ such that
	\begin{align}\label{SIMILARITYEQUATION}
	\omega_n=T \tau_n, \quad\forall n \in \mathbb{N}.
	\end{align} 
\end{definition}
Given frames $\{\tau_n\}_n$ and $\{\omega_n\}_n$, it is rather difficult to check whether they are similar because one has to get an invertible operator and verify Equation (\ref{SIMILARITYEQUATION}) for every natural number. Thus it is better if there is a characterization which does not  involve natural numbers and involves only operators. Further, it is natural to ask whether there is a formula for the operator $T$ which gives similarity. This was done by  \cite{BALAN} and independently by  \cite{HANLARSON} which states as follows. 
\begin{theorem}(\cite{BALAN, HANLARSON})\label{BALANCHARSIM}
	For two frames $\{\tau_n\}_n$ and $\{\omega_n\}_n$ for $  \mathcal{H}$, the following are equivalent.	
	\begin{enumerate}[label=(\roman*)]
		\item $\{\tau_n\}_n$ and $\{\omega_n\}_n$ are similar, i.e., there exists a bounded invertible operator $T:\mathcal{H} \to \mathcal{H}$ such that $\omega_n=T \tau_n$, $\forall n \in \mathbb{N}$.
		\item $\theta_\omega=\theta_\tau T$,  for some bounded invertible operator $T:\mathcal{H} \to \mathcal{H}$.
		\item $ P_\omega=P_\tau$.
	\end{enumerate}
	If one of the above conditions is satisfied, then the invertible operator in (i) and (ii) is unique and is given by $T=S_\tau^{-1}\theta_\tau^*\theta_\omega$.
\end{theorem}

For a given subset $\mathbb{M} $ of $ \mathbb{N}$, set $
S_\mathbb{M} :\mathcal{H} \ni h \mapsto \sum_{n\in \mathbb{M}} \langle h, \tau_n\rangle\tau_n\in
\mathcal{H}$. Because of Inequalities (\ref{SEQUENTIALEQUATION1}), $
S_\mathbb{M} $
is a well-defined bounded positive operator (which may not be invertible). Let $\mathbb{M}^\text{c}$ denote the complement of $\mathbb{M}$ in $\mathbb{N}$. 
Casazza, Edidin, and Kutyniok derived following identities for frames for Hilbert spaces (\cite{BALACASAZZAEDIDINKUTYNIOKFIRST, BALANSIGNAL}). 
\begin{theorem}(\cite{BALACASAZZAEDIDINKUTYNIOK, BALACASAZZAEDIDINKUTYNIOKFIRST}) (\textbf{Frame identity})  \label{CASAZZAGENERAL}
	Let $\{\tau_n\}_n$ be a  frame for  $\mathcal{H}$. Then for every $\mathbb{M} \subseteq \mathbb{N}$, 	
	\begin{align*}
	\sum_{n\in \mathbb{M}}|\langle h, \tau_n\rangle|^2-\sum_{n=1}^\infty|\langle S_\mathbb{M}h, \tilde{\tau}_n\rangle|^2=\sum_{n\in \mathbb{M}^\text{c}}|\langle h, \tau_n\rangle|^2-\sum_{n=1}^\infty|\langle S_{\mathbb{M}^\text{c}}h, \tilde{\tau}_n\rangle|^2,\quad \forall h \in \mathcal{H}.
	\end{align*}	
\end{theorem}
\begin{theorem}(\cite{BALACASAZZAEDIDINKUTYNIOK, BALACASAZZAEDIDINKUTYNIOKFIRST}) (\textbf{Parseval frame identity}) \label{SECOND}
	Let $\{\tau_n\}_n$ be a Parseval frame for  $\mathcal{H}$. Then for every $\mathbb{M} \subseteq \mathbb{N}$, 
	\begin{align*}
	\sum_{n\in \mathbb{M}}|\langle h, \tau_n\rangle|^2-\left\|\sum_{n\in \mathbb{M}}\langle h, \tau_n\rangle \tau_n\right\|^2=\sum_{n\in \mathbb{M}^\text{c}}|\langle h, \tau_n\rangle|^2-\left\|\sum_{n\in \mathbb{M}^\text{c}}\langle h, \tau_n\rangle \tau_n\right\|^2,\quad \forall h \in \mathcal{H}.
	\end{align*}
\end{theorem}
Theorem \ref{SECOND} has applications. It was applied to get the  following remarkable lower estimate for Parseval frames.
\begin{theorem}(\cite{BALACASAZZAEDIDINKUTYNIOK, GAVRUTA})\label{THIRD}
	Let $\{\tau_n\}_n$ be a Parseval frame for  $\mathcal{H}$. Then for every $\mathbb{M} \subseteq \mathbb{N}$, 
	\begin{align*}
	\sum_{n\in \mathbb{M}}|\langle h, \tau_n\rangle|^2+\left\|\sum_{n\in \mathbb{M}^\text{c}}\langle h, \tau_n\rangle \tau_n\right\|^2&=\sum_{n\in \mathbb{M}^\text{c}}|\langle h, \tau_n\rangle|^2+\left\|\sum_{n\in \mathbb{M}}\langle h, \tau_n\rangle \tau_n\right\|^2\\
	&\geq \frac{3}{4}\|h\|^2,\quad \forall h \in \mathcal{H}.
	\end{align*}
	Further, the bound 3/4 is optimal.
\end{theorem}
As another application, Theorem \ref{THIRD} was used in the study of Parseval frames with finite excesses (\cite{BAKICBERIC, BALANCASAZZAHEIL12}). \\
Like duality, there is another notion called as orthogonality for frames for Hilbert spaces. This was first introduced by  \cite{BALANTHESIS} in his Ph.D. thesis  and further  studied by  \cite{HANLARSON}. 
\begin{definition}(\cite{BALANTHESIS, HANLARSON})
		Let $\{\tau_n\}_n$ be a frame for $\mathcal{H}$. A  frame  $\{\omega_n\}_n$ for $  \mathcal{H}$ is said to be an \textbf{orthogonal}  frame for $\{\tau_n\}_n$ if 
	\begin{align*}
	0=\sum_{n=1}^\infty \langle h, \omega_n\rangle \tau_n=\sum_{n=1}^\infty
	\langle h, \tau_n\rangle \omega_n, \quad \forall h \in
	\mathcal{H}.
	\end{align*}
\end{definition}  
Remarkable property of orthogonal frames is that we can interpolate as well as  we can take direct sum of them to get new frames. These are illustrated in the following two results.
\begin{proposition}(\cite{HANKORNELSONLARSON, HANLARSON})
	Let $  \{\tau_n\}_n $ and $ \{\omega_n\}_n$ be  two Parseval  frames for     $\mathcal{H}$ which are  orthogonal. If $C,D \in \mathcal{B}(\mathcal{H})$ are such that $ C^*C+D^*D=I_\mathcal{H}$, then  $\{C\tau_n+D\omega_n\}_{n}$ is a  Parseval frame for   $\mathcal{H}$. In particular,  if scalars $ c,d,e,f$ satisfy $|c|^2+|d|^2 =1$, then $ \{c\tau_n+d\omega_n\}_{n} $ is   a Parseval frame.
\end{proposition} 
\begin{proposition}(\cite{HANKORNELSONLARSON, HANLARSON})
	If $  \{\tau_n\}_n $ and $ \{\omega_n\}_n$ are   orthogonal frames  for $\mathcal{H}$, then  $\{\tau_n\oplus \omega_n\}_{n}$ is a  frame for  $ \mathcal{H}\oplus \mathcal{H}.$    Further, if both $  \{\tau_n\}_n $ and $ \{\omega_n\}_n$ are  Parseval, then $\{\tau_n\oplus \omega_n\}_{n}$ is Parseval.
\end{proposition}
Recall that Paley-Wiener theorem \ref{PALEYWIENERTHEOREM}  says that sequences which are close to orthonormal  bases are  Riesz bases. Since a frame will also give a series representation,  it is natural to ask whether a sequence close to frame is a frame. This was first derived by   \cite{PALEY1} which showed that sequences which are quadratically close to frames are again frames. 
\begin{theorem}(\cite{PALEY1})\label{FIRSTPER} (\textbf{Christensen's quadratic perturbation})
	Let $ \{\tau_n\}_{n=1}^\infty$ be a frame for  $\mathcal{H} $ with bounds $ a$ and $b$. If  $ \{\omega_n\}_{n=1}^\infty$  in $\mathcal{H} $ satisfies
	$$ c \coloneqq\sum_{n=1}^{\infty}\|\tau_n-\omega_n\|^2<a,$$
	then $ \{\omega_n\}_{n=1}^\infty$ is  a frame for $\mathcal{H} $  with bounds $a\left(1-\sqrt{\frac{c}{a}}\right)^2 $ and $b\left(1+\sqrt{\frac{c}{b}}\right)^2.$
\end{theorem}
 Three months later, Christensen generalized Theorem \ref{FIRSTPER}. 
 \begin{theorem}(\cite{PALEY2})\label{SECONDPER} (\textbf{Christensen perturbation})
 	Let $ \{\tau_n\}_{n=1}^\infty$ be a frame for  $\mathcal{H} $ with bounds $ a$ and $b$.  If  $ \{\omega_n\}_{n=1}^\infty$  in $\mathcal{H} $ is  such that there exist $ \alpha, \gamma \geq0$ with $\alpha+\frac{\gamma}{\sqrt{a}}< 1 $ and
 	$$\left\|\sum_{n=1}^{m}c_n(\tau_n-\omega_n) \right\|\leq \alpha\left\|\sum_{n=1}^{m}c_n\tau_n\right \|+\gamma \left(\sum_{n=1}^{m}|c_n|^2\right)^\frac{1}{2},  \quad\forall c_1,  \dots, c_m \in \mathbb{K}, m=1, 2,\dots, $$
 	then $ \{\omega_n\}_{n=1}^\infty$ is  a frame for $\mathcal{H} $  with bounds $a\left(1-(\alpha+\frac{\gamma}{\sqrt{a}})\right)^2 $ and $b\left(1+(\alpha+\frac{\gamma}{\sqrt{b}})\right)^2.$
 \end{theorem}
 After two years, Casazza and Christensen further extended Theorem \ref{SECONDPER}. 
\begin{theorem}(\cite{CASAZZACHRISTENSTENPERTURBATION})\label{OLECAZASSA} (\textbf{Casazza-Christensen perturbation})
	Let $ \{\tau_n\}_{n=1}^\infty$ be a frame for  $\mathcal{H} $ with bounds $ a$ and $b$.  If  $ \{\omega_n\}_{n=1}^\infty$  in $\mathcal{H} $ is  such that there exist $ \alpha, \beta, \gamma \geq0$ with $ \max\{\alpha+\frac{\gamma}{\sqrt{a}}, \beta\}<1$ and
	\begin{align*}
	\left\|\sum_{n=1}^{m}c_n(\tau_n-\omega_n) \right\|&\leq \alpha\left\|\sum_{n=1}^{m}c_n\tau_n\right \|+\gamma \left(\sum_{n=1}^{m}|c_n|^2\right)^\frac{1}{2}+\beta\left\|\sum_{n=1}^{m}c_n\omega_n\right \|, \\
	&  \quad \quad\forall c_1,  \dots, c_m \in \mathbb{K}, m=1, 2, \dots,
	\end{align*}
	then $ \{\omega_n\}_{n=1}^\infty$ is a frame for $\mathcal{H} $  with bounds $a\left(1-\frac{\alpha+\beta+\frac{\gamma}{\sqrt{a}}}{1+\beta}\right)^2 $ and $b\left(1+\frac{\alpha+\beta+\frac{\gamma}{\sqrt{b}}}{1-\beta}\right)^2.$
\end{theorem}



We next consider Bessel sequences which is next level of generalization of frames.
\begin{definition}(cf. \cite{CHRISTENSEN})
A collection $  \{\tau_n\}_{n}$ in  a Hilbert space $ \mathcal{H}$ is said to be a \textbf{Bessel sequence} for $\mathcal{H}$ if there exists a real constant  $  b >0$ such that
\begin{equation}\label{BESSEL INEQUALITY}
\text{ (\textbf{General Bessel's inequality}) } \quad \sum_{n=1}^\infty|\langle h, \tau_n \rangle|^2 \leq b\|h\|^2  ,\quad \forall h \in \mathcal{H}.
\end{equation}
Constant  $ b$ is called as  a Bessel bound for $  \{\tau_n\}_{n}$.	
\end{definition}
Inequality \ref{SEQUENTIALEQUATION1} is stronger than  Inequality \ref{BESSEL INEQUALITY}. Hence every frame is a Bessel sequence. Using Cauchy-Schwarz inequality, it follows that every finite set of vectors is a Bessel sequence (cf. \cite{CHRISTENSEN}). Now  using Theorem \ref{FINITEDIMESIONALCHARAC}, we  get  plenty of Bessel sequences which are not frames (in finite dimensions). As an example in infinite dimensions, we claim that $\{e_2,e_3, \dots\}$ is a Bessel sequence for $\ell^2(\mathbb{N})$ but not a frame. Clearly $\{e_2,e_3, \dots\}$ satisfies Inequality \ref{BESSEL INEQUALITY}. If this is a frame, let $a>0$ be such that first inequality in \ref{SEQUENTIALEQUATION1} holds. Then by taking $h=e_1$, we get $a\|e_1\|^2\leq \sum_{n=2}^{\infty}|\langle e_1, e_n \rangle|^2=0 ~ \Rightarrow a=0,$ which is a contradiction.

\begin{example}
\begin{enumerate}[label=(\roman*)]
\item If $ \{\tau_n\}_{n} $  is a frame for  $\mathcal{H}$, then for every subset $\mathbb{S}$ of $\mathbb{N}$,  $\{\tau_n\}_{n \in \mathbb{S}} $  is a Bessel sequence for  $\mathcal{H}$, because $\sum_{n \in \mathbb{S}}|\langle h, \tau_n \rangle|^2 \leq \sum_{n=1}^\infty|\langle h, \tau_n \rangle|^2 , \forall h \in \mathcal{H}.$
\item (cf. \cite{CHRISTENSEN}) Let $g\in \mathcal{L}^2(\mathbb{R})$ be bounded, compactly supported function and  $a,b>0$. Define $f_{n,m}:\mathbb{R}\ni x \mapsto e^{2\pi imbx}g(x-na)\in \mathbb{R},$ $\forall n, m\in \mathbb{Z}$. Then 
$  \{f_{n,m}\}_{n,m \in \mathbb{Z}}$ is a Bessel sequence for $\mathcal{L}^2(\mathbb{R})$. 
\item (cf. \cite{CHRISTENSEN}) If $\{\tau_n\}_{n=1}^\infty$ is an orthonormal basis for $\mathcal{H}$, then $\{\tau_n+\tau_{n+1}\}_{n=1}^\infty$ is a Bessel sequence  for $\mathcal{H}$ (but not a frame for $\mathcal{H}$).
\end{enumerate}
\end{example}
\begin{theorem}(cf. \cite{CHRISTENSEN})\label{OLEBESSELCHARACTERIZATION12}
A collection $  \{\tau_n\}_{n}$ is  a Bessel sequence for $\mathcal{H}$	if and only if the map $\ell^2(\mathbb{N}) \ni  \{\tau_n\}_{n} \mapsto \sum_{n=1}^\infty a_n\tau_n \in \mathcal{H}$ is a well-defined bounded-linear operator. Moreover, if $  \{\tau_n\}_{n}$ is  a Bessel sequence for $\mathcal{H}$, then the operator $ \mathcal{H} \ni h \mapsto  \sum_{n=1}^\infty\langle h, \tau_n\rangle \tau_n \in \mathcal{H}$ is positive.
\end{theorem}	
Since  a Bessel sequence need not be a frame, it is natural to ask the following question: Given a Bessel sequence, can we add extra elements to it so that the resulting sequence is a frame? Answer is positive. This result was obtained by \cite{LISUN}.
\begin{theorem}(\cite{LISUN})\label{BESSELEXPANSIONHILBERT}
	Every Bessel sequence in a Hilbert space can be expanded to a tight frame. Moreover,  we can expand a Bessel sequence in infinitely many ways to tight frames.
\end{theorem}
\cite{LISUN} further observed that  if a Bessel sequence for a Hilbert space  can be expanded finitely to get a tight frame, then the number of elements added can not be small. Precise statement reads as follows.
\begin{theorem}(\cite{LISUN})\label{BESSELNUMBERHILBERT}
	Let $  \{\tau_n\}_{n}$ be  a Bessel sequence for $\mathcal{H}$. If  $  \{\tau_n \}_{n}\cup \{\omega_k\}_{k=1}^N $ is a $\lambda$-tight frame  for $\mathcal{H}$, then 
	\begin{align}\label{LISUNNUMBER}
	N\geq \dim (\lambda I_\mathcal{H}-S_{\tau}) (\mathcal{H}).
	\end{align}
Further, the Inequality (\ref{LISUNNUMBER}) can not be improved.
\end{theorem}

\section{RIESZ BASES, FRAMES AND BESSEL SEQUENCES FOR BANACH SPACES}
Definition of Riesz basis, as given in Definition \ref{RIESZBASISDEFINITION} requires the notion of inner product. Due to the lack of inner product in a  Banach space, Definition \ref{RIESZBASISDEFINITION} can not be carried over to Banach spaces. However, Theorem \ref{RIESZBASISTHM} allows to define Riesz basis for Banach spaces as follows. 
 \begin{definition}(\cite{ALDROUBISUNTANG})\label{RIESZBASISDEFINITIONBANACHSPACE}
	Let $1<q<\infty$ and $\mathcal{X}$ be a Banach space. A collection 	$\{\tau_n\}_{n}$  in $\mathcal{X}$ is said to be a 
	\begin{enumerate}[label=(\roman*)]
		\item \textbf{q-Riesz sequence} for $\mathcal{X}$ if there exist $a,b>0$ such that for every finite subset $\mathbb{S}$ of  $\mathbb{N}$, 
		\begin{align}\label{RIESZSEQUENCEINEQUALITY}
		a \left(\sum_{n \in \mathbb{S}}|c_n|^q\right)^\frac{1}{q}\leq \left\|\sum_{n \in \mathbb{S}}c_n\tau_n\right\|\leq b \left(\sum_{n \in \mathbb{S}}|c_n|^q\right)^\frac{1}{q}, \quad \forall c_n \in \mathbb{K}.
		\end{align}
		\item \textbf{q-Riesz basis} for $\mathcal{X}$ if it is a q-Riesz sequence  for
		$\mathcal{X}$ and $\overline{\operatorname{span}}\{\tau_n\}_{n}=\mathcal{X}$.
	\end{enumerate}
\end{definition}
\begin{example}
	Let $\{e_n\}_{n}$ be the standard Schauder basis for $\ell^p(\mathbb{N})$ and $A:\ell^p(\mathbb{N})\to \ell^p(\mathbb{N})$ be a bounded linear invertible operator. Then it follows that $\{Ae_n\}_{n}$ is a p-Riesz basis for $\ell^p(\mathbb{N})$.
\end{example}
Like Theorem \ref{RIESZISAFRAME}, we have a similar result for p-Riesz basis.
\begin{theorem}(\cite{CHRISTENSENSTOEVA})
	Let $\{f_n\}_{n}$ be a q-Riesz basis for $\mathcal{X}^*$ and let $p$ be the conjugate index of $q$. Then there exists a unique p-Riesz basis $\{\tau_n\}_{n}$  for  $\mathcal{X}$ such that 
	\begin{align*}
	x=\sum_{n=1}^{\infty}f_n(x)\tau_n, \quad \forall x \in \mathcal{X} \quad \text{ and } \quad f=\sum_{n=1}^{\infty}f(\tau_n)f_n, \quad \forall f \in \mathcal{X}^*.
	\end{align*}
\end{theorem}
By realizing that the functional $\mathcal{H} \ni h \mapsto \langle h, \tau_n\rangle \in \mathbb{K}$ is bounded linear, Definition \ref{OLE} leads to the following in Banach spaces.
\begin{definition}(\cite{ALDROUBISUNTANG, CHRISTENSENSTOEVA})\label{FRAMEDEFINITIONBANACH}
	Let $1<p<\infty$ and $\mathcal{X}$ be a Banach space. 
	\begin{enumerate}[label=(\roman*)]
		\item A collection 	$\{f_n\}_{n}$ of bounded linear functionals in $\mathcal{X}^*$ is said to be a  \textbf{p-frame} for $\mathcal{X}$ if there exist $a,b>0$ such that 	
		\begin{align*}
		a\|x\|\leq \left(\sum_{n=1}^{\infty}|f_n(x)|^p\right)^\frac{1}{p}\leq b\|x\|,\quad \forall x \in \mathcal{X}.
		\end{align*}
		If $a$ can take the value 0,
		then we say $\{f_n\}_{n}$ is a p-Bessel sequence for $\mathcal{X}$.
		\item A collection 	$\{\tau_n\}_{n}$  in $\mathcal{X}$ is said to
		be a \textbf{p-frame} for $\mathcal{X}^*$ if there exist $a,b>0$ such that 
		\begin{align*}
		a\|f\|\leq \left(\sum_{n=1}^{\infty}|f(\tau_n)|^p\right)^\frac{1}{p}\leq b\|f\|,\quad \forall f \in \mathcal{X}^*.
		\end{align*}
	\end{enumerate}	
\end{definition}
\begin{example}
	\begin{enumerate}[label=(\roman*)]
		\item 	Let $\{e_n\}_{n}$ be the standard Schauder basis for $\ell^p(\mathbb{N})$,   $\{\zeta_n\}_{n}$ be the coordinate functionals associated to  $\{e_n\}_{n}$ and $A:\ell^p(\mathbb{N})\to \ell^p(\mathbb{N})$ be a bounded linear invertible operator. Then it follows that $\{\zeta_nA\}_{n}$ is a p-frame for $\ell^p(\mathbb{N})$.
		\item (\cite{ALDROUBISUNTANG}) Let $1\leq p <\infty$ and  $a \in \mathbb{R}$. Define  
		\begin{align*}
		T_a:\mathcal{L}^p(\mathbb{R}) \ni f \mapsto T_af \in  \mathcal{L}^p(\mathbb{R}), \quad T_af:\mathbb{R} \mapsto (T_af) (x)\coloneqq f(x-a) \in \mathbb{C}.
		\end{align*}
		Define  
		\begin{align*}
	W\coloneqq\left\{f:\mathbb{R} \to \mathbb{C}\bigg| \sup_{x \in \mathbb{R}}\sum_{k\in \mathbb{Z}} |T_kf(x)|<\infty\right\} 
		\end{align*}
		and for $1<p<\infty$, $\phi \in W$, 
		\begin{align*}
		S_p\coloneqq \left\{\sum_{k \in \mathbb{Z}}c_kT_k\phi \bigg|\{c_k\}_{k \in \mathbb{Z}}\in \ell^p(\mathbb{Z})\right\}.
		\end{align*}
	\end{enumerate}
Then $S_p$ is a closed subspace of $\mathcal{L}^p(\mathbb{R})$ and $\{T_k\phi\}_{k \in \mathbb{Z}}$ is a p-frame for $\mathcal{L}^p(\mathbb{R})$.
\end{example}
Like Theorem \ref{OLEBESSELCHARACTERIZATION12}, we have a similar result for Banach spaces.
 \begin{theorem}(\cite{CHRISTENSENSTOEVA})\label{pFRAMECHAR}
	Let $\mathcal{X}$ be  a Banach space and $\{f_n\}_{n}$ be a sequence in
	$\mathcal{X}^*$.
	\begin{enumerate}[label=(\roman*)]
		\item $\{f_n\}_{n}$ is a p-Bessel sequence for 
		$\mathcal{X}$ with bound $b$ if and only if 
		\begin{align}\label{BASSELOPERATORCHARACTERIZATION}
		T: \ell^q (\mathbb{N}) \ni \{a_n\}_{n} \mapsto  \sum_{n=1}^\infty a_nf_n \in \mathcal{X}^*
		\end{align}
		is a well-defined (hence bounded) linear operator and $\|T\|\leq b$ (where $q$ is the conjugate
		index of $p$).
		\item If $\mathcal{X}$ is reflexive, then $\{f_n\}_{n}$ is a p-frame 
		for $\mathcal{X}$ if and only if the operator $T$ in
		(\ref{BASSELOPERATORCHARACTERIZATION}) is surjective.
	\end{enumerate}
\end{theorem}
Rather working on p-frames, one can consider a general notion of frames, which are generalizations of $\ell^p (\mathbb{N})$ spaces. For this, we need the notion of BK-space (Banach scalar valued sequence
space or Banach coordinate space).
\begin{definition} (cf. \cite{BANASMURSALEEN})
A sequence space $\mathcal{X}_d$ is said to be  a  \textbf{BK-space} if it is a Banach space and all the coordinate functionals are continuous, i.e.,
whenever $\{x_n\}_n$ is a sequence in $\mathcal{X}_d$ converging to $x \in \mathcal{X}_d$, then each coordinate of 
$x_n$ converges to each coordinate of $x$.
\end{definition}
Familiar sequence spaces like  $\ell^p(\mathbb{N})$, $c(\mathbb{N})$ (space of convergent sequences) and $c_0(\mathbb{N})$ (space of  sequences converging to zero) are examples of  BK-spaces. We now recall an example of a sequence space which is not a BK-space.
\begin{example}
	The space $\mathcal{X}_d\coloneqq\{\{x_n\}_{n=0}^\infty: x_n \in \mathbb{K}, \forall n \in \mathbb{N}\cup \{0\}\}$ equipped with the metric 
	\begin{align*}
		d(\{x_n\}_{n=0}^\infty, \{y_n\}_{n=0}^\infty)\coloneqq\sum_{n=0}^{\infty} \frac{1}{2^n}\frac{|x_n-y_n|}{1+|x_n-y_n|}, \quad \forall \{x_n\}_{n=0}^\infty, \{y_n\}_{n=0}^\infty \in \mathcal{X}_d
	\end{align*}
	is not a BK-space.
\end{example}
\begin{definition}(\cite{CASAZZACHRISTENSENSTOEVA})\label{XDFRAME}
	Let $\mathcal{X}$ be a Banach space and $\mathcal{X}_d$  be  an associated  BK-space. A collection $\{f_n\}_n$  in  $\mathcal{X}^*$ is said to be a   \textbf{$\mathcal{X}_d$-frame} 
	for $\mathcal{X}$ if the following holds.
	\begin{enumerate}[label=(\roman*)]
		\item $\{f_n(x)\}_n \in \mathcal{X}_d$, for  each $x \in \mathcal{X}$. 
		\item There exist $a,b>0$ such that $	a\|x\|\leq \|\{f_n(x)\}_n\|\leq b\|x\|,  \forall x \in \mathcal{X}.$
	\end{enumerate}
	Constants $a$ and $b$ are called as $\mathcal{X}_d$-frame bounds. 
\end{definition}
	
\begin{definition}(\cite{CASAZZACHRISTENSENSTOEVA})
	Let $\mathcal{X}$ be a Banach space and $\mathcal{X}_d$  be  an associated  BK-space. A collection $\{\tau_n\}_n$ in $\mathcal{X}$ is said to be a   \textbf{$\mathcal{X}_d$-frame} 
	for $\mathcal{X}$ if the following holds.
	\begin{enumerate}[label=(\roman*)]
		\item $\{f(\tau_n)\}_n \in \mathcal{X}_d$, for  each $f \in \mathcal{X}^*$. 
		\item There exist $a,b>0$ such that $	a\|f\|\leq \|\{f(\tau_n)\}_n\|\leq b\|f\|,  \forall f \in \mathcal{X}^*.$
	\end{enumerate}
\end{definition}
	
\begin{definition}(\cite{GROCHENIG})\label{BANACHFRAMEDEF}
	Let $\mathcal{X}$ be a Banach space and $\mathcal{X}_d$  be  an associated  BK-space. Let $\{f_n\}_n$ be a collection in $\mathcal{X}^*$ and $S:\mathcal{X}_d \to \mathcal{X}$ 
	be a bounded linear operator.
	The pair $(\{f_n\}_n, S)$ is said to be a  \textbf{Banach frame} for $\mathcal{X}$  
	if the following holds.
	\begin{enumerate}[label=(\roman*)]
		\item $\{f_n(x)\}_n \in \mathcal{X}_d$, for  each $x \in \mathcal{X}$. 
		\item There exist $a,b>0$ such that 
		$
		a\|x\|\leq \|\{f_n(x)\}_n\|\leq b\|x\|,  \forall x \in \mathcal{X}.
		$
		\item $S(\{f_n(x)\}_n)=x$, for  each $x \in \mathcal{X}$.
	\end{enumerate}
	Constants $a$ and $b$ are called as  \textbf{lower Banach frame bound} and  \textbf{upper Banach frame bound}, respectively. The operator $S$ is called as \textbf{reconstruction operator} and the operator $\theta_f:\mathcal{X} \ni x \mapsto \theta_f(x)\coloneqq \{f_n(x)\}_n\in  \mathcal{X}_d$ is called as \textbf{analysis operator}. 
\end{definition}
\begin{example}
\begin{enumerate}[label=(\roman*)]
	\item Let $\{\tau_n\}_n$ be a frame for a Hilbert space $\mathcal{H}$ with bounds $a$ and $b$. Let $f \in  \mathcal{H}^* $. Let $h_f \in \mathcal{H}$ be such that $f(h) =\langle h, h_f\rangle $, $\forall h \in \mathcal{H}$ and $\|f\|=\|h_f\|$. Then 
	\begin{align*}
	a\|f\|^2=a\|h_f\|^2\leq \sum_{n=1}^\infty |\langle h_f, \tau_n\rangle |^2= \sum_{n=1}^\infty|f(\tau_n)|^2\leq b\|h_f\|^2=b\|f\|^2.
	\end{align*}Therefore $\{\tau_n\}_n$ is an $\ell^2(\mathbb{N})$-frame for $\mathcal{H}$.
	\item   Let $\{\tau_n\}_n$ be a frame for Hilbert space $\mathcal{H}$.  We  define $f_n (h)\coloneqq \langle h,  \tau_n\rangle $, $ \forall h \in \mathcal{H}$, $ \forall n$ and $S\coloneqq S_\tau^{-1}\theta_\tau^*$. Then $(\{f_n\}_n, S)$ is  a Banach frame for $\mathcal{H}$.
	\item  Let $\{\tau_n\}_n$ be a frame for $\mathcal{H}$.  We  define $f_n (h)\coloneqq \langle h,  S_\tau^{-1}\tau_n\rangle $, $ \forall h \in \mathcal{H}$, $ \forall n$ and $S\coloneqq \theta_\tau^*$. Then $(\{f_n\}_n, S)$ is  a Banach frame for $\mathcal{H}$.
	\item (\cite{CASAZZACHRISTENSENSTOEVA}) Let $\{\tau_n\}_n$ be an orthonormal basis   for a Hilbert space $\mathcal{H}$. We define 
	\begin{align*}
	\mathcal{X}_d\coloneqq \{\{\langle h,  \tau_n+\tau_{n+1}\rangle\}_n:h \in \mathcal{H} \}=\{\{  a_n+a_{n+1}\}_n:\{a\}_n \in \ell^2(\mathbb{N}) \}
	\end{align*}
	equipped with the norm 
	\begin{align*}
	\|\{  a_n+a_{n+1}\}_n\|\coloneqq \left(\sum_{n=1}^{\infty}|a_n|^2\right)^\frac{1}{2}.
	\end{align*}
	Then $\mathcal{X}_d$ is a BK-space.  Define $f_n:\mathcal{H} \ni h \mapsto \langle h,  \tau_n+\tau_{n+1}\rangle \in \mathbb{K}$, $\forall n$, $T: \mathcal{H} \ni h \mapsto \{f_n(h)\}_n \in \mathcal{X}_d$ and set $S\coloneqq T^{-1}$. Then $(\{f_n\}_n, S)$ is a  Banach frame for $\mathcal{H}$. However, $\{\tau_n+\tau_{n+1}\}_n$ is not a frame for $\mathcal{H}$.
\end{enumerate}	
\end{example}
For Hilbert spaces it follows immediately that every Hilbert space has a frame because separable spaces have
orthonormal bases and an orthonormal basis is a frame. 
Using Hahn-Banach theorem, the following result was proved in (\cite{CASAZZAHANLARSONFRAMEBANACH}).
\begin{theorem}(\cite{CASAZZAHANLARSONFRAMEBANACH})\label{BANACHFRAMEEXISTSSEPARABLE}
	Every separable Banach space admits a Banach frame.
\end{theorem}
The notion of atomic decomposition is studied along with the notion of frames for Banach spaces. This is defined as follows.
\begin{definition}(\cite{GROCHENIG})\label{ATOMICDECOMPODEFI}
	Let $\mathcal{X}$ be a Banach space and $\mathcal{X}_d$  be  an associated  BK-space. Let $\{f_n\}_n$ be a collection in $\mathcal{X}^*$ and 
	$\{\tau_n\}_n$ be a collection in $\mathcal{X}$. The pair $(\{f_n\}_n, \{\tau_n\}_n)$ is said to be an  \textbf{atomic decomposition} for $\mathcal{X}$ 
	if the following holds.
	\begin{enumerate}[label=(\roman*)]
		\item $\{f_n(x)\}_n \in \mathcal{X}_d$, for  each $x \in \mathcal{X}$. 
		\item There exist $a,b>0$ such that 
		$
		a\|x\|\leq \|\{f_n(x)\}_n\|\leq b\|x\|,  \forall x \in \mathcal{X}.
		$
		\item $x=\sum_{n=1}^\infty f_n(x)\tau_n$, for  each $x \in \mathcal{X}$.
	\end{enumerate}
	Constants $a$ and $b$ are called as  \textbf{lower atomic  bound} and  \textbf{upper  atomic bound},  respectively.
\end{definition}
\begin{example}
	\begin{enumerate}[label=(\roman*)]
		\item Let $\{\tau_n\}_n$ be a frame for $\mathcal{H}$. By defining $f_n (h)\coloneqq \langle h, S_\tau^{-1} \tau_n\rangle $, $ \forall h \in \mathcal{H}$, $ \forall n$, the pair $(\{f_n\}_n, \{\tau_n\}_n)$ satisfies all the conditions of  Definition \ref{ATOMICDECOMPODEFI} and hence it is an atomic decomposition for $\mathcal{H}$ .
		\item Let $\{\tau_n\}_n$ be a frame for $\mathcal{H}$. By defining $f_n (h)\coloneqq \langle h, \tau_n\rangle $, $ \forall h \in \mathcal{H}$,  and $\omega_n \coloneqq S_\tau^{-1}\tau_n$, $ \forall n$, the pair $(\{f_n\}_n, \{\omega_n\}_n)$ satisfies all the conditions of  Definition \ref{ATOMICDECOMPODEFI} and hence it is an atomic decomposition for $\mathcal{H}$.
	\end{enumerate}	
\end{example}
Now we make a detailed observation that the notions of atomic decompositions and Banach frames are completely different. Definition \ref{ATOMICDECOMPODEFI} of atomic decomposition demands the expression of every element of a Banach space as a convergent series in the same Banach space using a collection of bounded linear functionals on the space and a collection of elements  in the same Banach space. On the other hand, Definition \ref{BANACHFRAMEDEF} of Banach frame demands the expression of every element of a Banach space using a bounded linear operator from the BK-space to the Banach space and a collection of bounded linear functionals on the same space without bothering about a converging series expansion. Theorem \ref{BANACHFRAMEEXISTSSEPARABLE} guarantees the existence of Banach frame for every separable Banach space. However, the following remarkable result gives a lot of conditions on Banach spaces to admit an atomic decomposition.
\begin{theorem}(\cite{CASAZZAHANLARSONFRAMEBANACH,PELCZYNSKI, JOHNSONROSENTHALZIPPIN})
	For a Banach space $\mathcal{X}$, the following are equivalent.
\begin{enumerate}[label=(\roman*)]
	\item $\mathcal{X}$ has an atomic decomposition.
	\item $\mathcal{X}$ has a finite dimensional expansion of identity.
	\item $\mathcal{X}$ is complemented in a Banach space with a Schauder basis.
	\item $\mathcal{X}$ has bounded approximation property.
\end{enumerate}	
\end{theorem}
There is a  close relationship between atomic decomposition and Banach frames for certain classes of Banach spaces. These are exhibited in the following theorems.
\begin{theorem}\cite{CASAZZAHANLARSONFRAMEBANACH}
	Let $\mathcal{X}$ be a Banach space, $\mathcal{X}_d$ be a BK-space, $\{f_n\}_n$ be a collection in $\mathcal{X}^*$ and $S:\mathcal{X}_d \to \mathcal{X}$ be a bounded linear operator. If the canonical  unit vectors $\{e_n\}_n$ are in $\mathcal{X}_d$, then the following are equivalent.
	\begin{enumerate}[label=(\roman*)]
		\item $(\{f_n\}_n, S)$ is a Banach frame for $\mathcal{X}$  and $\{e_n\}_n$  is a Schauder basis for  $\mathcal{X}_d$.
		\item $(\{f_n\}_n, \{Se_n\}_n)$ is an atomic decomposition  for $\mathcal{X}$.
	\end{enumerate}
\end{theorem}
\begin{theorem}(\cite{CASAZZAHANLARSONFRAMEBANACH, PELCZYNSKI, JOHNSONROSENTHALZIPPIN})
Let $\mathcal{X}$	be a Banach space and $\mathcal{X}_d$ be a BK-space. Let $\{f_n\}_n$ be a collection in $\mathcal{X}^*$ and  $S:\mathcal{X}_d \to \mathcal{X}$ be a bounded linear operator.  If the canonical  unit vectors $\{e_n\}_n$ are in $\mathcal{X}_d$, then the following are equivalent.
\begin{enumerate}[label=(\roman*)]
	\item There exists a BK-space $\mathcal{X}_d$ such that $(\{f_n\}_n, \{\tau_n\}_n)$ is an atomic decomposition  for $\mathcal{X}$.
	\item There exists a BK-space $\mathcal{Y}_d$ which has canonical unit vectors $\{e_n\}_n$ as Schauder basis and a bounded linear operator $S:\mathcal{Y}_d \to \mathcal{X}$ such that $(\{f_n\}_n, S)$ is a Banach frame for $\mathcal{X}$. Further, $S$ can be taken as a projection and $Se_n=\tau_n$, for all $n \in \mathbb{N}$.
	\end{enumerate}
\end{theorem}
Next theorem shows that there is a relation between atomic decompositions and projections.
\begin{theorem}(\cite{CASAZZAHANLARSONFRAMEBANACH})
Let $\mathcal{X}$	be a Banach space, $\{\tau_n\}_n$ be a collection  in $\mathcal{X}$, $\{f_n\}_n$ be a collection in $\mathcal{X}^*$ and  $S:\mathcal{X}_d \to \mathcal{X}$ be a bounded linear operator. Let $\{e_n\}_n$ be the standard unit vectors in $\mathcal{X}_d$. Then the following are equivalent.
\begin{enumerate}[label=(\roman*)]
	\item There is a BK-space $\mathcal{X}_d$ such that $(\{f_n\}_n, \{\tau_n\}_n)$ is  an atomic decomposition for $\mathcal{X}$.
	\item There is a Banach space $\mathcal{Z}$ with a Schauder basis $\{\omega_n\}_n$ such that  $\mathcal{X} \subseteq \mathcal{Z}$ and there is a bounded linear projection $P:\mathcal{Z}\to \mathcal{X} $ such that $P\omega_n=\tau_n,$ $ \forall n \in \mathbb{N}$. 
\end{enumerate}	
\end{theorem}
We mention here before passing that it is known,  one can simultaneously construct Banach frames and atomic decompositions for certain classes of Banach spaces such as coorbit spaces (\cite{GROCHENIG}), $\alpha$-modulation spaces (\cite{FORNASIERALPHA}), decomposition spaces (\cite{BORUPNIELSEN}), homogeneous spaces (\cite{DAHLKESTEIDLTESCHKE}), weighted coorbit spaces (\cite{DAHLKESTEIDLTESCHKE1}), generalized coorbit spaces (\cite{DAHLKEFORNASIERRAUHUTTESCHKE}), inhomogeneous function spaces (\cite{RAUHUTULLRICH}) and Bergman spaces (\cite{CHRISTENSENGROCHENIGOLAFSSON}).

Using approximation property for Banach spaces, \cite{CASAZZARECONSTRUCTION} proved that  $\mathcal{X}_d$-frame for a Banach space need not admit representation of every element of the Banach space. With regard to this, the following result gives information when it is possible to express element of the Banach space using series. 
 \begin{theorem}(\cite{CASAZZACHRISTENSENSTOEVA})\label{CASAZZASEPARABLECHARACTERIZATION}
	Let   $\mathcal{X}_d$ be a BK-space and 
	$\{f_n\}_n$ be an $\mathcal{X}_d$-frame for $\mathcal{X}$. Let $ \theta_f:\mathcal{X} \ni x \mapsto \{f_n(x)\}_n \in \mathcal{X}_d$ (this map is a well-defined  linear bounded below operator). 
	The following are equivalent. 
	\begin{enumerate}[label=(\roman*)]
		\item $ \theta_f(\mathcal{X})$ is complemented in $\mathcal{X}_d$.
		\item The operator $\theta_f^{-1}:\theta_f(\mathcal{X}) \rightarrow \mathcal{X}$ can be extended to a bounded linear operator $T_f: \mathcal{X}_d \rightarrow \mathcal{X}.$
		\item There exists a bounded linear operator $S: \mathcal{X}_d \rightarrow \mathcal{X}$ such that $(\{f_n\}_n, S)$ is a Banach frame for  $\mathcal{X}$.
	\end{enumerate}
	Also, the condition 
	\begin{enumerate}[label=(\roman*)]\addtocounter{enumi}{3}
		\item  There exists a sequence $\{\tau_n\}_n$ in $\mathcal{X}$ such that $\sum_{n=1}^\infty a_n \tau_n$ is convergent in $\mathcal{X}$
		for all $\{a_n\}_n$  in $\mathcal{X}_d$ and  $x=\sum_{n=1}^\infty f_n(x) \tau_n, \forall x \in \mathcal{X}.$
	\end{enumerate}
	implies each of (i)-(iii). If we also assume that the canonical unit vectors $\{e_n\}_n$ form a Schauder basis for $\mathcal{X}_d$,
	(iv) is equivalent to the above (i)-(iii) and to the following condition (v).
	\begin{enumerate}[label=(\roman*)]\addtocounter{enumi}{4}
		\item There exists an $\mathcal{X}_d^*$-Bessel sequence  $\{\tau_n\}_n\subseteq \mathcal{X}\subseteq \mathcal{X}^{**}$ for 
		$\mathcal{X}^*$ such that $x=\sum_{n=1}^\infty f_n(x) \tau_n, $ $ \forall x \in \mathcal{X}.$
	\end{enumerate}
	If the canonical unit vectors $\{e_n\}_n$ form a Schauder basis for both $\mathcal{X}_d$ and $\mathcal{X}_d^*$, then (i)-(v) is equivalent
	to the following condition (vi).
	\begin{enumerate}[label=(\roman*)]\addtocounter{enumi}{5}
		\item There exists an $\mathcal{X}_d^*$-Bessel sequence  $\{\tau_n\}_n\subseteq \mathcal{X}\subseteq \mathcal{X}^{**}$ for 
		$\mathcal{X}^*$ such that $f=\sum_{n=1}^\infty f(\tau_n) f_n, $ $ \forall f \in \mathcal{X}^*.$
	\end{enumerate}
	In each of the cases (v) and (vi), $\{\tau_n\}_n$ is actually an $\mathcal{X}_d^*$-frame for $\mathcal{X}^*$. 
\end{theorem}
We end this section by mentioning a perturbation theorem for Banach frames.
\begin{theorem}(\cite{CHRISTENSENHEIL})\label{PERTURBATIONBANACH12} (\textbf{Christensen-Heil perturbation})
Let $(\{f_n\}_{n}, S)$ be a
Banach  frame for a Banach space    $\mathcal{X}$.  Let $\{g_n\}_{n}$ be a collection in 
$\mathcal{X}^*$ satisfying the following.
\begin{enumerate}[label=(\roman*)]
	\item There exist $\alpha, \gamma\geq 0$ such that 
	\begin{align*}
	\|\{(f_n-g_n)(x)\}_n\|\leq \alpha \|\{f_n(x)\}_n\|+\gamma \|x-y\|, ~ \forall x, y \in  \mathcal{X}.
	\end{align*}
	\item $\alpha \|\theta_f\|+\gamma\leq \|S\|^{-1}.$
\end{enumerate}
Then there exists a reconstruction operator $T$ such that $(\{f_n\}_{n}, T)$ is  a
Banach frame for     $\mathcal{X}$  with bounds 
$
\|S\|^{-1}-(\alpha\|\theta_f\|+\gamma) $ and $ 
\|\theta_f\|+(\alpha\|\theta_f\|+\gamma).$
\end{theorem}

\section{MULTIPLIERS FOR BANACH SPACES}
Let $\{\lambda_n\}_n \in \ell^\infty(\mathbb{N})$ and  $\{\tau_n\}_n$, $\{\omega_n\}_n$ be
sequences in a Hilbert space $\mathcal{H}$. For $x,y \in \mathcal{H}$, the
operator  defined by  	$\mathcal{H}\ni h \mapsto \langle h, y\rangle x\in \mathcal{H}$ is denoted by $x\otimes \overline{y}$ .

The study of operators of the form 
\begin{align}\label{FIRST EQUATION}
\sum_{n=1}^{\infty}\lambda_n (\tau_n\otimes \overline{\omega_n})
\end{align}
began with  \cite{SCHATTEN}, in connection with the study of compact
operators. Schatten studied the operator in (\ref{FIRST EQUATION}) whenever
$\{\tau_n\}_n$ and  $\{\omega_n\}_n$ are orthonormal sequences in a Hilbert space
$\mathcal{H}$. He showed that  if $\{\lambda_n\}_n \in
\ell^\infty(\mathbb{N})$ and $\{\tau_n\}_n$, $\{\omega_n\}_n$ are orthonormal sequences
in a Hilbert space $\mathcal{H}$, then the map in (\ref{FIRST EQUATION}) is a well-defined bounded linear
operator. Later,  operators in (\ref{FIRST EQUATION}) are studied mainly in
connection with Gabor analysis (\cite{FEICHTINGERNOWAK}, \cite{BENEDETTOPFANDER},
\cite{DORFLERTORRESANI}, \cite{GIBSONLAMOUREUXMARGRAVE}, \cite{CORDEROGROCHENIG}, \cite{SKRETTINGLAND}). This was generalized by 
\cite{BALAZSBASIC} who replaced orthonormal sequences by Bessel sequences.

Let $\{f_n\}_n$
be a sequence in  the  dual space $\mathcal{X}^*$ of a Banach
space $\mathcal{X}$ and $\{\tau_n\}_n$ be a sequence in a Banach space
$\mathcal{Y}$. The operator $\tau \otimes f$ is defined by  	$\tau
\otimes f:\mathcal{X}\ni x \mapsto f(x)\tau\in \mathcal{Y}$. It was   \cite{RAHIMIBALAZSMUL} who extended the operator in (\ref{FIRST
	EQUATION}) from Hilbert spaces to Banach spaces. For a Banach space
$\mathcal{X}$ and dual space $\mathcal{X}^*$, they considered the operators (called as \textbf{multipliers}) of the form 
\begin{align}\label{SECOND EQUATION}
 M_{\lambda,f, \tau} \coloneqq\sum_{n=1}^{\infty}\lambda_n (\tau_n\otimes f_n).
\end{align}
Rahimi and Balazs  studied the operator in (\ref{SECOND EQUATION}), whenever
$\{\tau_n\}_n$ is a p-Bessel sequence  for
$\mathcal{X}^*$ and $\{f_n\}_n$ is a q-Bessel sequence for $\mathcal{X}$ ($q$ is the 
conjugate index of $p$). Besides theoretical importance, multipliers also play
important role in Physics, signal processing, acoustics (\cite{BALAZSSTOEVAMULAPP}).

Fundamental result obtained by  \cite{RAHIMIBALAZSMUL}   is the
following. In this section, $q$ denotes the 
conjugate index of $p$.
\begin{theorem}(\cite{RAHIMIBALAZSMUL})\label{RAHIMIBALAZS}
	Let $\{f_n\}_{n}$   be a
	p-Bessel sequence  for  a Banach space $\mathcal{X}$ with bound $b$ and
	$\{\tau_n\}_{n}$   be a
	q-Bessel sequence  for the dual of  a Banach space $\mathcal{Y}$ with bound
	$d$.  If
	$\{\lambda_n\}_n \in \ell^\infty(\mathbb{N})$, then the map
	\begin{align*}
	T: \mathcal{X} \ni x \mapsto \sum_{n=1}^{\infty}\lambda_n (\tau_n\otimes
	f_n) x \in \mathcal{Y}
	\end{align*}
	is a well-defined bounded linear   operator and  
	 $\|T\|\leq bd\|\{\lambda_n\}_n\|_\infty.$
\end{theorem}
By varying only the symbol in a multiplier, we get a bounded linear operator which has the nice property stated in the following theorem.
\begin{proposition}(\cite{RAHIMIBALAZSMUL})
	Let $\{f_n\}_{n}$ be a
	p-Bessel sequence  for  a Banach space $\mathcal{X}$  with non-zero
	elements, $\{\tau_n\}_{n}$   be a  q-Riesz sequence  for 
	$\mathcal{Y}$ and let $\{\lambda_n\}_n \in
	\ell^\infty(\mathbb{N})$. Then the mapping 
	\begin{align*}
	T:\ell^\infty(\mathbb{N})\ni \{\lambda_n\}_n \mapsto M_{\lambda,f, \tau}
	\in \mathcal{B}(\mathcal{X}, \mathcal{Y})
	\end{align*}
	is a well-defined injective bounded linear operator.	
\end{proposition}
From the spectral theory of compact operators in Hilbert spaces, it easily follows that the symbol of compact operator converges to zero.  Following is a general result for Banach spaces.
\begin{proposition}(\cite{RAHIMIBALAZSMUL})\label{RAHIMIBALAZSMULTIPLIERCOMPACT}
	Let $\{f_n\}_{n}$  be
	a  p-Bessel sequence  for  a Banach space $\mathcal{X}$ with bound $b$
	and $\{\tau_n\}_{n}$    be a  q-Bessel sequence 
	for $\mathcal{Y}$ with bound $d$. 	If
	$\{\lambda_n\}_n \in c_0(\mathbb{N})$, then $M_{\lambda,f, \tau}$ is a 
	compact operator.
\end{proposition}

Following theorem shows that multipliers behave nicely with respect to  change in its parameters. These are known as continuity properties of multipliers in the literature.
\begin{theorem}(\cite{RAHIMIBALAZSMUL})\label{RAHIMIBALAZSMULTIPLIERCONTINUITY}
	Let $\{f_n\}_{n}$  be a
	 p-Bessel sequence  for  $\mathcal{X}$ with bound $b$,
	$\{\tau_n\}_{n}$   be a  q-Bessel sequence  for 
 $\mathcal{Y}$ with bound $d$ and
	$\{\lambda_n\}_n \in \ell^\infty(\mathbb{N})$. 	Let $k \in \mathbb{N}$ and let 
	$\lambda^{(k)}=\{\lambda_1^{(k)},\lambda_2^{(k)}, \dots \}$,
	$\lambda=\{\lambda_1,\lambda_2, \dots \}$, 
	$\tau^{(k)}=\{\tau_1^{(k)}, \tau_2^{(k)}, \dots\}$,
	$\tau_n^{k} \in \mathcal{X}$, $\tau=\{\tau_1, \tau_2, \dots\}$. Assume that for
	each $k$, $\lambda^{(k)}\in \ell^\infty(\mathbb{N})$  and
	$\tau^{(k)}$  is  a  q-Bessel sequence  for
	$\mathcal{Y}$.
	\begin{enumerate}[label=(\roman*)]
		\item If $\lambda^{(k)} \to \lambda $ as $k \rightarrow \infty $ in p-norm,  then
		\begin{align*}
		\|M_{\lambda^{(k)},f, \tau}-M_{\lambda,f, \tau}\| \to 0 \quad  \text{ as } \quad k \to \infty.
		\end{align*}
		\item If $\{\lambda_n\}_n \in \ell^p(\mathbb{N})$ and  $\sum_{n=1}^{\infty}\|\tau_n^{(k)}-\tau_n\|^q \to 0 \text{ as }  k \to \infty $, then 
		\begin{align*}
		\|M_{\lambda, f, \tau^{(k)}}-M_{\lambda,f, \tau}\| \to 0 \quad \text{ as } \quad k \to \infty.
		\end{align*}
	\end{enumerate}
\end{theorem}

\section{LIPSCHITZ OPERATORS AND LIPSCHITZ COMPACT \\
	OPERATORS}
We
first  recall the definition of Lipschitz function. 
\begin{definition}(cf. \cite{WEAVER})
	Let $\mathcal{M}$, 
$\mathcal{N}$ be  metric spaces. A function $f:\mathcal{M}  \rightarrow
\mathcal{N}$ is said to be \textbf{Lipschitz} if there exists $b> 0$ such that 
\begin{align*}
d(f(x), f(y)) \leq b\, d(x,y), \quad \forall x, y \in \mathcal{M}.
\end{align*}
\end{definition}
One of the most important results in the study of Lipschitz functions is the following (which is similar to the Banach-Mazur theorem for Banach spaces which shows that every separable Banach space embeds isometrically in the Banach space of continuous functions on $[0,1]$ (cf. \cite{ALBIACKALTON})).
\begin{theorem} (cf. \cite{KALTONLANCIEN, AHARONI}) (\textbf{Aharoni's theorem})  \label{AHARONITHEOREM} 
If $\mathcal{M}$ is  a  separable metric space, then there exists a function $f: \mathcal{M} \to c_0(\mathbb{N})$ and a constant $b>0$ such that 
\begin{align*}
d(x,y)\leq \|f(x)-f(y)\|\leq b \, d(x,y), \quad \forall x, y \in \mathcal{M}.
\end{align*}	
\end{theorem}
\begin{definition}(cf. \cite{WEAVER})
	A 	metric space $\mathcal{M}$ with a reference point which is usually denoted by 0 is called as a \textbf{pointed metric space}. In this case, we write  $(\mathcal{M}, 0)$ is a pointed metric space.
\end{definition}
Note that every metric space is a pointed metric space by fixing any point of the space.

Just like norm of linear operator, a reasonable measure of a Lipschitz function can be defined. This is exhibited in the next definition.
\begin{definition}(cf. \cite{WEAVER})
	Let	$\mathcal{X}$ be a Banach space.
	\begin{enumerate}[label=(\roman*)]
		\item Let $\mathcal{M}$ be a  metric space. The collection 	$\operatorname{Lip}(\mathcal{M}, \mathcal{X})$
		is defined as $\operatorname{Lip}(\mathcal{M}, \mathcal{X})\coloneqq \{f:f:\mathcal{M} 
		\rightarrow \mathcal{X}  \operatorname{ is ~ Lipschitz} \}.$ For $f \in \operatorname{Lip}(\mathcal{M}, \mathcal{X})$, the \textbf{Lipschitz number}
		is defined as 
		\begin{align*}
		\operatorname{Lip}(f)\coloneqq \sup_{x, y \in \mathcal{M}, ~x\neq
			y} \frac{\|f(x)-f(y)\|}{d(x,y)}.
		\end{align*}
		\item Let $(\mathcal{M}, 0)$ be a pointed metric space. The collection 	$\operatorname{Lip}_0(\mathcal{M}, \mathcal{X})$
		is defined as $\operatorname{Lip}_0(\mathcal{M}, \mathcal{X})\coloneqq \{f:f:\mathcal{M} 
		\rightarrow \mathcal{X}  \operatorname{ is ~ Lipschitz ~ and } f(0)=0\}.$
		For $f \in \operatorname{Lip}_0(\mathcal{M}, \mathcal{X})$, the \textbf{Lipschitz norm}
		is defined as 
		\begin{align*}
		\|f\|_{\operatorname{Lip}_0}\coloneqq \sup_{x, y \in \mathcal{M}, x\neq
			y} \frac{\|f(x)-f(y)\|}{d(x,y)}.
		\end{align*}
	\end{enumerate}
\end{definition}
It is well-known that given two Banach spaces $\mathcal{X}$ and $\mathcal{Y}$, the collection of all bounded linear maps from $\mathcal{X}$ to $\mathcal{Y}$ is a Banach space with respect to operator-norm. A similar result holds for base point preserving Lipschitz maps from pointed metric spaces to Banach spaces.
\begin{theorem}(cf. \cite{WEAVER})\label{LIPISABANACHALGEBRA}
	Let	$\mathcal{X}$ be a Banach space.
	\begin{enumerate}[label=(\roman*)]
		\item If $\mathcal{M}$ is a  metric space, then 	$\operatorname{Lip}(\mathcal{M},
		\mathcal{X})$ is a semi-normed vector  space with respect to  the semi-norm  $\operatorname{Lip}(\cdot)$.
		\item If $(\mathcal{M}, 0)$ is a pointed metric space, then 	$\operatorname{Lip}_0(\mathcal{M},
		\mathcal{X})$ is a Banach space with respect to  the  norm
		$\|\cdot\|_{\operatorname{Lip}_0}$. Further, $\operatorname{Lip}_0(\mathcal{X})\coloneqq\operatorname{Lip}_0(\mathcal{X},
		\mathcal{X})$ is a unital Banach algebra. In particular, if $T \in \operatorname{Lip}_0(\mathcal{X})$ satisfies $
		\|T-I_\mathcal{X}\|_{\operatorname{Lip}_0}<1,$ 
		then $T $ is invertible and $T^{-1} \in \operatorname{Lip}_0(\mathcal{X})$.
	\end{enumerate}
\end{theorem}	
In the study of Lipschitz functions it is natural to shift from metric space to the setting of Banach spaces and use functional analysis tools on Banach spaces. This is achieved through the following theorem. 
\begin{theorem}(cf. \cite{WEAVER, KALTON, ARENSEELLS})\label{POINTEDSPLITS}
	Let $(\mathcal{M},0)$ be a pointed metric space. Then there exists a Banach space  $\mathcal{F}(\mathcal{M})$ and an isometric embedding $e:\mathcal{M} \to \mathcal{F}(\mathcal{M})$ satisfying
	the following  universal property: for each Banach space $\mathcal{X}$ and each $f \in \operatorname{Lip}_0(\mathcal{M}, \mathcal{X})$, there is a unique bounded linear operator 
	$T_f :\mathcal{F}(\mathcal{M})\to \mathcal{X} $ such that $T_fe=f$, i.e., the following diagram commutes.
	
	\begin{center}
		\[
		\begin{tikzcd}
		\mathcal{M} \arrow[d,"e"] \arrow[dr,"f"]\\
		\mathcal{F}(\mathcal{M}) \arrow[r,"T_f"] & \mathcal{X}
		\end{tikzcd}
		\]
	\end{center}
	Further, $\|T_f\|=\|f\|_{\operatorname{Lip}_0}$. This property characterizes the pair $(\mathcal{F}(\mathcal{M}),e)$ uniquely upto isometric isomorphism. Moreover, the map 
	$\operatorname{Lip}_0(\mathcal{M}, \mathcal{X})\ni f \mapsto T_f \in \mathcal{B}(\mathcal{F}(\mathcal{M}),\mathcal{X}) $  is an isometric isomorphism.
\end{theorem}
The space $\mathcal{F}(\mathcal{M})$ is known as \textbf{Arens-Eells space} or  \textbf{Lipschitz-free 	Banach space} (\cite{GODEFROYSURVEY}). Theorem \ref{POINTEDSPLITS} tells that in order to `find' the space $\operatorname{Lip}_0(\mathcal{M}, \mathcal{X})$, we can find first $\mathcal{F}(\mathcal{M})$ and then $\mathcal{B}(\mathcal{F}(\mathcal{M}),\mathcal{X}) $. In particular, $\operatorname{Lip}_0(\mathcal{M}, \mathbb{K})$ is isometrically isomorphic to  $\mathcal{F}(\mathcal{M})^*$. For this reason $\mathcal{F}(\mathcal{M})$ is also called as \textbf{predual} of metric space $\mathcal{M}$. The bounded linear  operator $T_f$ is called as \textbf{linearization} of $f$. We now mention some examples of Lipschitz-free spaces for certain metric spaces.
\begin{example}
	\begin{enumerate}[label=(\roman*)] (cf. \cite{WEAVER}, \cite{DUBEITYMCHATYN})
		\item If $\mathbb{R}$ is considered with usual metric, then $\mathcal{F}(\mathbb{R}) \cong \mathcal{L}^1(\mathbb{R})$.
		\item If  $\mathcal{M}$ is any  separable metric tree, then $\mathcal{F}(\mathcal{M})\cong \mathcal{L}^1([0,1])$.
		\item If  $\mathcal{M}$ is any  set equipped with the metric $d(x,y)\coloneqq 2$ whenever $x,y \in \mathcal{M}$ with $x\neq y$ and $d(x,x)\coloneqq 0$, $\forall x \in \mathcal{M}$, then $\mathcal{F}(\mathcal{M})\cong \ell^1(\mathcal{M})$.
		\item 	If $\mathbb{N}$ is considered with usual metric, then $\mathcal{F}(\mathbb{N})\cong \ell^1(\mathbb{N})$.
		\item $\mathcal{F}(\ell^1(\mathbb{N}))\cong \mathcal{L}^1(\mathbb{R})$.
	\end{enumerate}
\end{example}
In the theory of bounded linear operators between Banach spaces, an operator is
said to be compact if the image of the unit ball under the operator is
precompact  (\cite{FABIAN}). Linearity of the operator now gives various
characterizations of
compactness and plays an important role in rich theories such as theory of integral equations, spectral
theory, theory of Fredholm operators, operator algebra (C*-algebra),
K-theory, Calkin algebra, (operator) ideal theory,  approximation
properties of Banach spaces, Schauder basis theory. Lack of linearity is a
hurdle when one tries to define compactness of non-linear maps. This hurdle was
successfully crossed in the paper (\cite{JIMENEZSEPUILCREMOISES}) which began the study of Lipschitz compact
operators. 
\begin{definition}(\cite{JIMENEZSEPUILCREMOISES})
	If $\mathcal{M}$ is a  metric space and $\mathcal{X}$ is a Banach space, then the  \textbf{Lipschitz image} of a Lipschitz map (also called as \textbf{Lipschitz operator}) $f:\mathcal{M}\rightarrow \mathcal{X}$ is defined as the set 
	\begin{align}\label{LIPSCHITZIMAGE}
	\left\{\frac{f(x)-f(y)}{d(x,y)}:x, y \in \mathcal{M}, x\neq y\right\}.
	\end{align}
\end{definition}
We observe that whenever an operator is linear, the set in (\ref{LIPSCHITZIMAGE}) is simply the image
of the unit sphere.
\begin{definition}(\cite{JIMENEZSEPUILCREMOISES})\label{LIPSCHITZCOMPACTDEFINITION}
	If $(\mathcal{M}, 0)$ is a pointed metric space and $\mathcal{X}$ is a
	Banach space, then a Lipschitz map  $f:\mathcal{M}\rightarrow \mathcal{X}$ such
	that $f(0)=0$ is said to be \textbf{Lipschitz compact} if its Lipschitz image is
	relatively compact  in $\mathcal{X}$, i.e., the closure of the set in 
	(\ref{LIPSCHITZIMAGE}) is compact  in $\mathcal{X}$.
\end{definition}
As showed in (\cite{JIMENEZSEPUILCREMOISES}), there is a large collection of Lipschitz compact operators. To state this, first we need a definition.
\begin{definition}(\cite{CHENZHENG})
	Let $(\mathcal{M}, 0)$ be a pointed metric space and $\mathcal{X}$ be a Banach space. A Lipschitz operator $f:\mathcal{M}\rightarrow \mathcal{X}$ such that $f(0)=0$ is said to be \textbf{strongly Lipschitz p-nuclear} ($1\leq p <\infty$) if there exist operators $A \in \mathcal{B}(\ell^p(\mathbb{N}), \mathcal{X})$, $g \in \operatorname{Lip}_0(\mathcal{M},
	\ell^\infty(\mathbb{N}))$ and a diagonal operator $M_\lambda \in \mathcal{B}(\ell^\infty(\mathbb{N}), \ell^p(\mathbb{N}))$ induced by a sequence $\lambda \in \ell^p(\mathbb{N})$ such that $f=AM_\lambda g$, i.e., the following diagram commutes.
\begin{center}
	\[
	\begin{tikzcd}
	\mathcal{M} \arrow[r, "f"]\arrow[d, "g"]& \mathcal{X} \\
	\ell^\infty(\mathbb{N})\arrow[r, "M_\lambda" ]& \ell^p(\mathbb{N})\arrow[u, "A"
	]
	\end{tikzcd}	
	\]
\end{center}
\end{definition} 
\begin{proposition}(\cite{JIMENEZSEPUILCREMOISES})
	Every strongly Lipschitz p-nuclear operator from a pointed metric space to a Banach space is Lipschitz compact.
\end{proposition}
Since the image of a linear operator is a subspace, the natural definition of finite rank operator is that image is a finite dimensional subspace. The image of Lipschitz map may not be a subspace. Thus care has to be taken while defining rank of such maps.
\begin{definition}(\cite{JIMENEZSEPUILCREMOISES})\label{LFR}
	If $(\mathcal{M}, 0)$ is a pointed metric space and $\mathcal{X}$ is a
	Banach space, then a Lipschitz function $f:\mathcal{M}\rightarrow \mathcal{X}$
	such that $f(0)=0$ is said to have \textbf{Lipschitz finite dimensional rank} if the linear
	hull of its Lipschitz image is a finite dimensional subspace of $\mathcal{X}$.
\end{definition}
\begin{definition}(\cite{JIMENEZSEPUILCREMOISES})\label{FR}
	If $\mathcal{M}$ is a  metric space and $\mathcal{X}$ is a Banach space, then a Lipschitz function $f:\mathcal{M}\rightarrow \mathcal{X}$  is said to have  \textbf{finite dimensional rank} if the linear hull of its  image is a finite dimensional subspace of $\mathcal{X}$.
\end{definition}
Next theorem shows that for pointed metric spaces, Definitions \ref{LFR} and \ref{FR}   are
equivalent. 
\begin{theorem}(\cite{JIMENEZSEPUILCREMOISES, ACHOUR})\label{LIPSCHITCOMPACTIFFLINEAR}
	Let  $(\mathcal{M}, 0)$ be a pointed metric space and $\mathcal{X}$ be a Banach space.  For  a Lipschitz function $f:\mathcal{M}\rightarrow \mathcal{X}$ such that $f(0)=0$, the following are equivalent.
	\begin{enumerate}[label=(\roman*)]
		\item $f$ has Lipschitz finite dimensional rank.
		\item $f$ has finite dimensional rank.
		\item There exist $ f_1, \dots, f_n$ in  $\operatorname{Lip}_0(\mathcal{M}, \mathbb{K})$ and $\tau_1, \dots, \tau_n$ in $\mathcal{X}$ such that
		\begin{align*}
		f(x)=\sum_{k=1}^{n}f_k(x)\tau_k, \quad \forall x \in \mathcal{M}.
		\end{align*} 
	\end{enumerate}	
\end{theorem}
In Hilbert spaces (and not in Banach spaces), every compact operator is approximable by finite rank
operators in the operator norm (cf. \cite{FABIAN}). Following is the definition of approximable
operator for Lipschitz maps.
\begin{definition}(\cite{JIMENEZSEPUILCREMOISES})
	If $(\mathcal{M}, 0)$ is a pointed metric space and $\mathcal{X}$ is a Banach space, then a Lipschitz function $f:\mathcal{M}\rightarrow \mathcal{X}$ such that $f(0)=0$ is said to be \textbf{Lipschitz approximable} if it is the limit in the Lipschitz norm of a sequence of Lipschitz finite rank operators from  $\mathcal{M}$ to  $\mathcal{X}$.
\end{definition}
\begin{theorem}(\cite{JIMENEZSEPUILCREMOISES})\label{LIPSCHITZAPPROMABLEISCOMPACT}
	Every Lipschitz approximable operator from pointed metric space $(\mathcal{M}, 0)$ to  a Banach space $\mathcal{X}$ is Lipschitz compact.
\end{theorem}

\vspace{0.5cm}
{\onehalfspacing \section{OPERATOR-VALUED ORTHONORMAL BASES, RIESZ \\BASES, FRAMES AND BESSEL SEQUENCES IN HILBERT SPACES}
	Through a decade long research, the frame theory  for Hilbert spaces  was extended to larger extent by 
	Kaftal, Larson and Zhang  with the introduction of  the notion of operator-valued frame (OVF) in 2009. In the theory of operator-valued frames, the sequence $\{L_n\}_{n} $ of operators play an important role. These are defined as follows. 	
	\begin{definition}(\cite{KAFTALLARSONZHANG})\label{LJDEFINITION}
		 Given $n \in \mathbb{N}$, we define 
		\begin{align*}
		L_n : \mathcal{H}_0 \ni h \mapsto L_nh\coloneqq e_n\otimes h \in  \ell^2(\mathbb{N}) \otimes \mathcal{H}_0,
		\end{align*}  where $\{e_n\}_{n} $ is  the standard orthonormal basis for $\ell^2(\mathbb{N})$.
	\end{definition}
Following proposition shows properties of the operator $L_n$.	
	\begin{proposition}(\cite{KAFTALLARSONZHANG})\label{LJORTHO}
		The operators $L_n$ in Definition \ref{LJDEFINITION} satisfy the following.
		\begin{enumerate}[label=(\roman*)]
			\item  Each $L_n$ is an  isometry from  $\mathcal{H}_0 $ to $ \ell^2(\mathbb{N}) \otimes \mathcal{H}_0$, and  for
			$  n,m \in \mathbb{N}$ we have  
			\begin{align}\label{LEQUATION}
			L_n^*L_m =
			\left\{
			\begin{array}{ll}
			I_{\mathcal{H}_0 } & \mbox{if } n=m \\
			0 & \mbox{if } n\neq m
			\end{array}
			\right.
			~\text{and} \quad 
			\sum\limits_{n=1}^\infty L_nL_n^*=I_{\ell^2(\mathbb{N})}\otimes I_{\mathcal{H}_0}
			\end{align}
			where  the convergence is in the strong-operator topology.
			\item $L_m^*(\{a_n\}_{n}\otimes y) =a_my,$ $ \forall  \{a_n\}_{n} \in \ell^2(\mathbb{N}), \forall y \in \mathcal{H}_0$, for each $ m $ in $ \mathbb{N}.$
		\end{enumerate}
	\end{proposition}
Orthonormal basis for Hilbert spaces are defined in Definition \ref{ONBDEFINITIONOLE}. Considering Definition \ref{ONBDEFINITIONOLE} and Parseval identity  ((iv) in Theorem \ref{CHARORTHONORMALINTRO}) for orthonormal basis, 	  \cite{SUN1} defined the notion of orthonormal basis for operators.
	\begin{definition}(\cite{SUN1})\label{ONBDEFINITIONOVHS}
		A collection  $ \{F_n \}_{n}$ in $ \mathcal{B}(\mathcal{H}, \mathcal{H}_0)$ is said to be an \textbf{orthonormal basis} or \textbf{a G-basis} in  $ \mathcal{B}(\mathcal{H},\mathcal{H}_0)$ if 
		$$\langle F_n ^*y, F_k^*z\rangle=\delta_{n,k}\langle y, z\rangle  , \quad \forall y, z \in \mathcal{H}_0, ~\forall n , k \in \mathbb{N}$$
		and 
		$$ \sum\limits_{n=1}^\infty\|F_n h\|^2=\|h\|^2, \quad \forall h \in \mathcal{H}.$$
	\end{definition}
	We observe $\langle F_n ^*y, F_k^*z\rangle=\delta_{n ,k}\langle y, z\rangle  , \forall y, z \in \mathcal{H}_0, $ $\forall n , k \in \mathbb{N}$ if and only if $F_n F_k^*=\delta_{n ,k}I_{\mathcal{H}_0} , $ $ \forall n , k \in \mathbb{N}$. Hence if $ \{F_n \}_{n}$ is an orthonormal basis, then $\|F_n \|^2=\|F_n F_n ^*\|=1, \forall n  \in \mathbb{N} $ and   $ \sum_{n=1}^\infty F_n ^*F_n =I_\mathcal{H}$. Further, using Proposition \ref{LJORTHO} we have 
	\begin{align*}
	\left (\sum_{n=1}^\infty A_n^*L^*_n\right)\left(\sum_{k=1}^\infty L_kA_k\right)=I_\mathcal{H}.
	\end{align*}
	Consider the case $\mathcal{H}_0=\mathbb{K}$. For each $n\in \mathbb{N}$, via Riesz representation theorem (cf. \cite{LIMAYE}), there  exists a unique $\tau_n\in \mathcal{H}$ such that $F_nh=\langle h, \tau_n\rangle, \forall h \in \mathcal{H}$. Now first condition in Definition \ref{ONBDEFINITIONOVHS} tells $\langle F_n^*y, F_k^*z\rangle=y\overline{z}\langle \tau_n, \tau_k\rangle=y\overline{z}\delta_{j,k}, \forall j,k\in \mathbb{N},\forall y, z \in \mathbb{K}$ which shows that    $  \{\tau_n\}_{n}$ is orthonormal. Second condition in Definition \ref{ONBDEFINITIONOVHS} says that $\sum_{n=1}^\infty|\langle h, \tau_n\rangle|^2=\|h\|^2,  \forall h \in \mathcal{H}$. Theorem \ref{CHARORTHONORMALINTRO} now tells that $  \{\tau_n\}_{n}$ is an orthonormal basis for $\mathcal{H}$.  Hence Definition \ref{ONBDEFINITIONOVHS} generalizes the definition of orthonormal basis.
	\begin{example}\label{ONBOVF}
		\begin{enumerate}[label=(\roman*)]
			\item (\cite{SUN1}) Let $ \{\tau_n\}_{n}$ be  an orthonormal basis    for   $\mathcal{H}$. Define $F_n : \mathcal{H} \ni  h \mapsto \langle h, \tau_n\rangle \in \mathbb{K} $, for each $ n \in \mathbb{N}$. Then   $ \{F_n\}_{n } $  is an operator-valued  orthonormal basis in    $ \mathcal{B}(\mathcal{H}, \mathbb{K})$.
			\item If $ U: \mathcal{H}\rightarrow \mathcal{H}_0$ is unitary, then $\{U\}$ is an orthonormal basis in  $ \mathcal{B}(\mathcal{H}, \mathcal{H}_0)$.
			\item \label{CUNTZ} Let $n\geq 2 $ and $A_1, \dots, A_n$ be $n$ isometries on $\ell^2(\mathbb{N})$ such that $A_1A_1^*+\cdots +A_nA_n^*=I_{\ell^2(\mathbb{N})}$ (these are Cuntz algebra generators (\cite{CUNTZ})). We then have 
			\begin{align*}
			(A_jA_j^*)^2=A_j(A_j^*A_j)A_j^*=A_jI_\mathcal{H}A_j^*=A_jA_j^*.
			\end{align*} 
			Hence $A_jA_j^*$'s are projections. Further $A_j^*A_k=0$, $ \forall j\neq k$. 
			 Therefore $ \{A_j^*\}_{j=1}^n $  is an operator-valued  orthonormal basis in    $ \mathcal{B}(\ell^2(\mathbb{N}))$.
			\item Equation (\ref{LEQUATION}) says  that $ \{L^*_n\}_{n}$  is an orthonormal basis in $ \mathcal{B}(\ell^2(\mathbb{N})\otimes\mathcal{H}_0, \mathcal{H}_0)$.
		\end{enumerate}
	\end{example}
Using Theorem \ref{RIESZBASISTHM}, \cite{SUN1} defined the notion of Riesz  basis for operators.
	\begin{definition}(\cite{SUN1})
	A	collection $\{A_n\}_{n}$ in $ \mathcal{B}(\mathcal{H}, \mathcal{H}_0)$ is 	said to be an  \textbf{operator-valued   Riesz basis} or \textbf{G-Riesz basis} in $ \mathcal{B}(\mathcal{H}, \mathcal{H}_0)$ if it satisfies the following.	
	\begin{enumerate}[label=(\roman*)]
		\item $\{h \in \mathcal{H}:A_nh=0, \forall n \in \mathbb{N}\}=\{0\}$.
		\item There exist $a,b>0$ such that for every finite subset $\mathbb{S}$ of $\mathbb{N}$, 
		\begin{align*}
		a\sum_{n \in \mathbb{S}}\|h_n\|^2\leq \left\|\sum_{n \in \mathbb{S}}A_n^*h_n\right\|^2 \leq  b \sum_{n \in \mathbb{S}}\|h_n\|^2, \quad\forall h_n \in \mathcal{H}_0.
		\end{align*}
	\end{enumerate}	
	\end{definition}
	\begin{example}
		\begin{enumerate}[label=(\roman*)]
			\item (\cite{SUN1}) Let $ \{\tau_n\}_{n}$ be  a Riesz  basis    for   $\mathcal{H}$. Define $A_n : \mathcal{H} \ni  h \mapsto \langle h, \tau_n\rangle \in \mathbb{K} $, for each $ n \in \mathbb{N}$.    Then $A_n^*y=y\tau_n, \forall y \in \mathbb{K}$. Now from Theorem \ref{RIESZBASISTHM},  $ \{A_n\}_{n} $  is an operator-valued  Riesz  basis in    $ \mathcal{B}(\mathcal{H}, \mathbb{K})$.
			\item If $ U: \mathcal{H}\rightarrow \mathcal{H}_0$ is invertible, then $\{U\}$ is an operator-valued  Riesz basis in  $ \mathcal{B}(\mathcal{H}, \mathcal{H}_0)$.
			\item Let $A_1,\dots, A_n$ be as in \ref{CUNTZ} of Example \ref{ONBOVF} and  let $ A,B \in \mathcal{B}(\mathcal{H})$ be invertible. Then $ \{AA_j^*B\}_{j=1}^n $  is an operator-valued  Riesz basis in    $ \mathcal{B}(\mathcal{H})$. In fact, if $AA_j^*Bh=0, \forall j$, then $A_j^*Bh=0, \forall j $ which gives
			\begin{align*}
			h=B^{-1}Bh=B^{-1}\left(\sum_{k=1}^nA_kA_k^*Bh\right)=0
			\end{align*}
			 and  every finite subset $ \mathbb{S}$ of $ \mathbb{N}$, 
			\begin{align*}
			(\|B^{-1}\|\|A^{-1}\|)^{-1}\sum\limits_{j\in \mathbb{S}}\|h_j\|^2\leq \left\| \sum\limits_{j\in\mathbb{S}}B^*A_jA^*h_j  \right \|^2 \leq \|B\|\|A\|\sum\limits_{j\in \mathbb{S}}\|h_j\|^2,
			\end{align*}
			for all $h_j \in \mathcal{H}$.
		\end{enumerate}
	\end{example}
	 \cite{SUN1} derived the following result which tells that we can define the notion of operator-valued Riesz basis in a manner similar to Definition \ref{RIESZBASISDEFINITION}. 
	\begin{theorem}(\cite{SUN1})
	A	collection $\{A_n\}_{n}$ in $ \mathcal{B}(\mathcal{H}, \mathcal{H}_0)$ is 	an operator-valued   Riesz basis  if and only if  there exist an operator-valued orthonormal basis  $ \{F_n\}_{n}$ in $ \mathcal{B}(\mathcal{H}, \mathcal{H}_0)$ and an invertible $T\in \mathcal{B}(\mathcal{H})$ such that $A_n=F_nU, \forall n$.	
	\end{theorem}
	Historically, many generalizations of frames for Hilbert spaces are proposed such as frames for subspaces (\cite{CASAZZASUBSPACE}), fusion frames (\cite{CASAZZAFUSION}), outer frames (\cite{OUTER}), oblique frames (\cite{CHRISTENSENOBLIQUE}), pseudo frames (\cite{LIPSEUDO}), quasi-projectors (\cite{FORNASIERQUASI}). It was in 2006, when Sun gave the definition of G-frame which unified all these notions of frames for Hilbert spaces. 
	\begin{definition}(\cite{SUN1})\label{SUNDEF}
		A collection  $ \{A_n\}_{n} $  in $ \mathcal{B}(\mathcal{H}, \mathcal{H}_0)$ is said to be a \textbf{G-frame} in  $ \mathcal{B}(\mathcal{H}, \mathcal{H}_0)$ if there exist $ a, b >0$ such that 
		\begin{align*}
			a\|h\|^2\leq\sum_{n=1}^\infty\|A_nh\|^2 \leq b\|h\|^2  ,\quad \forall h \in \mathcal{H}.	
		\end{align*}
	\end{definition}
 Basic idea for the notion of OVF is the following.   Definition \ref{OLE} can be written in an equivalent form as 
	\begin{align}\label{SEQUENTIALEQUATION2}
		&\text{the map}~ S_\tau: \mathcal{H} \ni h \mapsto \sum_{n=1}^\infty\langle h, \tau_n\rangle \tau_n \in \mathcal{H} ~\text{is a well-defined bounded} \nonumber\\
		&\text{positive invertible operator}.
	\end{align}
	If we now define $ A_n : \mathcal{H} \ni  h \mapsto \langle h, x_n\rangle \in \mathbb{K} $, for each $ n \in \mathbb{N}$, then  Statement (\ref{SEQUENTIALEQUATION2}) can be rewritten as 
	\begin{align}\label{SEQUENTIALEQUATION3}
		&\sum_{n=1}^\infty A_n^*A_n ~\text{converges in the strong-operator topology  on } \mathcal{B}(\mathcal{H}) \text{ to a} \nonumber\\
		&\text{ bounded positive invertible operator.}
	\end{align}
	Now Statement (\ref{SEQUENTIALEQUATION3}) leads to 
	\begin{definition}(\cite{KAFTALLARSONZHANG})\label{KAFTAL}
		A collection  $ \{A_n\}_{n} $  in $ \mathcal{B}(\mathcal{H}, \mathcal{H}_0)$ is said to be an \textbf{operator-valued frame} (OVF) in $ \mathcal{B}(\mathcal{H}, \mathcal{H}_0)$ if the series 
		\begin{align*}
			\text{(\textbf{Operator-valued frame operator})}\quad 	S_A\coloneqq \sum_{n=1}^\infty A_n^*A_n
		\end{align*}
		converges in the strong-operator topology on $ \mathcal{B}(\mathcal{H})$ to a  bounded invertible operator.\\
		Constants $ a$ and $ b$ are called as lower and upper frame bounds, respectively. Supremum (resp. infimum) of the set of all lower (resp. upper) frame bounds is called optimal lower (resp. upper) frame bound. If the optimal frame bounds are equal, then the frame is called as tight operator-valued frame. A tight operator-valued frame whose optimal frame bound is one is termed as Parseval operator-valued frame.
	\end{definition}
Before proceeding, we first show that Definitions \ref{SUNDEF} and \ref{KAFTAL} are equivalent. For this, we first need a result. 
\begin{theorem}(\cite{SUN1})\label{SUNIMPO}
	If $ \{A_n\}_{n} $  in $ \mathcal{B}(\mathcal{H}, \mathcal{H}_0)$  is  a G-frame, then the map $S_A:\mathcal{H}\ni h  \mapsto \sum_{n=1}^\infty A^*_nA_nh \in \mathcal{H}$ is a well-defined bounded linear invertible operator.
\end{theorem}
Following theorem will reflect the equivalence of  notions of OVF and G-frame. 
\begin{theorem}\label{OVFIFANDONLYIFGFRAME}
	A collection $ \{A_n\}_{n} $  in $ \mathcal{B}(\mathcal{H}, \mathcal{H}_0)$  is  a  G-frame in $ \mathcal{B}(\mathcal{H},  \mathcal{H}_0)$ if and only if $ \{A_n\}_{n}$ is an operator-valued frame in $ \mathcal{B}(\mathcal{H}, \mathcal{H}_0)$.
\end{theorem}
\begin{proof}
	$(\Rightarrow)$ Theorem \ref{SUNIMPO} says that $S_A$ is a bounded linear invertible operator. Since $A_n^*A_n\geq0, \forall n \in \mathbb{N}$,  $S_A$ is positive. Hence $ \{A_n\}_{n}$ is an OVF.\\
	$(\Leftarrow)$ Since 	$\sum_{n=1}^\infty A^*_nA_n$ is positive invertible, there are $a,b>0$ such that $aI_\mathcal{H}\leq\sum_{n=1}^\infty A^*_nA_n $ $\leq b I_\mathcal{H}$. This implies $a\|h\|^2\leq \langle\sum_{n=1}^\infty A^*_nA_nh, h\rangle =\sum_{n=1}^\infty\|A_nh\|^2\leq b \|h\|^2, \forall h \in \mathcal{H}$. Hence $ \{A_n\}_{n}$ is a G-frame.
\end{proof}
\begin{remark}
	Even though Sun's paper (\cite{SUN1}) published earlier than the paper  by \cite{KAFTALLARSONZHANG}, it is mentioned in the  introduction of paper  (\cite{KAFTALLARSONZHANG})  that authors of   paper (\cite{KAFTALLARSONZHANG}) started  the work in January 1999.
\end{remark}
\begin{example}
	\begin{enumerate}[label=(\roman*)]
		\item (\cite{SUN1}) Let $ \{\tau_n\}_{n}$ be  a  frame for   $\mathcal{H}$. Define $A_n : \mathcal{H} \ni  h \mapsto \langle h, \tau_n\rangle \in \mathbb{K} $, for each $ n \in \mathbb{N}$. Then $A_n^*y=y\tau_n, \forall y \in \mathbb{K}$. Now from Theorem \ref{MOSTIMPORTANT}, the map $\mathcal{H}\ni h \mapsto \sum_{n=1}^\infty A_n^*A_nh=\sum_{n=1}^\infty \langle h, \tau_n\rangle \tau_n\in \mathcal{H}$ is a well-defined positive invertible operator. Hence  $ \{A_n\}_{n} $  is an operator-valued frame in    $ \mathcal{B}(\mathcal{H}, \mathbb{K})$.
		\item  If $ A: \mathcal{H}\rightarrow \mathcal{H}_0$ is a  bounded below linear operator, then $\{A\}$ is an operator-valued frame in  $ \mathcal{B}(\mathcal{H}, \mathcal{H}_0)$.
		\item Let $A_1,\dots, A_n$ be as in \ref{CUNTZ} of Example \ref{ONBOVF} and  let $ A,B \in \mathcal{B}(\mathcal{H})$ be bounded below. Then $ \{AA_j^*B\}_{j=1}^n $  is an operator-valued  frame in    $ \mathcal{B}(\mathcal{H})$.
	\end{enumerate}
\end{example}
	The fundamental tool used in the study of OVF is the factorization of frame operator $S_A$. This and other important properties of OVFs are stated in the following theorem.  
\begin{theorem}(\cite{KAFTALLARSONZHANG}, \cite{SUN1})
		Let  $ \{A_n\}_{n} $  be an OVF in $ \mathcal{B}(\mathcal{H}, \mathcal{H}_0)$. Then
		\begin{enumerate}[label=(\roman*)]
			\item 	 $\overline{\operatorname{span}} \cup_{n=1}^\infty A_n^*(\mathcal{H}_0)=\mathcal{H}$.
			\item The analysis operator 
			\begin{align*}
				\theta_A:\mathcal{H} \ni h \mapsto   \theta_A h\coloneqq\sum_{n=1}^\infty L_nA_n h \in \ell^2(\mathbb{N}) \otimes \mathcal{H}_0
			\end{align*}
			is a well-defined bounded  linear  operator. Further, $\sqrt{a}\|h\|\leq \|\theta_A h\|\leq \sqrt{b}\|h\|, \forall h \in \mathcal{H}$. In particular, $\theta_A$ is injective and its range is closed.
		\item We have 
			\begin{align*}
			a\|h\|^2\leq \langle S_A h, h\rangle\leq b\|h\|^2, \forall h \in \mathcal{H}\text{ and } a\|h\|\leq \|S_A h\|\leq b\|h\|, \forall h \in \mathcal{H}.
			\end{align*}
			\item  $h=\sum_{n=1}^\infty(A_nS_A^{-1})^*A_nh=\sum_{n=1}^\infty A_n^*(A_nS_A^{-1})h, \forall h \in \mathcal{H}$.
		\item $\theta_A^*z=	\sum\limits_{n=1}^\infty A_n^*L_n^*z $, $ \forall z  \in \ell^2(\mathbb{N})\otimes \mathcal{H}_0$.
			\item Frame operator 
			factors  as $S_A=\theta_A^*\theta_A.$
	\item $\theta_A^*$ is surjective.
				\item $\|S_A^{-1}\|^{-1}$ is the optimal lower frame bound and  $\|S_A\|=\|\theta_A\|^2$ is the optimal upper frame bound.
				\item $ P_A \coloneqq \theta_A S_A^{-1} \theta_A^*:\ell^2(\mathbb{N})\otimes \mathcal{H}_0 \to \ell^2(\mathbb{N})\otimes \mathcal{H}_0$ is an orthogonal  projection onto $ \theta_A(\mathcal{H})$.
			\item $  \{A_n\}_{n}$ is Parseval if and only if $\theta_A$ is an isometry if and only if $\theta_A\theta_A^*$ is a projection.
			\item $ \{A_nS_A^{-1}\}_{n}$  is an OVF in  $ \mathcal{B}(\mathcal{H}, \mathcal{H}_0)$  with bounds $b^{-1}$ and $a^{-1}$.
			\item $ \{A_nS_A^{-1/2}\}_{n}$  is a Parseval OVF in  $ \mathcal{B}(\mathcal{H}, \mathcal{H}_0)$. 
			\item (\textbf{Best approximation}) If $ h \in \mathcal{H}$ has representation  $ h=\sum_{n=1}^\infty A_n^*y_n,$ for some  sequence  $ \{y_n\}_{n}$ in $\mathcal{H}_0$,  then 
			$$ \sum\limits_{n=1}^\infty \|y_n\|^2 =\sum\limits_{n=1}^\infty \|A_nS_A^{-1}h\|^2+\sum\limits_{n=1}^\infty \| y_n-A_nS_A^{-1}h\|^2. $$
		\end{enumerate}	
	\end{theorem}
Similar to the notion of duality, orthogonality and similarity for frames in Hilbert spaces, there are similar notions for operator-valued frames. We now recall these notions and mention some results.
\begin{definition}(\cite{KAFTALLARSONZHANG})
	An  OVF  $ \{B_n\}_{n}$  in $\mathcal{B}(\mathcal{H}, \mathcal{H}_0)$ is said to be a \textbf{dual} for an   OVF $   \{A_n\}_{n}$ in $\mathcal{B}(\mathcal{H}, \mathcal{H}_0)$  if  $ \sum_{n=1}^\infty B^*_nA_n=I_{\mathcal{H}}.$
\end{definition}
\begin{definition}(\cite{KAFTALLARSONZHANG})
	An   OVF  $ \{B_n\}_{n}$  in $\mathcal{B}(\mathcal{H}, \mathcal{H}_0)$ is said to be \textbf{orthogonal} to an   OVF  $   \{A_n\}_{n} $ in $\mathcal{B}(\mathcal{H}, \mathcal{H}_0)$ if $ \sum_{n=1}^\infty B^*_nA_n=0. $
\end{definition}  
\begin{proposition}(\cite{KAFTALLARSONZHANG})
	Let $   \{A_n\}_{n}$ and $ \{B_n\}_{n}$ be  two Parseval  OVFs in   $\mathcal{B}(\mathcal{H}, \mathcal{H}_0)$ which are  orthogonal. If $C,D,E,F \in \mathcal{B}(\mathcal{H})$ are such that $ C^*C+D^*D=I_\mathcal{H}$, then  $	\{A_nC+B_nD\}_{n} $ is a  Parseval   OVF in  $\mathcal{B}(\mathcal{H}, \mathcal{H}_0)$. In particular,  if scalars $ c,d,$ satisfy $|c|^2+|d|^2 =1$, then $ \{cA_n+dB_n\}_{n} $ is   a Parseval   OVF.
\end{proposition} 
\begin{proposition}(\cite{KAFTALLARSONZHANG})
	If $   \{A_n\}_{n},$  and $ \{B_n\}_{n} $ are   orthogonal  OVFs in $ \mathcal{B}(\mathcal{H}, \mathcal{H}_0)$, then  $\{A_n\oplus B_n\}_{n}$ is an   OVF in $ \mathcal{B}(\mathcal{H}\oplus \mathcal{H}, \mathcal{H}_0).$    Further, if both $ \{A_n\}_{n} $  and $ \{B_n\}_{n} $ are  Parseval, then $\{A_n\oplus B_n\}_{n}$ is Parseval.
\end{proposition}

\begin{definition}(\cite{KAFTALLARSONZHANG})\label{SIMILARITYOVFKAFTALLARSONZHANG}
	An  OVF   $ \{B_n\}_{n} $  in $ \mathcal{B}(\mathcal{H}, \mathcal{H}_0)$    is said to be \textbf{similar} or \textbf{equivalent}  to an  OVF  $   \{A_n\}_{n} $ in $ \mathcal{B}(\mathcal{H}, \mathcal{H}_0)$  if there exists a bounded  invertible  $ R_{A,B} \in \mathcal{B}(\mathcal{H})$   such that $	B_n=A_nR_{A,B} ,  \forall n \in \mathbb{N}.$
\end{definition}
Similar frames share some  nice properties that knowing analysis, synthesis and frame operators of one give that of another.
\begin{lemma}(\cite{KAFTALLARSONZHANG})
	Let $   \{A_n\}_{n} $ and  $   \{B_n\}_{n}$ be similar  OVFs  and   $B_n=A_nR_{A,B} ,  \forall n \in \mathbb{N}$, for some invertible $ R_{A,B}  \in \mathcal{B}(\mathcal{H}).$ Then 
	\begin{enumerate}[label=(\roman*)]
		\item $ \theta_B=\theta_A R_{A,B}$.
		\item $S_{B}=R_{A,B}^*S_{A}R_{A,B}$.
		\item $P_{B}=P_{A}.$
	\end{enumerate}
\end{lemma}
There is a complete classification of  similarity using operators.
\begin{theorem}(\cite{KAFTALLARSONZHANG})
	For two  OVFs  $   \{A_n\}_{n} $ and  $   \{B_n\}_{n} $, the following are equivalent.
	\begin{enumerate}[label=(\roman*)]
		\item $B_n=A_nR_{A,B}  ,  \forall n \in \mathbb{N},$ for some invertible  $ R_{A,B}  \in \mathcal{B}(\mathcal{H}). $
		\item $\theta_B=\theta_AR_{A,B}  $ for some invertible  $ R_{A,B}  \in \mathcal{B}(\mathcal{H}). $
		\item $P_{B}=P_{A}.$
	\end{enumerate}
	If one of the above conditions is satisfied, then  invertible operators in  $ \operatorname{(i)}$ and  $ \operatorname{(ii)}$ are unique and are given by $R_{A,B}=S_{A}^{-1}\theta_A^*\theta_B$. In the case that $   \{A_n\}_{n} $ is Parseval, then $  \{B_n\}_{n}$ is  Parseval if and only if $R_{A,B}$ is unitary.  
\end{theorem}
In the study of frames, rather indexing with natural numbers or other indexing sets, it is more useful in some cases to study frames indexed by groups and generated by a single operator. 

 Let $ G$ be a discrete topological group and  $ \{\chi_g\}_{g\in G}$ be the  standard orthonormal  basis for $\ell^2(G) $.  Let $\lambda $ be the left regular representation of $ G$ defined by $ \lambda_g\chi_q(r)=\chi_q(g^{-1}r), $ $ \forall  g, q, r \in G$ and     $\rho $ be the right regular representation of $ G$ defined by $ \rho_g\chi_q(r)=\chi_q(rg), $ $\forall g, q, r \in G.$ By $\mathscr{L}(G) $, we mean  the von Neumann algebra generated by unitaries $\{\lambda_g\}_{g\in G} $ in $ \mathcal{B}(\ell^2(G))$. Similarly $\mathscr{R}(G) $ denotes the von Neumann algebra generated by $\{\rho_g\}_{g\in G} $ in $ \mathcal{B}(\ell^2(G))$. We  recall that $\mathscr{L}(G)'=\mathscr{R}(G)$, $ \mathscr{R}(G)'=\mathscr{L}(G) $ (cf. \cite{CONWAY}), where $\mathcal{A}'$ denotes the commutant of $\mathcal{A}\subseteq  \mathcal{B}(\mathcal{H})$.
\begin{definition}(\cite{KAFTALLARSONZHANG})
	Let $ \pi$ be a unitary representation of a discrete 
	group $ G$ on  a Hilbert space $ \mathcal{H}.$ An operator $ A$ in $ \mathcal{B}(\mathcal{H}, \mathcal{H}_0)$ is called a    \textbf{operator frame generator} (resp. a  Parseval frame generator) w.r.t. an operator $ \Psi$ in $ \mathcal{B}(\mathcal{H}, \mathcal{H}_0)$ if $\{A_g\coloneqq A \pi_{g^{-1}}\}_{g\in G}$ is a factorable weak OVF (resp.  Parseval)  in $ \mathcal{B}(\mathcal{H}, \mathcal{H}_0)$. In this case, we say $ A$ is an operator  frame generator for $\pi$.
\end{definition}
Frames generated by groups have the remarkable property that the frame operator belongs to the commutant of $\pi(G)$. These and other properties are given in the following proposition.
\begin{proposition}(\cite{KAFTALLARSONZHANG})
	Let $ A$ and $ B$ be   operator frame generators    in $\mathcal{B}(\mathcal{H},  \mathcal{H}_0)$ for a unitary representation $ \pi$ of  $G$ on $ \mathcal{H}.$ Then
	\begin{enumerate}[label=(\roman*)]
		\item $ \theta_A\pi_g=(\lambda_g\otimes I_{\mathcal{H}_0})\theta_A,  \theta_B \pi_g=(\lambda_g\otimes I_{\mathcal{H}_0})\theta_B,  \forall g \in G.$
		\item $ \theta_A^*\theta_B$ is in the commutant $ \pi(G)'$ of $ \pi(G)''.$ Further, $ S_{A} \in \pi(G)'$. 
		\item $ \theta_AT\theta_A^*, \theta_AT\theta_B^* \in \mathscr{R}(G)\otimes \mathcal{B}(\mathcal{H}_0), \forall T \in \pi(G)'.$ In particular, $ P_{A} \in \mathscr{R}(G)$ $\otimes \mathcal{B}(\mathcal{H}_0). $
	\end{enumerate}
\end{proposition}
Following theorem gives a characterization of frames generated by unitary representation without using representation.
\begin{theorem}(\cite{KAFTALLARSONZHANG})
	Let $ G$ be a discrete group,  $ e$ be the identity of $G$ and $ \{A_g\}_{g\in G}$ be a Parseval  OVF  in $ \mathcal{B}(\mathcal{H},\mathcal{H}_0).$ Then there is a  unitary representation $ \pi$  of $ G$ on  $ \mathcal{H}$  for which 
	$$ A_g=A_e\pi_{g^{-1}},  \quad\forall  g \in G$$
	if and only if 
	$$A_{gp}A_{gq}^*=A_pA_q^* ,\quad    \forall g,p,q \in G.$$
\end{theorem} 
One of the most important results obtained by \cite{KAFTALLARSONZHANG} is the connectedness of the set of all generators of operator-valued frames generated by groups.
\begin{theorem}(\cite{KAFTALLARSONZHANG})\label{KAFTALPATHCONNECTED}
Let $\pi$  be a unitary representation of a discrete group $G$
on $ \mathcal{H}$. Suppose 
\begin{align*}
\emptyset \neq \mathscr{F}_{G}\coloneqq\{&A \in \mathcal{B}(\mathcal{H},\mathcal{H}_0): \{A\pi_{g^{-1}}\}_{g \in G} ~\text{is an  operator-valued frame  in}\\
& \mathcal{B}(\mathcal{H},\mathcal{H}_0)\}.
\end{align*}
\begin{enumerate}[label=(\roman*)]
\item  If $ \dim \mathcal{H}_0<\infty$, then $ \mathscr{F}_{G}$  is \textbf{path-connected}  in the \textbf{operator-norm topology} on $\mathcal{B}(\mathcal{H},\mathcal{H}_0)$.
\item If $ \dim \mathcal{H}_0=\infty$, then $ \mathscr{F}_{G}$ is path-connected  in the operator-norm topology on $\mathcal{B}(\mathcal{H},\mathcal{H}_0)$ if and only if the von Neumann algebra $\mathscr{R}(G)$ generated by the right regular representations of $G$ is \textbf{diffuse} (i.e., $\mathscr{R}(G)$ has no nonzero minimal projections).
\end{enumerate}
\end{theorem}
Stability result for OVFs is due to \cite{SUNSTABILITY}.
\begin{theorem}(\cite{SUNSTABILITY})
 	Let $ \{A_n\}_{n} $  be an OVF in $ \mathcal{B}(\mathcal{H}, \mathcal{H}_0) $ with frame bounds $ a$ and $b$. Suppose  $ \{B_n\}_{n } $ in $ \mathcal{B}(\mathcal{H}, \mathcal{H}_0) $ is such that  there exist $\alpha, \beta, \gamma \geq 0  $ with $ \max\{\alpha+\frac{\gamma}{\sqrt{a}}, \beta\}<1$ 
 	and for all $m=1,2, \dots, $
 	\begin{align*}
 	\left\|\sum\limits_{n=1}^m(A_n^*-B_n^*)y_n\right\|\leq \alpha\left\|\sum\limits_{n=1}^mA_n^*y_n\right\|+\beta\left\|\sum\limits_{n=1}^mB_n^*y_n\right\|+\gamma \left(\sum\limits_{n=1}^m\|y_n\|^2\right)^\frac{1}{2},\nonumber
 	\quad \forall y_n \in  \mathcal{H}_0.
 	\end{align*} 
 Then  $\{B_n\}_{n} $ is an operator-valued frame in $ \mathcal{B}(\mathcal{H}, \mathcal{H}_0) $   with frame bounds
 	\begin{align*}
 	a\left(1-\frac{\alpha+\beta+\frac{\gamma}{\sqrt{a}}}{1+\beta}\right)^2 \text{~and~} b\left(1+\frac{\alpha+\beta+\frac{\gamma}{\sqrt{b}}}{1-\beta}\right)^2.
 	\end{align*} 
 \end{theorem}
Like Bessel sequences for Hilbert spaces, there is a similar notion for operators.
\begin{definition}(\cite{SUN1})
A collection $\{A_n\}_{n} $ in $ \mathcal{B}(\mathcal{H}, \mathcal{H}_0) $  is said to be an \textbf{operator-valued Bessel sequence} (or \textbf{G-Bessel sequence}) if there exists $b>0$ such that 
\begin{align*}
\text{ (\textbf{Operator-valued   Bessel's inequality}) }\quad  \sum\limits_{n=1}^\infty \|A_nh\|^2\leq b\|h\|^2,\quad  \forall h \in \mathcal{H}.
\end{align*}
Constant $b$ is called as a Bessel bound for $ \{A_n\}_{n }$.
\end{definition}
Similar to Theorem \ref{OVFIFANDONLYIFGFRAME}    we have the following result.
\begin{theorem}
A collection $\{A_n\}_{n} $ in $ \mathcal{B}(\mathcal{H}, \mathcal{H}_0) $ is an operator-valued Bessel sequence if and only if   the series $ \sum_{n} A_n^*A_n$  converges in the strong-operator topology on $ \mathcal{B}(\mathcal{H})$ to a  bounded  operator.
\end{theorem}
Following are some examples of operator-valued Bessel sequences.
\begin{example}
\begin{enumerate}[label=(\roman*)]
	\item 	 (\cite{SUN1}) Let $ \{\tau_n\}_{n}$ be  a  Bessel sequence  for   $\mathcal{H}$. Define $A_n : \mathcal{H} \ni  h \mapsto \langle h, \tau_n\rangle \in \mathbb{K} $, for each $ n \in \mathbb{N}$. Then $A_n^*y=y\tau_n, \forall y \in \mathbb{K}$. Now from Theorem \ref{OLEBESSELCHARACTERIZATION12}, the map $\mathcal{H}\ni h \mapsto \sum_{n=1}^\infty A_n^*A_nh=\sum_{n=1}^\infty \langle h, \tau_n\rangle \tau_n\in \mathcal{H}$ is a well-defined positive  operator. Hence  $ \{A_n\}_{n} $  is an operator-valued Bessel sequence  in    $ \mathcal{B}(\mathcal{H}, \mathbb{K})$.
	\item From operator-norm inequality, we see that any finite collection of operators is an operator-valued Bessel sequence.
	\end{enumerate}
\end{example}




{\onehalfspacing \chapter{FRAMES FOR METRIC SPACES}\label{chap2} }
\vspace{0.5cm}
{\onehalfspacing \section{BASIC PROPERTIES}
In this chapter, we define frames for metric spaces and derive several fundamental properties. 
 \begin{definition}\label{PFRAMEFORMETRIC}(\textbf{p-frame for metric space})
	Let $\mathcal{M}$ be a metric space and $p \in [1,\infty)$. A collection $\{f_n\}_{n}$ of Lipschitz functions in    $ \operatorname{Lip}(\mathcal{M}, \mathbb{K})$  is said to be a \textbf{metric p-frame} or \textbf{Lipschitz p-frame} for  $\mathcal{M}$ if there exist $a,b>0$ such that 
	\begin{align*}
	a\,d(x,y)\leq \left(\sum_{n=1}^{\infty}|f_n(x)-f_n(y)|^p\right)^\frac{1}{p}\leq b\,d(x,y),\quad \forall x, y \in \mathcal{M}.
	\end{align*}
	If we do not demand the first inequality, then we say $\{f_n\}_{n}$ is a  metric p-Bessel sequence for  $\mathcal{M}$.
\end{definition}
We now see that whenever $\mathcal{M}$ is a Banach space and $f_n$'s are linear functionals, then Definition \ref{PFRAMEFORMETRIC} reduces to Definition \ref{FRAMEDEFINITIONBANACH}.  We now give  various examples.
\begin{example}
	Let $\{f_n\}_{n}$ be a p-frame for a Banach space $\mathcal{X}$. Choose any bi-Lipschitz function $A:\mathcal{X}\to \mathcal{X}$. Then it follows that $\{f_nA\}_{n}$ is a metric  p-frame for $\mathcal{X}$.
\end{example}
\begin{example}\label{1FRAMEFIRST}
	Let $1<a<\infty.$ Let us take $\mathcal{M}\coloneqq[a,\infty)$ and define $f_n:\mathcal{M}\to \mathbb{R}$ by 
	\begin{align*}
	f_0(x)&\coloneqq 1, \quad \forall x \in \mathcal{M}\\
	f_n(x)&\coloneqq \frac{(\log x)^n}{n!}, \quad \forall x \in \mathcal{M}, \forall n\geq 1.
	\end{align*}
	Then $f_n'(x)=\frac{(\log x)^{(n-1)}}{(n-1)!x}$, $\forall x \in \mathcal{M}, \forall n\geq1.$ Since   $f_n'$ is bounded on $\mathcal{M}$, $\forall n\geq1$, it follows that $f_n$ is Lipschitz on $\mathcal{M}$, $\forall n\geq1$. For $x, y \in \mathcal{M}$ with $x<y$, we now see that 
	\begin{align*}
	\sum_{n=0}^{\infty}|f_n(x)-f_n(y)|&=\sum_{n=0}^{\infty}\left|\frac{(\log
		x)^n}{n!}-\frac{(\log y)^n}{n!}\right|=\sum_{n=0}^{\infty}\frac{(\log
		y)^n}{n!}-\sum_{n=0}^{\infty}\frac{(\log x)^n}{n!}
	\\
	&=e^{\log y}-e^{\log x}=y-x=|x-y|.
	\end{align*}
	Hence $\{f_n\}_n$ is a metric 1-frame for $\mathcal{M}$.
\end{example}
\begin{example}\label{1FRAMESECOND}
	Let $1<a<b<\infty.$ We  take $\mathcal{M}\coloneqq[\frac{1}{1-a},\frac{1}{1-b}]$ and define $f_n:\mathcal{M}\to \mathbb{R}$ by 
	\begin{align*}
	f_n(x)&\coloneqq\left(1-\frac{1}{x}\right)^n, \quad \forall x \in \mathcal{M}, \forall n \geq 0.
	\end{align*}
	Then $f_n'(x)=\frac{n}{-x^2}\left(1-\frac{1}{x}\right)^{n-1}$, $\forall x \in \mathcal{M}, \forall n\geq1.$ Therefore $f_n$ is a Lipschitz function, for each $n\geq1.$ We now see that $\{f_n\}_n$ is a metric 1-frame for $\mathcal{M}$. In fact, for  $x, y \in \mathcal{M},$ with $x<y$, we have
	\begin{align*}
	\sum_{n=0}^{\infty}|f_n(x)-f_n(y)|&=\sum_{n=0}^{\infty}\left|\left(1-\frac{1}{x}\right)^n-\left(1-\frac{1}{y}\right)^n\right|=\sum_{n=0}^{\infty}\left(1-\frac{1}{y}\right)^n-\sum_{n=0}^{\infty}\left(1-\frac{1}{x}\right)^n
	\\
	&=y-x=|x-y|.
	\end{align*}
\end{example}
\begin{example}\label{LINEARGOOD}
	Let $\{f_n\}_{n}$ be a p-frame for a Banach space $\mathcal{X}$. Let $\phi: \mathbb{K}\to  \mathbb{K}$ be bi-Lipschitz  and define   $g_n (x)\coloneqq \phi (f_n(x)), $ $\forall x \in \mathcal{X}, $  $\forall n \in \mathbb{N}$. It then follows that $\{g_n\}_n  $ is a metric p-frame for $\mathcal{X}$.
\end{example}
By looking at Theorem \ref{pFRAMECHAR} we can ask whether there is a result similar for metric p-frames. We answer this partially through the following theorem.
\begin{theorem}\label{PBESSELCHAR}
	Let $(\mathcal{M},0)$ be  a pointed metric  space and $\{f_n\}_{n}$ be a sequence in $\operatorname{Lip}_0(\mathcal{M},
	\mathbb{K})$.  Then $\{f_n\}_{n}$ is a metric p-Bessel sequence for 
	$\mathcal{M}$ with bound $b$ if and only if 
	\begin{align}\label{LIPBASSELOPERATORCHARACTERIZATION}
	&T: \ell^q (\mathbb{N})\ni \{a_n\}_{n} \mapsto T\{a_n\}_{n} \in \operatorname{Lip}_0(\mathcal{M}\times \mathcal{M},
	\mathbb{K}),\\
	&T\{a_n\}_{n}: \mathcal{M}\times \mathcal{M} \ni (x,y) \mapsto  \sum_{n=1}^\infty a_n(f_n(x)-f_n(y)) \in \mathbb{K} \nonumber
	\end{align}
	is a well-defined (hence bounded) linear operator and $\|T\|\leq b$ (where $q$ is the conjugate
	index of $p$).
	\end{theorem}
\begin{proof}
	$(\Rightarrow)$ Let $\{a_n\}_{n} \in \ell^q (\mathbb{N})$ and $n, m\in \mathbb{N}$ with $n<m$.  First we have to show that the series in (\ref{LIPBASSELOPERATORCHARACTERIZATION}) is convergent. For all $x, y \in \mathcal{M}$, 
	\begin{align*}
	\left|\sum_{k=n}^{m}a_k(f_k(x)-f_k(y))\right|&\leq \left(\sum_{k=n}^{m}|a_k|^q\right)^\frac{1}{q}\left(\sum_{k=n}^{m}|f_k(x)-f_k(y)|^p\right)^\frac{1}{p}\\
	&\leq b \left(\sum_{k=n}^{m}|a_k|^q\right)^\frac{1}{q}\, d(x,y).
	\end{align*}
	Therefore the series in (\ref{LIPBASSELOPERATORCHARACTERIZATION})  converges. We next show that the map $T\{a_n\}_{n}$ is Lipschitz. Consider 
	\begin{align*}
	&\left\|T\{a_n\}_{n}\right\|_{\operatorname{Lip}_0} =\sup_{(x, y), (u,v) \in \mathcal{M}\times \mathcal{M}, (x, y)\neq (u,v)}\frac{|T\{a_n\}_{n}(x,y)-T\{a_n\}_{n}(u,v)|}{d(x,u)+d(y,v)}\\
	&\quad=\sup_{(x, y), (u,v) \in \mathcal{M}\times \mathcal{M}, (x, y)\neq (u,v)}\frac{|\sum_{n=1}^{\infty}a_n(f_n(x)-f_n(u))-\sum_{n=1}^{\infty}a_n(f_n(y)-f_n(v))|}{d(x,u)+d(y,v)}\\
	&\quad\leq \sup_{(x, y), (u,v) \in \mathcal{M}\times \mathcal{M}, (x, y)\neq (u,v)}\frac{|\sum_{n=1}^{\infty}a_n(f_n(x)-f_n(u))|+|\sum_{n=1}^{\infty}a_n(f_n(y)-f_n(v))|}{d(x,u)+d(y,v)}\\
	&\quad\leq b\sup_{(x, y), (u,v) \in \mathcal{M}\times \mathcal{M}, (x, y)\neq (u,v)}\frac{\left(\sum_{n=1}^{\infty}|a_n|^q\right)^\frac{1}{q}\, d(x,u)+\left(\sum_{n=1}^{\infty}|a_n|^q\right)^\frac{1}{q}\, d(y,v)}{d(x,u)+d(y,v)}\\
	&\quad=b\left(\sum_{n=1}^{\infty}|a_n|^q\right)^\frac{1}{q}.
	\end{align*}
	Hence $T$ is well-defined. Clearly $T$ is linear.  Boundedness of $T$ with bound $b$ will follow from previous calculation.\\
	$(\Leftarrow)$ From the definition of $T$, it is bounded by 	Banach-Steinhaus. Given $x, y \in \mathcal{M}$,  we define a map 
	\begin{align*}
	\Phi_{x,y}: \ell^q (\mathbb{N})  \ni \{a_n\}_{n} \mapsto \Phi_{x,y}\{a_n\}_{n}\coloneqq \sum_{n=1}^{\infty}a_n(f_n(x)-f_n(y))\in \mathbb{K}
	\end{align*} 
	which is a bounded linear functional. Hence $\{f_n(x)-f_n(y)\}_{n}\in \ell^p (\mathbb{N})$. Let $\{e_n\}_{n}$ be the standard Schauder basis for $ \ell^p (\mathbb{N})$. Then 
	\begin{align*}
	\|\Phi_{x,y}\|=\left(\sum_{n=1}^{\infty}|\Phi_{x,y}\{e_n\}_{n}|^p\right)^\frac{1}{p}= \left(\sum_{n=1}^{\infty}|f_n(x)-f_n(y)|^p\right)^\frac{1}{p}.
	\end{align*}
	Now 
	\begin{align*}
	b\left(\sum_{n=1}^{\infty}|a_n|^q\right)^\frac{1}{q}&=b\|\{a_n\}_{n}\|\geq \|T\{a_n\}_{n}\|_{\operatorname{Lip}_0}\\
	&\geq \sup_{(x, 0), (y,0) \in \mathcal{M}\times \mathcal{M}, (x, 0)\neq (y,0)}\frac{|T\{a_n\}_{n}(x,0)-T\{a_n\}_{n}(y,0)|}{d(x,y)}\\
	&=\sup_{(x, 0), (y,0) \in \mathcal{M}\times \mathcal{M}, (x, 0)\neq (y,0)}\frac{|\sum_{n=1}^{\infty}a_n(f_n(x)-f_n(y))|}{d(x,y)}\\
	&=\sup_{(x, 0), (y,0) \in \mathcal{M}\times \mathcal{M}, (x, 0)\neq (y,0)}\frac{|\Phi_{x,y}\{a_n\}_{n}|}{d(x,y)}
	\end{align*}
	which implies 
	\begin{align*}
	|\Phi_{x,y}\{a_n\}_{n}|\leq b \left(\sum_{n=1}^{\infty}|a_n|^q\right)^\frac{1}{q}\,d(x,y), \quad \forall x, y \in \mathcal{M}.
	\end{align*}
	Using all these, we finally get 
	\begin{align*}
	\left(\sum_{n=1}^{\infty}|f_n(x)-f_n(y)|^p\right)^\frac{1}{p}=\|\Phi_{x,y}\|\leq b\, d(x,y), \quad \forall x, y \in \mathcal{M}.
	\end{align*}
	Hence 	$\{f_n\}_{n}$ is a metric p-Bessel sequence for 
	$\mathcal{M}$ with bound $b$.
\end{proof}
In the spirit of definition of $\mathcal{X}_d$-frame, Definition \ref{PFRAMEFORMETRIC} can be generalized.
\begin{definition}\label{XDMETRICFRAME}
	Let $\mathcal{M}$ be a metric space and $\mathcal{M}_d$ be an associated BK-space. Let
	$\{f_n\}_{n}$ be a sequence in $\operatorname{Lip}(\mathcal{M}, \mathbb{K})$. We say that $\{f_n\}_{n}$ is a  \textbf{metric $\mathcal{M}_d$-frame} (or \textbf{Lipschitz $\mathcal{M}_d$-frame}) for $\mathcal{M}$  if the following conditions hold.
	\begin{enumerate}[label=(\roman*)]
		\item $\{f_n(x)\}_{n} \in \mathcal{M}_d$, for each  $x \in \mathcal{M}$,
		\item There exist positive $a, b$  such that 
		$
		a\, d(x,y) \leq \|\{f_n(x)-f_n(y)\}_n\|_{\mathcal{M}_d} \leq b\, d(x,y), $ $ \forall x
		, y\in \mathcal{M}.
		$
	\end{enumerate}
 We say  constant $a$ as \textbf{lower metric $\mathcal{M}_d$-frame bound}
 and constant $b$ as \textbf{upper metric $\mathcal{M}_d$-frame bound}.	If we do not demand the first inequality, then we say $\{f_n\}_{n}$ is a  metric \textbf{$\mathcal{M}_d$-Bessel sequence}.
\end{definition}
An easier  way of producing metric $\mathcal{M}_d$-frame is the following. Let $\mathcal{M}_d$ be a BK-space which admits a Schauder basis $\{\tau_n\}_{n}$. Let  $\{f_n\}_{n}$ be the coefficient functionals associated with $\{\tau_n\}_{n}$. Let $\mathcal{M}$ be a metric space and $A:\mathcal{M} \rightarrow \mathcal{M}_d$ be bi-Lipschitz with bounds $a$ and $b$. Define $g_n\coloneqq f_n A, \forall n$. Then $g_n$ is a Lipschitz function for all $n$. Now 
\begin{align*}
a\, d(x,y) &\leq \|Ax-Ay\|_{\mathcal{M}_d}=\|\{f_n(Ax-Ay)\}_n\|_{\mathcal{M}_d}\\
&=\|\{f_n(Ax)-f_n(Ay)\}_n\|_{\mathcal{M}_d} 
=\|\{g_n(x)-g_n(y)\}_n\|_{\mathcal{M}_d} \leq b\, d(x,y), \quad \forall x
, y\in \mathcal{M}.
\end{align*}
Hence $\{g_n\}_{n}$ is a  metric $\mathcal{M}_d$-frame for $\mathcal{M}$.\\
Following result ensures that metric frames are universal in nature.
\begin{theorem}\label{METRICFRAMEEXISTS3}
	Every separable metric space $ \mathcal{M}$ admits a metric $\mathcal{M}_d$-frame.
\end{theorem}
\begin{proof}
	From Theorem \ref{AHARONITHEOREM} it follows that there exists a bi-Lipschitz function $f: \mathcal{M} \to c_0(\mathbb{N})$. Let $ p_n: c_0(\mathbb{N}) \to \mathbb{K}$ be the coordinate projection, for each $n$. If we now set 	$f_n\coloneqq p_nf$, for each $n$, then $\{f_n\}_{n}$ is a metric $c_0(\mathbb{N})$-frame  for $\mathcal{M}$.
\end{proof}
Given metric $\mathcal{M}_d$-frames  $\{f_n\}_{n}$, $\{g_n\}_{n}$ and a nonzero  scalar $\lambda$, one can naturally ask whether we can scale and add them to get new frames? i.e.,  whether $\{f_n+\lambda g_n\}_{n}$ is a frame? In the case of Hilbert spaces, a use of Minkowski's inequality shows that whenever $\{\tau_n\}_{n}$ and $\{\omega_n\}_{n}$ are frames for a Hilbert space $\mathcal{H}$, then  $\{\tau_n+\lambda \omega_n\}_{n}$ is a Bessel sequence for $\mathcal{H}$. In general, this sequence need not be a frame for $\mathcal{H}$. Thus we have to impose extra conditions to ensure the existence of  lower frame bound. For Hilbert spaces this is done by  \cite{FAVIERZALIK}. We now obtain similar results for metric spaces.
\begin{theorem}
	Let 	$\{f_n\}_{n}$ be a  metric $\mathcal{M}_d$-frame for metric space $\mathcal{M}$ with bounds $a$ and $b$. Let $\lambda$ be a non-zero  scalar. Then 
	\begin{enumerate}[label=(\roman*)]
		\item   $\{\lambda f_n\}_{n}$ is a metric  $\mathcal{M}_d$-frame for  $\mathcal{M}$ with bounds $a\lambda$ and $b\lambda$.
		\item If $A:\mathcal{M} \rightarrow \mathcal{M}$ is bi-Lipschitz with bounds $c$ and $d$, then $\{f_nA\}_{n}$ is a metric  $\mathcal{M}_d$-frame for  $\mathcal{M}$ with bounds $ac$ and $bd$.
		\item If $\{g_n\}_{n}$ is a metric $\mathcal{M}_d$-Bessel sequence for $\mathcal{M}$ with bound $d$ and $|\lambda|<\frac{a}{d}$, then $\{f_n+\lambda g_n\}_{n}$ is a metric $\mathcal{M}_d$-frame for  $\mathcal{M}$ with bounds $a-|\lambda|d$ and $b+|\lambda|d$.
	\end{enumerate}
\end{theorem}
\begin{proof}
	First two conclusions follow easily. For the upper frame bound of $\{f_n+\lambda g_n\}_{n}$ 	we use  triangle inequality. Now for lower frame bound, using reverse triangle inequality, we get 
		\begin{align*}
	&\|\{(f_n+\lambda g_n)(x)-(f_n+\lambda g_n)(y)\}_n\|_{\mathcal{M}_d}=\|\{f_n(x)-f_n(y)+ \lambda( g_n(x)- g_n(y))\}_n\|_{\mathcal{M}_d}\\
	&\quad\geq \|\{f_n(x)-f_n(y)\}_n\|_{\mathcal{M}_d}-\|\{ \lambda( g_n(x)- g_n(y))\}_n\|_{\mathcal{M}_d}\\
	&\quad \geq  a\, d(x,y)- |\lambda| \, d(x,y)
	=(a-|\lambda|)\, d(x,y), \quad \forall x
	, y\in \mathcal{M}.
	\end{align*}
\end{proof}
We next define  ``metric frame" which is stronger than    Definition \ref{XDMETRICFRAME} in light of definition of Banach frame.
 \begin{definition}\label{METRICBANACHFRAME}
	Let $\mathcal{M}$ be a metric space and $\mathcal{M}_d$ be an associated  BK-space. Let
	$\{f_n\}_{n}$ be a sequence in $\operatorname{Lip}(\mathcal{M}, \mathbb{K})$
	and $S: \mathcal{M}_d \rightarrow \mathcal{M}$. We say that  $(\{f_n\}_{n}, S)$ is a \textbf{metric frame} or \textbf{Lipschitz metric
		frame} for $\mathcal{M}$ if the following conditions hold.
	\begin{enumerate}[label=(\roman*)]
		\item $\{f_n(x)\}_{n} \in \mathcal{M}_d$, for each  $x \in \mathcal{M}$,
		\item There exist positive $a, b$  such that 
		$
		a\, d(x,y) \leq \|\{f_n(x)-f_n(y)\}_n\|_{\mathcal{M}_d} \leq b\, d(x,y), $ $  \forall x
		, y\in \mathcal{M},
		$
		\item $S$ is Lipschitz and $S(\{f_n(x)\}_{n})=x$, for each $x \in \mathcal{M}$.
	\end{enumerate}
	 Mapping $S$ is called as
	Lipschitz reconstruction operator. We say  constant $a$ as \textbf{lower frame bound}
	and constant $b$ as \textbf{upper frame bound}. If we do not demand the first inequality, then we say $(\{f_n\}_{n}, S)$ is a \textbf{metric  Bessel sequence}.
\end{definition}
We observe that if $(\{f_n\}_{n}, S)$ is a metric frame  for $\mathcal{M}$, then condition (i) in Definition \ref{METRICBANACHFRAME} tells 
that the mapping (we call as analysis map)
\begin{align*}
\theta_f:\mathcal{M} \ni x \mapsto \theta_f x\coloneqq \{f_n(x)\}_{n} \in \mathcal{M}_d
\end{align*}
is well-defined and condition (ii) in Definition \ref{METRICBANACHFRAME} tells that $\theta_f$ satisfies 
\begin{align*}
a\, d(x,y)\leq \|\theta_f x -\theta_fy \|\leq b\, d(x,y), \quad \forall x
, y\in \mathcal{M}.
\end{align*}
Hence $\theta_f$ is bi-Lipschitz and injective. Thus a  metric frame puts the space 
into $\mathcal{M}_d$ via $\theta_f$ and reconstructs 
every element by using reconstruction operator $S$. Now note that $S\theta_f =I_\mathcal{M}$. This operator description helps us to derive the following propositions easily.
\begin{proposition}
	If $(\{f_n\}_{n}, S)$ is a metric frame  for $\mathcal{M}$, then 	$P_{f, S}\coloneqq \theta_f S: \mathcal{M}_d \to \mathcal{M}_d$ is idempotent and $P_{f, S}(\mathcal{M}_d)=\theta_f(\mathcal{M}_d).$
\end{proposition} 
\begin{proposition}
	Let $\{f_n\}_{n}$ be a  $\mathcal{M}_d$-frame  for $\mathcal{M}$ and $S: \mathcal{M}_d \rightarrow \mathcal{M}$ be Lipschitz. Then $(\{f_n\}_{n}, S)$ is a metric frame  for $\mathcal{M}$  if and only if $S$ is a left-Lipschitz inverse of $\theta_f$ if and only if 	$\theta_f$ is a  right-Lipschitz inverse of $S$.
\end{proposition}
We now give some explicit examples of metric frames.
\begin{example}
	Let $\mathcal{M}$, $\{f_n\}_{n}$ be as in Example \ref{1FRAMEFIRST} and let $a=1$.	Take\\ $\mathcal{M}_d \coloneqq \ell^1(\{0\}\cup \mathbb{N})$ and define 
	\begin{align*}
	S:\mathcal{M}_d \ni \{a_n\}_{n} \mapsto S \{a_n\}_{n} \coloneqq 1+\left| \sum_{n=1}^{\infty} a_n\right| \in \mathcal{M}.
	\end{align*}
	Then 
	
	\begin{align*}
	|S \{a_n\}_{n}-S \{b_n\}_{n}|&=\left|| \sum_{n=1}^{\infty} a_n|-| \sum_{n=1}^{\infty} b_n| \right|\leq \left| \sum_{n=1}^{\infty} a_n- \sum_{n=1}^{\infty} b_n \right|\\
	&=\left| \sum_{n=1}^{\infty} (a_n-b_n)\right|\leq \sum_{n=1}^{\infty} |a_n-b_n|\leq \sum_{n=0}^{\infty} |a_n-b_n|\\
	&=\|\{a_n\}_{n}-\{b_n\}_{n}\|,\quad \forall \{a_n\}_{n}, \{b_n\}_{n} \in \ell^1(\{0\}\cup \mathbb{N}).
	\end{align*}
	Thus $S$ is Lipschitz. Further, 
	\begin{align*}
	S(\{f_n(x)\}_{n})=1+\left| \sum_{n=1}^{\infty} f_n(x)\right|=1+\sum_{n=1}^{\infty}\frac{(\log x)^n}{n!}=x,\quad \forall x \in \mathcal{M}. 
	\end{align*}
	Hence $(\{f_n\}_{n}, S)$ is a metric frame  for $\mathcal{M}$.
	Note that if we define \begin{align*}
	T:\mathcal{M}_d \ni \{a_n\}_{n} \mapsto S \{a_n\}_{n} \coloneqq 1+ \sum_{n=1}^{\infty} |a_n| \in \mathcal{M},
	\end{align*}
	then 
	\begin{align*}
	|T \{a_n\}_{n}-T \{b_n\}_{n}|&=\left| \sum_{n=1}^{\infty} |a_n|- \sum_{n=1}^{\infty} |b_n| \right|= \left| \sum_{n=1}^{\infty} (|a_n|-|b_n|)  \right|\\
	&\leq  \sum_{n=1}^{\infty} \bigg||a_n|-|b_n|\bigg|\leq \sum_{n=1}^{\infty} |a_n-b_n|\leq \sum_{n=0}^{\infty} |a_n-b_n|\\
	&=\|\{a_n\}_{n}-\{b_n\}_{n}\|,\quad \forall \{a_n\}_{n}, \{b_n\}_{n} \in \ell^1(\{0\}\cup \mathbb{N}).
	\end{align*}
	Thus $T$ is Lipschitz. Hence  $(\{f_n\}_{n}, T)$ is also  a metric frame  for $\mathcal{M}$.
\end{example}
\begin{example}
	Let $f_1: \mathbb{K}\to  \mathbb{K}$ be bi-Lipschitz and let  $f_2, \dots, f_m: \mathbb{K}\to  \mathbb{K}$ be $m$ Lipschitz maps such that 
	\begin{align*}
	f_1(x)+\dots+ f_m(x)=x, \quad \forall x \in \mathbb{K}.
	\end{align*}
	We now define $S: \mathbb{K}^m \ni (x_1, \dots, x_m)\mapsto  \sum_{j=1}^{m}x_j \in \mathbb{K}$. Then $(\{f_n\}_{n}, S)$ is   a metric frame  for $\mathbb{K}$. Note that the operator $S$ is linear. 	
\end{example}
After the definition of metric frame, the first question which comes is the
existence. In Theorem \ref{BANACHFRAMEEXISTSSEPARABLE} it was proved that every separable Banach
space admits a Banach frame. Even though this result is not known in metric space settings,  two results are derived one is  close to the definition of metric frame and another gives existence under certain assumptions. To do this
we prove a result which we derive from Mc-Shane extension theorem.
\begin{theorem}\label{MCSHANE}(\textbf{Mc-Shane extension theorem}) (cf. \cite{WEAVER})
	Let $\mathcal{M}$ be a metric space and $\mathcal{M}_0$  be a nonempty subset
	of $\mathcal{M}$. If $f_0:\mathcal{M}_0 \rightarrow \mathbb{R} $ is Lipschitz,
	then there exists a Lipschitz function $f:\mathcal{M} \rightarrow \mathbb{R}
	$ such that $f|{\mathcal{M}_0}=f_0$ and
	$\operatorname{Lip}(f)=\operatorname{Lip}(f_0)$.
\end{theorem}
Using Theorem \ref{MCSHANE} we derive the following.

\begin{corollary}\label{MCSHANECORO}
	If $(\mathcal{M}, 0)$ is a pointed metric space, then for every $x
	\in \mathcal{M}$, there exists a Lipschitz function $f:\mathcal{M} \rightarrow
	\mathbb{R}$ such that $f(x)=d(x,0)$, $f(0)=0$ and $\operatorname{Lip}(f)=1$.
\end{corollary}
\begin{proof}
	Case (i) : $x\neq0$.
	Define $\mathcal{M}_0\coloneqq \{0,x\}$ and $f_0(0)=0$, $f_0(x)=d(x,0)$. Then
	$|f_0(x)-f_0(0)|=d(x,0)$ and hence $f_0$ is Lipschitz. Application of Theorem \ref{MCSHANE} now
	completes the proof.\\
	Case (ii) : $x=0$.
	Take a non-zero point $y
	\in \mathcal{M}$. We use the argument in case (i) by
	replacing $y$ in the place of $x$.
\end{proof}
\begin{theorem}\label{METRICFRAMEEXISTS}
	Let $ \mathcal{M}$ be a separable metric space.  Then there exist a BK-space $ \mathcal{M}_d$,  a sequence $\{f_n\}_{n}$  in $\operatorname{Lip}_0(\mathcal{M}, \mathbb{R})$
	and a function $S: \mathcal{M}_d \rightarrow \mathcal{M}$ such that 
	\begin{enumerate}[label=(\roman*)]
		\item $\{f_n(x)\}_{n} \in \mathcal{M}_d$, for each  $x \in \mathcal{M}$,
		\item 
		$
		\|\{f_n(x)-f_n(y)\}_n\|_{\mathcal{M}_d} \leq \, d(x,y), \forall x
		, y\in \mathcal{M},
		$
		\item  $S(\{f_n(x)\}_{n})=x$, for each $x \in \mathcal{M}$.
	\end{enumerate}
\end{theorem}
\begin{proof}
	Let $\{x_n\}_{n}$ be a dense set in $ \mathcal{M}$.  Then for each $n \in \mathbb{N}$, from Corollary \ref{MCSHANECORO}
	there exists a Lipschitz function $f_n:\mathcal{M} \rightarrow
	\mathbb{R}$ such that $f_n(x_n)=d(x_n,0)$, $f_n(0)=0$ and
	$\operatorname{Lip}(f_n)=1$. Let $ x \in \mathcal{M}$ be fixed. Now for each $n
	\in\mathbb{N}$, 
	\begin{align*}
	|f_n(x)|=|f_n(x)-f_n(0)|\leq \|f_n\|_{\operatorname{Lip}_0}\, d (x,0)=d (x,0)
	\end{align*}
	which gives $\sup_{n \in\mathbb{N}}|f_n(x)|\leq d (x,0)$. Since $\{x_n\}_{n}$
	is dense, there exists a sub sequence $\{x_{n_k}\}_{k}$ of $\{x_n\}_{n}$ such
	that $x_{n_k} \rightarrow x$ as $n \to \infty.$ From the inequality 
	\begin{align*}
	|d(y,z)-d(y,w)|\leq d(z,w), \quad \forall y,z,w \in \mathcal{M}
	\end{align*}
	we see then that $d(x_{n_k}, 0) \rightarrow d(x,0)$ as $n \to \infty.$ Consider 
	\begin{align*}
	d(x_{n_k}, 0)&=f_{n_k}(x_{n_k})\leq 
	|f_{n_k}(x_{n_k})-f_{n_k}(x)|+|f_{n_k}(x)|\\
	&\leq 1 .d(x_{n_k},
	x)+|f_{n_k}(x)|,\quad \forall k \in\mathbb{N}\\
	\implies &\lim_{k \to \infty}(d(x_{n_k}, 0)-d(x_{n_k},
	x))\leq\sup_{k \in \mathbb{N} }(d(x_{n_k}, 0)-d(x_{n_k},
	x)) \leq \sup_{k \in\mathbb{N}}|f_{n_k}(x)|.
	\end{align*}
	Therefore 
	\begin{align*}
	\sup_{n \in\mathbb{N}}|f_n(x)|&\leq d (x,0)=\lim_{k \to \infty}d(x_{n_k},
	0)=\lim_{k \to \infty}(d(x_{n_k}, 0)-d(x_{n_k},x))\\
	&\leq  \sup_{k
		\in\mathbb{N}}|f_{n_k}(x)|\leq\sup_{n \in\mathbb{N}}|f_n(x)|.
	\end{align*}
	So we proved that 
	\begin{align}\label{ALMOST}
	d(x,0)=\sup_{n \in\mathbb{N}}|f_n(x)|, \quad \forall x \in \mathcal{M}.
	\end{align}
	Define 
	$
	\mathcal{M}^0_d\coloneqq \{\{f_n(x)\}_n: x \in \mathcal{M}\}.
	$
	Equality (\ref{ALMOST}) then tells that $\mathcal{M}^0_d$ is a subset of $\ell^\infty(\mathbb{N}).$ Now we define
	$S_0:\mathcal{M}_d^0 \ni \{f_n(x)\}_n \mapsto x \in \mathcal{M}$. Then from
	Equality (\ref{ALMOST}),
	\begin{align*}
	d(S_0(\{f_n(x)\}_n),S_0(\{f_n(y)\}_n))&=d(x,y)
	\leq d(x,0)+d(0,y)\\
	&=\sup_{n \in\mathbb{N}}|f_n(x)|+\sup_{n \in\mathbb{N}}|f_n(y)|\\
	&=\|\{f_n(x)\}_n\|+\|\{f_n(y)\}_n\|, \quad \forall x, y \in
	\mathcal{M}.
	\end{align*}
	We  also have 
	\begin{align*}
	\|\{f_n(x)-f_n(y)\}_n\|_{\mathcal{M}_d}&=\sup_{n \in\mathbb{N}}|f_n(x)-f_n(y)| \\
	&\leq \sup_{n \in\mathbb{N}}\|f_n\|_{\operatorname{Lip}_0}\, d (x,y)=d (x,y), \quad \forall  x
	, y\in \mathcal{M}.
	\end{align*}
	We can now take $S$ as Lipschitz extension 
	of $S_0$ to $\ell^\infty(\mathbb{N})$ and $\mathcal{M}_d=\ell^\infty(\mathbb{N})$ which completes the proof. 
\end{proof}
\begin{theorem}\label{METRICFRAMEEXISTS2}
	If 	$A:\mathcal{M} \to \mathcal{M}_d$ is bi-Lipschitz and there is a Lipschitz projection $P:\mathcal{M}_d \to A(\mathcal{M})$, then $\mathcal{M}$ admits a metric frame.
\end{theorem}
\begin{proof}
	Let $\{h_n\}_n$ be the sequence of coordinate functionals associated with $\mathcal{M}_d$. 	Define $f_n\coloneqq h_nA$ and $S \coloneqq A^{-1}P$. Then 
	\begin{align*}
	S(\{f_n(x)\}_n)=A^{-1}P(\{h_n(Ax)\}_n)=A^{-1}PAx=A^{-1}Ax=x, \quad \forall x \in \mathcal{M}.
	\end{align*}
	Hence  $(\{f_n\}_{n}, S)$ is  a metric frame  for $\mathcal{M}$.
\end{proof}
It is well-known that Mc-Schane extension theorem fails for complex valued Lipschitz functions. Thus we may ask whether we can take a complex
sequence space in Theorem \ref{METRICFRAMEEXISTS}. It is possible for certain metric spaces due to the following theorem. 
\begin{theorem}(\textbf{Kirszbraun extension theorem}) (cf. \cite{VALENTINEL})
	Let $\mathcal{H}$ be a Hilbert  space and $\mathcal{M}_0$  be a nonempty subset
	of $\mathcal{H}$. If $f_0:\mathcal{M}_0 \rightarrow \mathbb{K} $ is Lipschitz,
	then there exists a Lipschitz function $f:\mathcal{H} \rightarrow \mathbb{K}
	$ such that $f|{\mathcal{M}_0}=f_0$ and
	$\operatorname{Lip}(f)=\operatorname{Lip}(f_0)$.
\end{theorem}

Following proposition shows that given a metric frame, we can generate more metric frames.	
\begin{proposition}
	Let $(\{f_n\}_{n}, S)$ be a metric frame  for $\mathcal{M}$. If maps $A, B :\mathcal{M} \to \mathcal{M}$  are such that $A$ is 
	bi-Lipschitz, $B$ is Lipschitz and $BA=I_\mathcal{M}$, then $(\{f_nA\}_{n}, BS)$ is a metric frame  for $\mathcal{M}$. In particular, if $A :\mathcal{M} \to \mathcal{M}$   is 
	bi-Lipschitz invertible, then $(\{f_nA\}_{n}, A^{-1}S)$ is a metric frame  for $\mathcal{M}$.
\end{proposition}
\begin{proof}
	Bi-Lipschitzness of $A$ tells that  condition (ii) in Definition \ref{METRICBANACHFRAME} holds. Now by using $BA=I_\mathcal{M}$ we get $BS (\{f_nAx\}_{n})=BAx=x, \forall x \in \mathcal{M}.$
\end{proof}
Previous proposition not only helps to generate metric frames from metric frames but also from Banach frames. Since there are  large number of examples of Banach frames for a variety of Banach spaces, just by operating with  bi-Lipschitz invertible functions on subsets, it produces metric frames for that subset. Next we  characterize metric frames using Lipschitz functions. The following theorem precisely says when an $\mathcal{M}_d$-frame can be converted into a metric frame.
\begin{theorem}\label{CHARLIPMETRIC}
	Let $\{f_n\}_{n}$ be a metric $\mathcal{M}_d$-frame  for $\mathcal{M}$. Then the following are equivalent. 
	\begin{enumerate}[label=(\roman*)]
		\item There exists a Lipschitz projection  $P:\mathcal{M}_d \to \theta_f(\mathcal{M})$. 
		\item There exists a Lipschitz map $V:\mathcal{M}_d \to \mathcal{M}$ such that $V|_{\theta_f(\mathcal{M})}=\theta_f^{-1}$.
		\item There exists a Lipschitz map $S:\mathcal{M}_d \to \mathcal{M}$ such that $(\{f_n\}_{n}, S)$ is  a metric frame  for $\mathcal{M}$.
	\end{enumerate}
\end{theorem}
\begin{proof}
	(i)	$\Rightarrow$ (ii) Define $V\coloneqq \theta_f^{-1} P$. Then for $y=\theta_f(x),  x \in \mathcal{M}$ we get  $Vy=V\theta_f(x)=\theta_f^{-1} P\theta_f(x)=\theta_f^{-1}\theta_f(x)=\theta_f^{-1} y$.\\
	(ii)	$\Rightarrow$ (i) Set $P\coloneqq \theta_fV$. Now $P^2=\theta_fV\theta_fV=\theta_fI_\mathcal{M}V=P$.\\
	(ii)	$\Rightarrow$ (iii) Define $S\coloneqq V$. Then $S\{f_n(x)\}_n=S\theta_f (x)=V\theta_f (x)=x$, for all $x \in \mathcal{M}$. Hence  $(\{f_n\}_{n}, S)$ is  a metric frame  for $\mathcal{M}$.\\
	(iii)	$\Rightarrow$ (ii) Define $V\coloneqq S$. Then $V\theta_f (x)=S\theta_f (x)=S\{f_n(x)\}_n=x$, for all $x \in \mathcal{M}$.
\end{proof}
\section{METRIC FRAMES FOR BANACH SPACES}
Now we turn onto the representation of elements using metric frames. Naturally, to deal with sums we  must look in to Banach space structure. Following theorem can be compared with Theorem \ref{CASAZZASEPARABLECHARACTERIZATION}. 
\begin{theorem}\label{PALL}
	Let $\{f_n\}_{n}$ be a metric p-frame  for a Banach space $\mathcal{X}$. Assume that $f_n(0)=0$ for all $n$. Then the following are equivalent. 
	\begin{enumerate}[label=(\roman*)]
		\item There exists a bounded linear map $V:\mathcal{M}_d \to \mathcal{X}$ such that $V|_{\theta_f(\mathcal{M})}=\theta_f^{-1}$.
		\item There exists a bounded linear  map $S:\mathcal{M}_d \to \mathcal{X}$ such that $(\{f_n\}_{n}, S)$ is  a metric p-frame  for $\mathcal{X}$. 
		\item There exists a sequence $\{\tau_n\}_{n}$ in  $\mathcal{X}$ such that $\sum_{n=1}^{\infty}c_n\tau_n$ converges for all $\{c_n\}_{n}\in \ell^p(\mathbb{N})$ and 
		$
		x=\sum_{n=1}^{\infty}f_n(x)\tau_n,  \forall x \in \mathcal{X}.
		$
		\item There exists a  q-Bessel sequence $\{\tau_n\}_{n}$ in  $\mathcal{X}\subseteq \mathcal{X}^{**}$ such that 
		$	x=\sum_{n=1}^{\infty}f_n(x)\tau_n, $ $  \forall x \in \mathcal{X}.
		$
		\item   There exists a  q-Bessel sequence $\{\tau_n\}_{n}$ in  $\mathcal{X}\subseteq \mathcal{X}^{**}$ such that 
		$
		f=\sum_{n=1}^{\infty}f(\tau_n)f_n, $ $ \forall f \in \mathcal{X}^*.
		$
	\end{enumerate}	
	In each of the cases (iv) and (v), $\{\tau_n\}_n$ is actually a q-frame for $\mathcal{X}^*$. 	 
\end{theorem}
\begin{proof}
	Proof of (i)	$\iff$ (ii) is similar to the proof of (ii)	$\iff$ (iii) in Theorem \ref{CHARLIPMETRIC}.\\
	(iii)	$\Rightarrow$ (i) Given information tells that the map 
	\begin{align*}
	V: \ell^p(\mathbb{N}) \ni \{c_n\}_{n} \to \sum_{n=1}^{\infty}c_n\tau_n \in \mathcal{X}
	\end{align*}
	is well-defined. Banach-Steinhaus theorem now asserts that $V$ is bounded. Now for $y=\theta_f(x), $ $ x \in \mathcal{X}$ we get 
	\begin{align*}
	Vy=V\theta_f(x)=V(\{f_n(x)\}_{n})=\sum_{n=1}^{\infty}f_n(x)\tau_n=x=\theta_f^{-1} \theta_f(x)=\theta_f^{-1} y.
	\end{align*} 
	(i)	$\Rightarrow$ (iii) Let  $\{e_n\}_{n}$ be the standard Schauder basis for $\ell^p(\mathbb{N})$ and define $\tau_n \coloneqq Ve_n$, for all $n$. Since $V$ is bounded linear and $\sum_{n=1}^{\infty}c_ne_n$ converges for all $\{c_n\}_{n}\in \ell^p(\mathbb{N})$, it follows that $\sum_{n=1}^{\infty}c_n\tau_n$ converges for all $\{c_n\}_{n}\in \ell^p(\mathbb{N})$. Moreover, 
	\begin{align*}
	x=V\theta_f(x)=V(\{f_n(x)\}_{n})=\sum_{n=1}^{\infty}f_n(x)\tau_n, \quad \forall x \in \mathcal{X}.
	\end{align*}
	(iii)	$\iff$ (iv) By considering $\tau_n$ in $\mathcal{X}^{**}$ through James embedding and using Theorem \ref{CASAZZASEPARABLECHARACTERIZATION} we get that  $\{\tau_n\}_{n}$ is a q-Bessel sequence  in  $\mathcal{X}$ if and only if $\sum_{n=1}^{\infty}c_n\tau_n$ converges for all $\{c_n\}_{n}\in \ell^p(\mathbb{N})$. \\
	(iv)	$\Rightarrow$ (v) Let $b$ be a Bessel bound for $\{\tau_n\}_{n}$. Then for all $f \in \mathcal{X}^*$ and $n\in \mathbb{N}$,
	\begin{align*}
	&\left\|f-\sum_{k=1}^{n}f(\tau_k)f_k\right\|_{\operatorname{Lip}_0}=\sup_{x, y \in \mathcal{X},~ x\neq y} \frac{\left|\left(f-\sum_{k=1}^{n}f(\tau_k)f_k\right)(x)-\left(f-\sum_{k=1}^{n}f(\tau_k)f_k\right)(y)\right|}{\|x-y\|}\\
	&\quad=\sup_{x, y \in \mathcal{X},~ x\neq y} \frac{\left|f\left(\sum_{k=1}^{\infty}f_k(x)\tau_k\right)-f\left(\sum_{k=1}^{\infty}f_k(y)\tau_k\right)-\sum_{k=1}^{\infty}f(\tau_k)(f_k(x)-f_k(y))\right|}{\|x-y\|}\\
	&\quad=\sup_{x, y \in \mathcal{X}, ~x\neq y} \frac{\left|\sum_{k=1}^{n}f(\tau_k)(f_k(x)-f_k(y))-\sum_{k=1}^{\infty}f(\tau_k)(f_k(x)-f_k(y))\right|}{\|x-y\|}\\
	&\quad=\sup_{x, y \in \mathcal{X}, ~x\neq y} \frac{\left|\sum_{k=n+1}^{\infty}f(\tau_k)(f_k(x)-f_k(y))\right|}{\|x-y\|}\\
	&\quad\leq \sup_{x, y \in \mathcal{X}, ~x\neq y} \frac{\left(\sum_{k=n+1}^{\infty}|f(\tau_k)|^q\right)^\frac{1}{q}\left(\sum_{k=n+1}^{\infty}|f_k(x)-f_k(y)|^p\right)^\frac{1}{p}}{\|x-y\|}\\
	&\quad\leq b \left(\sum_{k=n+1}^{\infty}|f_k(x)-f_k(y)|^p\right)^\frac{1}{p} \to 0 \text{ as } n \to \infty.
	\end{align*}
	(v)	$\Rightarrow$ (iv) Let $b$ be a Bessel bound for $\{\tau_n\}_{n}$. Let  $x \in \mathcal{X}$ and $n\in \mathbb{N}$. Then 
	\begin{align*}
	\left\|x-\sum_{k=1}^{n}f_k(x)\tau_k\right\|&=\sup_{f\in \mathcal{X}^*, \|f\|=1} 	\left|f(x)-\sum_{k=1}^{n}f_k(x)f(\tau_k)\right|\\
	&=\sup_{f\in \mathcal{X}^*, \|f\|=1} 	\left|\left(\sum_{k=1}^{\infty}f(\tau_k)f_k\right)(x)-\sum_{k=1}^{n}f_k(x)f(\tau_k)\right|\\
	&=\sup_{f\in \mathcal{X}^*, \|f\|=1} 	\left|\sum_{k=n+1}^{\infty}f_k(x)f(\tau_k)\right|\\
	&\leq \left(\sum_{k=n+1}^{\infty}|f(\tau_k)|^q\right)^\frac{1}{q}\left(\sum_{k=n+1}^{\infty}|f_k(x)-f_k(0)|^p\right)^\frac{1}{p}\\
	&\leq b \left(\sum_{k=n+1}^{\infty}|f_k(x)|^p\right)^\frac{1}{p} \to 0 \text{ as } n \to \infty.
	\end{align*}
	Now we are left with proving  that $\{\tau_n\}_{n}$ is a q-frame for $\mathcal{X}$. Assume (iv). Let $f \in \mathcal{X}^*.$ Then 
	\begin{align*}
	\|f\|&=\sup_{x\in \mathcal{X}, \|x\|=1} \left|f(x)\right|=\sup_{x\in \mathcal{X}, \|x\|=1} \left|f\left(\sum_{n=1}^{\infty}f_n(x)\tau_n\right)\right|\\
	&=\sup_{x\in \mathcal{X}, \|x\|=1} \left|\sum_{n=1}^{\infty}f_n(x)f(\tau_n)\right|\leq b\left(\sum_{n=1}^{\infty}|f_n(x)|^p\right)^\frac{1}{p} .
	\end{align*}
	Since $f$ was arbitrary, the conclusion follows.
\end{proof}
Theorem \ref{CASAZZASEPARABLECHARACTERIZATION} and Theorem \ref{PALL} suggest the following question. 
For which metric spaces and BK-spaces, does Theorem \ref{PALL} hold?  We next present a result which demands only reconstruction of elements using Lipschitz functions on a Banach space and not frame conditions. First we see a result for this purpose. 
\begin{lemma}
	(cf. \cite{CASAZZACHRISTENSENSTOEVA})\label{ANOTHER}
	Given  a Banach space $\mathcal{X}$ and a sequence $\{\tau_n\}_{n}$ of non-zero elements in $\mathcal{X}$, let 
	\begin{align*}
	\mathcal{Y}_d \coloneqq \left\{\{a_n\}_{n}:\sum_{n=1}^\infty a_n \tau_n \text{ converges in }   \mathcal{X}\right\}. 
	\end{align*}
	Then $\mathcal{Y}_d$ is a Banach space with respect to  the norm 
	\begin{align*}
	\| \{a_n\}_{n}\|\coloneqq \sup_{m }\left\|\sum_{n=1}^m a_n \tau_n\right\|.
	\end{align*}
	Further, the canonical unit vectors form a Schauder basis for $\mathcal{Y}_d$.
	\end{lemma}
\begin{theorem}
	Let $\mathcal{X}$ be a Banach space and   $\{f_n\}_{n}$ be a sequence in   $\operatorname{Lip}_0(\mathcal{X}, \mathbb{K})$.   Then the following are equivalent. 
	\begin{enumerate}[label=(\roman*)]
		\item There exists a  sequence $\{\tau_n\}_{n}$ in  $\mathcal{X}$ such that 
		$
		x=\sum_{n=1}^{\infty}f_n(x)\tau_n,  \forall x \in \mathcal{X}.
		$
		\item Let $\{\tau_n\}_{n}$ be a sequence in $\mathcal{X}$ and  define $S_n(x)\coloneqq \sum_{k=1}^{n}f_k(x)\tau_k$, $\forall x \in \mathcal{X}$, for each $n \in \mathbb{N}$. Then $\sup_{n \in\mathbb{N}}\|S_n\|_{\operatorname{Lip}_0} <\infty$ and    there exist a BK-space $\mathcal{M}_d$ and a bounded linear  map $S:\mathcal{M}_d \to \mathcal{M}$ such that $(\{f_n\}_{n}, S)$ is  a metric frame  for $\mathcal{X}$. 
	\end{enumerate}
	Further, a choice for $\tau_n$ is $\tau_n=Se_n$ for each $n \in \mathbb{N}$, where $\{e_n\}_{n}$ is the standard Schauder basis for $\ell^p(\mathbb{N})$.
\end{theorem}
\begin{proof}
	(ii)	$\Rightarrow$ (i) This follows from Theorem \ref{PALL}.\\
	(i)	$(\Rightarrow)$ (ii) We give an argument which is similar to   the arguments given in \cite{CASAZZACHRISTENSENSTOEVA}. Define $A\coloneqq \{n \in \mathbb{N}: \tau_n=0\}$ and $B\coloneqq \mathbb{N} \setminus A$. Let $c_0(A) $ be the space of sequences converging to zero, indexed by $A$, equipped with sup-norm.  Let $\{e_n\}_{n\in A}$ be the canonical Schauder basis for $c_0(A) $. Since the norm is sup-norm, it easily follows that $\{\frac{1}{n(\|f_n\|_{\operatorname{Lip}_0}+1)}e_n\}_{n\in A}$ is also a Schauder basis for $c_0(A) $. Define 
	\begin{align*}
	\mathcal{Z}_d\coloneqq \left\{\{c_n\}_{n\in A}: \sum_{n\in A}\frac{c_n}{n(\|f_n\|_{\operatorname{Lip}_0}+1)}e_n \text{ converges in } A\right\}.
	\end{align*}
	We equip $\mathcal{Z}_d $ with the norm 
	\begin{align*}
	\|\{c_n\}_{n\in A}\|_{\mathcal{Z}_d}\coloneqq \left\|\frac{c_n}{n(\|f_n\|_{\operatorname{Lip}_0}+1)}\right\|_{c_0(A)}=\sup _{n\in A}\left|\frac{c_n}{n(\|f_n\|_{\operatorname{Lip}_0}+1)}\right|.
	\end{align*}
	Then $\{e_n\}_{n\in A}$ is a Schauder basis for $\mathcal{Z}_d $. Clearly $\mathcal{Z}_d $ is a BK-space. Let $\mathcal{Y}_d $ be as defined in  Lemma \ref{ANOTHER}, for the index set $B$. Now set $\mathcal{M}_d \coloneqq \mathcal{Y}_d \oplus \mathcal{Z}_d $ equipped with norm $\|y \oplus z \|_{\mathcal{M}_d}\coloneqq \|y\|_{\mathcal{Y}_d} +\|z\|_{\mathcal{Z}_d}$. It then follows that, for each $x \in \mathcal{X}$, $\{f_n(x)\}_{n\in B}\oplus \{f_n(x)\}_{n\in A} \in \mathcal{M}_d$. We next show that $\{f_n\}_{n}$ is a metric  $\mathcal{M}_d$-frame for $\mathcal{X}$. Let $x, y \in \mathcal{X}$. Then 
	\begin{align*}
	\|x-y\|&=\left\|\sum_{n=1}^{\infty}(f_n(x)-f_n(y))\tau_n\right\|=\lim_{n\to\infty}\left\|\sum_{k=1}^{n}(f_k(x)-f_k(y))\tau_k\right\|\\
	&\leq \sup _{n\in \mathbb{N}}\left\|\sum_{k=1}^{n}(f_k(x)-f_k(y))\tau_k\right\|=\sup _{n\in B}\left\|\sum_{k=1}^{n}(f_k(x)-f_k(y))\tau_k\right\|\\
	&=\|\{f_n(x)-f_n(y)\}_{n\in B}\|_{\mathcal{Y}_d}\\
	&\leq \|\{f_n(x)-f_n(y)\}_{n\in B}\|_{\mathcal{Y}_d}+\|\{f_n(x)-f_n(y)\}_{n\in A}\|_{\mathcal{Z}_d}\\
	&= \|\{f_n(x)-f_n(y)\}_{n\in B}\oplus \{f_n(x)-f_n(y)\}_{n\in A}\|_{\mathcal{M}_d}
	\end{align*}
	and 
	\begin{align*}
	&\|\{f_n(x)-f_n(y)\}_{n\in B}\oplus \{f_n(x)-f_n(y)\}_{n\in A}\|_{\mathcal{M}_d}\\
	&=\|\{f_n(x)-f_n(y)\}_{n\in B}\|_{\mathcal{Y}_d}+\|\{f_n(x)-f_n(y)\}_{n\in A}\|_{\mathcal{Z}_d}\\
	&=\sup _{n\in B}\left\|\sum_{k=1}^{n}(f_k(x)-f_k(y))\tau_k\right\|+\sup _{n\in A}\left|\frac{f_n(x)-f_n(y)}{n(\|f_n\|_{\operatorname{Lip}_0}+1)}\right|\\
	&=\sup _{n\in B}\left\|S_n(x)-S_n(y)\right\|+\sup _{n\in A}\left|\frac{f_n(x)-f_n(y)}{n(\|f_n\|_{\operatorname{Lip}_0}+1)}\right|\\
	&\leq \sup_{n \in B}\|S_n\|_{\operatorname{Lip}_0}\|x-y\|+\sup _{n\in A}\frac{\|f_n\|_{\operatorname{Lip}_0}\|x-y\|}{n(\|f_n\|_{\operatorname{Lip}_0}+1)}\\
	&\leq\left(\sup_{n \in B}\|S_n\|_{\operatorname{Lip}_0}+1\right)\|x-y\|.
	\end{align*}
	We now define 
	\begin{align*}
	S:\mathcal{M}_d \ni \{a_n\}_{n\in B}\oplus \{b_n\}_{n\in A}\mapsto \sum_{n\in B}a_n\tau_n \in \mathcal{X}.
	\end{align*}
	Clearly $S$ is linear. Boundedness of  $S$  follows from the following calculation.
	\begin{align*}
	\|S(\{a_n\}_{n\in B}\oplus \{b_n\}_{n\in A})\|&=\left\|\sum_{n\in B}a_n\tau_n\right\|\leq \sup _{n\in B}\left\|\sum_{k=1}^n a_k\tau_k\right\|\\
	&=\|\{a_n\}_{n\in B}\|_{\mathcal{Y}_d} \leq \|\{a_n\}_{n\in B}\oplus \{b_n\}_{n\in A}\|_{\mathcal{M}_d}.
	\end{align*}
\end{proof}
Using Theorem \ref{POINTEDSPLITS} we derive the following result which tells that given a metric frame for a metric space we can get a metric frame using linear functionals for a subset of the Banach space.
\begin{theorem}\label{LIPIFFLINEAR}
	Let $\{f_n\}_{n}$  be a sequence in $ \operatorname{Lip}_0(\mathcal{M}, \mathbb{K})$. For  each $n\in \mathbb{N}$, let  $T_{f_n}$ be linearization of $f_n$. Let $e$ and $\mathcal{F}(\mathcal{M})$ 
	be as in Theorem \ref{POINTEDSPLITS}. Then $\{f_n\}_{n}$ is a 
	metric frame  for $\mathcal{M}$ with bounds $a$ and $b$ if and only if $\{T_{f_n}\}_{n}$ is a 
	metric frame  for $e(\mathcal{M})$ with bounds $a$ and $b$. In particular, $(\{f_n\}_{n}, S)$ is a metric frame  for $\mathcal{M}$  if and only if $(\{T_{f_n}\}_{n}, eS)$ is a metric frame  for $e(\mathcal{M})$.
\end{theorem}
\begin{proof}
	$(\Rightarrow)$ Let $u,v \in e(\mathcal{M})$. Then $u=e(x), v=e(y)$, for some $x, y \in \mathcal{M}$. Now using the fact that $e$ is an isometry,
	\begin{align*}
	a\|u-v\|&=a\|e(x)-e(y)\|=a\,d(x,y)\leq \|\{f_n(x)-f_n(y)\}_n\|\\
	&=\|\{(T_{f_n}e)(x)-(T_{f_n}e)(y)\}_n\|=\|\{T_{f_n}(e(x))-T_{f_n}(e(y))\}_n\|\\
	&=\|\{T_{f_n}(u)-T_{f_n}(v)\}_n\|\leq b \,d(x,y)=b\|e(x)-e(y)\|=b\|u-v\|.
	\end{align*}
	$(\Leftarrow)$ Let $x, y \in \mathcal{M}$. Then $e(x), e(y) \in e(\mathcal{M})$. Hence 
	\begin{align*}
	a\,d(x,y)&=a\|e(x)-e(y)\|\leq \|\{T_{f_n}(e(x))-T_{f_n}(e(y))\}_n\|\\
	&=\|\{f_n(x)-f_n(y)\}_n\|\leq b\|e(x)-e(y)\|=b\,d(x,y).
	\end{align*}
	Since $x,y$ were arbitrary, the result follows.
\end{proof}
\begin{remark}
	We can not use Theorem \ref{LIPIFFLINEAR}	 to view metric frames as Banach frames. The reason is that $e(\mathcal{M})$  is just a subset of $\mathcal{F}(\mathcal{M})$ and need not be a vector space. Moreover, the map $eS$ is Lipschitz and may not be linear. 	
\end{remark}
{\onehalfspacing \section{PERTURBATIONS}
 Here we present some stability results. These are important as they say that sequences which are close to metric frames are again metric frames. On the other hand, it asserts that if we perturb a metric frame we again get a metric frame.
\begin{theorem}\label{FIRSTPERTURB}
	Let $\{f_n\}_{n}$ be a
	p-metric frame  for $\mathcal{M}$ with bounds $a$ and $b$. Let $\{g_n\}_{n}$ be a sequence in $\operatorname{Lip}(\mathcal{M}, \mathbb{K})$ satisfying the following.
	\begin{enumerate}[label=(\roman*)]
		\item There exist $\alpha, \beta, \gamma \geq 0$ such that $\beta<1$, $\alpha<1$, $\gamma<(1-\alpha)a$.
			\item For all $x, y \in  \mathcal{M}$, and $ m=1, 2,\dots ,$ 
		\begin{align}\label{PERINEQUA}
		\left(\sum_{n=1}^m|(f_n-g_n)(x)-(f_n-g_n)(y)|^p\right)^\frac{1}{p}&\leq \alpha \left(\sum_{n=1}^m|f_n(x)-f_n(y)|^p\right)^\frac{1}{p}\nonumber\\
		&+
		\beta \left(\sum_{n=1}^m|g_n(x)-g_n(y)|^p\right)^\frac{1}{p}+\gamma \,d(x,y).
		\end{align}
	\end{enumerate}
	Then $\{g_n\}_{n}$ is a
	p-metric frame  for $\mathcal{M}$ with bounds
	$
	\frac{((1-\alpha)a-\gamma)}{1+\beta}$ and  $ \frac{((1+\alpha)b+\gamma)}{1-\beta}.
	$
\end{theorem}
\begin{proof}
	Using Minkowski's inequality and Inequality (\ref{PERINEQUA}), we get, for all $x, y \in  \mathcal{M}$ and $m\in \mathbb{N}$,
	\begin{align*}
	&\left(\sum_{n=1}^m|g_n(x)-g_n(y)|^p\right)^\frac{1}{p}\\
	&\leq \left(\sum_{n=1}^m|(f_n-g_n)(x)-(f_n-g_n)(y)|^p\right)^\frac{1}{p}+ \left(\sum_{n=1}^m|f_n(x)-f_n(y)|^p\right)^\frac{1}{p}\\
	&\leq (1+\alpha) \left(\sum_{n=1}^m|f_n(x)-f_n(y)|^p\right)^\frac{1}{p}+\beta \left(\sum_{n=1}^m|g_n(x)-g_n(y)|^p\right)^\frac{1}{p}+\gamma \,d(x,y)
	\end{align*}
	which implies 
	\begin{align*}
	(1-\beta)\left(\sum_{n=1}^m|g_n(x)-g_n(y)|^p\right)^\frac{1}{p}\leq (1+\alpha)\left(\sum_{n=1}^m|f_n(x)-f_n(y)|^p\right)^\frac{1}{p}+\gamma  \,d(x,y), 
	\end{align*}
for all $x, y \in  \mathcal{M}$. 	Since the sum $\sum_{n=1}^\infty|f_n(x)-f_n(y)|^p$ converges, $\sum_{n=1}^\infty|g_n(x)-g_n(y)|^p$ will also converge. Inequality (\ref{PERINEQUA}) now gives 
	\begin{align}\label{PERINEQUA2}
	\left(\sum_{n=1}^\infty|(f_n-g_n)(x)-(f_n-g_n)(y)|^p\right)^\frac{1}{p}&\leq \alpha \left(\sum_{n=1}^\infty|f_n(x)-f_n(y)|^p\right)^\frac{1}{p}\nonumber\\
	&+
	\beta \left(\sum_{n=1}^\infty|g_n(x)-g_n(y)|^p\right)^\frac{1}{p}+\gamma \,d(x,y) .
	\end{align}
	By doing a similar calculation and using Inequality (\ref{PERINEQUA2}) we get  for all $x, y \in  \mathcal{M}$,
		\begin{align*}
	&\left(\sum_{n=1}^\infty|g_n(x)-g_n(y)|^p\right)^\frac{1}{p}\\
	&\quad \leq \left(\sum_{n=1}^\infty|(f_n-g_n)(x)-(f_n-g_n)(y)|^p\right)^\frac{1}{p}+ \left(\sum_{n=1}^\infty|f_n(x)-f_n(y)|^p\right)^\frac{1}{p}\\
	&\quad\leq (1+\alpha) \left(\sum_{n=1}^\infty|f_n(x)-f_n(y)|^p\right)^\frac{1}{p}+\beta \left(\sum_{n=1}^\infty|g_n(x)-g_n(y)|^p\right)^\frac{1}{p}+\gamma \,d(x,y)\\
	&\quad\leq (1+\alpha)b  \,d(x,y)+\beta \left(\sum_{n=1}^\infty|g_n(x)-g_n(y)|^p\right)^\frac{1}{p}+\gamma \,d(x,y)\\
	&\quad=((1+\alpha)b+\gamma)  \,d(x,y)+\beta \left(\sum_{n=1}^\infty|g_n(x)-g_n(y)|^p\right)^\frac{1}{p}
	\end{align*}
	which gives
	\begin{align*}
	(1-\beta)\left(\sum_{n=1}^\infty|g_n(x)-g_n(y)|^p\right)^\frac{1}{p}\leq ((1+\alpha)b+\gamma)  \,d(x,y), \quad \forall x, y \in  \mathcal{M}\\
	\text{i.e.,} \left(\sum_{n=1}^\infty|g_n(x)-g_n(y)|^p\right)^\frac{1}{p}\leq \frac{((1+\alpha)b+\gamma)}{1-\beta}\,d(x,y), \quad \forall x, y \in  \mathcal{M}.
	\end{align*}
	Hence we obtained upper frame bound for $\{g_n\}_n$. For lower frame bound, let  $x, y$ $  \in  \mathcal{M}$. Then 
	\begin{align*}
	&\left(\sum_{n=1}^\infty|f_n(x)-f_n(y)|^p\right)^\frac{1}{p}\\
	&\quad \leq \left(\sum_{n=1}^\infty|(f_n-g_n)(x)-(f_n-g_n)(y)|^p\right)^\frac{1}{p}+ \left(\sum_{n=1}^\infty|g_n(x)-g_n(y)|^p\right)^\frac{1}{p}\\
	&\quad \leq \alpha \left(\sum_{n=1}^\infty|f_n(x)-f_n(y)|^p\right)^\frac{1}{p}+(1+\beta)\left(\sum_{n=1}^\infty|g_n(x)-g_n(y)|^p\right)^\frac{1}{p}+\gamma \,d(x,y)
	\end{align*}
	which implies 
	\begin{align*}
	(1-\alpha)a\,d(x,y)&\leq (1-\alpha)\left(\sum_{n=1}^\infty|f_n(x)-f_n(y)|^p\right)^\frac{1}{p}\\
	&\leq (1+\beta)\left(\sum_{n=1}^\infty|g_n(x)-g_n(y)|^p\right)^\frac{1}{p}+\gamma \,d(x,y), \quad \forall x, y \in  \mathcal{M}\\
	\text{i.e.,} ~\frac{((1-\alpha)a-\gamma)}{1+\beta} &\leq \left(\sum_{n=1}^\infty|g_n(x)-g_n(y)|^p\right)^\frac{1}{p}, \quad \forall x, y \in  \mathcal{M}.
	\end{align*}
\end{proof}
 Using Theorem \ref{FIRSTPERTURB} we obtain the following   result.	
\begin{corollary}
	Let $\{f_n\}_{n}$ be a
	p-metric frame  for $\mathcal{M}$  with bounds $a$ and $b$. Let $\{g_n\}_{n}$ be a sequence in $\operatorname{Lip}(\mathcal{M}, \mathbb{K})$ such that 
	\begin{align*}
	r\coloneqq \left(\sum_{n=1}^\infty \operatorname{Lip}(f_n-g_n)^p\right)^\frac{1}{p} <a.
	\end{align*}
	Then $\{g_n\}_{n}$ is a
	p-metric frame  for $\mathcal{M}$  with bounds $a-r$ and $b+r$.
\end{corollary}
\begin{proof}
	Define $\alpha\coloneqq 0$, $\beta\coloneqq 0$ and $\gamma\coloneqq r$. Then
	for all $x, y \in  \mathcal{M}$, 
	\begin{align*}
&	\left(\sum_{n=1}^\infty|(f_n-g_n)(x)-(f_n-g_n)(y)|^p\right)^\frac{1}{p}\leq \left(\sum_{n=1}^\infty \operatorname{Lip}(f_n-g_n)^p\,d(x,y)^p\right)^\frac{1}{p}\\
	&\quad=\left(\sum_{n=1}^\infty \operatorname{Lip}(f_n-g_n)^p\right)^\frac{1}{p}\,d(x,y)=r\,d(x,y)\\
	&\quad=\alpha \left(\sum_{n=1}^\infty|f_n(x)-f_n(y)|^p\right)^\frac{1}{p}\nonumber
	+
	\beta \left(\sum_{n=1}^\infty|g_n(x)-g_n(y)|^p\right)^\frac{1}{p}+\gamma \,d(x,y).
	\end{align*}
	Thus the hypothesis in Theorem \ref{FIRSTPERTURB} holds. Hence the corollary.
\end{proof}
\begin{corollary}
	Let $\{f_n\}_{n}$ be a
	p-metric Bessel sequence   for $\mathcal{M}$   with bound $b$. Let $\{g_n\}_{n}$ be a sequence in $\operatorname{Lip}(\mathcal{M}, \mathbb{K})$ satisfying the following.
	\begin{enumerate}[label=(\roman*)]
		\item There exist $\alpha, \beta, \gamma \geq 0$ such that $\beta<1$.
		\item For all $x, y \in  \mathcal{M}$, and $m=1,2, \dots ,$ 
		\begin{align*}
		\left(\sum_{n=1}^m|(f_n-g_n)(x)-(f_n-g_n)(y)|^p\right)^\frac{1}{p}&\leq \alpha \left(\sum_{n=1}^m|f_n(x)-f_n(y)|^p\right)^\frac{1}{p}\\
		&+
		\beta \left(\sum_{n=1}^m|g_n(x)-g_n(y)|^p\right)^\frac{1}{p}+\gamma \,d(x,y).
		\end{align*}
	\end{enumerate}
	Then $\{g_n\}_{n}$ is a
	p-metric Bessel sequence for $\mathcal{M}$ with bound $	\frac{((1+\alpha)b+\gamma)}{1-\beta}.$
\end{corollary}
We next derive a stability result in which we perturb the Lipschitz functions and then derive the existence of reconstruction operator.  This is motivated from Theorem  \ref{PERTURBATIONBANACH12}.	
\begin{theorem}\label{STABILITYMA}
	Let $(\{f_n\}_{n}, S)$ be a
	metric frame for a Banach space    $\mathcal{X}$. Assume that
	$f_n(0)=0,$ for all $n \in \mathbb{N}$, and $S(0)=0$. Let $\{g_n\}_{n}$ be a collection in 
	$\operatorname{Lip}_0(\mathcal{X}, \mathbb{K})$ satisfying the following.
	\begin{enumerate}[label=(\roman*)]
		\item There exist $\alpha, \gamma\geq 0$ such that 
		\begin{align}\label{MFPER}
		\|\{(f_n-g_n)(x)-(f_n-g_n)(y)\}_n\|\leq \alpha \|\{f_n(x)-f_n(y)\}_n\|+\gamma \|x-y\|, ~ \forall x, y \in  \mathcal{X}.
		\end{align}
		\item $\alpha \|\theta_f\|_{\operatorname{Lip}_0}+\gamma\leq \|S\|_{\operatorname{Lip}_0}^{-1}.$
	\end{enumerate}
	Then there exists a reconstruction Lipschitz operator $T$ such that $(\{f_n\}_{n}, T)$ is  a
	metric frame for     $\mathcal{X}$ with bounds 
	$
	\|S\|_{\operatorname{Lip}_0}^{-1}-(\alpha\|\theta_f\|_{\operatorname{Lip}_0}+\gamma) $ and $ 
	\|\theta_f\|_{\operatorname{Lip}_0}+(\alpha\|\theta_f\|_{\operatorname{Lip}_0}+\gamma).$
	\end{theorem}
\begin{proof}
	Let $ x\in \mathcal{X}$. Since $g_n(0)=0$ and $f_n(0)=0$ for all $n \in \mathbb{N}$, using Inequality (\ref{MFPER}),
	\begin{align*}
	\|\{g_n(x)\}_n\| &\leq \|\{(f_n-g_n)(x)\}_n\|+\|\{f_n(x)\}_n\|\\
	&\leq (\alpha+1)\|\{f_n(x)\}_n\|+\gamma \|x\|.
	\end{align*}
	Therefore if we define $\theta_g:\mathcal{X} \ni x \mapsto \{g_n(x)\}_{n} \in\mathcal{M}_d$, then this map is well-defined. Again 
	using Inequality (\ref{MFPER}), we show that $\theta_g$ is Lipschitz. For $x, y \in  \mathcal{X}$,
		\begin{align*}
	\|\theta_gx-\theta_gy\|&=\|\{g_n(x)-g_n(y)\}_n\|=\|\{-g_n(x)+g_n(y)\}_n\| \\
	&\leq \|\{(f_n-g_n)(x)-(f_n-g_n)(y)\}_n\|+\|\{f_n(x)-f_n(y)\}_n\|\\
	&\leq (1+\alpha)\|\{f_n(x)-f_n(y)\}_n\|+\gamma \|x-y\|=(1+\alpha)\|\theta_fx-\theta_fy\|+\gamma \|x-y\|\\
	&\leq (1+\alpha) \|\theta_f\|_{\operatorname{Lip}_0}\|x-y\|+\gamma \|x-y\|
	=((1+\alpha) \|\theta_f\|_{\operatorname{Lip}_0}+\gamma)\|x-y\|.
	\end{align*}
	Thus $\|\theta_g\|_{\operatorname{Lip}_0}\leq (1+\alpha) \|\theta_f\|_{\operatorname{Lip}_0}+\gamma$. Previous calculation
	also tells that upper frame bound is $((1+\alpha) \|\theta_f\|_{\operatorname{Lip}_0}+\gamma)$. We see further that  
	Inequality (\ref{MFPER}) can be written as 
	\begin{align}\label{THETAFTHETAG}
	\|(\theta_f-\theta_g)x-(\theta_f-\theta_g)y\|&\leq \alpha \|\theta_fx-\theta_fy\|+\gamma \|x-y\|\nonumber\\
	&\leq (\alpha \|\theta_f\|_{\operatorname{Lip}_0}+\gamma)\|x-y\|, \quad \forall x, y \in  \mathcal{X}.
	\end{align}
	Now noting 
	$S\theta_f=I_\mathcal{X}$ and using Inequality (\ref{THETAFTHETAG}) we see that 
	\begin{align*}
	\|I_\mathcal{X}-S\theta_g\|_{\operatorname{Lip}_0}&=\|S\theta_f-S\theta_g\|_{\operatorname{Lip}_0}\\
	&\leq \|S\|_{\operatorname{Lip}_0}\|\theta_f-\theta_g\|_{\operatorname{Lip}_0}\\
	&\leq \|S\|_{\operatorname{Lip}_0}(\alpha \|\theta_f\|_{\operatorname{Lip}_0}+\gamma)<1.
	\end{align*}
	Since $\operatorname{Lip}_0(\mathcal{X})$ is a unital Banach algebra (Theorem \ref{LIPISABANACHALGEBRA}), last inequality tells that $S\theta_g$ is invertible and its inverse 
	is also a Lipschitz operator and 
	\begin{align*}
	\|(S\theta_g)^{-1}\|_{\operatorname{Lip}_0}\leq 
	\frac{1}{1-\|S\|_{\operatorname{Lip}_0}(\alpha \|\theta_f\|_{\operatorname{Lip}_0}+\gamma)}.
	\end{align*}
	Define $T\coloneqq (S\theta_g)^{-1} S$. Then $T\theta_g=I_\mathcal{X}$ and 
	\begin{align*}
	\|x-y\|&=\|T\theta_gx-T\theta_gy\|\leq \|T\|_{\operatorname{Lip}_0}\|\theta_gx-\theta_gy\|\\
	&\leq \frac{1}{1-\|S\|_{\operatorname{Lip}_0}(\alpha \|\theta_f\|_{\operatorname{Lip}_0}+\gamma)}\|\theta_gx-\theta_gy\|, 
	\quad \forall x, y \in  \mathcal{X}
	\end{align*}
	which gives the lower bound stated in the theorem.
\end{proof}
\begin{corollary}
	Let $(\{f_n\}_{n}, S)$ be a
	metric Bessel sequence for a Banach space    $\mathcal{X}$. Assume that
	$f_n(0)=0,$ for all $n \in \mathbb{N}$, and $S(0)=0$. Let $\{g_n\}_{n}$ be a collection in 
	$\operatorname{Lip}_0(\mathcal{X}, \mathbb{K})$ satisfying the following.
	There exist $\alpha, \gamma\geq 0$ such that 
	\begin{align*}
	\|\{(f_n-g_n)(x)-(f_n-g_n)(y)\}_n\|\leq \alpha \|\{f_n(x)-f_n(y)\}_n\|+\gamma \|x-y\|, \quad \forall x, y \in  \mathcal{X}.
	\end{align*}
	Then there exists a reconstruction Lipschitz operator $T$ such that $(\{f_n\}_{n}, T)$ is  a
	metric Bessel sequence for      $\mathcal{X}$ with bound 
	$\|\theta_f\|_{\operatorname{Lip}_0}+(\alpha\|\theta_f\|_{\operatorname{Lip}_0}+\gamma)$.
\end{corollary}

{\onehalfspacing \chapter{MULTIPLIERS FOR  METRIC SPACES}\label{chap3} }

\section{DEFINITION AND BASIC PROPERTIES OF MULTIPLIERS}
In this chapter, we introduce and study multipliers  for metric spaces. 
We use the following  notation  in this chapter. Let $\mathcal{M}$ be a  metric
space and $\mathcal{X}$ be a  Banach space. Given $f \in
\operatorname{Lip}(\mathcal{M}, \mathbb{K})$ and $\tau \in \mathcal{X}$,
define 
\begin{align*}
\tau\otimes f:\mathcal{M} \ni x \mapsto (\tau\otimes f)(x)\coloneqq f(x)\tau \in \mathcal{X}.
\end{align*}
Then it follows that $\tau\otimes f$ is a Lipschitz operator and $\operatorname{Lip}(\tau\otimes f)=\|\tau\|\operatorname{Lip}(f)$.

We first derive a result which allows us to define multipliers for metric spaces. In the sequel, $1<p<\infty$ and $q$ denotes the conjugate index of $p$.
\begin{theorem}\label{DEFINITIONEXISTENCE}
	Let $\{f_n\}_{n}$ in $\operatorname{Lip}_0(\mathcal{M}, \mathbb{K})$  be a
	 Lipschitz p-Bessel sequence  for a pointed metric space $(\mathcal{M},
	0)$ with bound $b$ and $\{\tau_n\}_{n}$ in a Banach space $\mathcal{X}$  be a
 Lipschitz q-Bessel sequence for $\operatorname{Lip}_0(\mathcal{X},
	\mathbb{K})$ with bound $d$.  If
	$\{\lambda_n\}_n \in \ell^\infty(\mathbb{N})$, then the map
	\begin{align*}
	T: \mathcal{M} \ni x \mapsto \sum_{n=1}^{\infty}\lambda_n (\tau_n\otimes
	f_n) x \in \mathcal{X}
	\end{align*}
	is a well-defined Lipschitz  operator such that $T(0)=0$ with Lipschitz
	norm at most $bd\|\{\lambda_n\}_n\|_\infty.$
\end{theorem}
\begin{proof}
	Let $n , m \in \mathbb{N}$ with $n \leq m$. Then for each $x \in
	\mathcal{M}$, using Holder's inequality,
		\begin{align*}
	\left\|\sum_{k=n}^{m}\lambda_k(\tau_k\otimes f_k)(x)\right\|&=\left\|\sum_{k=n}^{m}\lambda_k f_k(x)\tau_k\right\|=\sup_{\phi \in \mathcal{X}^*,\|\phi\|\leq 1}\left|\phi\left(\sum_{k=n}^{m}\lambda_k f_k(x)\tau_k\right)\right|\\
	&=\sup_{\phi \in \mathcal{X}^*,\|\phi\|\leq 1}\left|\sum_{k=n}^{m}\lambda_k f_k(x)\phi(\tau_k)\right|\\
	&\leq\sup_{\phi \in \mathcal{X}^*,\|\phi\|\leq 1}\sum_{k=n}^{m}|\lambda_k| |f_k(x)||\phi(\tau_k)|\\
	&\leq \sup_{n \in \mathbb{N}}|\lambda_n|\sup_{\phi \in \mathcal{X}^*,\|\phi\|\leq 1}\sum_{k=n}^{m} |f_k(x)||\phi(\tau_k)|\\
	&\leq \sup_{n \in \mathbb{N}}|\lambda_n|\sup_{\phi \in \mathcal{X}^*,\|\phi\|\leq 1}\left(\sum_{k=n}^{m}|f_k(x)|^p\right)^\frac{1}{p}\left(\sum_{k=n}^{m}|\phi(\tau_k)|^q\right)^\frac{1}{q}\\
	&\leq \sup_{n \in \mathbb{N}}|\lambda_n|\sup_{\phi \in
		\mathcal{X}^*,\|\phi\|\leq
		1}\left(\sum_{k=n}^{m}|f_k(x)|^p\right)^\frac{1}{p}d\|\phi\|\\
	&=d\sup_{n \in \mathbb{N}}|\lambda_n|\left(\sum_{k=n}^{m}|f_k(x)|^p\right)^\frac{1}{p}.
	\end{align*}
	Since $\left(\sum_{k=1}^{\infty}|f_k(x)|^p\right)^\frac{1}{p}$
	converges, $\sum_{k=1}^{\infty}\lambda_k(\tau_k\otimes f_k)(x)$ also converges.
	Now for all $x,y \in \mathcal{M}$,
		\begin{align*}
	\|Tx-Ty\|&=\left\|\sum_{n=1}^{\infty}\lambda_n f_n(x)\tau_n-\sum_{n=1}^{\infty}\lambda_n f_n(y)\tau_n\right\|=\left\|\sum_{n=1}^{\infty}\lambda_n (f_n(x)-f_n(y))\tau_n\right\|\\
	&=\sup_{\phi \in \mathcal{X}^*,\|\phi\|\leq 1}\left|\phi\left(\sum_{n=1}^{\infty}\lambda_n (f_n(x)-f_n(y))\tau_k\right)\right|\\
	&=\sup_{\phi \in \mathcal{X}^*,\|\phi\|\leq 1}\left|\sum_{n=1}^{\infty}\lambda_n (f_n(x)-f_n(y))\phi(\tau_k)\right|\\
	&\leq \sup_{n \in \mathbb{N}}|\lambda_n|\sup_{\phi \in \mathcal{X}^*,\|\phi\|\leq 1}\left(\sum_{n=1}^{\infty}|f_n(x)-f_n(y)|^p\right)^\frac{1}{p}\left(\sum_{n=1}^{\infty}|\phi(\tau_n)|^q\right)^\frac{1}{q}\\
	&\leq \sup_{n \in \mathbb{N}}|\lambda_n|\sup_{\phi \in \mathcal{X}^*,\|\phi\|\leq 1}\left(\sum_{n=1}^{\infty}|f_n(x)-f_n(y)|^p\right)^\frac{1}{p}d\|\phi\|\\
	&=d\sup_{n \in
		\mathbb{N}}|\lambda_n|\left(\sum_{n=1}^{\infty}|f_n(x)-f_n(y)|^p\right)^\frac{1}
	{p}\leq bd\sup_{n \in \mathbb{N}}|\lambda_n|d(x,y).
	\end{align*}
	Hence 
	\begin{align*}
	\|T\|_{\operatorname{Lip}_0}=\sup_{x, y \in \mathcal{M},~ x\neq y}
	\frac{\|Tx-Ty\|}{d(x,y)}\leq bd\sup_{n \in \mathbb{N}}|\lambda_n|.
	\end{align*}
\end{proof}
\begin{corollary}
	Let $\{f_n\}_{n}$ in $\operatorname{Lip}(\mathcal{M}, \mathbb{K})$  be a
	Lipschitz p-Bessel sequence  for  a metric space $\mathcal{M}$ with
	bound $b$ and $\{\tau_n\}_{n}$ in a Banach space $\mathcal{X}$  be a
	 Lipschitz q-Bessel sequence for $\operatorname{Lip}_0(\mathcal{X},
	\mathbb{K})$ with bound $d$.  If
	$\{\lambda_n\}_n \in \ell^\infty(\mathbb{N})$, then for fixed $z \in
	\mathcal{M}$, the map
		\begin{align*}
	T: \mathcal{M} \ni x \mapsto \sum_{n=1}^{\infty}\lambda_n (\tau_n\otimes
	(f_n-f(z)) )x \in \mathcal{X}
	\end{align*}
	is a well-defined Lipschitz  operator  with Lipschitz
	number at most $bd\|\{\lambda_n\}_n\|_\infty.$
\end{corollary}
\begin{proof}
	Define $g_n\coloneqq f_n-f(z), \forall n \in \mathbb{N}$. Then for all $x, y 
	\in \mathcal{M}$,
	\begin{align*}
	\left(\sum_{n=1}^{\infty}|g_n(x)-g_n(y)|^p\right)^\frac{1}{p}=\left(\sum_{n=1}^
	{\infty}|f_n(x)-f_n(y)|^p\right)^\frac{1}{p}\leq b\,d(x,y).
	\end{align*}
	Hence $\{g_n\}_{n}$
	is a 
	Lipschitz p-Bessel sequence  for   the pointed metric space $(\mathcal{M},
	z)$ and we apply Theorem \ref{DEFINITIONEXISTENCE} to $\{g_n\}_{n}$ which gives the result.
\end{proof}
\begin{definition}\label{DEFINITION}
	Let $\{f_n\}_{n}$ in $\operatorname{Lip}_0(\mathcal{M}, \mathbb{K})$  be
	a  Lipschitz p-Bessel sequence  for a pointed metric space
	$(\mathcal{M},
	0)$ and $\{\tau_n\}_{n}$ in a Banach space $\mathcal{X}$  be a  Lipschitz
	q-Bessel sequence for $\operatorname{Lip}_0(\mathcal{X}, \mathbb{K})$.  Let 
	$\{\lambda_n\}_n \in \ell^\infty(\mathbb{N})$. The Lipschitz operator
	\begin{align*}
	M_{\lambda,f, \tau}\coloneqq  \sum_{n=1}^{\infty}\lambda_n  (\tau_n\otimes f_n)
	\end{align*}
	is called as the \textbf{Lipschitz $(p,q)$-Bessel multiplier}. The sequence $\{\lambda_n\}_n$ is called as \textbf{symbol} for $	M_{\lambda,f, \tau}.$
\end{definition} 
We easily see that 	Definition \ref{DEFINITION} generalizes Definition 3.2
in (\cite{RAHIMIBALAZSMUL}). By varying the symbol and fixing other parameters in
the multiplier we get map from $\ell^\infty(\mathbb{N})$ to  $\operatorname{Lip}_0(\mathcal{M}, \mathcal{X})$. Property of
this map for Hilbert space was derived by Balazs (Lemma 7.1 in (\cite{BALAZSBASIC}))
and for Banach spaces it is due to Rahimi and Balazs (Proposition 3.3 in 
(\cite{RAHIMIBALAZSMUL})). In the next proposition we study it in the context of
metric spaces.
\begin{proposition}\label{INJECTIVE}
	Let $\{f_n\}_{n}$ in $\operatorname{Lip}_0(\mathcal{M}, \mathbb{K})$  be
	a  Lipschitz p-Bessel sequence  for  $(\mathcal{M},0)$  with non-zero
	elements, $\{\tau_n\}_{n}$ in $\mathcal{X}$  be a  q-Riesz sequence  for 
	$\operatorname{Lip}_0(\mathcal{X}, \mathbb{K})$ and $\{\lambda_n\}_n \in
	\ell^\infty(\mathbb{N})$. Then the mapping 
	\begin{align*}
	T:\ell^\infty(\mathbb{N})\ni \{\lambda_n\}_n \mapsto M_{\lambda,f, \tau}
	\in \operatorname{Lip}_0(\mathcal{M}, \mathcal{X})
	\end{align*}
	is a well-defined injective bounded linear operator.	
\end{proposition}
\begin{proof}
	From the norm estimate of $M_{\lambda,f, \tau}$, we see that  $T$ is a 
	well-defined bounded linear operator. Let $\{\lambda_n\}_n, \{\mu_n\}_n \in
	\ell^\infty(\mathbb{N})$ be such that $M_{\lambda,f, \tau}
	=T\{\lambda_n\}_n=T\{\mu_n\}_n=M_{\mu,f, \tau} $. Then
	$\sum_{n=1}^{\infty}\lambda_n  f_n (x)\tau_n =M_{\lambda,f, \tau}x =M_{\mu,f,
		\tau}x=\sum_{n=1}^{\infty}\mu_n  f_n (x)\tau_n $, $\forall x \in
	\mathcal{M}$ $\Rightarrow$ $\sum_{n=1}^{\infty}(\lambda_n-\mu_n)  f_n
	(x)\tau_n=0$,  $\forall x \in \mathcal{M}$. Now using Inequality
	(\ref{RIESZSEQUENCEINEQUALITY}), 
	\begin{align*}
	&a \left(\sum_{n=1 }^\infty|(\lambda_n-\mu_n)  f_n
	(x)|^q\right)^\frac{1}{q}\leq \left\|\sum_{n=1}^\infty (\lambda_n-\mu_n)  f_n
	(x)\tau_n\right\|=0, \quad \forall x \in \mathcal{M}\\
	&\implies (\lambda_n-\mu_n)  f_n (x)=0, \quad \forall n \in \mathbb{N}, 
	\forall x \in \mathcal{M}.
	\end{align*}
	Let  $n \in \mathbb{N}$ be fixed. Since $f_n\neq 0$, there exists $x \in
	\mathcal{M}$ such that $f_n(x)\neq 0$. Therefore we get $\lambda_n-\mu_n=0$. By
	varying $n \in \mathbb{N}$ we arrive at $\lambda_n=\mu_n$, $\forall n \in
	\mathbb{N}$. Hence $T$ is injective.
\end{proof}
\section{CONTINUITY PROPERTIES OF MULTIPLIERS}
 In Proposition \ref{RAHIMIBALAZSMULTIPLIERCOMPACT},  it was obtained  that whenever the symbol is in $c_0(\mathbb{N})$, then the multiplier is compact. Using the notion of Lipschitz compact operator (Definition \ref{LIPSCHITZCOMPACTDEFINITION}), we derive non linear analogue of Proposition \ref{RAHIMIBALAZSMULTIPLIERCOMPACT}.
\begin{proposition}\label{MULTIPLIERISCOMPACT}
	Let $\{f_n\}_{n}$ in $\operatorname{Lip}_0(\mathcal{M}, \mathbb{K})$  be
	a  Lipschitz p-Bessel sequence  for  $(\mathcal{M}, 0)$ with bound $b$
	and $\{\tau_n\}_{n}$ in $\mathcal{X}$  be a  Lipschitz q-Bessel sequence 
	for $\operatorname{Lip}_0(\mathcal{X}, \mathbb{K})$ with bound $d$. 	If
	$\{\lambda_n\}_n \in c_0(\mathbb{N})$, then $M_{\lambda,f, \tau}$ is a Lipschitz
	compact operator.
\end{proposition}
\begin{proof}
	For each  $m \in \mathbb{N}$, define $	M_{\lambda_m,f, \tau}\coloneqq   \sum_{n=1}^{m}\lambda_n  (\tau_n\otimes f_n )$. Then $	M_{\lambda_m,f, \tau}$ is a Lipschitz finite rank operator (Theorem \ref{LIPSCHITCOMPACTIFFLINEAR}). Now 
	\begin{align*}
	\|M_{\lambda_m,f, \tau}-M_{\lambda,f,
		\tau}\|_{\operatorname{Lip}_0}&=\sup_{x, y \in \mathcal{M}, ~x\neq y}
	\frac{\|(M_{\lambda_m,f, \tau}-M_{\lambda,f, \tau})x-(M_{\lambda_m,f,
			\tau}-M_{\lambda,f, \tau})y\|}{d(x,y)}\\
	&=\sup_{x, y \in \mathcal{M}, ~x\neq y}
	\frac{\left\|\sum_{n=m+1}^{\infty}\lambda_n  f_n
		(x)\tau_n-\sum_{n=m+1}^{\infty}\lambda_n  f_n (y)\tau_n\right\|}{d(x,y)}\\
	&=\sup_{x, y \in \mathcal{M}, ~x\neq y}
	\frac{\left\|\sum_{n=m+1}^{\infty}\lambda_n  (f_n
		(x)-f_n(y))\tau_n\right\|}{d(x,y)}\\
	&\leq bd\sup_{m+1\leq n<\infty }|\lambda_n| \to 0 \text{ as } m \to \infty.
	\end{align*}
	Hence $M_{\lambda,f, \tau}$ is the limit of a sequence
	of Lipschitz finite rank operators $\{M_{\lambda_m,f, \tau}\}_{m=1}^\infty$ with respect to the Lipschitz norm. Thus $M_{\lambda,f, \tau}$ is Lipschitz approximable and from Theorem
	\ref{LIPSCHITZAPPROMABLEISCOMPACT} it follows that $M_{\lambda,f, \tau}$ is
	Lipschitz compact.
\end{proof}
We now study the properties of multiplier by changing its parameters. 
Following result extends Theorem \ref{RAHIMIBALAZSMULTIPLIERCONTINUITY}.
\begin{theorem}\label{MULTIPLIERISWELL}
	Let $\{f_n\}_{n}$ in $\operatorname{Lip}_0(\mathcal{M}, \mathbb{K})$  be a
	 Lipschitz p-Bessel sequence  for  $\mathcal{M}$ with bound $b$ and
	$\{\tau_n\}_{n}$ in $\mathcal{X}$  be a  Lipschitz q-Bessel sequence  for
	$\operatorname{Lip}_0(\mathcal{X}, \mathbb{K})$ with bound $d$ and
	$\{\lambda_n\}_n \in \ell^\infty(\mathbb{N})$. 	Let $k \in \mathbb{N}$ and let 
	$\lambda^{(k)}=\{\lambda_1^{(k)},\lambda_2^{(k)}, \dots \}$,
	$\lambda=\{\lambda_1,\lambda_2, \dots \}$, 
	$\tau^{(k)}=\{\tau_1^{(k)}, \tau_2^{(k)}, \dots\}$,
	$\tau_n^{k} \in \mathcal{X}$, $\tau=\{\tau_1, \tau_2, \dots\}$. Assume that for
	each $k$, $\lambda^{(k)}\in \ell^\infty(\mathbb{N})$  and
	$\tau^{(k)}$  is  a pointed Lipschitz q-Bessel sequence  for
	$\operatorname{Lip}_0(\mathcal{X}, \mathbb{K})$.
	\begin{enumerate}[label=(\roman*)]
		\item If $\lambda^{(k)} \to \lambda $ as $k \rightarrow \infty $ in p-norm,  then
		\begin{align*}
		\|M_{\lambda^{(k)},f, \tau}-M_{\lambda,f, \tau}\|_{\operatorname{Lip}_0} \to 0 \text{ as } k \to \infty.
		\end{align*}
		\item If $\{\lambda_n\}_n \in \ell^p(\mathbb{N})$ and  $\sum_{n=1}^{\infty}\|\tau_n^{(k)}-\tau_n\|^q \to 0 \text{ as } k \to \infty $, then 
		\begin{align*}
		\|M_{\lambda, f, \tau^{(k)}}-M_{\lambda,f, \tau}\|_{\operatorname{Lip}_0} \to 0 \text{ as } k \to \infty.
		\end{align*}
	\end{enumerate}
\end{theorem}
\begin{proof}
	\begin{enumerate}[label=(\roman*)]
		\item Using Theorem \ref{DEFINITIONEXISTENCE},
		\begin{align*}
		&\|M_{\lambda^{(k)},f, \tau}-M_{\lambda,f,
			\tau}\|_{\operatorname{Lip}_0}\\&=\sup_{x, y \in \mathcal{M}, ~x\neq y}
		\frac{\|(M_{\lambda^{(k)},f, \tau}-M_{\lambda,f, \tau})x-(M_{\lambda^{(k)},f,
				\tau}-M_{\lambda,f, \tau})y\|}{d(x,y)}\\
		&=\sup_{x, y \in \mathcal{M}, ~x\neq y}
		\frac{\left\|\sum_{n=1}^{\infty}(\lambda_n^{(k)}-\lambda_n)f_n(x)\tau_n-\sum_{
				n=1}^{\infty}(\lambda_n^{(k)}-\lambda_n)f_n(y)\tau_n\right\|}{d(x,y)}\\
		&=\sup_{x, y \in \mathcal{M}, ~x\neq y}
		\frac{\left\|\sum_{n=1}^{\infty}(\lambda_n^{(k)}
			-\lambda_n)(f_n(x)-f_n(y))\tau_n\right\|}{d(x,y)}\\
		&\leq bd \sup_{n\in \mathbb{N}}|\lambda_n^{(k)}-\lambda_n|=bd \|\{\lambda_n^{(k)}-\lambda_n\}_n\|_\infty \\
		& \leq bd \|\{\lambda_n^{(k)}-\lambda_n\}_n\|_p \to 0 \text{ as } k \to \infty.
		\end{align*}
		\item Using Holder's inequality,
		\begin{align*}
		&	\|M_{\lambda,f, \tau^{(k)}}-M_{\lambda,f, \tau}\|_{\operatorname{Lip}_0}\\
		&=\sup_{x, y \in \mathcal{M}, ~x\neq y} \frac{\|(M_{\lambda,f,
				\tau^{(k)}}-M_{\lambda,f, \tau})x-(M_{\lambda,f, \tau^{(k)}}-M_{\lambda,f,
				\tau})y\|}{d(x,y)}\\
		&=\sup_{x, y \in \mathcal{M}, ~x\neq y}
		\frac{\left\|\sum_{n=1}^{\infty}\lambda_nf_n(x)(\tau_n^{(k)}-\tau_n)-\sum_{n=1}^
			{\infty}\lambda_nf_n(y)(\tau_n^{(k)}-\tau_n)\right\|}{d(x,y)}\\
		&=\sup_{x, y \in \mathcal{M}, ~x\neq y}
		\frac{\left\|\sum_{n=1}^{\infty}\lambda_n(f_n(x)-f_n(y))(\tau_n^{(k)}
			-\tau_n)\right\|}{d(x,y)}\\
		&=\sup_{x, y \in \mathcal{M}, ~x\neq y} \ \ \sup _{\phi \in
			\mathcal{X}^*,\|\phi\|\leq 1}
		\frac{\left|\sum_{n=1}^{\infty}\lambda_n(f_n(x)-f_n(y))\phi(\tau_n^{(k)}
			-\tau_n)\right|}{d(x,y)}\\
		&\leq \sup_{x, y \in \mathcal{M}, ~x\neq y} \ \ \sup _{\phi
			\in \mathcal{X}^*,\|\phi\|\leq 1}
		\frac{\left(\sum_{n=1}^{\infty}|\lambda_n(f_n(x)-f_n(y))|^p\right)^\frac{1}{p}
			\left(\sum_{n=1}^{\infty}|\phi(\tau_n^{(k)}-\tau_n)|^q\right)^\frac{1}{q}}{d(x,
			y)}\\
		&\leq \sup_{x, y \in \mathcal{M}, ~x\neq y} \ \ \sup _{\phi
			\in \mathcal{X}^*,\|\phi\|\leq 1}\\
		&\quad 
		\left\{\frac{\left(\sum_{n=1}^{\infty}|\lambda_n|^p\right)^\frac{1}{p}\left(\sum_{n=1}^
			{\infty}|f_n(x)-f_n(y)|^p\right)^\frac{1}{p}\left(\sum_{n=1}^{\infty}
			|\phi(\tau_n^{(k)}-\tau_n)|^q\right)^\frac{1}{q}}{d(x,y)}\right\}\\
		&\leq b \|\{\lambda_n\}_n\|_p\left(\sum_{n=1}^{\infty}\|\tau_n^{(k)}-\tau_n\|^q\right)^\frac{1}{q} \to 0 \text{ as } k \to \infty.
		\end{align*}
	\end{enumerate}	
\end{proof}

{\onehalfspacing \chapter{p-APPROXIMATE SCHAUDER FRAMES FOR BANACH SPACES}\label{chap4} }
\section{p-APPROXIMATE SCHAUDER FRAMES}
Let $\mathcal{X}$ be a separable Banach space and $\mathcal{X}^*$ be its dual.  Equation (\ref{GFE}) motivated Casazza, Dilworth, Odell, Schlumprecht, and Zsak,  to define the notion of Schauder frame for $\mathcal{X}$ in 2008.
\begin{definition}(\cite{CASAZZA})\label{FRAMING}
	Let $\{\tau_n\}_n$ be a sequence in  $\mathcal{X}$ and 	$\{f_n\}_n$ be a sequence in  $\mathcal{X}^*.$ The pair $ (\{f_n \}_{n}, \{\tau_n \}_{n}) $ is said to be a \textbf{Schauder frame} for $\mathcal{X}$ if 
	\begin{align}\label{SFEQUA}
	x=\sum_{n=1}^\infty
	f_n(x)\tau_n, \quad \forall x \in
	\mathcal{X}.
	\end{align}
\end{definition} 
 Definition \ref{FRAMING} was generalized for $\mathbb{R}^n$  by Thomas in her Master's thesis and later to Banach spaces by  Freeman, Odell,  Schlumprecht, and Zsak. 
\begin{definition}(\cite{FREEMANODELL, THOMAS})\label{ASFDEF}
	Let $\{\tau_n\}_n$ be a sequence in  $\mathcal{X}$ and 	$\{f_n\}_n$ be a sequence in  $\mathcal{X}^*.$ The pair $ (\{f_n \}_{n}, \{\tau_n \}_{n}) $ is said to be an \textbf{approximate Schauder frame} (ASF) for $\mathcal{X}$ if 
	\begin{align}\label{ASFEQUA}
	\text {(\textbf{Frame operator})}\quad	S_{f, \tau}:\mathcal{X}\ni x \mapsto S_{f, \tau}x\coloneqq \sum_{n=1}^\infty
	f_n(x)\tau_n \in
	\mathcal{X}
	\end{align}
	is a well-defined bounded linear, invertible operator.
\end{definition} 
Note that whenever $S_{f, \tau}=I_\mathcal{X}$, the identity operator on $\mathcal{X}$, Definition \ref{ASFDEF} reduces to Definition \ref{FRAMING}. Since $S_{f, \tau}$ is invertible, it follows that there are $a,b>0$ such that 
\begin{align*}
a\|x\|\leq \left\|\sum_{n=1}^\infty
f_n(x)\tau_n \right\|\leq b\|x\|, \quad \forall x \in  \mathcal{X}.
\end{align*} 
We call $a$ as \textbf{lower ASF bound} and $b$ as \textbf{upper ASF bound}. Supremum (resp. infimum) of the set of all lower (resp. upper) ASF bounds is called \textbf{optimal lower} (resp.   \textbf{optimal upper}) ASF bound. From the theory of bounded linear operators between Banach spaces, one sees that optimal lower frame bound is $ \|S_{f, \tau}^{-1}\|^{-1}$ and optimal upper frame bound is $  \|S_{f,\tau}\|$. Advantage of ASF over Schauder frame is that it is more easier to get the operator in (\ref{ASFEQUA}) as invertible than obtaining Equation (\ref{SFEQUA}).
\begin{example}(\cite{FREEMANODELL})
Let  $2<p<\infty$ and $\{\lambda_n\}_n$  be an unbounded sequence of scalars. For $a\in \mathbb{R}$, define 
\begin{align*}
		T_a:	\mathcal{L}^p(\mathbb{R})\ni f \mapsto T_af \in \mathcal{L}^p(\mathbb{R});\quad  T_af: \mathbb{R} \ni x \mapsto ( T_af)(x)\coloneqq f(x-a)\in \mathbb{C}. 
\end{align*}
Then there exist  $\phi \in \mathcal{L}^p(\mathbb{R})$ and a sequence $\{f_n \}_{n}$, $f_n \in ( \mathcal{L}^p(\mathbb{R}))^*$, $\forall n \in \mathbb{N}$ such that $ (\{f_n \}_{n}, \{T_{\lambda_n}\phi \}_{n}) $ is an ASF for $	\mathcal{L}^p(\mathbb{R})$.
\end{example}
\begin{definition}\label{PASFDEF}
	An ASF $ (\{f_n \}_{n}, \{\tau_n \}_{n}) $  for $\mathcal{X}$	is said to be a \textbf{p-approximate Schauder frame} (p-ASF), $p \in [1, \infty)$ if both the maps 
	\begin{align}\label{ALIGNP}
	&\text{(\textbf{Analysis operator})} \quad \theta_f: \mathcal{X}\ni x \mapsto \theta_f x\coloneqq \{f_n(x)\}_n \in \ell^p(\mathbb{N}) \text{ and } \\
	&\text{(\textbf{Synthesis operator})} \quad\theta_\tau : \ell^p(\mathbb{N}) \ni \{a_n\}_n \mapsto \theta_\tau \{a_n\}_n\coloneqq \sum_{n=1}^\infty a_n\tau_n \in \mathcal{X}\label{AAA}
	\end{align}
	are well-defined bounded linear operators. A Schauder frame which is a p-ASF is called as a \textbf{simple  p-ASF} or \textbf{Parseval  p-ASF}.
\end{definition}
It can be easily observed  that a p-approximate Schauder frame is an approximate Schauder frame and a Schauder frame is an approximate Schauder frame. We now give an example to  show that the set of all p-approximate Schauder frames is strictly smaller than the set of all approximate Schauder frames. Let  $\mathcal{X}=\mathbb{K}$. Define $\tau_n\coloneqq \frac{1}{n^2}$, $f_n(x)=x, \forall x \in \mathbb{K}$, $\forall n \in \mathbb{N}$. Then $\sum_{n=1}^{\infty}f_n(x)\tau_n=\frac{\pi^2}{6}x$, $\forall x \in \mathbb{K}$. Therefore $ (\{f_n\}_{n}, \{\tau_n\}_{n}) $ is an  approximate Schauder frame for 	$\mathcal{X}$. Let $ x \in \mathbb{K}$ be non zero. Then for every $p\in[1,\infty)$,
\begin{align*}
\sum_{n=1}^{m}|f_n(x)|^p=m|x|^p \to \infty \quad \text{as} \quad m \to \infty.
\end{align*}
Thus $\{f_n(x)\}_n\notin \ell^p(\mathbb{N})$ for any $p\in[1,\infty)$ and hence $ (\{f_n\}_{n}, \{\tau_n\}_{n}) $ is not a p-ASF for  any $p\in[1,\infty)$. We next note that  there is a bijection between the set of approximate Schauder frames and the set of all Schauder frames (Lemma 3.1 in (\cite{FREEMANODELL})).
We observe that, in terms of inequalities, (\ref{ALIGNP}) and (\ref{AAA}) say that there exist $c,d>0$, such that 
\begin{align}
&\left(\sum_{n=1}^\infty
|f_n(x)|^p\right)^\frac{1}{p}\leq c \|x\|, \quad \forall x \in \mathcal{X} \text{ and }\label{FIRSTINEQUALITYPASF} \\
&\left\|\sum_{n=1}^\infty a_n\tau_n\right\|\leq d \left(\sum_{n=1}^\infty
|a_n|^p\right)^\frac{1}{p}, \quad \forall \{a_n\}_n  \in \ell^p(\mathbb{N})\label{SECONDINEQUALITYPASF}.
\end{align}
We now give various examples of p-ASFs.
\begin{example}
	Let $p\in[1,\infty)$ and $U:\mathcal{X} \rightarrow\ell^p(\mathbb{N})$, $ V: \ell^p(\mathbb{N})\to \mathcal{X}$ be bounded linear operators such that $VU$ is bounded invertible. Let $\{e_n\}_n$ denote the standard Schauder basis for  $\ell^p(\mathbb{N})$  and let $\{\zeta_n\}_n$ denote the coordinate functionals associated with $\{e_n\}_n$.	Define 
	\begin{align*}
	f_n\coloneqq \zeta_n U, \quad \tau_n\coloneqq Ve_n, \quad \forall n \in \mathbb{N}.
	\end{align*}
	Then $ (\{f_n\}_{n}, \{\tau_n\}_{n}) $ is a p-ASF for 	$\mathcal{X}$. In particular, if $U:\ell^p(\mathbb{N}) \rightarrow\ell^p(\mathbb{N})$ is bounded invertible, then  $(\{\zeta_nU\}_{n}, \{U^{-1}e_n\}_{n}) $ is a p-ASF for $\ell^p(\mathbb{N})$.
\end{example}
\begin{example}
	Let $p\in[1,\infty)$ and  $\{\tau_n\}_{n=1}^m $	be a basis for a finite dimensional Banach space $\mathcal{X}$. Choose any basis $\{f_n\}_{n=1}^m $	 for  $\mathcal{X}^*$. We claim that $(\{f_n\}_{n=1}^m, \{\tau_n\}_{n=1}^m) $ is a p-ASF for $\mathcal{X}$. To prove this, since $\mathcal{X}$ is finite dimensional, it suffices to prove that the map $\mathcal{X}\ni x \mapsto \sum_{n=1}^{m}f_n(x)\tau_n\in \mathcal{X}$ is injective. Let $x \in \mathcal{X}$ be such that $\sum_{n=1}^{m}f_n(x)\tau_n=0$. Since $\{\tau_n\}_{n=1}^m $	is a basis for   $\mathcal{X}$, we then have $f_1(x)=\cdots=f_n(x)=0$. We then have $f(x)=0, \forall f \in \mathcal{X}^*$. Hahn-Banach theorem now says that $x=0$. Hence the claim holds and consequently $(\{f_n\}_{n=1}^m, \{\tau_n\}_{n=1}^m) $ is a p-ASF for $\mathcal{X}$.
\end{example}
\begin{example}
	Recall that  a spanning set is a frame for a finite dimensional Hilbert space (\cite{HANKORNELSONLARSON}). We now generalize this for p-ASFs. Let $p\in[1,\infty)$, $\mathcal{X}$ be a finite dimensional Banach space and $\{\tau_n\}_{n=1}^m $	be a spanning set for $\mathcal{X}$. We claim that there exists a collection $\{f_n\}_{n=1}^m $ in  $\mathcal{X}^*$ such that $ (\{f_n\}_{n=1}^m, \{\tau_n\}_{n=1}^m) $ is a p-ASF for 	$\mathcal{X}$. Since $\{\tau_n\}_{n=1}^m $	spans $\mathcal{X}$, there exists a basis in the collection $\{\tau_n\}_{n=1}^m $. By rearranging, if necessary, we may assume that $\{\tau_n\}_{n=1}^r $ is a basis for $\mathcal{X}$. Let $\{f_n\}_{n=1}^r $ be the dual basis for $\{\tau_n\}_{n=1}^r $. Choose linear operators $U,V:\mathcal{X}\to \mathcal{X}$ such that $VU$ is injective or surjective. If we now set $f_{r+1}=\cdots=f_n=0$, it then follows that $(\{f_nU\}_{n=1}^m, \{V\tau_n\}_{n=1}^m) $ is a p-ASF for $\mathcal{X}$.
\end{example}
\begin{example}
	Let $\mathcal{X}$ be a Banach space which admits a Schauder basis $\{\omega_n\}_{n} $ and let $\{g_n\}_{n}$ be the coordinate functionals associated with $\{e_n\}_n$.  Let   $U, V:\mathcal{X} \rightarrow \mathcal{X}$ be bounded linear operators such that $VU$ is invertible.  	Define 
	\begin{align*}
	f_n\coloneqq g_n U, \quad \tau_n\coloneqq V\omega_n, \quad \forall n \in \mathbb{N}.
	\end{align*}
	Then $ (\{f_n\}_{n}, \{\tau_n\}_{n}) $ is an approximate Schauder frame  for 	$\mathcal{X}$.   	 If $VU=I_\mathcal{X}$, then $ (\{f_n\}_{n}, \{\tau_n\}_{n}) $ is a Schauder frame  for 	$\mathcal{X}$.  
\end{example}
\begin{example}\label{EXAMPLE2}
	Let $p\in[1,\infty)$ and $U:\mathcal{X} \rightarrow\ell^p(\mathbb{N})$, $ V: \ell^p(\mathbb{N})\to \mathcal{X}$ be bounded linear operators such that $VU$ is  invertible. Let $\{e_n\}_n$ denote the canonical Schauder basis for  $\ell^p(\mathbb{N})$  and let $\{\zeta_n\}_n$ denote the coordinate functionals associated with $\{e_n\}_n$.	Define 
	\begin{align*}
	f_n\coloneqq \zeta_n U, \quad \tau_n\coloneqq Ve_n, \quad \forall n \in \mathbb{N}.
	\end{align*}
	Then $ (\{f_n\}_{n}, \{\tau_n\}_{n}) $ is a p-ASF for 	$\mathcal{X}$. 
\end{example}
Now we have Banach space analogous of Theorem \ref{MOSTIMPORTANT}.
\begin{theorem}\label{OURS}
	Let $ (\{f_n \}_{n}, \{\tau_n \}_{n}) $ be a p-ASF for $\mathcal{X}$.  Then
	\begin{enumerate}[label=(\roman*)]
		\item We have 
		\begin{align}\label{REPBANACH}
		x=\sum_{n=1}^\infty (f_nS_{f, \tau}^{-1})(x) \tau_n=\sum_{n=1}^\infty
		f_n(x) S_{f, \tau}^{-1}\tau_n, \quad \forall x \in
		\mathcal{X}.
		\end{align}
		\item $ (\{f_nS_{f, \tau}^{-1} \}_{n}, \{S_{f, \tau}^{-1} \tau_n \}_{n}) $ is a p-ASF for $\mathcal{X}$.
		\item The analysis operator 
		$
		\theta_f: \mathcal{X} \ni x \mapsto \{f_n(x) \}_n \in \ell^p(\mathbb{N})
		$
		is injective. 
		\item 
		The synthesis operator 
		$
		\theta_\tau: \ell^p(\mathbb{N}) \ni \{a_n \}_n \mapsto \sum_{n=1}^\infty a_n\tau_n \in \mathcal{X}
		$
		is surjective.
		\item Frame operator 
		splits as $S_{f, \tau}=\theta_\tau\theta_f.$
		\item  $P_{f, \tau}\coloneqq\theta_fS_{f,\tau}^{-1}\theta_\tau:\ell^p(\mathbb{N})\to \ell^p(\mathbb{N})$ is a projection onto   $\theta_f(\mathcal{X})$.
	\end{enumerate}
\end{theorem}
\begin{proof}
	First follows from the continuity and linearity of $S_{f, \tau}^{-1}$.  Because $S_{f, \tau}$ is invertible, we have (ii). Again invertibility of  $S_{f, \tau}$ makes $\theta_f$  injective and $\theta_\tau$  surjective.  (v) and (vi) are routine calculations.
\end{proof}
Now we can derive a generalization of Theorem \ref{HOLUBTHEOREM} for Banach spaces.
\begin{theorem}\label{THAFSCHAR}
 Let $\{e_n\}_n$ denote the standard Schauder basis for  $\ell^p(\mathbb{N})$  and let $\{\zeta_n\}_n$ denote the coordinate functionals associated with $\{e_n\}_n$.	A pair  $ (\{f_n\}_{n}, \{\tau_n\}_{n}) $ is a p-ASF for 	$\mathcal{X}$
	if and only if 
	\begin{align*}
	f_n=\zeta_n U, \quad \tau_n=Ve_n, \quad \forall n \in \mathbb{N},
	\end{align*}  where $U:\mathcal{X} \rightarrow\ell^p(\mathbb{N})$, $ V: \ell^p(\mathbb{N})\to \mathcal{X}$ are bounded linear operators such that $VU$ is bounded invertible.
\end{theorem}
\begin{proof}
	$(\Leftarrow)$ Clearly $\theta_f$ and $\theta_\tau$ are bounded linear operators. Now let $x\in \mathcal{X}$. Then 
	\begin{align}\label{PASFFIRSTTHEOREMEQUATION}
	S_{f, \tau}x= \sum_{n=1}^\infty
	f_n(x)\tau_n=\sum_{n=1}^\infty \zeta_n(Ux)Ve_n=V\left(\sum_{n=1}^\infty \zeta_n(Ux)e_n\right)=VUx.
	\end{align} 
	Hence $S_{f, \tau}$ is bounded invertible. \\
	$(\Rightarrow)$ Define $U\coloneqq \theta_f$, $V\coloneqq \theta_\tau$. Then $\zeta_nUx=\zeta_n\theta_fx=\zeta_n(\{f_k(x)\}_k)=f_n(x)$, $\forall x \in \mathcal{X}$, $Ve_n=\theta_\tau e_n=\tau_n$, $\forall n \in \mathbb{N}$ and $VU=\theta_\tau \theta_f=S_{f, \tau}$ which is bounded invertible.
\end{proof}
Note that Theorem \ref{THAFSCHAR}  generalizes Theorem \ref{HOLUBTHEOREM}. In fact, in the case of Hilbert spaces, Theorem \ref{THAFSCHAR} reads  as ``A sequence $\{\tau_n\}_n$ in  $\mathcal{H}$ is a
frame for $\mathcal{H}$	if and only if there exists a  bounded linear operator $T:\ell^2(\mathbb{N}) \to \mathcal{H}$ such that $Te_n=\tau_n$, for all $n \in \mathbb{N}$ and $TT^*$ is invertible". Now we know that $TT^*$ is invertible if and only if $T$ is surjective.\\
Since every separable Hilbert space admits an orthonormal basis, the existence of orthonormal basis in  Theorem \ref{OLECHA} is automatic. On the other hand, Enflo showed that there are  separable Banach spaces which do not  have Schauder basis (\cite{JAMES, ENFLO}). Thus to obtain analogous of Theorem \ref{OLECHA} for  Banach spaces, we need to impose condition on $\mathcal{X}$.
\begin{theorem}\label{SCHFRS}
	Assume that $\mathcal{X}$ admits a Schauder basis $\{\omega_n\}_n$. Let $\{g_n\}_n$ denote the coordinate functionals associated with $\{\omega_n\}_n$. Assume that 
	\begin{align}\label{ASSUMPTIONNEEDED}
	\{g_n(x)\}_n \in \ell^p(\mathbb{N}), \quad \forall x \in \mathcal{X}.
	\end{align} 
	Then a  pair  $ (\{f_n\}_{n}, \{\tau_n\}_{n}) $ is a p-ASF for 	$\mathcal{X}$
	if and only if 
	\begin{align*}
	f_n=g_n U, \quad \tau_n=V\omega_n, \quad \forall n \in \mathbb{N},
	\end{align*}  where $U, V:\mathcal{X} \rightarrow  \mathcal{X}$ are bounded linear operators such that $VU$ is bounded invertible.
\end{theorem}
\begin{proof}
	$(\Leftarrow)$ This is similar to the calculation done in (\ref{PASFFIRSTTHEOREMEQUATION}).\\
	$(\Rightarrow)$ Let $T$ be the map defined by 
	\begin{align*}
	T:\mathcal{X}\ni \sum_{n=1}^\infty a_n\omega_n\mapsto \sum_{n=1}^\infty a_ne_n\in \ell^p(\mathbb{N}).
	\end{align*}
	Assumption (\ref{ASSUMPTIONNEEDED}) then says  that $T$ is a bounded invertible operator with inverse $ T^{-1} :\ell^p(\mathbb{N}) \ni \sum_{n=1}^\infty b_ne_n \mapsto \sum_{n=1}^\infty b_n\omega_n \in \mathcal{X}$. Define $ U\coloneqq T^{-1}\theta_f$ and $V\coloneqq\theta_\tau T$. Then $ U,V$ are bounded  such that  $ VU=(\theta_\tau T)(T^{-1}\theta_f)=\theta_\tau\theta_f=S_{f,\tau}$ is invertible  and  for $ x \in \mathcal{X}$ we have 
	\begin{align*}
	(g_nU)(x)&= g_n(T^{-1}\theta_fx)=g_n(T^{-1}(\{f_k(x)\}_{k})) 
	= g_n\left(\sum_{k=1}^\infty f_k(x)T^{-1}e_k\right)
	\\&=g_n\left(\sum_{k=1}^\infty f_k(x)\omega_k\right)
	=\sum_{k=1}^\infty f_k(x)g_n(\omega_k)=f_n(x), \quad \forall x \in \mathcal{X}
	\end{align*} 
	and $ V\omega_n=\theta_\tau T\omega_n=\theta_\tau e_n=\tau_n, \forall  n \in \mathbb{N}$.
\end{proof}
\section{DUAL FRAMES FOR p-APPROXIMATE SCHAUDER FRA-MES}
Equation (\ref{REPBANACH}) motivates us to define the notion of dual frame as follows.
\begin{definition}\label{SIMILARITYMINE}
	Let $ (\{f_n\}_{n}, \{\tau_n\}_{n}) $ be a p-ASF for 	$\mathcal{X}$. 	A p-ASF $ (\{g_n \}_{n}, \{\omega_n \}_{n}) $ for $\mathcal{X}$ is a \textbf{dual p-ASF} for $ (\{f_n \}_{n}, \{\tau_n \}_{n}) $ if 
	\begin{align*}
	x=\sum_{n=1}^\infty g_n(x) \tau_n=\sum_{n=1}^\infty
	f_n(x) \omega_n, \quad \forall x \in
	\mathcal{X}.
	\end{align*}
\end{definition}
Note that dual frames always exist. In fact, the  Equation (\ref{REPBANACH}) shows that the frame $ (\{f_nS_{f, \tau}^{-1} \}_{n}, \{S_{f, \tau}^{-1} \tau_n \}_{n}) $ is a  dual for $ (\{f_n\}_{n}, \{\tau_n\}_{n}) $.  We call the  frame $ (\{f_nS_{f, \tau}^{-1} \}_{n}, $ $ \{S_{f, \tau}^{-1} \tau_n \}_{n}) $ as the  \textbf{canonical dual} for  $ (\{f_n\}_{n}, \{\tau_n\}_{n}) $. With this notion, the following theorem  follows easily.
\begin{theorem}
Let $ (\{f_n\}_{n}, \{\tau_n\}_{n}) $ be a  p-ASF for $ \mathcal{X}$ with frame bounds $ a$ and $ b.$ Then
\begin{enumerate}[label=(\roman*)]
	\item The canonical dual p-ASF for the canonical dual p-ASF  for  $ (\{f_n\}_{n}, \{\tau_n\}_{n}) $ is itself.
	\item$ \frac{1}{b}, \frac{1}{a}$ are frame bounds for the canonical dual for $ (\{f_n\}_{n}, \{\tau_n\}_{n}) $.
	\item If $ a, b $ are optimal frame bounds for $ (\{f_n\}_{n}, \{\tau_n\}_{n}) $, then $ \frac{1}{b}, \frac{1}{a}$ are optimal  frame bounds for its canonical dual.
\end{enumerate} 
\end{theorem}
One can naturally ask when  a p-ASF has unique dual? An affirmative answer is in the following result.
\begin{proposition}
	Let $ (\{f_n\}_{n}, \{\tau_n\}_{n}) $ be a  p-ASF for $ \mathcal{X}$. If $\{\tau_n\}_{n}$ is a Schauder basis for   $\mathcal{X}$ and $ f_k(\tau_n)=\delta_{k,n},\forall k,n \in \mathbb{N}$, then  $ (\{f_n\}_{n}, \{\tau_n\}_{n}) $ has unique dual.
\end{proposition}
\begin{proof}
	Let $ (\{g_n \}_{n}, \{\omega_n \}_{n}) $  and $ (\{u_n \}_{n}, \{\rho_n \}_{n}) $ be two dual p-ASFs for $ (\{f_n\}_{n}, \{\tau_n\}_{n}) $. Then 
	\begin{align*}
	\sum_{n=1}^\infty(g_n(x)-u_n(x))\tau_n=0= \sum_{n=1}^\infty f_n(x)(\omega_n-\rho_n), \quad \forall x \in \mathcal{X}.
	\end{align*}
	First equality gives $ g_n=u_n, \forall n \in \mathbb{N}$ and  evaluating second equality at a fixed $ \tau_k$ gives $ \omega_k=\rho_k$. Since $k$ was arbitrary, proposition follows.
\end{proof}
We now characterize dual frames  by using analysis and synthesis operators.
\begin{proposition}\label{ORTHOGONALPRO}
	For two p-ASFs $ (\{f_n\}_{n}, \{\tau_n\}_{n}) $ and $ (\{g_n \}_{n}, \{\omega_n \}_{n}) $ for $\mathcal{X}$, the following are equivalent.
	\begin{enumerate}[label=(\roman*)]
		\item  $ (\{g_n \}_{n}, \{\omega_n \}_{n}) $ is a dual  for $ (\{f_n \}_{n}, \{\tau_n \}_{n}) $.
		\item $\theta_\tau\theta_g =\theta_\omega\theta_f =I_\mathcal{X}$.
	\end{enumerate}
\end{proposition}
Like Lemmas \ref{LILEMMA1}, \ref{LILEMMA2} and Theorem \ref{LITHM}, we now  characterize dual frames using standard Schauder basis for $\ell^p(\mathbb{N})$.
\begin{lemma}\label{ASFLEMMA1}
Let $\{e_n\}_n$ denote the standard Schauder basis for  $\ell^p(\mathbb{N})$  and let $\{\zeta_n\}_n$ denote the coordinate functionals associated with $\{e_n\}_n$.	Let  $ (\{f_n \}_{n}, \{\tau_n \}_{n}) $  be a  p-ASF for   $\mathcal{X}$. Then a  p-ASF  $ (\{g_n \}_{n}, \{\omega_n \}_{n}) $ for $\mathcal{X}$ is a dual  for $ (\{f_n \}_{n}, \{\tau_n \}_{n}) $ if and only if
	\begin{align*}
	g_n=\zeta_n U, \quad \omega_n=Ve_n, \quad \forall n \in \mathbb{N},
	\end{align*} 
	where $ U:\mathcal{X} \rightarrow\ell^p(\mathbb{N})$ is  a bounded right-inverse of $ \theta_\tau$, and  $V: \ell^p(\mathbb{N}) \rightarrow \mathcal{X}$ is a bounded left-inverse of $ \theta_f$ such that $ VU$ is bounded invertible.
\end{lemma}
\begin{proof}
	$(\Leftarrow)$ From the `if' part of proof of Theorem \ref{THAFSCHAR}, we get that $ (\{g_n \}_{n}, \{\omega_n \}_{n}) $ is a p-ASF for $\mathcal{X}$. We have to check for duality of $ (\{g_n \}_{n}, \{\omega_n \}_{n}) $. Also, we have $\theta_\tau\theta_g=\theta_\tau U=I_\mathcal{X} $, $ \theta_\omega\theta_f=V\theta_f =I_\mathcal{X}$.\\
	$(\Rightarrow)$ Let $ (\{g_n \}_{n}, \{\omega_n \}_{n}) $ be a dual p-ASF for  $ (\{f_n \}_{n}, \{\tau_n \}_{n}) $.  Then $\theta_\tau\theta_g =I_\mathcal{X} $, $ \theta_\omega\theta_f =I_\mathcal{X}$. Define $ U\coloneqq\theta_g, V\coloneqq\theta_\omega.$ Then $ U:\mathcal{X} \rightarrow\ell^p(\mathbb{N})$ is a bounded right-inverse of $ \theta_\tau$, and  $V: \ell^p(\mathbb{N}) \rightarrow \mathcal{X}$ is  a bounded left-inverse of $ \theta_f$ such that the operator $ VU=\theta_\omega\theta_g=S_{g,\omega}$ is invertible. Further,
	\begin{align*}
	(\zeta_nU)x=\zeta_n\left(\sum_{k=1}^\infty g_k(x)e_k\right)=\sum_{k=1}^\infty g_k(x)\zeta_n(e_k)=g_n(x), \quad \forall x \in \mathcal{X}
	\end{align*} 
	and $Ve_n=\theta_\omega e_n=\omega_n, \forall n \in \mathbb{N} $.
\end{proof}
\begin{lemma}\label{ASFLEMMA2}
	Let $ (\{f_n \}_{n}, \{\tau_n \}_{n}) $ be a  p-ASF for   $\mathcal{X}$. Then 
	\begin{enumerate}[label=(\roman*)]
		\item $R: \mathcal{X} \rightarrow \ell^p(\mathbb{N})$ is a bounded right-inverse of $ \theta_\tau$  if and only if 
		\begin{align*}
		R=\theta_fS_{f,\tau}^{-1}+(I_{\ell^p(\mathbb{N})}-\theta_fS_{f,\tau}^{-1}\theta_\tau)U
		\end{align*} where $U:\mathcal{X} \to \ell^p(\mathbb{N})$ is a bounded linear operator.
		\item  $ L:\ell^p(\mathbb{N})\rightarrow \mathcal{X}$ is a bounded left-inverse of $ \theta_f$ if and only if 
		\begin{align*}
		L=S_{f,\tau}^{-1}\theta_\tau+V(I_{\ell^p(\mathbb{N})}-\theta_fS_{f,\tau}^{-1}\theta_\tau),
		\end{align*} 
		where $V:\ell^p(\mathbb{N}) \to \mathcal{X}$ is a bounded linear operator. 
	\end{enumerate}		
\end{lemma}
 \begin{proof}
	\begin{enumerate}[label=(\roman*)]
		\item  $(\Leftarrow)$  $\theta_\tau(\theta_fS_{f,\tau}^{-1}+(I_{\ell^p(\mathbb{N})}-\theta_fS_{f,\tau}^{-1}\theta_\tau)U)=I_\mathcal{X}+\theta_\tau U-I_\mathcal{X}\theta_\tau U=I_\mathcal{X}$. Therefore $\theta_fS_{f,\tau}^{-1}+(I_{\ell^p(\mathbb{N})}-\theta_fS_{f,\tau}^{-1}\theta_\tau)U$ is a bounded right-inverse of $ \theta_\tau$.  
		
		$(\Rightarrow)$  Define $U\coloneqq R $. Then $\theta_fS_{f,\tau}^{-1}+(I_{\ell^p(\mathbb{N})}-\theta_fS_{f,\tau}^{-1}\theta_\tau)U=\theta_fS_{f,\tau}^{-1}+(I_{\ell^p(\mathbb{N})}-\theta_fS_{f,\tau}^{-1}\theta_\tau)R=\theta_fS_{f,\tau}^{-1}+R-\theta_fS_{f,\tau}^{-1}=R$.
		\item
		$(\Leftarrow)$  $(S_{f,\tau}^{-1}\theta_\tau+V(I_{\ell^p(\mathbb{N})}-\theta_fS_{f,\tau}^{-1}\theta_\tau))\theta_f=I_\mathcal{X}+V\theta_f-V\theta_fI_\mathcal{X}=I_\mathcal{X}$. Therefore  $S_{f,\tau}^{-1}\theta_\tau+V(I_{\ell^p(\mathbb{N})}-\theta_fS_{f,\tau}^{-1}\theta_\tau)$ is a bounded left-inverse of $\theta_f$.
		
		$(\Rightarrow)$  Define $V\coloneqq L$. Then $S_{f,\tau}^{-1}\theta_\tau+V(I_{\ell^p(\mathbb{N})}-\theta_fS_{f,\tau}^{-1}\theta_\tau) =S_{f,\tau}^{-1}\theta_\tau+L(I_{\ell^p(\mathbb{N})}-\theta_fS_{f,\tau}^{-1}\theta_\tau)=S_{f,\tau}^{-1}\theta_\tau+L-S_{f,\tau}^{-1}\theta_\tau= L$.
	\end{enumerate}		
\end{proof}
\begin{theorem}\label{ALLDUAL}
Let $\{e_n\}_n$ denote the standard Schauder basis for  $\ell^p(\mathbb{N})$  and let $\{\zeta_n\}_n$ denote the coordinate functionals associated with $\{e_n\}_n$.	Let $ (\{f_n \}_{n}, \{\tau_n \}_{n}) $ be a  p-ASF for   $\mathcal{X}$. Then a  p-ASF  $ (\{g_n \}_{n}, \{\omega_n \}_{n}) $ for $\mathcal{X}$ is a dual  for $ (\{f_n \}_{n}, \{\tau_n \}_{n}) $ if and only if
	\begin{align*}
	&g_n=f_nS_{f,\tau}^{-1}+\zeta_nU-f_nS_{f,\tau}^{-1}\theta_\tau U,\\
	&\omega_n=S_{f,\tau}^{-1}\tau_n+Ve_n-V\theta_fS_{f,\tau}^{-1}\tau_n, \quad \forall n \in \mathbb{N}
	\end{align*}
	such that the operator 
	\begin{align*}
	S_{f,\tau}^{-1}+VU-V\theta_fS_{f,\tau}^{-1}\theta_\tau U
	\end{align*}
	is bounded invertible, where   $U:\mathcal{X} \to \ell^p(\mathbb{N})$ and $ V:\ell^p(\mathbb{N})\to \mathcal{X}$ are bounded linear operators.
\end{theorem}
\begin{proof}
	Lemmas \ref{ASFLEMMA1} and  \ref{ASFLEMMA2} give the characterization of dual frame as 
	\begin{align*}
	&g_n=\zeta_n\theta_fS_{f,\tau}^{-1}+\zeta_nU-\zeta_n\theta_fS_{f,\tau}^{-1}\theta_\tau U=f_nS_{f,\tau}^{-1}+\zeta_nU-f_nS_{f,\tau}^{-1}\theta_\tau U,\\
	&\omega_n=S_{f,\tau}^{-1}\theta_\tau e_n+Ve_n-V\theta_fS_{f,\tau}^{-1}\theta_\tau e_n=S_{f,\tau}^{-1}\tau_n+Ve_n-V\theta_fS_{f,\tau}^{-1}\tau_n, \quad \forall n \in \mathbb{N}
	\end{align*}
	such that the operator 
	$$(S_{f,\tau}^{-1}\theta_\tau+V(I_{\ell^p(\mathbb{N})}-\theta_fS_{f,\tau}^{-1}\theta_\tau))(\theta_fS_{f,\tau}^{-1}+(I_{\ell^p(\mathbb{N})}-\theta_fS_{f,\tau}^{-1}\theta_\tau)U) $$
	is bounded invertible, where $U:\mathcal{X} \to \ell^p(\mathbb{N})$ and $ V:\ell^p(\mathbb{N})\to \mathcal{X}$ are bounded linear operators. By a direct expansion and simplification we get 
	\begin{align*}
	&(S_{f,\tau}^{-1}\theta_\tau+V(I_{\ell^p(\mathbb{N})}-\theta_fS_{f,\tau}^{-1}\theta_\tau))(\theta_fS_{f,\tau}^{-1}+(I_{\ell^p(\mathbb{N})}-\theta_fS_{f,\tau}^{-1}\theta_\tau)U)
	\\
	&\quad=S_{f,\tau}^{-1}+VU-V\theta_fS_{f,\tau}^{-1}\theta_\tau U.
	\end{align*}
\end{proof}
We know that a bounded linear operator from $\ell^2(\mathbb{N}) $ to $\mathcal{H} $ is given by a Bessel sequence (Theorem \ref{OLEBESSELCHARACTERIZATION12}). Thus, for Hilbert spaces, Theorem \ref{ALLDUAL} becomes Theorem \ref{LITHM}.  

\section{SIMILARITY FOR p-APPROXIMATE SCHAUDER FRAM-ES}
We define Definition \ref{SIMILARDEFHILBERT} to Banach spaces as follows.
\begin{definition}
	Two p-ASFs $ (\{f_n\}_{n}, \{\tau_n\}_{n}) $ and $ (\{g_n \}_{n}, \{\omega_n \}_{n}) $ for $\mathcal{X}$   are said to be \textbf{similar} or \textbf{equivalent} if there exist bounded invertible operators  $T_{f,g}, T_{\tau,\omega} :\mathcal{X} \to \mathcal{X}$ such that 
	\begin{align*}
	g_n=f_nT_{f,g},\quad  \omega_n= T_{\tau,\omega}\tau_n, \quad \forall  n \in \mathbb{N}.
	\end{align*}
\end{definition}
Since the operators giving similarity are bounded invertible, the notion of similarity is symmetric. Further, a routine calculation shows that it is an equivalence relation (hence the name equivalent) on the set 
\begin{align*}
\{(\{f_n\}_{n}, \{\tau_n\}_{n}): (\{f_n\}_{n}, \{\tau_n\}_{n}) \text{ is a p-ASF for } \mathcal{X}\}.
\end{align*}
We now characterize similarity using just operators. In the sequel, given a p-ASF $ (\{f_n\}_{n}, \{\tau_n\}_{n}) $,  we set $P_{f, \tau}\coloneqq \theta_fS_{f,\tau}^{-1}\theta_\tau$. 
\begin{theorem}\label{SEQUENTIALSIMILARITY}
	For two p-ASFs $ (\{f_n\}_{n}, \{\tau_n\}_{n}) $ and $ (\{g_n \}_{n}, \{\omega_n \}_{n}) $ for $\mathcal{X}$, the following are equivalent.
	\begin{enumerate}[label=(\roman*)]
		\item   $g_n=f_nT_{f, g} , \omega_n=T_{\tau,\omega}\tau_n,  \forall  n \in \mathbb{N}$, for some bounded invertible operators $T_{f,g}, T_{\tau,\omega}:\mathcal{X} \to \mathcal{X}.$ 
		\item $\theta_g=\theta_f T_{f,g}, \theta_\omega=T_{\tau,\omega}\theta_\tau$,  for some bounded invertible operators $T_{f,g}, T_{\tau,\omega}:\mathcal{X} \to \mathcal{X}.$
		\item $P_{g,\omega}=P_{f, \tau}.$
	\end{enumerate}
	If one of the above conditions is satisfied, then  invertible operators in  $\operatorname{(i)}$ and  $\operatorname{(ii)}$ are unique and are given by  $T_{f,g}= S_{f,\tau}^{-1}\theta_\tau\theta_g, T_{\tau, \omega}=\theta_\omega\theta_fS_{f,\tau}^{-1}.$ In the case that $ (\{f_n \}_{n}, \{\tau_n \}_{n}) $ is a simple  p-ASF, then $ (\{g_n \}_{n}, \{\omega_n \}_{n}) $ is  a simple p-ASF if and only if $T_{\tau, \omega}T_{f,g} =I_\mathcal{X}$   if and only if $ T_{f,g}T_{\tau, \omega} =I_\mathcal{X}$. 
\end{theorem}
\begin{proof}
	(i) $\Rightarrow $ (ii) $ \theta_gx=\{g_n(x)\}_{n}=\{f_n(T_{f,g}x)\}_{n}=\theta_f(T_{f,g}x), \forall x \in \mathcal{X}$, 	 $ \theta_\omega(\{a_n\}_{n})=\sum_{n=1}^\infty a_n\omega_n=\sum_{n=1}^\infty a_nT_{\tau,\omega}\tau_n=T_{\tau,\omega}( \theta_\tau(\{a_n\}_{n})) , \forall \{a_n\}_{n} \in \ell^p(\mathbb{N})$.\\
	(ii) $\Rightarrow $ (iii) $  S_{g,\omega}= \theta_\omega\theta_g=T_{\tau,\omega} \theta_\tau\theta_f T_{f,g} =T_{\tau,\omega} S_{f, \tau}T_{f,g}$ and 
	\begin{align*}
	P_{g,\omega}=\theta_g S_{g,\omega}^{-1} \theta_\omega=(\theta_f T_{f,g})(T_{\tau,\omega} S_{f, \tau}T_{f,g})^{-1}(T_{\tau,\omega} \theta_\tau)= P_{f, \tau}.
	\end{align*}  
	(ii) $\Rightarrow $ (i)  $ \sum_{n=1}^\infty g_n(x)e_n=\theta_g(x)=\theta_f(T_{f,g}x)=\sum_{n=1}^\infty f_n(T_{f,g}x)e_n, \forall x \in \mathcal{X}.$ This clearly gives (i).\\
	(iii) $\Rightarrow $ (ii) $\theta_g=P_{g,\omega} \theta_g= P_{f,\tau}\theta_g=\theta_f(S_{f,\tau}^{-1}\theta_{\tau}\theta_g)$, and $$\theta_\omega=\theta_\omega P_{g,\omega}=\theta_\omega P_{f,\tau}=(\theta_\omega\theta_fS_{f,\tau}^{-1})\theta_\tau .$$ We  show that $S_{f,\tau}^{-1}\theta_{\tau}\theta_g$ and $\theta_\omega\theta_fS_{f,\tau}^{-1} $ are invertible. For,
	\begin{align*}
	&(S_{f,\tau}^{-1}\theta_{\tau}\theta_g)(S_{g,\omega}^{-1}\theta_{\omega}\theta_f)=S_{f,\tau}^{-1}\theta_{\tau}P_{g,\omega}\theta_f=S_{f,\tau}^{-1}\theta_{\tau} P_{f,\tau}\theta_f=I_\mathcal{X},\\
	&(S_{g,\omega}^{-1}\theta_{\omega}\theta_f)(S_{f,\tau}^{-1}\theta_{\tau}\theta_g)=S_{g,\omega}^{-1}\theta_{\omega} P_{f,\tau}\theta_g=S_{g,\omega}^{-1}\theta_{\omega}P_{g,\omega}\theta_g=I_\mathcal{X} 
	\end{align*}
	and 
	\begin{align*}
	&(\theta_\omega\theta_fS_{f,\tau}^{-1})(\theta_\tau\theta_gS_{g,\omega}^{-1})=\theta_\omega P_{f,\tau}\theta_gS_{g,\omega}^{-1}=\theta_\omega P_{g,\omega}\theta_gS_{g,\omega}^{-1}=I_\mathcal{X},\\
	&(\theta_\tau\theta_gS_{g,\omega}^{-1})(\theta_\omega\theta_fS_{f,\tau}^{-1})=\theta_\tau P_{g,\omega}\theta_fS_{f,\tau}^{-1}=\theta_\tau P_{f,\tau}\theta_fS_{f,\tau}^{-1}=I_\mathcal{X}.
	\end{align*}    
	Let $T_{f,g}, T_{\tau,\omega}:\mathcal{X} \to \mathcal{X}$ be bounded invertible and $g_n=f_nT_{f, g}, \omega_n=T_{\tau,\omega}\tau_n,  \forall  n \in \mathbb{N}$. Then $\theta_g=\theta_fT_{f, g} $ says that $\theta_\tau\theta_g=\theta_\tau\theta_fT_{f, g}=S_{f,\tau}T_{f, g}  $ which implies $ T_{f, g} =S_{f,\tau}^{-1}\theta_\tau\theta_g$, and $\theta_\omega=T_{\tau,\omega}\theta_\tau $ says $\theta_\omega\theta_f=T_{\tau,\omega}\theta_\tau\theta_f=T_{\tau,\omega}S_{f,\tau} $. Hence $T_{\tau,\omega}=\theta_\omega\theta_fS_{f,\tau}^{-1} $. 
\end{proof}
It is easy to see that for Hilbert spaces,  Theorem \ref{SEQUENTIALSIMILARITY} reduces to Theorem \ref{BALANCHARSIM}. \\
 Definition \ref{SIMILARITYMINE} introduced the notion of dual frames. A twin notion associated is the notion of orthogonality.
\begin{definition}\label{ORTHOGONALDEF}
	Let $ (\{f_n\}_{n}, \{\tau_n\}_{n}) $ be a p-ASF for 	$\mathcal{X}$. 	A p-ASF $ (\{g_n \}_{n}, \{\omega_n \}_{n}) $ for $\mathcal{X}$ is \textbf{orthogonal}   for $ (\{f_n \}_{n}, \{\tau_n \}_{n}) $ if 
	\begin{align*}
	0=\sum_{n=1}^\infty g_n(x) \tau_n=\sum_{n=1}^\infty
	f_n(x) \omega_n, \quad \forall x \in
	\mathcal{X}.
	\end{align*}
\end{definition}
Unlike duality, the notion orthogonality is symmetric but not reflexive. Further, dual p-ASFs cannot be orthogonal to each other and orthogonal p-ASFs cannot be dual to each other. Moreover,  if $ (\{g_n\}_{n}, \{\omega_n\}_n)$ is orthogonal for $ (\{f_n\}_{n}, \{\tau_n\}_n)$, then  both $ (\{f_n\}_{n}, \{\omega_n\}_n)$ and $ (\{g_n\}_{n}, \{\tau_n\}_n)$ are not p-ASFs. Similar to Proposition \ref{ORTHOGONALPRO} we have the following proposition.
\begin{proposition}
	For two p-ASFs $ (\{f_n\}_{n}, \{\tau_n\}_{n}) $ and $ (\{g_n \}_{n}, \{\omega_n \}_{n}) $ for $\mathcal{X}$, the following are equivalent.
	\begin{enumerate}[label=(\roman*)]
		\item  $ (\{g_n \}_{n}, \{\omega_n \}_{n}) $ is  orthogonal  for $ (\{f_n \}_{n}, \{\tau_n \}_{n}) $.
		\item $\theta_\tau\theta_g =\theta_\omega\theta_f =0$.
	\end{enumerate}
\end{proposition}
Usefulness of orthogonal frames is that we have interpolation result, i.e., these frames can be stitched along certain curves (in particular, on the unit circle centered at the origin) to get new frames.
\begin{theorem}\label{INTERPOLATION}
	Let $ (\{f_n\}_{n}, \{\tau_n\}_{n}) $ and $ (\{g_n \}_{n}, \{\omega_n \}_{n}) $ be  two Parseval p-ASFs for  $\mathcal{X}$ which are  orthogonal. If $A,B,C,D :\mathcal{X}\to \mathcal{X}$ are bounded linear  operators and  $ CA+DB=I_\mathcal{X}$, then  
	\begin{align*}
	(\{f_nA+g_nB\}_{n}, \{C\tau_n+D\omega_n\}_{n})
	\end{align*}
	is a  simple  p-ASF for  $\mathcal{X}$. In particular,  if scalars $ a,b,c,d$ satisfy $ca+db =1$, then 
	$ (\{af_n+bg_n\}_{n}, \{c\tau_n+d\omega_n\}_{n}) $ is a  simple  p-ASF for  $\mathcal{X}$.
\end{theorem} 
\begin{proof}
	By a calculation we find  
	\begin{align*}
	\theta_{fA+gB} x = \{(f_nA+g_nB)(x) \}_{n}=\{f_n(Ax) \}_{n}+\{g_n(Bx) \}_{n}=\theta_f(Ax)+\theta_g(Bx), \quad \forall x \in \mathcal{X}
	\end{align*}
	and 
	\begin{align*}
	\theta_{C\tau+D\omega}(\{a_n \}_{n})=\sum_{n=1}^\infty a_n(C\tau_n+D\omega_n)=C\theta_\tau(\{a_n \}_{n})+D\theta_\omega(\{a_n \}_{n}), \quad  \forall \{a_n\}_n  \in \ell^p(\mathbb{N}). 
	\end{align*} 
	So 	
	\begin{align*}
	S_{fA+gB,C\tau+D\omega} &=\theta_{C\tau+D\omega} \theta_{fA+gB}= ( C\theta_\tau+ D\theta_\omega)(\theta_fA+\theta_gB)\\
	&=C\theta_\tau\theta_fA+C\theta_\tau\theta_gB+D\theta_\omega\theta_fA+D\theta_\omega\theta_gB\\
	&=CS_{f,\tau}A+0+0+DS_{g,\omega}B
	=CI_\mathcal{X}A+DI_\mathcal{X}B=I_\mathcal{X}.
	\end{align*}
\end{proof} 
Using Theorem \ref{SEQUENTIALSIMILARITY} we finally relate three notions duality, similarity and orthogonality.
\begin{proposition}\label{LASTONE}
	For every p-ASF $(\{f_n\}_{n}, \{\tau_n\}_{n})$, the canonical dual for $(\{f_n\}_{n}, \{\tau_n\}_{n})$ is the only dual p-ASF that is similar to $(\{f_n\}_{n}, \{\tau_n\}_{n})$.
\end{proposition}
\begin{proof}
	Let us suppose that  two p-ASFs $(\{f_n\}_{n}, \{\tau_n\}_{n})$ and $ (\{g_n \}_{n}, \{\omega_n \}_{n}) $  are similar and dual to each other. Then there exist bounded invertible operators  $T_{f,g}, T_{\tau,\omega} :\mathcal{X}\to \mathcal{X}$  such that $ g_n=f_nT_{f,g},\omega_n=T_{\tau,\omega}\tau_n ,\forall n \in \mathbb{N}$. Theorem \ref{SEQUENTIALSIMILARITY} then gives
	\begin{align*}
	T_{f,g}=S_{f,\tau}^{-1}\theta_\tau\theta_g=S_{f,\tau}^{-1}I_\mathcal{X}=S_{f,\tau}^{-1}\text{ and }T_{\tau, \omega}=\theta_\omega\theta_fS_{f,\tau}^{-1}=I_\mathcal{X}S_{f,\tau}^{-1}=S_{f,\tau}^{-1}.
	\end{align*} 
	Hence $ (\{g_n \}_{n}, \{\omega_n \}_{n}) $ is the canonical dual for  $(\{f_n\}_{n}, \{\tau_n\}_{n})$.	
\end{proof}
\begin{proposition}\label{LASTTWO}
	Two similar  p-ASFs cannot be orthogonal.
\end{proposition}
\begin{proof}
	Let $(\{f_n\}_{n}, \{\tau_n\}_{n})$ and $ (\{g_n \}_{n}, \{\omega_n \}_{n}) $  be two p-ASFs which are similar.  Then there exist bounded  invertible  operators $T_{f,g}, T_{\tau,\omega} :\mathcal{X}\to \mathcal{X}$  such that $ g_n=f_nT_{f,g},\omega_n=T_{\tau,\omega}\tau_n ,\forall n \in \mathbb{N}$.  Theorem \ref{SEQUENTIALSIMILARITY} then says   $\theta_g=\theta_f T_{f,g}, \theta_\omega=T_{\tau,\omega}\theta_\tau  $.    Therefore 
	\begin{align*}
	\theta_\tau \theta_g=\theta_\tau\theta_f T_{f,g}=S_{f,\tau}T_{f,g}\neq0. 
	\end{align*}
\end{proof}
\begin{remark}
	For every p-ASF  $(\{f_n\}_{n}, \{\tau_n\}_{n}),$ both  p-ASFs
	\begin{align*}
	(\{f_n{S}_{f, \tau}^{-1}\}_{n}, \{\tau_n\}_{n})	 \text{ and }  (\{f_n\}_{n}, \{{S}_{f, \tau}^{-1}\tau_n\}_{n})
	\end{align*}   
	are  simple p-ASFs and are  similar to  $(\{f_n\}_{n}, \{\tau_n\}_{n})$.  Therefore  each p-ASF is similar to  simple  p-ASFs.
\end{remark}

\section{DILATION THEOREM FOR p-APPROXIMATE SCHAUDER FRAMES}
Here we derive a generalization of  Theorem \ref{DILATIONTHEOREMHILBERTSPACE} (Naimark-Han-Larson dilation theorem) for frames in Hilbert spaces  to p-ASFs for Banach spaces.  In order to derive the  dilation result we must have a notion of Riesz basis for Banach space. Theorem  \ref{RIESZBASISTHM}  gives various characterizations for Riesz bases for Hilbert spaces  but all uses (implicitly or explicitly) inner product structures and orthonormal bases. These characterizations lead to the   notion of p-Riesz basis for Banach spaces using a single sequence in the Banach space (Definition \ref{RIESZBASISDEFINITIONBANACHSPACE}) but we  consider a different notion in this chapter. \\
To define the notion of Riesz basis, which is compatible with Hilbert space situation, we first derive an operator-theoretic  characterization for Riesz basis in Hilbert spaces, which does not use the inner product of Hilbert space. To do so, we need a result from Hilbert space frame theory.

\begin{theorem}\label{RIESZBASISCHAROURS}
	For  sequence $\{\tau_n\}_n$ in  $\mathcal{H}$, the following are equivalent.
	\begin{enumerate}[label=(\roman*)]
		\item $\{\tau_n\}_n$ is a Riesz basis for  $ \mathcal{H}$. 
		\item $\{\tau_n\}_n$ is a frame  for  $ \mathcal{H}$ and 
		\begin{align}\label{RIESZEQUATIONTHEOREM}
		\theta_\tau S_\tau^{-1} \theta_\tau^*=I_{\ell^2(\mathbb{N})}.
		\end{align}
	\end{enumerate}	
\end{theorem}
\begin{proof}
	(i) $\implies$ (ii) From Theorem \ref{RIESZISAFRAME}  that a Riesz basis is a frame.  Now there exist an orthonormal basis $\{\omega_n\}_n$ for  $\mathcal{H}$ and   a bounded invertible operator $T: \mathcal{H}\to \mathcal{H}$ such that $T\omega_n=\tau_n$, for all $n \in \mathbb{N}$. We then have 
		\begin{align*}
		S_\tau h&= \sum_{n=1}^\infty \langle h, \tau_n\rangle\tau_n= \sum_{n=1}^\infty \langle h, T\omega_n\rangle T \omega_n\\
		&=T\left(\sum_{n=1}^\infty \langle T^*h, \omega_n\rangle  \omega_n\right)=TT^*h, \quad \forall h \in \mathcal{H}.
		\end{align*}
		Therefore 
		\begin{align*}
		\theta_\tau S_\tau^{-1} \theta_\tau^*\{a_n\}_n&=\theta_\tau (TT^*)^{-1} \theta_\tau^*\{a_n\}_n=\theta_\tau (T^*)^{-1}T^{-1} \theta_\tau^*\{a_n\}_n\\
		&=\theta_\tau (T^*)^{-1}T^{-1}\left(\sum_{n=1}^\infty a_n\tau_n\right)=\theta_\tau (T^*)^{-1}T^{-1}\left(\sum_{n=1}^\infty a_nT\omega_n\right)\\
		&=\theta_\tau \left(\sum_{n=1}^\infty a_n(T^*)^{-1}\omega_n\right)=\sum_{k=1}^{\infty}\left\langle \sum_{n=1}^\infty a_n(T^*)^{-1}\omega_n, \tau_k\right\rangle e_k \\
		&=\sum_{k=1}^{\infty}\left\langle \sum_{n=1}^\infty a_n(T^*)^{-1}\omega_n, T\omega_k\right\rangle e_k\\
		&=\sum_{k=1}^{\infty}\left\langle \sum_{n=1}^\infty a_n\omega_n, \omega_k\right\rangle e_k=\{a_k\}_k, \quad\forall\{a_n\}_n \in  \ell^2(\mathbb{N}).
		\end{align*}
	(ii) $\implies$ (i) From Holub's theorem  (Theorem \ref{HOLUBTHEOREM}), there exists a surjective bounded linear operator $T:\ell^2(\mathbb{N}) \to \mathcal{H}$ such that $Te_n=\tau_n$, for all $n \in \mathbb{N}$. Since all separable Hilbert spaces are isometrically isomorphic to one another and orthonormal bases map into orthonormal bases, without loss of generality we may assume that $\{e_n\}_n$ is an orthonormal basis for $\mathcal{H}$ and the domain of $T$ is $\mathcal{H}$. Our job now reduces in showing $T$ is invertible. Since $T$ is already surjective, to show it is invertible, it suffices to show it is injective. Let $\{a_n\}_n \in  \ell^2(\mathbb{N}).$ Then $\{a_n\}_n=\theta_\tau (S_\tau^{-1} \theta_\tau^*\{a_n\}_n)$. Hence $\theta_\tau$ is surjective. We now find 
		\begin{align*}
		\theta_\tau h=\sum_{n=1}^{\infty}\langle h, \tau_n\rangle e_n=\sum_{n=1}^{\infty}\langle h, Te_n\rangle e_n=T^*h, \quad \forall h \in \mathcal{H}.
		\end{align*}
		Therefore 
		\begin{align*}
		\operatorname{Kernel} (T)=T^*(\mathcal{H})^\perp=\theta_\tau(\mathcal{H})^\perp=\mathcal{H}^\perp=\{0\}.
		\end{align*}
		Hence $T$ is injective.	
\end{proof}
Theorem \ref{RIESZBASISCHAROURS} leads to the following definition of p-approximate Riesz basis.
\begin{definition}
	A pair $ (\{f_n \}_{n}, \{\tau_n \}_{n}) $ is said to be a \textbf{p-approximate Riesz basis}   for $\mathcal{X}$ if 	it is a p-ASF for $ \mathcal{X}$ and $\theta_fS_{f,\tau}^{-1}\theta_\tau=I_{\ell^p(\mathbb{N})}$.
\end{definition}
\begin{example}\label{EXAMPLE3}
	Let $p\in[1,\infty)$ and $U:\mathcal{X} \rightarrow\ell^p(\mathbb{N})$, $ V: \ell^p(\mathbb{N})\to \mathcal{X}$ be bounded invertible linear operators.  Let $\{e_n\}_n$, $\{\zeta_n\}_n$, $ \{f_n\}_{n}$, and $ \{\tau_n\}_{n}$ be as in Example \ref{EXAMPLE2}.	Then $ (\{f_n\}_{n}, \{\tau_n\}_{n}) $ is a p-approximate Riesz basis for 	$\mathcal{X}$. 
\end{example}
We now derive the dilation theorem. 
\begin{theorem}\label{DILATIONTHEOREMPASF}
	(\textbf{Dilation theorem for p-approximate Schauder frames})	Let   $ (\{f_n \}_{n}$, $\{\tau_n \}_{n}) $ be a p-ASF 	for  $\mathcal{X}$. Then there exist a Banach space $\mathcal{X}_1$	which contains $\mathcal{X}$ isometrically and a p-approximate Riesz basis  $ (\{g_n \}_{n}, \{\omega_n \}_{n}) $   	for  $\mathcal{X}_1$ such that 
	\begin{align*}
	f_n=g_nP_{|\mathcal{X}}, \quad\tau_n=P\omega_n, \quad \forall n \in \mathbb{N},
	\end{align*}
	where $P:\mathcal{X}_1\rightarrow \mathcal{X}$ is onto  projection. 
\end{theorem}
\begin{proof}
	Let  $\{e_n\}_n$ denote the standard Schauder basis for  $\ell^p(\mathbb{N})$ and let $\{\zeta_n\}_n$ denote the coordinate functionals associated with $\{e_n\}_n$. 	Define
	\begin{align*}
	\mathcal{X}_1\coloneqq\mathcal{X}\oplus(I_{\ell^p(\mathbb{N})}-P_{f, \tau})(\ell^p(\mathbb{N})), \quad P:\mathcal{X}_1 \ni x\oplus y\mapsto x\oplus 0 \in \mathcal{X}_1
	\end{align*}
	and 
	\begin{align*}
	\omega_n\coloneqq	\tau_n\oplus (I_{\ell^p(\mathbb{N})}-P_{f, \tau})e_n \in \mathcal{X}_1, \quad \quad g_n\coloneqq f_n \oplus \zeta_n (I_{\ell^p(\mathbb{N})}-P_{f, \tau})\in \mathcal{X}_1^*, \quad \forall  n \in \mathbb{N}.
	\end{align*}
	Then clearly $\mathcal{X}_1$	 contains $\mathcal{X}$ isometrically, $P:\mathcal{X}_1\rightarrow \mathcal{X}$ is onto projection and 
	\begin{align*}
	&(g_nP_{|\mathcal{X}})(x)=g_n(P_{|\mathcal{X}}x)=g_n(x)=(f_n \oplus \zeta_n (I_{\ell^p(\mathbb{N})}-P_{f, \tau}))(x\oplus 0)=f_n(x),\quad \forall x \in \mathcal{X},\\
	& P\omega_n=P(\tau_n\oplus (I_{\ell^p(\mathbb{N})}-P_{f, \tau})e_n)=\tau_n, \quad \forall n \in \mathbb{N}.
	\end{align*}
	Since the operator $I_{\ell^p(\mathbb{N})}-P_{f, \tau}$ is idempotent, it follows that $(I_{\ell^p(\mathbb{N})}-P_{f, \tau})(\ell^p(\mathbb{N}))$ is a closed subspace of $\ell^p(\mathbb{N})$ and hence a Banach space. Therefore $\mathcal{X}_1$ is a Banach space.  Let $x\oplus y \in \mathcal{X}_1$ and we shall write  $y=\{a_n\}_n \in \ell^p(\mathbb{N})$. We then see that 
	\begin{align*}
	&\sum_{n=1}^{\infty}(\zeta_n (I_{\ell^p(\mathbb{N})}-P_{f, \tau}))(y)\tau_n=\sum_{n=1}^{\infty}\zeta_n(y)\tau_n-\sum_{n=1}^{\infty}\zeta_n(P_{f, \tau}(y))\tau_n\\
	&\quad=\sum_{n=1}^{\infty}\zeta_n(\{a_k\}_{k})\tau_n-\sum_{n=1}^{\infty}\zeta_n(\theta_fS_{f, \tau}^{-1}\theta_\tau(\{a_k\}_{k}))\tau_n\\
	&\quad=\sum_{n=1}^{\infty}a_n\tau_n-\sum_{n=1}^{\infty}\zeta_n\left(\theta_fS_{f, \tau}^{-1}\left(\sum_{k=1}^{\infty}a_k\tau_k\right)\right)\tau_n\\
	&\quad
	=\sum_{n=1}^{\infty}a_n\tau_n-\sum_{n=1}^{\infty}\zeta_n\left(\sum_{k=1}^{\infty}a_k\theta_fS_{f, \tau}^{-1}\tau_k\right)\tau_n\\
	&\quad=\sum_{n=1}^{\infty}a_n\tau_n-\sum_{n=1}^{\infty}\zeta_n\left(\sum_{k=1}^{\infty}a_k\sum_{r=1}^{\infty}f_r(S_{f, \tau}^{-1}\tau_k)e_r\right)\tau_n\\
	&\quad
	=\sum_{n=1}^{\infty}a_n\tau_n-\sum_{n=1}^{\infty}\sum_{k=1}^{\infty}a_k\sum_{r=1}^{\infty}f_r(S_{f, \tau}^{-1}\tau_k)\zeta_n(e_r)\tau_n\\
	&\quad=\sum_{n=1}^{\infty}a_n\tau_n-\sum_{n=1}^{\infty}\sum_{k=1}^{\infty}a_kf_n(S_{f, \tau}^{-1}\tau_k)\tau_n\\
	&\quad=\sum_{n=1}^{\infty}a_n\tau_n-\sum_{k=1}^{\infty}a_k\sum_{n=1}^{\infty}f_n(S_{f, \tau}^{-1}\tau_k)\tau_n\\
	&\quad=\sum_{n=1}^{\infty}a_n\tau_n-\sum_{k=1}^{\infty}a_k\tau_k=0  
	\end{align*}
	and 
	\begin{align*}
	& \sum_{n=1}^{\infty}f_n(x)(I_{\ell^p(\mathbb{N})}-P_{f, \tau})e_n=\sum_{n=1}^{\infty}f_n(x)e_n-\sum_{n=1}^{\infty}f_n(x)P_{f, \tau}e_n\\
	&=\sum_{n=1}^{\infty}f_n(x)e_n-\sum_{n=1}^{\infty}f_n(x)\theta_fS_{f, \tau}^{-1}\theta_\tau e_n\\
	&=\sum_{n=1}^{\infty}f_n(x)e_n-\sum_{n=1}^{\infty}f_n(x)\theta_fS_{f, \tau}^{-1}\tau_n\\
	&=\sum_{n=1}^{\infty}f_n(x)e_n-\sum_{n=1}^{\infty}f_n(x)\sum_{k=1}^{\infty}f_k(S_{f, \tau}^{-1}\tau_n)e_k
	\\
	&=\sum_{n=1}^{\infty}f_n(x)e_n-\sum_{n=1}^{\infty}\sum_{k=1}^{\infty}f_n(x)f_k(S_{f, \tau}^{-1}\tau_n)e_k\\
	&=\sum_{n=1}^{\infty}f_n(x)e_n-\sum_{k=1}^{\infty}\sum_{n=1}^{\infty}f_n(x)f_k(S_{f, \tau}^{-1}\tau_n)e_k\\
	&=\sum_{n=1}^{\infty}f_n(x)e_n-\sum_{k=1}^{\infty}f_k\left(\sum_{n=1}^{n}f_n(x)S_{f, \tau}^{-1}\tau_n\right)e_k\\
	&=\sum_{n=1}^{\infty}f_n(x)e_n-\sum_{k=1}^{\infty}f_k(x)e_k=0.
	\end{align*}
	By using previous two calculations, we  get 
	\begin{align*}
&	S_{g, \omega}(x\oplus y) =\sum_{n=1}^{\infty}g_n(x\oplus y)\omega_n\\
	&=\sum_{n=1}^{\infty}(f_n \oplus \zeta_n (I_{\ell^p(\mathbb{N})}-P_{f, \tau}))(x\oplus y)(\tau_n\oplus (I_{\ell^p(\mathbb{N})}-P_{f, \tau})e_n)\\
	&=\sum_{n=1}^{\infty}(f_n(x) + (\zeta_n (I_{\ell^p(\mathbb{N})}-P_{f, \tau}))(y))(\tau_n\oplus (I_{\ell^p(\mathbb{N})}-P_{f, \tau})e_n)\\
	&=\left(\sum_{n=1}^{\infty}f_n(x)\tau_n+\sum_{n=1}^{\infty}(\zeta_n (I_{\ell^p(\mathbb{N})}-P_{f, \tau}))(y)\tau_n\right)\oplus\\
	&\quad \left(\sum_{n=1}^{\infty}f_n(x)(I_{\ell^p(\mathbb{N})}-P_{f, \tau})e_n+\sum_{n=1}^{\infty}(\zeta_n (I_{\ell^p(\mathbb{N})}-P_{f, \tau}))(y)(I_{\ell^p(\mathbb{N})}-P_{f, \tau})e_n\right)\\
	&=(S_{f, \tau}x+0)\oplus \left(0+(I_{\ell^p(\mathbb{N})}-P_{f, \tau})\sum_{n=1}^{\infty}\zeta_n ((I_{\ell^p(\mathbb{N})}-P_{f, \tau})y)e_n\right)\\
	&=S_{f, \tau}x\oplus (I_{\ell^p(\mathbb{N})}-P_{f, \tau})(I_{\ell^p(\mathbb{N})}-P_{f, \tau})y=S_{f, \tau}x\oplus (I_{\ell^p(\mathbb{N})}-P_{f, \tau})y\\
	&=(S_{f, \tau}\oplus (I_{\ell^p(\mathbb{N})}-P_{f, \tau}))(x\oplus y).
	\end{align*}
	Since the operator $I_{\ell^p(\mathbb{N})}-P_{f, \tau}$ is idempotent, $I_{\ell^p(\mathbb{N})}-P_{f, \tau}$ becomes the identity operator on the space $(I_{\ell^p(\mathbb{N})}-P_{f, \tau})(\ell^p(\mathbb{N}))$.  Hence we get that the operator  $S_{g, \omega}=S_{f, \tau}\oplus (I_{\ell^p(\mathbb{N})}-P_{f, \tau})$ is bounded invertible from  $\mathcal{X}_1$ onto itself. We next show that $ (\{g_n \}_{n}, \{\omega_n \}_{n}) $ is a p-approximate Riesz basis for $\mathcal{X}_1$. For this, first we find $\theta_g$ and $\theta_\omega$. Consider
	\begin{align*}
	\theta_g(x\oplus y)&=\{g_n(x\oplus y)\}_{n}=\{(f_n \oplus \zeta_n (I_{\ell^p(\mathbb{N})}-P_{f, \tau}))(x\oplus y)\}_{n}\\
	&=\{f_n (x) +\zeta_n ((I_{\ell^p(\mathbb{N})}-P_{f, \tau}) y)\}_{n}=\{f_n (x)\}_{n} +\{\zeta_n ((I_{\ell^p(\mathbb{N})}-P_{f, \tau}) y)\}_{n}\\
	&=\theta_fx+\sum_{n=1}^{\infty}\zeta_n ((I_{\ell^p(\mathbb{N})}-P_{f, \tau}) y )e_n=\theta_fx+(I_{\ell^p(\mathbb{N})}-P_{f, \tau}) y , \quad\forall x\oplus y \in \mathcal{X}_1
	\end{align*}
	and 
	\begin{align*}
	\theta_\omega\{a_n\}_n&=\sum_{n=1}^{\infty}a_n\omega_n=\sum_{n=1}^{\infty}a_n(\tau_n\oplus (I_{\ell^p(\mathbb{N})}-P_{f, \tau})e_n)\\
	&= \left(\sum_{n=1}^{\infty}a_n\tau_n\right) \oplus \left(\sum_{n=1}^{\infty}a_n(I_{\ell^p(\mathbb{N})}-P_{f, \tau})e_n\right)\\
	&=\theta_\tau\{a_n\}_n\oplus (I_{\ell^p(\mathbb{N})}-P_{f, \tau})\left(\sum_{n=1}^{\infty}a_ne_n\right)\\
	&=\theta_\tau\{a_n\}_n\oplus (I_{\ell^p(\mathbb{N})}-P_{f, \tau})\{a_n\}_n, \quad\forall \{a_n\}_n \in \ell^p(\mathbb{N}).
	\end{align*}
	Therefore
	\begin{align*}
	P_{g, \omega}\{a_n\}_n&=\theta_g S_{g, \omega}^{-1}\theta_\omega\{a_n\}_n=\theta_gS_{g, \omega}^{-1}(\theta_\tau\{a_n\}_n\oplus (I_{\ell^p(\mathbb{N})}-P_{f, \tau})\{a_n\}_n)\\
	&=\theta_g(S_{f, \tau}^{-1}\oplus (I_{\ell^p(\mathbb{N})}-P_{f, \tau}) )(\theta_\tau\{a_n\}_n\oplus (I_{\ell^p(\mathbb{N})}-P_{f, \tau})\{a_n\}_n)\\
	&=\theta_g(S_{f, \tau}^{-1} \theta_\tau\{a_n\}_n\oplus  (I_{\ell^p(\mathbb{N})}-P_{f, \tau})^2\{a_n\}_n)\\
	&=\theta_g(S_{f, \tau}^{-1} \theta_\tau\{a_n\}_n\oplus  (I_{\ell^p(\mathbb{N})}-P_{f, \tau})\{a_n\}_n)\\
	&=\theta_f(S_{f, \tau}^{-1} \theta_\tau\{a_n\}_n)+(I_{\ell^p(\mathbb{N})}-P_{f, \tau})(I_{\ell^p(\mathbb{N})}-P_{f, \tau})\{a_n\}_n\\
	&=P_{f, \tau}\{a_n\}_n+(I_{\ell^p(\mathbb{N})}-P_{f, \tau})\{a_n\}_n=\{a_n\}_n, \quad \forall \{a_n\}_n \in \ell^p(\mathbb{N}).
	\end{align*}	
\end{proof}
\begin{corollary}(\cite{HANLARSON, KASHINKULIKOVA})
Let $\{\tau_n\}_n$ be  a  frame  for  $ \mathcal{H}$.  Then there exist a Hilbert space $ \mathcal{H}_1 $ which contains $ \mathcal{H}$ isometrically and  a Riesz basis $\{\omega_n\}_n$ for  $ \mathcal{H}_1$ such that 
\begin{align*}
\tau_n=P\omega_n, \quad\forall n \in \mathbb{N},
\end{align*}
where $P$ is the orthogonal projection from $\mathcal{H}_1$ onto $\mathcal{H}$. 		
\end{corollary}
\begin{proof}
	Let $\{\tau_n\}_n$ be a frame for  $\mathcal{H}$. Define
	\begin{align*}
	f_n:\mathcal{H} \ni h \mapsto f_n(h)\coloneqq \langle h, \tau_n\rangle \in \mathbb{K}, \quad \forall n \in \mathbb{N}.
	\end{align*}	
	Then $\theta_f=\theta_\tau$. Note that now $ (\{f_n \}_{n}, \{\tau_n \}_{n}) $ is  a 2-approximate frame   for $\mathcal{H}$. Theorem  \ref{DILATIONTHEOREMPASF} now says that  there exist a Banach space $\mathcal{X}_1$	which contains $\mathcal{H}$ isometrically and a 2-approximate Riesz basis  $ (\{g_n \}_{n}, \{\omega_n \}_{n}) $   	for  $\mathcal{X}_1=\mathcal{H}\oplus(I_{\ell^2(\mathbb{N})}-P_{\tau})(\ell^2(\mathbb{N}))$ such that 
	\begin{align*}
	f_n=g_nP_{|\mathcal{H}}, \quad\tau_n=P\omega_n, \quad \forall n \in \mathbb{N},
	\end{align*}
	where $P:\mathcal{X}_1\rightarrow \mathcal{H}$ is onto  projection. Since   $(I_{\ell^2(\mathbb{N})}-P_{\tau})(\ell^2(\mathbb{N}))$ is a closed  subspace of the Hilbert space $\ell^2(\mathbb{N})$, $\mathcal{X}_1$ now becomes a Hilbert space. From the definition of $P$ we get that it is an  orthogonal projection.  To prove $\{\omega_n\}_n$ is a Riesz basis for  $\mathcal{X}_1$, we use  Theorem \ref{RIESZBASISCHAROURS}. Since $\{\tau_n\}_n$  is  a
	frame for $\mathcal{H}$,  there exist $a,b>0$ such that 
	\begin{align*}
	a\|h\|^2 \leq \sum_{n=1}^\infty |\langle h, \tau_n\rangle|^2\leq b\|h\|^2, \quad \forall h \in \mathcal{H}.
	\end{align*}
	Let $h\oplus  (I_{\ell^2(\mathbb{N})}-P_{f, \tau})\{a_k\}_k\in \mathcal{X}_1$. Then by noting $b\geq1$, we get 
	\begin{align*}
	&\sum_{n=1}^{\infty}|\langle h\oplus (I_{\ell^2(\mathbb{N})}-P_{ \tau})\{a_k\}_k, \omega_n\rangle|^2\\
	&\quad=\sum_{n=1}^{\infty}|\langle h\oplus (I_{\ell^2(\mathbb{N})}-P_{ \tau})\{a_k\}_k, \tau_n\oplus (I_{\ell^2(\mathbb{N})}-P_{ \tau})e_n\rangle|^2\\
	&\quad=\sum_{n=1}^{\infty}|\langle h, \tau_n\rangle|^2+\sum_{n=1}^{\infty}|\langle  (I_{\ell^2(\mathbb{N})}-P_{ \tau})\{a_k\}_k,  (I_{\ell^2(\mathbb{N})}-P_{ \tau})e_n\rangle|^2\\
	&\quad=\sum_{n=1}^{\infty}|\langle h, \tau_n\rangle|^2+\sum_{n=1}^{\infty}|\langle    (I_{\ell^2(\mathbb{N})}-P_{ \tau})(I_{\ell^2(\mathbb{N})}-P_{ \tau})\{a_k\}_k, e_n\rangle|^2\\
	&\quad=\sum_{n=1}^{\infty}|\langle h, \tau_n\rangle|^2+\sum_{n=1}^{\infty}|\langle    (I_{\ell^2(\mathbb{N})}-P_{ \tau})\{a_k\}_k, e_n\rangle|^2\\
	&\quad=\sum_{n=1}^{\infty}|\langle h, \tau_n\rangle|^2+\|(I_{\ell^2(\mathbb{N})}-P_{ \tau})\{a_k\}_k\|^2\\
	&\quad\leq b\|h\|^2+\|(I_{\ell^2(\mathbb{N})}-P_{\tau})\{a_k\}_k\|^2\\
	&\quad\leq b(\|h\|^2+\|(I_{\ell^2(\mathbb{N})}-P_{ \tau})\{a_k\}_k\|^2)\\
	&\quad=b\| h\oplus (I_{\ell^2(\mathbb{N})}-P_{\tau})\{a_k\}_k\|^2.
	\end{align*}
	Previous calculation tells that $\{\omega_n\}_n$  is  a Bessel sequence
	for $\mathcal{X}_1$. Hence $S_\omega:\mathcal{X}_1 \ni x\oplus\{a_k\}_k\mapsto \sum_{n=1}^{\infty} \langle x\oplus\{a_k\}_k, \omega_n\rangle \omega_n\in \mathcal{X}_1$ is a well-defined bounded linear operator. Next we claim that 
	\begin{align}\label{CLAIM}
	g_n(x\oplus\{a_k\}_k)=\langle x\oplus\{a_k\}_k, \omega_n\rangle, \quad \forall x\oplus\{a_k\}_k \in \mathcal{X}_1, \forall n \in \mathbb{N}.
	\end{align}
	Consider 
	\begin{align*}
	&g_n(x\oplus\{a_k\}_k)=(f_n \oplus \zeta_n (I_{\ell^2(\mathbb{N})}-P_{ \tau}))(x\oplus\{a_k\}_k)\\
	&=f_n(x)+\zeta_n ((I_{\ell^2(\mathbb{N})}-P_{ \tau})\{a_k\}_k)
	=f_n(x)+\zeta_n\left(\{a_k\}_k\right)-\zeta_n (P_{ \tau}\{a_k\}_k)\\
	&=f_n(x)+\zeta_n\left(\{a_k\}_k\right)-\zeta_n (\theta_\tau S_{ \tau}^{-1}\theta_\tau^*\{a_k\}_k)
	=f_n(x)+a_n-\zeta_n\left(\theta_\tau S_{ \tau}^{-1}\left(\sum_{k=1}^{\infty}a_k\tau_k\right)\right)\\
	&=f_n(x)+a_n-\zeta_n\left(\sum_{k=1}^{\infty}a_k\theta_\tau S_{ \tau}^{-1}\tau_k\right)
	=f_n(x)+a_n-\zeta_n\left(\sum_{k=1}^{\infty}a_k\sum_{r=1}^{\infty}\langle S_{ \tau}^{-1}\tau_k, \tau_r \rangle e_r\right)\\
	&=f_n(x)+a_n-\sum_{k=1}^{\infty}a_k\langle S_{ \tau}^{-1}\tau_k, \tau_n \rangle= \langle x, \tau_n\rangle+a_n-\sum_{k=1}^{\infty}a_k\langle S_{ \tau}^{-1}\tau_k, \tau_n \rangle \quad \text{ and }
	\end{align*}
	\begin{align*}
	&\langle x\oplus\{a_k\}_k, \omega_n\rangle=\langle x\oplus\{a_k\}_k, \tau_n\oplus (I_{\ell^2(\mathbb{N})}-P_{ \tau})e_n\rangle\\
	&\quad=\langle x, \tau_n\rangle+\langle \{a_k\}_k,  (I_{\ell^2(\mathbb{N})}-P_{ \tau})e_n\rangle
	=\langle x, \tau_n\rangle+\langle \{a_k\}_k,  e_n\rangle+\langle \{a_k\}_k,  P_{ \tau}e_n\rangle\\
	&\quad=\langle x, \tau_n\rangle+a_n-\left \langle\{a_k\}_k, \theta_\tau S_{ \tau}^{-1}\theta_\tau^*e_n\right\rangle=\langle x, \tau_n\rangle+a_n-\left \langle\{a_k\}_k, \theta_\tau S_{ \tau}^{-1}\tau_n\right\rangle\\
	&\quad=\langle x, \tau_n\rangle+a_n- \langle\{a_k\}_k,  \{ \langle S_{ \tau}^{-1}\tau_n, \tau_k \rangle\}_k \rangle=\langle x, \tau_n\rangle+a_n-\sum_{k=1}^{\infty}a_k\overline{\langle S_{ \tau}^{-1}\tau_n, \tau_k \rangle}\\
	&\quad=\langle x, \tau_n\rangle+a_n-\sum_{k=1}^{\infty}a_k\langle\tau_k,S_{ \tau}^{-1}\tau_n \rangle=\langle x, \tau_n\rangle+a_n-\sum_{k=1}^{\infty}a_k\langle S_{ \tau}^{-1}\tau_k, \tau_n \rangle.
	\end{align*}
	Thus Equation (\ref{CLAIM}) holds. Therefore for all $x\oplus\{a_k\}_k\in \mathcal{X}_1$,
	\begin{align*}
	S_{g,\omega}(x\oplus\{a_k\}_k)=\sum_{n=1}^{\infty}g_n(x\oplus\{a_k\}_k)\omega_n=\sum_{n=1}^{\infty}\langle x\oplus\{a_k\}_k, \omega_n\rangle\omega_n=S_{\omega}(x\oplus\{a_k\}_k).
	\end{align*}
	Since $ S_{g,\omega}$ is invertible, $S_{\omega}$ becomes invertible. Clearly $S_{\omega}$ is positive. Therefore 
	\begin{align*}
	\frac{1}{\|S_{\omega}\|^{-1}}\|g\|^2\leq \langle S_{\omega}g, g\rangle\leq \|S_\omega\| \|g\|^2, \quad \forall g \in \mathcal{X}_1. 
	\end{align*}
	Hence
	\begin{align*}
	\frac{1}{\|S_{\omega}\|^{-1}}\|g\|^2\leq \sum_{n=1}^{\infty}|\langle g, \omega_n\rangle|^2\leq \|S_\omega\| \|g\|^2, \quad \forall g \in \mathcal{X}_1. 
	\end{align*}
	Hence  $\{\omega_n\}_n$  is  a frame 
	for $\mathcal{X}_1$.
	
	Finally we  show Equation (\ref{RIESZEQUATIONTHEOREM}) in  Theorem \ref{RIESZBASISCHAROURS} for the frame $\{\omega_n\}_n$. Consider 
	\begin{align*}
	&\theta_\omega S_\omega^{-1} \theta_\omega^*\{a_n\}_n=\theta_\omega S_\omega^{-1}\left(\sum_{n=1}^{\infty}a_n\omega_n\right)=\theta_\omega \left(\sum_{n=1}^{\infty}a_nS_\omega^{-1}\omega_n\right)\\
	&=\sum_{k=1}^{\infty} \left\langle \sum_{n=1}^{\infty}a_nS_\omega^{-1}\omega_n, \omega_k\right \rangle=\sum_{k=1}^{\infty}  \sum_{n=1}^{\infty}a_n\langle S_\omega^{-1}\omega_n, \omega_k \rangle\\
	&=\sum_{k=1}^{\infty}  \sum_{n=1}^{\infty}a_n\langle (S_\tau^{-1}\oplus(I_{\ell^2(\mathbb{N})}-P_{ \tau}) )(\tau_n\oplus (I_{\ell^2(\mathbb{N})}-P_{ \tau})e_n), \tau_k\oplus (I_{\ell^2(\mathbb{N})}-P_{ \tau})e_k \rangle\\
	&=\sum_{k=1}^{\infty}  \sum_{n=1}^{\infty}a_n\langle S_\tau^{-1}\tau_n\oplus(I_{\ell^2(\mathbb{N})}-P_{ \tau}) ^2e_n, \tau_k\oplus (I_{\ell^2(\mathbb{N})}-P_{ \tau})e_k \rangle\\
	&=\sum_{k=1}^{\infty}  \sum_{n=1}^{\infty}a_n(\langle S_\tau^{-1}\tau_n, \tau_k \rangle+\langle (I_{\ell^2(\mathbb{N})}-P_{ \tau}) e_n,  (I_{\ell^2(\mathbb{N})}-P_{ \tau})e_k \rangle)\\
	&=\sum_{k=1}^{\infty} \left\langle \sum_{n=1}^{\infty}a_nS_\tau^{-1}\tau_n, \tau_k\right \rangle+\sum_{k=1}^{\infty}  \sum_{n=1}^{\infty}a_n\langle (I_{\ell^2(\mathbb{N})}-P_{f, \tau}) e_n,  (I_{\ell^2(\mathbb{N})}-P_{ \tau})e_k \rangle\\
	&=P_\tau\{a_n\}_n+\sum_{k=1}^{\infty}  \sum_{n=1}^{\infty}a_n\langle (I_{\ell^2(\mathbb{N})}-P_{ \tau}) e_n,e_k \rangle\\
	&=P_\tau\{a_n\}_n+\sum_{k=1}^{\infty}\sum_{n=1}^{\infty}a_n\langle  e_n,e_k \rangle-\sum_{k=1}^{\infty}\sum_{n=1}^{\infty}a_n\langle  P_\tau e_n,e_k \rangle\\
	&=P_\tau\{a_n\}_n+\sum_{k=1}^{\infty}a_ke_k-\sum_{k=1}^{\infty}\sum_{n=1}^{\infty}a_n\langle  \theta_\tau S^{-1}_\tau  \theta_\tau ^*e_n,e_k \rangle\\
	&=P_\tau\{a_n\}_n+\sum_{k=1}^{\infty}a_ke_k-\sum_{k=1}^{\infty}\sum_{n=1}^{\infty}a_n\langle   S^{-1}_\tau  \tau_n,\theta_\tau^*e_k \rangle\\
	&=P_\tau\{a_n\}_n+\sum_{k=1}^{\infty}a_ke_k-\sum_{k=1}^{\infty}\sum_{n=1}^{\infty}a_n\langle   S^{-1}_\tau  \tau_n,\tau_k \rangle\\
	&=P_\tau\{a_n\}_n+\sum_{k=1}^{\infty}a_ke_k-P_\tau\{a_n\}_n= \{a_n\}_n,\quad\forall  \{a_n\}_n \in \ell^2(\mathbb{N}).
	\end{align*}
	Thus  $\{\omega_n\}_n$  is  a Riesz basis 
	for $\mathcal{X}_1$  which completes the proof.
\end{proof}
We now illustrate Theorem \ref{DILATIONTHEOREMPASF} with an example.
\begin{example}
	Let $p\in[1,\infty)$. Let $\{e_n\}_n$ denote the canonical Schauder basis for  $\ell^p(\mathbb{N})$  and let $\{\zeta_n\}_n$ denote the coordinate functionals associated with $\{e_n\}_n$ respectively. Define 
	\begin{align*}
	R: \ell^p(\mathbb{N}) \ni (x_n)_{n=1}^\infty\mapsto (0,x_1,x_2, \dots)\in \ell^p(\mathbb{N}),\\
	L: \ell^p(\mathbb{N}) \ni (x_n)_{n=1}^\infty\mapsto (x_2,x_3,x_4, \dots)\in \ell^p(\mathbb{N}).
	\end{align*}
	Then $LR=I_{\ell^p(\mathbb{N})}$. Example \ref{EXAMPLE3} says that $ (\{f_n\coloneqq \zeta_nR\}_{n}, \{\tau_n\coloneqq Le_n\}_{n}) $ is a p-ASF for 	$\ell^p(\mathbb{N})$. Note that $\theta_f=R$ and $\theta_\tau=L$. Therefore $S_{f, \tau}=LR=I_{\ell^p(\mathbb{N})}$ and $P_{f, \tau}=RL.$ Then 
	\begin{align*}
	(I_{\ell^p(\mathbb{N})}-P_{f, \tau})(x_n)_{n=1}^\infty&=(x_n)_{n=1}^\infty-
	RL(x_n)_{n=1}^\infty\\
	&=(x_n)_{n=1}^\infty-(0,x_2,x_3, \dots)=(x_1, 0, 0,  \dots),\quad \forall (x_n)_{n=1}^\infty\in \ell^p(\mathbb{N})
	\end{align*}
	which says that $(I_{\ell^p(\mathbb{N})}-P_{f, \tau})(\ell^p(\mathbb{N}))\cong \mathbb{K}$. Using Theorem \ref{DILATIONTHEOREMPASF},
	\begin{align*}
	&\mathcal{X}_1=\ell^p(\mathbb{N})\oplus (I_{\ell^p(\mathbb{N})}-P_{f, \tau})(\ell^p(\mathbb{N}))\cong\ell^p(\mathbb{N})\oplus\mathbb{K}\cong \ell^p(\mathbb{N}\cup \{0\})\\
	&P:\ell^p(\mathbb{N}\cup \{0\})\ni (x_n)_{n=0}^\infty \mapsto (x_n)_{n=1}^\infty \in \ell^p(\mathbb{N}),
	\end{align*} 
	\begin{align*}
	\omega_1&=\tau_1\oplus (I_{\ell^p(\mathbb{N})}-P_{f, \tau})\tau_1=Le_1\oplus (I_{\ell^p(\mathbb{N})}-P_{f, \tau})Le_1
	=0\oplus 0,\\
	\omega_2&=\tau_2\oplus (I_{\ell^p(\mathbb{N})}-P_{f, \tau})\tau_2=Le_n\oplus (I_{\ell^p(\mathbb{N})}-P_{f, \tau})Le_2\\
	&=e_{1}\oplus (I_{\ell^p(\mathbb{N})}-P_{f, \tau})e_{1}=e_{1}\oplus RLe_{1}=e_{1}\oplus 0,\\
	\omega_n&=\tau_n\oplus (I_{\ell^p(\mathbb{N})}-P_{f, \tau})\tau_n=Le_n\oplus (I_{\ell^p(\mathbb{N})}-P_{f, \tau})Le_n\\
	&=e_{n-1}\oplus (I_{\ell^p(\mathbb{N})}-P_{f, \tau})e_{n-1}=e_{n-1}\oplus RLe_{n-1}=e_{n-1}\oplus e_{n-1}, \quad \forall n \geq 3,\\
	g_n&=f_n\oplus \zeta_n(I_{\ell^p(\mathbb{N})}-P_{f, \tau})=\zeta_nR\oplus \zeta_nRL=\zeta_nR(I_{\ell^p(\mathbb{N})}\oplus L), \quad \forall n \in \mathbb{N}
	\end{align*} 
	and $ (\{g_n\}_{n}, \{\omega_n\}_{n}) $ is a p-approximate Riesz basis for $\ell^p(\mathbb{N})$.	
\end{example}
\section{NEW IDENTITY FOR PARSEVAL p-APPROXIMATE SCH-AUDER FRAMES}
Certain classes of Banach spaces known as homogeneous semi-inner product spaces admit a kind of inner product and can be studied with certain  similarities with Hilbert spaces. These spaces are introduced by  \cite{LUMER} and studied extensively by  \cite{GILES}. We now recall the fundamentals of semi-inner products. Let $\mathcal{X}$ be a vector space over $\mathbb{K}$. A map $[\cdot, \cdot]:\mathcal{X}\times \mathcal{X} \to \mathbb{K}$ is said to be a \textbf{homogeneous semi-inner product} if it satisfies the following.
\begin{enumerate}[label=(\roman*)]
	\item $[x,x]>0$, for all $x \in \mathcal{X}, x\neq 0$.
	\item $[\lambda x,y]=\lambda [x,y]$, for all $x,y \in \mathcal{X}$, for all $\lambda \in \mathbb{K}$.
	\item 	$[x, \lambda y]=\overline{\lambda} [x,y]$, for all $x,y \in \mathcal{X}$, for all $\lambda \in \mathbb{K}$.
	\item $[x+y, z]=[x, z]+[y, z]$, for all $x,y,z \in \mathcal{X}$.
	\item $| [x,y]|^2\leq [x,x][y,y]$, for all $x,y \in \mathcal{X}$.
\end{enumerate}
A homogeneous semi-inner product $[\cdot, \cdot]$ induces a \textbf{norm} which is defined as $\|x\|\coloneqq \sqrt{[x,x]}$. A prototypical example of homogeneous semi-inner product spaces is the standard $\ell^p(\mathbb{N})$ space, $1<p<\infty$,  equipped with semi-inner product defined as follows. For $x=\{x_n\}_n,$  $y=\{y_n\}_n\in \ell^p(\mathbb{N})$, define 
\begin{align*}
[x,y]\coloneqq \begin{cases*}
\frac{\sum\limits_{n=1}^{\infty}x_n\overline{y_n}|y_n|^{p-2}}{\|y\|_p^{p-2}} \quad& if $ y \neq 0 $ \\
0 & if $ y=0.$ \\
\end{cases*}
\end{align*}
For certain classes of Banach spaces we have  Riesz  representation theorem. 
\begin{theorem}(\cite{GILES})\label{RIESZ} (\textbf{Riesz  representation theorem for Banach spaces})
	Let $\mathcal{X}$ be a complete homogeneous semi-inner product space. If $\mathcal{X}$ is continuous and uniformly convex, then for every bounded linear functional $f:  \mathcal{X}\to \mathbb{K}$, there exists a unique $y \in \mathcal{X}$  such that $f(x)=[x,y]$, for all $x \in \mathcal{X}$.
\end{theorem}
Theorem \ref{RIESZ} leads to the notion of generalized adjoint  whose existence is assured by the  following theorem. Two state the result we need two definitions.
\begin{definition}(cf. \cite{GILES})
Let $\mathcal{X}$ be a complete homogeneous semi-inner product space. Space 	$\mathcal{X}$ is said to be \textbf{continuous} if 
\begin{align*}
\text{Re}([x,y+\lambda x])\to \text{Re}([x,y]), \text{ for all real } \lambda \to 0, \forall x,y \in \mathcal{X} \text{ such that } \|x\|=\|y\|=1.
\end{align*}
\end{definition}
\begin{definition}(\cite{GILES})
	A Banach space is said to be \textbf{uniformly convex} if given $\epsilon >0$, there exists an $\delta>0$ such that if $ x,y \in \mathcal{X}$ are such that $\|x\|=\|y\|=1$ and $\|x-y\|>\epsilon$, then $\|x+y\|\leq 2(1-\delta)$.
\end{definition}
\begin{theorem}(\cite{KOEHLER})\label{THEOREMK}
	Let $\mathcal{X}$ be a complete homogeneous semi-inner product space. If $\mathcal{X}$ is continuous and uniformly convex, then for every bounded linear operator  $A:  \mathcal{X}\to \mathcal{X}$, there exists a unique map $A^\dagger:\mathcal{X}\to \mathcal{X}$,  which may not be linear or continuous (called as \textbf{generalized adjoint} of $A$)   such that 
	\begin{align*}
	[Ax,y]=[x,A^\dagger y],\quad \forall x,y \in \mathcal{X}.
	\end{align*}
	Moreover, the following statements hold.
	\begin{enumerate}[label=(\roman*)]
		\item $(\lambda A)^\dagger=\overline{\lambda}A^\dagger$, for all $\lambda \in \mathbb{K}$.
		\item $A^\dagger$ is injective if and only if $\overline{ A(\mathcal{X})}=\mathcal{X}$.
		\item If the norm of $\mathcal{X}$ is strongly (Frechet) differentiable, then $A^\dagger$ is continuous. 
	\end{enumerate}	
\end{theorem}
Throughout this section we assume that  $\mathcal{X}$ is a continuous, uniformly convex, homogeneous semi-inner product space. Let $ (\{f_n \}_{n}, \{\tau_n \}_{n}) $ be a p-ASF for $\mathcal{X}$. Theorem \ref{RIESZ} then says that each $f_n$ can be identified with unique $\omega_n \in \mathcal{X}$ which satisfies $f_n(x)=[x,\omega_n]$, for all $x \in \mathcal{X}$. Note that 
\begin{align*}
\sum_{n=1}^{\infty}[x, (S_{\omega, \tau}^{-1})^\dagger\omega_n]S_{\omega, \tau}^{-1}\tau_n=S_{\omega, \tau}^{-1}\left(\sum_{n=1}^{\infty}[ S_{\omega, \tau}^{-1}x, \omega_n]\tau_n\right)=S_{\omega, \tau}^{-1}x, \quad \forall x \in \mathcal{X}.
\end{align*}
Hence $ (\{\tilde{\omega}_n\coloneqq(S_{\omega, \tau}^{-1})^\dagger\omega_n \}_{n}, \{\tilde{\tau}_n\coloneqq S_{\omega, \tau}^{-1}\tau_n \}_{n}) $ is a p-ASF for $\mathcal{X}$ which is called as canonical dual frame for $ (\{\omega_n \}_{n}, \{\tau_n \}_{n}) $. 
Given $\mathbb{M} \subseteq \mathbb{N}$, we define $
S_\mathbb{M} :\mathcal{X} \ni x \mapsto \sum_{n\in \mathbb{M}} \langle x, \omega_n\rangle\tau_n\in
\mathcal{X}$. Because of Inequalities (\ref{FIRSTINEQUALITYPASF}) and (\ref{SECONDINEQUALITYPASF}), the map $S_\mathbb{M}$ is a well-defined bounded linear operator. Note that the  operator $S_\mathbb{M}$  may not be invertible. In Proposition 2.2 in   (\cite{BALACASAZZAEDIDINKUTYNIOK}) it is derived that if operators $U, V:\mathcal{H} \to \mathcal{H}$ satisfy $U+V=I_\mathcal{H}$, then $U-V=U^2-V^2$. This remains valid for Banach spaces. 
\begin{lemma}\label{LEMMA}
	If operators $U, V:\mathcal{X} \to \mathcal{X}$ satisfy $U+V=I_\mathcal{X}$, then $U-V=U^2-V^2$.
\end{lemma}
\begin{proof}
	We follow the ideas in  the proof of Proposition 2.2 in (\cite{BALACASAZZAEDIDINKUTYNIOK}): 
	\begin{align*}
	U-V&=U-(I_\mathcal{X}-U)=2U-I_\mathcal{X}=U^2-(I_\mathcal{X}-2U+U^2)\\
	&=U^2-(I_\mathcal{X}-U)^2=U^2-V^2.
	\end{align*}
\end{proof}
We now have Banach space version of Theorem \ref{CASAZZAGENERAL}.
\begin{theorem}\label{OURGENERAL}
	Let $ (\{\omega_n \}_{n}, \{\tau_n \}_{n}) $  be a p-ASF for $\mathcal{X}$. 	Then for every $\mathbb{M} \subseteq \mathbb{N}$, and for all  	$x \in \mathcal{X}$, 
	\begin{align*}
	\sum_{n\in \mathbb{M}}[x, \omega_n][\tau_n,x]-\sum_{n=1}^{\infty}[S_\mathbb{M}x,\tilde{\omega}_n][\tilde{\tau}_n,S_\mathbb{M}^\dagger x]&=\sum_{n\in \mathbb{M}^\text{c}}[x, \omega_n][\tau_n,x]-\sum_{n=1}^{\infty}[S_{\mathbb{M}^\text{c}}x,\tilde{\omega}_n][\tilde{\tau}_n,S_{\mathbb{M}^\text{c}}^\dagger x] .
\end{align*}
\end{theorem}
\begin{proof}
	For notational convenience, we denote $S_{f, \tau}$ by $S$. We clearly have $S_\mathbb{M}+S_{\mathbb{M}^\text{c}}=S$. Using 	$S^{-1}S_\mathbb{M}+S^{-1}S_{\mathbb{M}^\text{c}}=I_\mathcal{X}$ and Lemma \ref{LEMMA},  we get $S^{-1}S_\mathbb{M}-S^{-1}S_{\mathbb{M}^\text{c}}=(S^{-1}S_\mathbb{M})^2-(S^{-1}S_{\mathbb{M}^\text{c}})^2=S^{-1}S_\mathbb{M}S^{-1}S_\mathbb{M}-S^{-1}S_{\mathbb{M}^\text{c}}S^{-1}S_{\mathbb{M}^\text{c}}$ which gives 
	\begin{align*}
	S^{-1}S_\mathbb{M}-S^{-1}S_\mathbb{M}S^{-1}S_\mathbb{M}=S^{-1}S_{\mathbb{M}^\text{c}}-S^{-1}S_{\mathbb{M}^\text{c}}S^{-1}S_{\mathbb{M}^\text{c}}.
	\end{align*}
	Therefore for all $x, y \in \mathcal{X}$,
	\begin{align*}
	[S^{-1}S_\mathbb{M}x,y]-[S^{-1}S_\mathbb{M}S^{-1}S_\mathbb{M}x, y]=[S^{-1}S_{\mathbb{M}^\text{c}}x,y]-[S^{-1}S_{\mathbb{M}^\text{c}}S^{-1}S_{\mathbb{M}^\text{c}}x,y].
	\end{align*}
	In particular,  for all  	$x \in \mathcal{X}$,
	\begin{align*}
	[S^{-1}S_\mathbb{M}x,S^\dagger x]-[S^{-1}S_\mathbb{M}S^{-1}S_\mathbb{M}x, S^\dagger x]&=[S^{-1}S_{\mathbb{M}^\text{c}}x,S^\dagger x]-[S^{-1}S_{\mathbb{M}^\text{c}}S^{-1}S_{\mathbb{M}^\text{c}}x,S^\dagger x]
	\end{align*}
	which gives 
	\begin{align}\label{FIRST}
	[S_\mathbb{M}x,x]-[S^{-1}S_\mathbb{M}x, S_\mathbb{M}^\dagger x]=[S_{\mathbb{M}^\text{c}}x,x]-[S^{-1}S_{\mathbb{M}^\text{c}}x,S_{\mathbb{M}^\text{c}}^\dagger x], \quad \forall x \in \mathcal{X}.
	\end{align}
	Now note that 
	\begin{align*}
	\sum_{n=1}^{\infty}[x,\tilde{\omega}_n][\tilde{\tau}_n,y]&=\sum_{n=1}^{\infty}[x,(S^{-1})^\dagger\omega_n][S^{-1}\tau_n,y]=\sum_{n=1}^{\infty}[S^{-1}x, \omega_n][S^{-1}\tau_n,y]\\
	&=\left[\sum_{n=1}^{\infty}[S^{-1}x, \omega_n]S^{-1}\tau_n,y \right]=\left[S^{-1}\left(\sum_{n=1}^{\infty}[S^{-1}x, \omega_n]\tau_n\right),y \right]\\
	&=[S^{-1}x,y],
	 \quad \forall x,y \in \mathcal{X}.
	\end{align*}
	Equation (\ref{FIRST}) now gives 
	\begin{align*}
	\sum_{n\in \mathbb{M}}[x, \omega_n][\tau_n,x]-\sum_{n=1}^{\infty}[S_\mathbb{M}x,\tilde{\omega}_n][\tilde{\tau}_n,S_\mathbb{M}^\dagger x]&=\sum_{n\in \mathbb{M}^\text{c}}[x, \omega_n][\tau_n,x]\\
	&\quad -\sum_{n=1}^{\infty}[S_{\mathbb{M}^\text{c}}x,\tilde{\omega}_n][\tilde{\tau}_n,S_{\mathbb{M}^\text{c}}^\dagger x] ,
	& \quad \forall x \in \mathcal{X}.
	\end{align*}
\end{proof}
A look at Theorem \ref{SECOND} makes  a guess of the following statement  for Banach spaces.  Let $ (\{\omega_n \}_{n}, \{\tau_n \}_{n}) $  be a Parseval p-ASF for $\mathcal{X}$. 	Then for every $\mathbb{M} \subseteq \mathbb{N}$,
\begin{align}\label{WRONGEQUATION}
\sum_{n\in \mathbb{M}}[x, \omega_n][\tau_n,x]-\left\|\sum_{n\in \mathbb{M}}[ x, \omega_n] \tau_n\right\|^2=\sum_{n\in \mathbb{M}^\text{c}}[x, \omega_n][\tau_n,x]-\left\|\sum_{n\in \mathbb{M}^\text{c}}[ x, \omega_n] \tau_n\right\|^2, \quad \forall x \in \mathcal{X}.
\end{align}	
However, the correct Banach space version  of Theorem \ref{SECOND} is not Equation (\ref{WRONGEQUATION}) but it is stated in the next theorem. 
\begin{theorem} (\textbf{Parseval p-ASF identity}) \label{PASFID}
	Let $ (\{\omega_n \}_{n}, \{\tau_n \}_{n}) $  be a Parseval p-ASF for $\mathcal{X}$. 	Then for every $\mathbb{M} \subseteq \mathbb{N}$,
	\begin{align*}
	&\sum_{n\in \mathbb{M}}[x, \omega_n][\tau_n,x]-\sum_{n\in \mathbb{M}}\sum_{k\in \mathbb{M}}[x, \omega_n][\tau_n,\omega_k][\tau_k,x]\\
	&\quad=\sum_{n\in \mathbb{M}^\text{c}}[x, \omega_n][\tau_n,x]-\sum_{n\in \mathbb{M}^\text{c}}\sum_{k\in \mathbb{M}^\text{c}}[x, \omega_n][\tau_n,\omega_k][\tau_k,x], \quad \forall x \in \mathcal{X}.
	\end{align*}	
\end{theorem}
\begin{proof}
	Using  	Theorem \ref{OURGENERAL}, for all $ x \in \mathcal{X}$,
	\begin{align*}
	&\sum_{n\in \mathbb{M}}[x, \omega_n][\tau_n,x]-\sum_{n\in \mathbb{M}}\sum_{k\in \mathbb{M}}[x, \omega_n][\tau_n,\omega_k][\tau_k,x]\\
	&= \sum_{n\in \mathbb{M}}[x, \omega_n][\tau_n,x]-\left[\sum_{n\in \mathbb{M}}[x, \omega_n]\sum_{k\in \mathbb{M}}[\tau_n,\omega_k]\tau_k, x\right]\\
	&=\sum_{n\in \mathbb{M}}[x, \omega_n][\tau_n,x]-\left[\sum_{n\in \mathbb{M}}[x, \omega_n]S_\mathbb{M}\tau_n, x\right]\\
	&=\sum_{n\in \mathbb{M}}[x, \omega_n][\tau_n,x]-\left[S_\mathbb{M}\left(\sum_{n\in \mathbb{M}}[x, \omega_n]\tau_n\right), x\right]\\
	&=\sum_{n\in \mathbb{M}}[x, \omega_n][\tau_n,x]-\left[S_\mathbb{M} S_\mathbb{M}x, x\right]\\
	&=\sum_{n\in \mathbb{M}}[x, \omega_n][\tau_n,x]-\left[ S_\mathbb{M}x, S_\mathbb{M}^\dagger x\right]\\
	&=\sum_{n\in \mathbb{M}}[x, \omega_n][\tau_n,x]-\left[\sum_{n=1}^{\infty}[ S_\mathbb{M}x, \omega_n]\tau_n,S_\mathbb{M}^\dagger x\right]\\
	&=\sum_{n\in \mathbb{M}}[x, \omega_n][\tau_n,x]-\sum_{n=1}^{\infty}[ S_\mathbb{M}x, \omega_n][\tau_n,S_\mathbb{M}^\dagger x]\\
	&=\sum_{n\in \mathbb{M}^\text{c}}[x, \omega_n][\tau_n,x]-\sum_{n=1}^{\infty}[S_{\mathbb{M}^\text{c}}x,\omega_n][\tau_n,S_{\mathbb{M}^\text{c}}^\dagger x]
	\\
	&=\sum_{n\in \mathbb{M}^\text{c}}[x, \omega_n][\tau_n,x]-\sum_{n\in \mathbb{M}^\text{c}}\sum_{k\in \mathbb{M}^\text{c}}[x, \omega_n][\tau_n,\omega_k][\tau_k,x].
	\end{align*}
\end{proof}
In terms of $S_\mathbb{M}$ and $S_\mathbb{M}^\text{c}$, Theorem \ref{PASFID} can be written as 
\begin{align}\label{OPERATORDESCRPTION}
S_\mathbb{M}-S_\mathbb{M}^2=S_{\mathbb{M}^\text{c}}-S_{\mathbb{M}^\text{c}}^2\quad \text{ or } \quad S_\mathbb{M}+S_{\mathbb{M}^\text{c}}^2=S_{\mathbb{M}^\text{c}}+S_\mathbb{M}^2.
\end{align}
We now give an application of Theorem \ref{PASFID}. This is Banach space version of Theorem \ref{THIRD}. 
\begin{theorem}\label{LOWER234}
	Let $ (\{\omega_n \}_{n}, \{\tau_n \}_{n}) $  be a Parseval p-ASF for $\mathcal{X}$. 	Let  $\mathbb{M} \subseteq \mathbb{N}$. If $x \in \mathcal{X}$ is such that $[(S_\mathbb{M}-\frac{1}{2}I_\mathcal{X})^2x, x]\geq 0$, then 
	\begin{align*}
	&\sum_{n\in \mathbb{M}}[x, \omega_n][\tau_n,x]+\sum_{n\in \mathbb{M}^\text{c}}\sum_{k\in \mathbb{M}^\text{c}}[x, \omega_n][\tau_n,\omega_k][\tau_k,x]\\
	&=\sum_{n\in \mathbb{M}^\text{c}}[x, \omega_n][\tau_n,x]+\sum_{n\in \mathbb{M}}\sum_{k\in \mathbb{M}}[x, \omega_n][\tau_n,\omega_k][\tau_k,x]
	\geq \frac{3}{4}\|x\|^2 , \quad \text{ for that } x.
	\end{align*}
\end{theorem}
\begin{proof}
	We first compute 
	\begin{align*}
	S_\mathbb{M}^2+S_{\mathbb{M}^\text{c}}^2 &= S_\mathbb{M}^2+ (I_\mathcal{X}-S_\mathbb{M})^2=2 S_\mathbb{M}^2-2S_\mathbb{M}+I_\mathcal{X}\\
	&
	=2\left(S_\mathbb{M}-\frac{1}{2}I_\mathcal{X}\right)^2+\frac{1}{2}I_\mathcal{X}.
	\end{align*} 
	Hence if $x \in \mathcal{X}$ satisfies $[(S_\mathbb{M}-\frac{1}{2}I_\mathcal{X})^2x, x]\geq 0$, then 
	\begin{align*}
	[( S_\mathbb{M}^2+S_{\mathbb{M}^\text{c}}^2)x,x]\geq \frac{1}{2}\|x\|^2.
	\end{align*}
	Now using Equation (\ref{OPERATORDESCRPTION}) we get 
	\begin{align*}
	&2\sum_{n\in \mathbb{M}}[x, \omega_n][\tau_n,x]+2\sum_{n\in \mathbb{M}^\text{c}}\sum_{k\in \mathbb{M}^\text{c}}[x, \omega_n][\tau_n,\omega_k][\tau_k,x]=2[S_\mathbb{M}x,x]+2[ S_{\mathbb{M}^\text{c}}^2x,x]\\
	&=[2(S_\mathbb{M}+S_{\mathbb{M}^\text{c}}^2)x,x]=[((S_\mathbb{M}+S_{\mathbb{M}^\text{c}}^2)+(S_\mathbb{M}+S_{\mathbb{M}^\text{c}}^2))x,x]\\
	&=[((S_\mathbb{M}+S_{\mathbb{M}^\text{c}}^2)+(S_{\mathbb{M}^\text{c}}+S_\mathbb{M}^2))x,x]=[(I_\mathcal{X}+S_{\mathbb{M}^\text{c}}^2+S_\mathbb{M}^2)x,x]\\
	&=\|x\|^2 +[( S_\mathbb{M}^2+S_{\mathbb{M}^\text{c}}^2)x,x] \geq \frac{3}{2}\|x\|^2 , \quad \forall x \in \mathcal{X}.
	\end{align*}
\end{proof}

\section{PALEY-WIENER THEOREM FOR p-APPROXIMATE  \\
	SCHAUDER FRAMES}\label{SECTIONTWO}

In order to derive Paley-Wiener theorem for p-ASFs, we need a generalization of result of    \cite{HILDING}.
\begin{theorem}(\cite{CASAZZACHRISTENSTENPERTURBATION, CASAZZAKALTON, VANEIJNDHOVEN})\label{cc1} (\textbf{Casazza-Christensen-Kalton-van Eijndhoven perturbation})
	Let $ \mathcal{X}, \mathcal{Y}$ be Banach spaces and  $ A: \mathcal{X}\rightarrow \mathcal{Y}$ be a bounded invertible operator. If  a bounded  operator $ B : \mathcal{X}\rightarrow \mathcal{Y}$ is  such that there exist  $ \alpha, \beta \in \left [0, 1  \right )$ with 
	$$ \|Ax-Bx\|\leq\alpha\|Ax\|+\beta\|Bx\|,\quad \forall x \in  \mathcal{X},$$
	then $ B $ is invertible and 
	$$ \frac{1-\alpha}{1+\beta}\|Ax\|\leq\|Bx\|\leq\frac{1+\alpha}{1-\beta} \|Ax\|, \quad\forall x \in  \mathcal{X};$$
	$$ \frac{1-\beta}{1+\alpha}\frac{1}{\|A\|}\|y\|\leq\|B^{-1}y\|\leq\frac{1+\beta}{1-\alpha} \|A^{-1}\|\|y\|, \quad\forall y \in  \mathcal{Y}.$$
\end{theorem}
In the sequel, the standard Schauder basis for $\ell^p(\mathbb{N})$ is denoted by $\{e_n \}_{n}$.
\begin{theorem}\label{OURPERTURBATION}
	Let $ (\{f_n \}_{n}, \{\tau_n \}_{n}) $ be a p-ASF for $\mathcal{X}$. Assume that a collection $\{\tau_n \}_{n} $ in $\mathcal{X}$ is such that there exist $\alpha, \beta, \gamma \geq 0$ with  $ \max\{\alpha+\gamma\|\theta_f S_{f,\tau}^{-1}\|, \beta\}<1$ and 
	\begin{align}\label{PEREQUATIONA}
	\left\|\sum_{n=1}^{m}c_n(\tau_n-\omega_n)\right\|\leq\nonumber& \alpha\left\|\sum_{n=1}^{m}c_n\tau_n\right \|+\gamma \left(\sum_{n=1}^{m}|c_n|^p\right)^\frac{1}{p}+\beta\left\|\sum_{n=1}^{m}c_n\omega_n\right \|,  \\
	&\quad\forall c_1,  \dots, c_m \in \mathbb{K}, ~ m=1,2,  \dots.
	\end{align}
	Then $ (\{f_n \}_{n}, \{\omega_n \}_{n}) $ is a p-ASF for $\mathcal{X}$ with bounds 
	\begin{align*}
	\frac{1-(\alpha+\gamma\|\theta_f S_{f,\tau}^{-1}\|)}{(1+\beta)\|S_{f,\tau}^{-1}\|} \quad \text{and} \quad  \left(\frac{1+\alpha}{1-\beta}\|\theta_\tau\|+\frac{\gamma}{1-\beta}\right)\|\theta_f\|.
	\end{align*}
\end{theorem}
\begin{proof}
	For $ m=1,2,  \dots$   and for every $c_1,  \dots, c_m \in \mathbb{K}$, 
	\begin{align*}
	\left\|\sum_{n=1}^{m}c_n\omega_n\right\|&\leq \left\|\sum_{n=1}^{m}c_n(\tau_n-\omega_n)\right\|+\left\|\sum_{n=1}^{m}c_n\tau_n\right\|\\
	&\leq (1+\alpha)\left\|\sum_{n=1}^{m}c_n\tau_n\right \|+\gamma \left(\sum_{n=1}^{m}|c_n|^p\right)^\frac{1}{p}+\beta\left\|\sum_{n=1}^{m}c_n\omega_n\right \|.
	\end{align*}
	Hence 
	\begin{align*}
	\left\| \sum\limits_{n=1}^mc_n\omega_n\right\|\leq\frac{1+\alpha}{1-\beta}\left\| \sum\limits_{n=1}^mc_n\tau_n\right\|+\frac{\gamma}{1-\beta}\left( \sum\limits_{n=1}^m|c_n|^p\right)^\frac{1}{p},  \quad\forall c_1,  \dots, c_m \in \mathbb{K}, m=1,2, \dots.
	\end{align*}
	Therefore   $\theta_\omega$ is well-defined bounded linear operator with 
	\begin{align*}
 \|\theta_\omega\|\leq\frac{1+\alpha}{1-\beta}\|\theta_\tau\|+\frac{\gamma}{1-\beta}.
	\end{align*} 
	Now Equation (\ref{PEREQUATIONA}) gives 
	\begin{align*}
	\left\|\sum_{n=1}^{\infty}c_n(\tau_n-\omega_n)\right\|\leq \alpha\left\|\sum_{n=1}^{\infty}c_n\tau_n\right \|+\gamma \left(\sum_{n=1}^{\infty}|c_n|^p\right)^\frac{1}{p}+\beta\left\|\sum_{n=1}^{\infty}c_n\omega_n\right \|,  \quad\forall  \{c_n\}_n \in \ell^p(\mathbb{N}).
	\end{align*}
	That is, 
	\begin{align}\label{PEREQUATIONB}
	\|\theta_\tau \{c_n\}_n-\theta_\omega \{c_n\}_n\|&\leq \alpha \|\theta_\tau \{c_n\}_n\|+\gamma\left( \sum\limits_{n=1}^\infty|c_n|^p\right)^\frac{1}{p}+\beta \|\theta_\omega \{c_n\}_n\|,\nonumber\\
	& \quad\forall \{c_n\}_n \in \ell^p(\mathbb{N}).
	\end{align}
	By taking $\{c_n\}_n =\{f_n(S_{f,\tau}^{-1}x)\}_n=\theta_fS_{f,\tau}^{-1}x$ in Equation (\ref{PEREQUATIONB}), we get 
	\begin{align*}
	\|\theta_\tau \theta_fS_{f,\tau}^{-1}x-\theta_\omega \theta_fS_{f,\tau}^{-1}x\|\leq \alpha \|\theta_\tau \theta_fS_{f,\tau}^{-1}x\|+\gamma\left( \sum\limits_{n=1}^\infty|f_n(S_{f,\tau}^{-1}x)|^p\right)^\frac{1}{p}+\beta \|\theta_\omega\theta_fS_{f,\tau}^{-1}x\|, 
	\end{align*}
for all  $x \in \mathcal{X}.$	That is, 
	\begin{align*}
	\|x-S_{f,\omega}S_{f,\tau}^{-1}x\|&\leq \alpha \| x\|+\gamma\|\theta_fS_{f,\tau}^{-1}x\|+\beta \|S_{f,\omega}S_{f,\tau}^{-1}x\|\\
	&\leq  (\alpha +\gamma\|\theta_fS_{f,\tau}^{-1}\|)\|x\|+\beta \|S_{f,\omega}S_{f,\tau}^{-1}x\|, \quad\forall  x \in \mathcal{X}.
	\end{align*}
	Since $ \max\{\alpha+\gamma\|\theta_f S_{f,\tau}^{-1}\|, \beta\}<1$, we can use Theorem \ref{cc1} to get the operator $S_{f,\omega}S_{f,\tau}^{-1}$ to be  invertible and 
	\begin{align*}
	\|(S_{g,\omega} S_{f,\tau}^{-1})^{-1}\| \leq \frac{1+\beta}{1-(\alpha+\gamma\|\theta_f S_{f,\tau}^{-1}\|)}.
	\end{align*}
	Hence the operator $S_{f,\omega}=(S_{f,\omega}S_{f,\tau}^{-1})S_{f,\tau}$ is invertible. Therefore $ (\{f_n \}_{n}, \{\omega_n \}_{n}) $ is a p-ASF for $\mathcal{X}$. We get the frame bounds from the following calculations:
	\begin{align*}
	&\| S_{f,\omega}^{-1}\|\leq\|S_{f,\tau}^{-1}\|\| S_{f,\tau}S_{f,\omega}^{-1}\| \leq \frac{\|S_{f,\tau}^{-1}\|(1+\beta)}{1-(\alpha+\gamma\|\theta_f S_{f,\tau}^{-1}\|)}\quad\text{ and }\\
	&\|S_{f,\omega}\|\leq \|\theta_\omega\|\|\theta_f\|\leq \left(\frac{1+\alpha}{1-\beta}\|\theta_\tau\|+\frac{\gamma}{1-\beta}\right)\|\theta_f\|.
	\end{align*}
\end{proof}
\begin{remark}\label{OURCOROLLARY}
	Theorem \ref{OLECAZASSA} is a corollary of Theorem \ref{OURPERTURBATION}. In particular, Theorems \ref{FIRSTPER} and \ref{SECONDPER} are corollaries of Theorem  \ref{OURPERTURBATION}. Indeed, 
	let $\{\tau_n\}_n$ be a frame for  $\mathcal{H}$. We define
	\begin{align*}
	f_n:\mathcal{H} \ni h \mapsto f_n(h)\coloneqq \langle h, \tau_n\rangle \in \mathbb{K}, \quad \forall n \in \mathbb{N}.
	\end{align*}	
	Then $\theta_f=\theta_\tau$ and  $ (\{f_n \}_{n}, \{\tau_n \}_{n}) $ is  a 2-approximate frame   for $\mathcal{H}$.	Theorem \ref{OURPERTURBATION} now says that $ (\{f_n \}_{n}, \{\omega_n \}_{n}) $ is a 2-ASF for $\mathcal{H}$. To prove Theorem \ref{OLECAZASSA}, it now suffices to prove that $\{\omega_n\}_n$ is a frame for  $\mathcal{H}$. Since $ (\{f_n \}_{n}, \{\omega_n \}_{n}) $ is a 2-ASF for $\mathcal{H}$, it follows that $\theta_\omega$ is surjective.  We now  use the following result to conclude that  $\{\omega_n\}_n$ is a frame for $\mathcal{H}$.
\end{remark}
\begin{theorem}(\cite{OCPSEUDOINVERSE})
A collection $\{\tau_n\}_n$ is a frame for $\mathcal{H}$ if and only if the map 
\begin{align*}
T:\ell^2(\mathbb{N})\ni \{c_n \}_{n}\mapsto  \sum_{n=1}^{\infty}c_n\tau_n \in \mathcal{H}
\end{align*}
is a well-defined bounded linear surjective operator.
\end{theorem}
\begin{corollary}
	Let $q$ be the conjugate index of $p$. Let $ (\{f_n \}_{n}, \{\tau_n \}_{n}) $ be a p-ASF for $\mathcal{X}$. Assume that a collection $\{\tau_n \}_{n} $ in $\mathcal{X}$ is    such that 
	and
	$$ \lambda \coloneqq \sum_{n=1}^\infty\|\tau_n-\omega_n\|^q <\frac{1}{\|\theta_f S_{f,\tau}^{-1}\|^q}.$$
	Then $ (\{f_n \}_{n}, \{\omega_n \}_{n}) $ is a p-ASF for $\mathcal{X}$ with bounds 
	\begin{align*}
	\frac{1-\lambda^{1/p}\|\theta_f S_{f,\tau}^{-1}\|}{\|S_{f,\tau}^{-1}\|} \quad \text{ and  } \quad  (\|\theta_\tau\|+\lambda^{1/p}).
	\end{align*}
\end{corollary}
\begin{proof}
	Take $  \alpha =0, \beta=0, \gamma=\lambda^{1/p}$. Then $ \max\{\alpha+\gamma\|\theta_f S_{f,\tau}^{-1}\|, \beta\}<1$ and 
	\begin{align*}
	&\left\|\sum\limits_{n=1}^{m}c_n(\tau_n-\omega_n)\right\|\leq \left(\sum\limits_{n=1}^{m}\|\tau_n-\omega_n\|^q \right)^\frac{1}{q}\left(\sum\limits_{n=1}^{m}|c_n|^p\right)^\frac{1}{p}\leq \gamma\left(\sum\limits_{n=1}^{m}|c_n|^p\right)^\frac{1}{p}, \\
	& \quad \quad \quad \quad \quad \quad\quad \quad\quad \quad\forall c_1,  \dots, c_m \in \mathbb{K},~ m=1, 2,\dots.
	\end{align*}
	By using Theorem \ref{OURPERTURBATION} we now get the result.
\end{proof}
We next derive stability result which does not demand  maximum condition on parameters $\alpha$ and $\gamma$.
\begin{theorem}
	Let $ (\{f_n \}_{n}, \{\tau_n \}_{n}) $ be a p-ASF for $\mathcal{X}$. Assume that a collection $\{\tau_n \}_{n} $ in $\mathcal{X}$ and a collection $ \{g_n \}_{n}	$ in $\mathcal{X}^*$  are such that there exist $r,s,t,\alpha, \beta, \gamma \geq 0$ with  $ \max\{ \beta,s\}<1$ and
	\begin{align*}
	&\left\|\sum_{n=1}^{m}(f_n-g_n)(x)e_n\right\|\leq r\left\|\sum_{n=1}^{m}f_n(x)e_n\right \|+t \|x\|+s\left\|\sum_{n=1}^{m}g_n(x)e_n\right \|,   \\
	&\quad \quad \quad \quad \quad \quad\quad \quad\quad \quad \quad \forall x  \in \mathcal{X}, m=1, 2, \dots,\\
	&\left\|\sum_{n=1}^{m}c_n(\tau_n-\omega_n)\right\|\leq \alpha\left\|\sum_{n=1}^{m}c_n\tau_n\right \|+\gamma \left(\sum_{n=1}^{m}|c_n|^p\right)^\frac{1}{p}+\beta\left\|\sum_{n=1}^{m}c_n\omega_n\right \|, \\
	& \quad \quad \quad \quad \quad \quad\quad \quad\quad \quad  \quad\forall c_1,  \dots, c_m \in \mathbb{K}, m=1,2, \dots.
	\end{align*}	
	Assume that one of the following holds. 
	\begin{enumerate}[label=(\roman*)]
		\item $\sum_{n=1}^{\infty}(\|f_n-g_n\|\|S_{f,\tau}^{-1}\tau_n\|+\|g_n\|\|S_{f,\tau}^{-1}(\tau_n-\omega_n)\|)<1.$
		\item $\sum_{n=1}^{\infty}(\|f_n-g_n\|\|S_{f,\tau}^{-1}\omega_n\|+\|f_n\|\|S_{f,\tau}^{-1}(\tau_n-\omega_n)\|)<1.$
		\item $\sum_{n=1}^{\infty}(\|(f_n-g_n)S_{f,\tau}^{-1}\|\|\tau_n\|+\|g_nS_{f,\tau}^{-1}\|\|\tau_n-\omega_n\|)<1.$
		\item $\sum_{n=1}^{\infty}(\|(f_n-g_n)S_{f,\tau}^{-1}\|\|\omega_n\|+\|f_nS_{f,\tau}^{-1}\|\|\tau_n-\omega_n\|)<1$.
	\end{enumerate}
	Then $ (\{g_n \}_{n}, \{\omega_n \}_{n}) $ is a p-ASF for $\mathcal{X}$. Moreover, an upper bound is  
	\begin{align*}
	\left(\frac{1+\alpha}{1-\beta}\|\theta_\tau\|+\frac{\gamma}{1-\beta}\right)\left(\frac{1+r}{1-s}\|\theta_f\|+\frac{t}{1-s}\right).
	\end{align*}
\end{theorem}
\begin{proof}
	Following the initial lines in the proof of Theorem \ref{OURPERTURBATION}, we see that $\theta_g$ and $\theta_\omega$ are well-defined bounded linear operators. 	We now consider four cases.\\
	Assume (i). Then 
	\begin{align*}
	&\left\|x-\sum_{n=1}^{\infty}g_n(x)S_{f,\tau}^{-1}\omega_n\right\|=\left\|\sum_{n=1}^{\infty}f_n(x)S_{f,\tau}^{-1}\tau_n-\sum_{n=1}^{\infty}g_n(x)S_{f,\tau}^{-1}\omega_n\right\|\\
	&\quad\leq \sum_{n=1}^{\infty}\|f_n(x)S_{f,\tau}^{-1}\tau_n-g_n(x)S_{f,\tau}^{-1}\omega_n\|\\
	&\quad\leq \sum_{n=1}^{\infty}\bigg\{\|f_n(x)S_{f,\tau}^{-1}\tau_n-g_n(x)S_{f,\tau}^{-1}\tau_n\|+\|g_n(x)S_{f,\tau}^{-1}\tau_n-g_n(x)S_{f,\tau}^{-1}\omega_n\|\bigg\}\\
	&\quad=\sum_{n=1}^{\infty}\bigg\{\|(f_n-g_n)(x)S_{f,\tau}^{-1}\tau_n\|+\|g_n(x)S_{f,\tau}^{-1}(\tau_n-\omega_n)\|\bigg\}\\
	&\quad\leq \left(\sum_{n=1}^{\infty}\bigg\{\|f_n-g_n\|\|S_{f,\tau}^{-1}\tau_n\|+\|g_n\|\|S_{f,\tau}^{-1}(\tau_n-\omega_n)\|\bigg\}\right)\|x\|.
	\end{align*}
	Therefore the operator  $S_{f,\tau}^{-1}S_{g,\omega}$ is invertible.\\
	Assume (ii). Then 
		\begin{align*}
	&\left\|x-\sum_{n=1}^{\infty}g_n(x)S_{f,\tau}^{-1}\omega_n\right\|=\left\|\sum_{n=1}^{\infty}f_n(x)S_{f,\tau}^{-1}\tau_n-\sum_{n=1}^{\infty}g_n(x)S_{f,\tau}^{-1}\omega_n\right\|\\
	&\quad\leq \sum_{n=1}^{\infty}\|f_n(x)S_{f,\tau}^{-1}\tau_n-g_n(x)S_{f,\tau}^{-1}\omega_n\|\\
	&\quad\leq \sum_{n=1}^{\infty}\bigg\{\|f_n(x)S_{f,\tau}^{-1}\tau_n-f_n(x)S_{f,\tau}^{-1}\omega_n\|+\|f_n(x)S_{f,\tau}^{-1}\omega_n-g_n(x)S_{f,\tau}^{-1}\omega_n\|\bigg\}\\
	&\quad=\sum_{n=1}^{\infty}\bigg\{\|f_n(x)S_{f,\tau}^{-1}(\tau_n-\omega_n)\|+\|(f_n-g_n)(x)S_{f,\tau}^{-1}\omega_n\|\bigg\}\\
	&\quad\leq \left(\sum_{n=1}^{\infty}\bigg\{\|f_n\|\|S_{f,\tau}^{-1}(\tau_n-\omega_n)\|+\|f_n-g_n\|\|S_{f,\tau}^{-1}\omega_n\|\bigg\}\right)\|x\|.
	\end{align*}
	Therefore the operator  $S_{f,\tau}^{-1}S_{g,\omega}$ is invertible.\\
	Assume (iii). Then 
	\begin{align*}
	&\left\|x-\sum_{n=1}^{\infty}g_n(S_{f,\tau}^{-1}x)\omega_n\right\|=\left\|\sum_{n=1}^{\infty}f_n(S_{f,\tau}^{-1}x)\tau_n-\sum_{n=1}^{\infty}g_n(S_{f,\tau}^{-1}x)\omega_n\right\|\\
	&\leq\sum_{n=1}^{\infty}\|f_n(S_{f,\tau}^{-1}x)\tau_n-g_n(S_{f,\tau}^{-1}x)\omega_n\| \\
	&\leq \sum_{n=1}^{\infty}\bigg\{\|f_n(S_{f,\tau}^{-1}x)\tau_n-g_n(S_{f,\tau}^{-1}x)\tau_n\|+\|g_n(S_{f,\tau}^{-1}x)\tau_n-g_n(S_{f,\tau}^{-1}x)\omega_n\|\bigg\}\\
	&=\sum_{n=1}^{\infty}\bigg\{\|(f_n-g_n)(S_{f,\tau}^{-1}x)\tau_n\|+\|g_n(S_{f,\tau}^{-1}x)(\tau_n-\omega_n)\|\bigg\}\\
	&\leq \left(\sum_{n=1}^{\infty}\bigg\{\|(f_n-g_n)S_{f,\tau}^{-1}\|\|\tau_n\|+\|g_nS_{f,\tau}^{-1}\|\|\tau_n-\omega_n\|\bigg\}\right)\|x\|.
	\end{align*}
	Therefore the operator  $S_{g,\omega}S_{f,\tau}^{-1}$ is invertible.\\
	Assume (iv). Then 
	\begin{align*}
	&\left\|x-\sum_{n=1}^{\infty}g_n(S_{f,\tau}^{-1}x)\omega_n\right\|=\left\|\sum_{n=1}^{\infty}f_n(S_{f,\tau}^{-1}x)\tau_n-\sum_{n=1}^{\infty}g_n(S_{f,\tau}^{-1}x)\omega_n\right\|\\
	&\quad\leq\sum_{n=1}^{\infty}\|f_n(S_{f,\tau}^{-1}x)\tau_n-g_n(S_{f,\tau}^{-1}x)\omega_n\|\\
	&\quad\leq \sum_{n=1}^{\infty}\bigg\{\|f_n(S_{f,\tau}^{-1}x)\tau_n-f_n(S_{f,\tau}^{-1}x)\omega_n\|+\|f_n(S_{f,\tau}^{-1}x)\omega_n-g_n(S_{f,\tau}^{-1}x)\omega_n\|\bigg\}\\
	&\quad=\sum_{n=1}^{\infty}\bigg\{\|f_n(S_{f,\tau}^{-1}x)(\tau_n-\omega_n)\|+\|(f_n-g_n)(S_{f,\tau}^{-1}x)\omega_n\|\bigg\}\\
	&\quad\leq \left(\sum_{n=1}^{\infty}\bigg\{\|f_nS_{f,\tau}^{-1}\|\|\tau_n-\omega_n\|+\|(f_n-g_n)S_{f,\tau}^{-1}\|\|\omega_n\|\bigg\}\right)\|x\|.
	\end{align*}
	Therefore the operator  $S_{g,\omega}S_{f,\tau}^{-1}$ is invertible.
	
	Hence in each of the assumptions we get that  $ (\{g_n \}_{n}, \{\omega_n \}_{n}) $ is a p-ASF for $\mathcal{X}$.
\end{proof} 
We end this chapter by deriving results on the expansion of sequences to approximate Schauder frames.

A routine Hilbert space argument shows that a sequence $\{\tau_n\}_n$ is  a Bessel sequence for  Hilbert space  $\mathcal{H}$ if and only if the map $ S_\tau :\mathcal{H} \ni h \mapsto \sum_{n=1}^\infty \langle h, \tau_n\rangle\tau_n\in
\mathcal{H} $ is a well-defined bounded linear operator. In fact, if $\{\tau_n\}_n$ is a Bessel sequence, then both maps $\theta_\tau:\mathcal{H} \ni h \mapsto \theta_\tau h \coloneqq\{\langle h, \tau_n\rangle \}_n \in \ell^2(\mathbb{N})$ and $\theta_\tau^*:\ell^2(\mathbb{N}) \ni \{a_n\}_n \mapsto \theta_\tau^*\{a_n\}_n\coloneqq \sum_{n=1}^{\infty}a_n\tau_n \in \mathcal{H}$  are well-defined bounded linear operators (Chapter 3   in \cite{CHRISTENSEN}). Now $\theta_\tau^*\theta_\tau=S_\tau$ and hence $S_\tau$ is a  well-defined bounded linear operator. Conversely, let $S_\tau$ be a  well-defined bounded linear operator. Definition of $S_\tau$ says that it is a positive operator. Thus there exists $b>0$ such that $\langle S_\tau h, h \rangle \leq b\|h\|^2$, $\forall h \in \mathcal{H}$. Again using the definition of $S_\tau$ gives that  $\{\tau_n\}_n$ is a Bessel sequence. This observation and Definition \ref{ASFDEF} make us to define the following.
\begin{definition}
	Let $\{\tau_n\}_n$ be a sequence in a Banach space  $\mathcal{X}$ and 	$\{f_n\}_n$ be a sequence in  $\mathcal{X}^*$. The pair $ (\{f_n \}_{n}, \{\tau_n \}_{n}) $ is said to be a \textbf{weak reconstruction sequence}  or \textbf{approximate Bessel sequence} (ABS) for $\mathcal{X}$ if  $ S_{f, \tau}:\mathcal{X}\ni x \mapsto S_{f, \tau}x\coloneqq \sum_{n=1}^\infty
	f_n(x)\tau_n \in
	\mathcal{X}$	is a well-defined bounded linear operator.
\end{definition} 
We next recall the reconstruction property of Banach spaces.
\begin{definition}(\cite{CASAZZARECONSTRUCTION})
	A Banach space $\mathcal{X}$ is said to have the \textbf{reconstruction property}  if there exists a sequence $\{\tau_n\}_n$   in  $\mathcal{X}$ and a sequence 	$\{f_n\}_n$  in  $\mathcal{X}^*$ such that $x=\sum_{n=1}^\infty
	f_n(x)\tau_n ,  \forall x \in \mathcal{X}.$
\end{definition}
Using \textbf{approximation property of Banach spaces}   (cf. \cite{CASAZZAAPPROXIMATION}), Casazza and Christensen proved the following result.
\begin{theorem}(\cite{CASAZZARECONSTRUCTION})\label{RECTHEOREM}
	There exists a Banach space $\mathcal{X}$ such that $\mathcal{X}$ does not have the reconstruction property.	
\end{theorem}
Now we have the following characterization. This is a result which is in contrast with Theorem \ref{BESSELEXPANSIONHILBERT}.
\begin{theorem}\label{CHARBESSELTOFRAME}
	Let 	$ (\{f_n \}_{n}, \{\tau_n \}_{n}) $ be a weak reconstruction sequence for $\mathcal{X}$. Then the following are equivalent.
	\begin{enumerate}[label=(\roman*)]
		\item $ (\{f_n \}_{n}, \{\tau_n \}_{n}) $ can be expanded to an ASF for $\mathcal{X}$.
		\item $\mathcal{X}$ has the reconstruction property. 
	\end{enumerate}	 
\end{theorem}
\begin{proof}
	\begin{enumerate}[label=(\roman*)]
		\item $\Rightarrow $ (ii) Let $\{\omega_n\}_n$ be a sequence in   $\mathcal{X}$ and 	$\{g_n\}_n$ be a sequence in  $\mathcal{X}^*$ such that $ (\{f_n \}_{n}\cup \{g_n \}_{n}, \{\tau_n \}_{n}\cup\{\omega_n \}_{n} ) $ is an ASF   for $\mathcal{X}$. Let $S_{(f,g), (\tau, \omega)}$ be the frame operator for $ (\{f_n \}_{n}\cup \{g_n \}_{n}, \{\tau_n \}_{n}\cup\{\omega_n \}_{n} ) $. Then 
		\begin{align*}
			x&=S_{(f,g), (\tau, \omega)}^{-1}S_{(f,g), (\tau, \omega)}x=S_{(f,g), (\tau, \omega)}^{-1}\left( \sum_{n=1}^{\infty}f_n(x)\tau_n+\sum_{n=1}^{\infty}g_n(x)\omega_n\right)\\
			&=\sum_{n=1}^{\infty}f_n(x)S_{(f,g), (\tau, \omega)}^{-1}\tau_n+\sum_{n=1}^{\infty}g_n(x)S_{(f,g), (\tau, \omega)}^{-1}\omega_n, \quad \forall x \in \mathcal{X} 
		\end{align*}
		which shows that $\mathcal{X}$ has the reconstruction property.
		\item $\Rightarrow  $ (i) $\{\omega_n\}_n$ be a sequence in   $\mathcal{X}$ and  	$\{g_n\}_n$  be a sequence in  $\mathcal{X}^*$  such that 
		$x=\sum_{n=1}^\infty
		g_n(x)\omega_n $,  $ \forall x \in \mathcal{X}.$  Define $h_n\coloneqq g_n $, $\rho_n \coloneqq (I_\mathcal{X}-S_{f, \tau})\omega_n $, for all $n \in \mathbb{N}$. Then 
		\begin{align*}
			\sum_{n=1}^{\infty}f_n(x)\tau_n+\sum_{n=1}^{\infty}h_n(x)\rho_n&=\sum_{n=1}^{\infty}f_n(x)\tau_n+\sum_{n=1}^{\infty}g_n(x)(I_\mathcal{X}-S_{f, \tau})\omega_n\\
			&=S_{f, \tau}x+(I_\mathcal{X}-S_{f, \tau})\left(\sum_{n=1}^{\infty}g_n(x)\omega_n\right)\\
			&=S_{f, \tau}x+(I_\mathcal{X}-S_{f, \tau})x=x, \quad \forall x \in \mathcal{X}. 
		\end{align*}
		Therefore $ (\{f_n \}_{n}\cup \{h_n \}_{n}, \{\tau_n \}_{n}\cup\{\rho_n \}_{n} ) $ is an ASF   for $\mathcal{X}$.	
	\end{enumerate}
\end{proof}
We now show that there are  infinitely many ways to expand a weak reconstruction sequence into an ASF. 
\begin{corollary}\label{NOT}
	There exists a Banach space $\mathcal{X}$ such that given any weak reconstruction sequence $ (\{f_n \}_{n}, \{\tau_n \}_{n}) $  for $\mathcal{X}$, $ (\{f_n \}_{n}, \{\tau_n \}_{n}) $ can not be expanded to an ASF for $\mathcal{X}$.
\end{corollary}
\begin{proof}
	From Theorem \ref{RECTHEOREM}, there exists a Banach space $\mathcal{X}$ which does not have the reconstruction property. Let $ (\{f_n \}_{n}, \{\tau_n \}_{n}) $ be any weak reconstruction sequence  for $\mathcal{X}$. Theorem \ref{CHARBESSELTOFRAME} now says that $ (\{f_n \}_{n}, \{\tau_n \}_{n}) $ can not be expanded to an ASF for $\mathcal{X}$.	
\end{proof}
Following corollary is an easy consequence of Theorem \ref{CHARBESSELTOFRAME}.
\begin{corollary}
	Let 	$ (\{f_n \}_{n}, \{\tau_n \}_{n}) $ be a weak reconstruction sequence for $\mathcal{X}$. If 	$\mathcal{X}$ admits a Schauder basis, then $ (\{f_n \}_{n}, \{\tau_n \}_{n}) $ can be expanded to an ASF for $\mathcal{X}$.	
\end{corollary}
Note that  Theorem \ref{CHARBESSELTOFRAME} may not add  countably many  elements to a weak reconstruction sequence to get an ASF. In the following example we show that  it adds just  one element to a weak reconstruction sequence and yields an ASF. 
\begin{example}\label{EXAMPLEOME}
	Let $p\in[1,\infty)$ and let $\{e_n\}_n$ denote the standard  Schauder basis for  $\ell^p(\mathbb{N})$  and  let $\{\zeta_n\}_n$ denote the coordinate functionals associated with $\{e_n\}_n$. Define 
	\begin{align*}
	&R: \ell^p(\mathbb{N}) \ni (x_n)_{n=1}^\infty\mapsto (0,x_1,x_2, \dots)\in \ell^p(\mathbb{N}),\\
	&L: \ell^p(\mathbb{N}) \ni (x_n)_{n=1}^\infty\mapsto (x_2,x_3,x_4, \dots)\in \ell^p(\mathbb{N}).
	\end{align*}
	Clearly  $ (\{f_n\coloneqq \zeta_nL\}_{n}, \{\tau_n\coloneqq Re_n\}_{n}) $ is a weak reconstruction sequence  for 	$\ell^p(\mathbb{N})$. Note that  $S_{f, \tau}=RL$ and 
	\begin{align*}
	&	(I_{\ell^p(\mathbb{N})}-S_{f, \tau})e_1=e_1-RLe_1=e_1-0=e_1,\\
	&	(I_{\ell^p(\mathbb{N})}-S_{f, \tau})e_n=e_n-RLe_n=e_n-Re_{n-1}=e_n-e_n=0, \quad \forall n \geq 2.
	\end{align*}
	Let $g_n\coloneqq \zeta_n$ and $\omega_n\coloneqq e_n$, $\forall n \in \mathbb{N}$. Theorem \ref{CHARBESSELTOFRAME} now says that 
	$ (\{f_n \}_{n}\cup \{h_1 \}, \{\tau_n \}_{n}\cup\{\rho_1 \} ) $ is an ASF   for $\ell^p(\mathbb{N})$.
\end{example}
It may be possible to expand a weak reconstruction sequence to a tight  ASF by adding finitely many elements Hilbert space. In this case, we can estimate the number of elements added to a tight ASF. This is given in the following theorem which can be compared with Theorem \ref{BESSELNUMBERHILBERT}.  
\begin{theorem}\label{NUMBERINEQUALITY}
	Let 	$ (\{f_n \}_{n}, \{\tau_n \}_{n}) $ be a weak reconstruction sequence for $\mathcal{X}$. If  $ (\{f_n \}_{n}\cup \{g_k\}_{k=1}^N, \{\tau_n \}_{n}\cup \{\omega_k\}_{k=1}^N) $ is a $\lambda$-tight ASF for $\mathcal{X}$, then 
	\begin{align}\label{KSNUMBER}
	N\geq \dim (\lambda I_\mathcal{X}-S_{f, \tau}) (\mathcal{X}).
	\end{align}
	Further, the Inequality (\ref{KSNUMBER}) can not be improved.	
\end{theorem}
\begin{proof}
	Let $S_{(f,g), (\tau, \omega)}$ be the frame operator for $ (\{f_n \}_{n}\cup \{g_k\}_{k=1}^N, \{\tau_n \}_{n}\cup \{\omega_k\}_{k=1}^N) $. Set $S_{g, \omega}(x)\coloneqq\sum_{k=1}^{N}g_k(x)\omega_k,  \forall x \in \mathcal{X}$. Then 
	\begin{align*}
	\lambda x =S_{(f,g), (\tau, \omega)}x=\sum_{n=1}^{\infty}f_n(x)\tau_n+\sum_{k=1}^{N}g_k(x)\omega_k=S_{f, \tau}x+S_{g,\omega}x,\quad \forall x \in \mathcal{X}.
	\end{align*}
	Therefore 
	\begin{align*}
	N \geq \dim S_{g, \omega} (\mathcal{X}) = \dim (\lambda I_\mathcal{X}-S_{f, \tau}) (\mathcal{X}).
	\end{align*}
	Example \ref{EXAMPLEOME}  says that inequality in Theorem \ref{NUMBERINEQUALITY} can not be improved.	
\end{proof}
We now state the definition of a p-weak reconstruction sequence and give  a extension theorem for p-weak reconstruction sequences.
\begin{definition}
	Let $p \in [1, \infty)$.	A weak reconstruction sequence   $ (\{f_n \}_{n}, \{\tau_n \}_{n}) $  for $\mathcal{X}$	is said to be a \textbf{p-weak reconstruction sequence} or \textbf{p-approximate Bessel sequence} (p-ABS) for $\mathcal{X}$ if both the maps $
	\theta_f: \mathcal{X}\ni x \mapsto \theta_f x\coloneqq \{f_n(x)\}_n \in \ell^p(\mathbb{N}) $ and $
	\theta_\tau : \ell^p(\mathbb{N}) \ni \{a_n\}_n \mapsto \theta_\tau \{a_n\}_n\coloneqq \sum_{n=1}^\infty a_n\tau_n \in \mathcal{X}$
	are well-defined bounded linear operators. 
\end{definition}
\begin{theorem}
	Let $p \in [1, \infty)$.	If    $ (\{f_n \}_{n}, \{\tau_n \}_{n}) $ is a p-weak reconstruction sequence for $\ell^p(\mathbb{N}), $ then 	$ (\{f_n \}_{n}, \{\tau_n \}_{n}) $ can be expanded to a p-ASF.
\end{theorem} 
\begin{proof}
	Let $ \{e_n \}_{n}$ and $ \{\zeta_n \}_{n}$ be as in Example \ref{EXAMPLEOME}.  Define $h_n\coloneqq \zeta_n $, $\rho_n \coloneqq (I_{\ell^p(\mathbb{N})}-S_{f, \tau})e_n $, for all $n \in \mathbb{N}$. Then  it follows that $ (\{f_n \}_{n}\cup \{h_n \}_{n}, \{\tau_n \}_{n}\cup\{\rho_n \}_{n} ) $ is a p-ASF for  $\ell^p(\mathbb{N}) $.
\end{proof}

{\onehalfspacing \chapter{WEAK OPERATOR-VALUED FRAMES}\label{chap5}}
\vspace{0.5cm}

\section{BASIC PROPERTIES}
Let $\mathcal{H},\mathcal{H}_0$   be Hilbert spaces and $ \mathcal{B}(\mathcal{H}, \mathcal{H}_0)$ be the collection of all bounded linear operators from $\mathcal{H}$ to $\mathcal{H}_0$.  In this  chapter, we study   a generalization of the notion of operator-valued frame by studying the convergence of the series 
$
\sum_{n=1}^\infty \Psi_n^*A_n
$
to a  bounded  invertible operator in  $ \mathcal{B}(\mathcal{H}, \mathcal{H}_0)$.  

\begin{definition}
	Let  $ \{A_n\}_{n} $ and  $ \{\Psi_n\}_{n} $ be collections in $ \mathcal{B}(\mathcal{H}, \mathcal{H}_0)$. The pair $( \{A_n\}_{n}, $ $  \{\Psi_n\}_{n} )$ is said to be a \textbf{weak operator-valued frame}  (weak OVF) in $ \mathcal{B}(\mathcal{H}, \mathcal{H}_0) $   if  the series 
	\begin{align*}
	\text{(\textbf{Operator-valued frame operator})}\quad S_{A, \Psi} \coloneqq  \sum_{n=1}^\infty \Psi_n^*A_n
	\end{align*}
	converges in the strong-operator topology on $ \mathcal{B}(\mathcal{H})$ to a  bounded  invertible operator. If $	S_{A, \Psi}=I_\mathcal{H}$, then it is called as a Parseval weak OVF. 
\end{definition}
We now give various examples of weak OVFs. 
\begin{example}
	\begin{enumerate}[label=(\roman*)]
		\item By taking $\Psi_n\coloneqq A_n$, for all $n \in \mathbb{N}$, it follows that an \textbf{operator-valued frame} is a weak OVF. In particular, a \textbf{G-frame} is a weak OVF. 
		\item Let $( \{\tau_n\}_{n},  \{f_n\}_{n} )$ be a \textbf{pseudo-frame} for $\mathcal{H}$. If we define  $\Psi_n\coloneqq f_n$ and $A_nh\coloneqq \langle h, \tau_n \rangle $,  for all $n \in \mathbb{N}$, for all $h \in \mathcal{H}$, then $( \{A_n\}_{n},  \{\Psi_n\}_{n} )$ is a weak OVF in $ \mathcal{B}(\mathcal{H}, \mathbb{K}) $. Similarly it follows that \textbf{frames for subspaces}, \textbf{fusion frames}, \textbf{outer frames}, \textbf{oblique frames} and  \textbf{quasi-projectors}  are all weak OVFs.
		\item Let $ C \in \mathcal{B}(\mathcal{H})$ be invertible and $ \{\tau_n\}_{n} $ be a \textbf{C-controlled frame} for $\mathcal{H}$ (\cite{BALAZSGRYBOS}). If we define $\Psi_n h \coloneqq \langle h, C\tau_n \rangle$ and $A_nh\coloneqq \langle h, \tau_n \rangle$,  for all $n \in \mathbb{N}$, for all $h \in \mathcal{H}$, then $( \{A_n\}_{n},  \{\Psi_n\}_{n} )$ is a weak OVF in $ \mathcal{B}(\mathcal{H}, \mathbb{K}) $. In particular, every \textbf{weighted frame} is a weak OVF.
		\item  Let $( \{\tau_n\}_{n},  \{f_n\}_{n} )$ be an \textbf{approximate  Schauder  frame} for $\mathcal{H}$ (\cite{FREEMANODELL}). Note that it is possible for Hilbert spaces to have approximate Schauder frames which are not frames. If we define $\Psi_n\coloneqq f_n$ and $A_nh\coloneqq \langle h, \tau_n \rangle $,  for all $n \in \mathbb{N}$, for all $h \in \mathcal{H}$, then  $( \{A_n\}_{n},  \{\Psi_n\}_{n} )$ is a weak OVF in $ \mathcal{B}(\mathcal{H}, \mathbb{K}) $. In particular, \textbf{atomic decompositions} (\cite{CASAZZAHANLARSONFRAMEBANACH}),  \textbf{framings} (\cite{CASAZZAHANLARSONFRAMEBANACH}), \textbf{cb-frames} (\cite{LIURUAN}) and \textbf{Schauder frames} (\cite{CASAZZA}) for Hilbert spaces are all weak OVFs.
		\item Let $ \{\tau_n\}_{n}$ be a \textbf{signed frame} for $\mathcal{H}$ with signs $\{\sigma_n\}_{n}$ (\cite{PENGWALDRON}). If we define  $\Psi_n\coloneqq h\coloneqq \langle h, \sigma_n\tau_n \rangle$ and $A_nh\coloneqq \langle h, \tau_n \rangle $,  for all $n \in \mathbb{N}$, for all $h \in \mathcal{H}$, then $( \{A_n\}_{n},  \{\Psi_n\}_{n} )$ is a weak OVF in $ \mathcal{B}(\mathcal{H}, \mathbb{K}) $.
	\end{enumerate}	
\end{example}
Unlike in the case of OVFs,  the frame operator $S_{A, \Psi}$ need not be positive.  Since $S_{A, \Psi}$ is invertible, there are $a,b>0$ such that 
\begin{align*}
a\|h\|\leq \|S_{A, \Psi}h\|\leq b \|h\|, \quad \forall h \in \mathcal{H}.
\end{align*}
We call such $a,b$ as lower and upper frame bounds, respectively. Supremum of the set of all lower frame bounds is called as optimal lower frame bound and infimum of the set of all upper frame bounds is called as optimal upper frame bound. We easily get that 
\begin{align*}
&\text{ optimal lower frame bound }=\|S_{A,\Psi}^{-1}\|^{-1},\\
&  \text{ optimal upper frame bound } = \|S_{A,\Psi}\|.
\end{align*}
 We now define the notion of dual   weak OVFs.
\begin{definition}
	A weak  OVF  $ (\{B_n\}_{n} , \{\Phi_n\}_{n} )$  in $\mathcal{B}(\mathcal{H}, \mathcal{H}_0)$ is said to be a  \textbf{dual} for a weak  OVF $  ( \{A_n\}_{n},  \{\Psi_n\}_{n} )$ in $\mathcal{B}(\mathcal{H}, \mathcal{H}_0)$  if  
	\begin{align*}
	\sum_{n=1}^\infty \Psi_n^*B_n= \sum_{n=1}^\infty\Phi^*_nA_n=I_{\mathcal{H}}.
	\end{align*}
\end{definition}
Note that  dual always exists for a given weak OVF.  In fact, a direct calculation shows that
\begin{align*}
( \{\widetilde{A}_n\coloneqq A_nS_{A,\Psi}^{-1}\}_{n},\{\widetilde{\Psi}_n\coloneqq\Psi_n(S_{A,\Psi}^{-1})^*\}_{n})
\end{align*}
is a weak OVF and is a dual for $  ( \{A_n\}_{n},  \{\Psi_n\}_{n} )$.  This weak  OVF  is  called as the \textbf{canonical dual} for $  ( \{A_n\}_{n},  \{\Psi_n\}_{n} )$. Canonical dual has two nice properties. Following two results establish them.

\begin{proposition}
	Let $( \{A_n\}_{n},  \{\Psi_n\}_{n} )$ be a weak  OVF in $ \mathcal{B}(\mathcal{H}, \mathcal{H}_0).$  If $ h \in \mathcal{H}$ has representation  $ h=\sum_{n=1}^\infty A_n^*y_n= \sum_{n=1}^\infty\Psi_n^*z_n, $ for some sequences $ \{y_n\}_{n},\{z_n\}_{n}$ in $ \mathcal{H}_0$, then 
	$$ \sum_{n=1}^\infty\langle y_n,z_n\rangle =\sum_{n=1}^\infty\langle \widetilde{\Psi}_nh,\widetilde{A}_nh\rangle+\sum_{n=1}^\infty\langle y_n-\widetilde{\Psi}_nh,z_n-\widetilde{A}_nh\rangle. $$
\end{proposition}  
\begin{proof}
	We start from the right side and see 	
	\begin{align*}
	&\sum\limits_{n=1}^\infty\langle \widetilde{\Psi}_nh,\widetilde{A}_nh\rangle+\sum\limits_{n=1}^\infty\langle y_n, z_n\rangle -\sum\limits_{n=1}^\infty\langle y_n, \widetilde{A}_nh\rangle-\sum\limits_{n=1}^\infty\langle \widetilde{\Psi}_nh, z_n\rangle +\sum\limits_{n=1}^\infty\langle \widetilde{\Psi}_nh,\widetilde{A}_nh\rangle\\
	&=2\sum\limits_{n=1}^\infty\langle \widetilde{\Psi}_nh,\widetilde{A}_nh\rangle+ \sum\limits_{n=1}^\infty\langle y_n, z_n\rangle-\sum\limits_{n=1}^\infty\langle y_n,A_nS_{A,\Psi}^{-1}h\rangle-\sum\limits_{n=1}^\infty\langle \Psi_n(S_{A,\Psi}^{-1})^*h, z_n\rangle\\
	&= 2\left\langle\sum\limits_{n=1}^\infty(S_{A,\Psi}^{-1})^*A_n^*\Psi_n(S_{A,\Psi}^{-1})^*h, h \right\rangle+ \sum\limits_{n=1}^\infty\langle y_n, z_n\rangle-\left\langle \sum\limits_{n=1}^\infty A_n^*y_n,S_{A,\Psi}^{-1}h\right \rangle \\
	&\quad -\left\langle (S_{A,\Psi}^{-1})^*h , \sum\limits_{n=1}^\infty\Psi_n^*z_n\right \rangle\\
	&=2 \langle (S_{A,\Psi}^{-1})^*h,h \rangle + \sum\limits_{n=1}^\infty\langle y_n, z_n\rangle -\langle h, S_{A,\Psi}^{-1}h\rangle-\langle (S_{A,\Psi}^{-1})^*h, h\rangle
	\end{align*}
	which gives the left side.	
\end{proof}
\begin{theorem}\label{CANONICALDUALFRAMEPROPERTYOPERATORVERSIONWEAK}
	Let $( \{A_n\}_{n},  \{\Psi_n\}_{n} )$ be a weak  OVF with frame bounds $ a$ and $ b.$ 
	\begin{enumerate}[label=(\roman*)]
		\item The canonical dual weak  OVF for the canonical dual weak  OVF  for $( \{A_n\}_{n},  \{\Psi_n\}_{n} )$ is itself.
		\item$ \frac{1}{b}, \frac{1}{a}$ are frame bounds for the canonical dual of $( \{A_n\}_{n},  \{\Psi_n\}_{n} )$.
		\item If $ a, b $ are optimal frame bounds for $( \{A_n\}_{n},  \{\Psi_n\}_{n} )$, then $ \frac{1}{b}, \frac{1}{a}$ are optimal  frame bounds for its canonical dual.
	\end{enumerate} 
\end{theorem} 
\begin{proof}
	Since (ii) and (iii)  follow from the property of invertible operators on Banach spaces, we have to argue for (i): frame operator for $(\{A_nS_{A,\Psi}^{-1}\}_{n}, \{\Psi_n(S_{A,\Psi}^{-1})^*\}_{n} )$ is 
	$$ \sum\limits_{n=1}^\infty(\Psi_n(S_{A,\Psi}^{-1})^*)^* (A_nS_{A,\Psi}^{-1}) =S_{A,\Psi}^{-1}\left(\sum\limits_{n=1}^\infty\Psi_n ^*A_n\right)S_{A,\Psi}^{-1} =S_{A,\Psi}^{-1}S_{A,\Psi}S_{A,\Psi}^{-1}= S_{A,\Psi}^{-1}.$$
	Therefore, its canonical dual is $(\{(A_nS_{A,\Psi}^{-1})S_{A,\Psi}\}_{n} , \{(\Psi_n(S_{A,\Psi}^{-1})^*)S_{A,\Psi}^*\}_{n})$ which is the original frame.
\end{proof}
For the further study of weak OVFs, we impose some conditions so that the frame operator splits.  
\begin{definition}
	A weak OVF $( \{A_n\}_{n},  \{\Psi_n\}_{n} )$ is said to be \textbf{factorable} if both the maps (called \textbf{analysis operator})
	\begin{align*}
	  \theta_A:\mathcal{H} \ni h \mapsto   \theta_A h\coloneqq\sum_{n=1}^\infty L_nA_n h \in \ell^2(\mathbb{N}) \otimes \mathcal{H}_0\\
		 \theta_\Psi:\mathcal{H}\ni h \mapsto  \theta_\Psi h\coloneqq \sum_{n=1}^\infty L_n\Psi_n h \in \ell^2(\mathbb{N}) \otimes \mathcal{H}_0
	\end{align*}
	are well-defined bounded linear operators. 
\end{definition}
We next give an example which shows that a weak OVF need not be factorable.
\begin{example}
	On $ \mathbb{C},$ define $ A_nx\coloneqq\frac{x}{\sqrt{n}},  \forall x \in \mathbb{C}, \forall n \in \mathbb{N}$, and $\Psi_1x\coloneqq x, \Psi_nx\coloneqq0, \forall x \in \mathbb{C}, \forall n \in \mathbb{N}\setminus\{1\} $. Then $ \sum_{n=1}^\infty\Psi_n^*A_nx$ converges to an  invertible operator but  $ \sum_{n=1}^\infty L_nA_nx$ does not converge. In fact, using Equation (\ref{LEQUATION}),
	\begin{align*}
	\left\| \sum_{n=1}^mL_nA_n1\right\|^2=\sum_{n=1}^m\|A_n1\|^2=\sum_{n=1}^m\frac{1}{n} \to \infty \quad \text{ as } \quad m \to \infty.
	\end{align*}
\end{example}
Equation (\ref{LEQUATION}) gives the following theorem  easily.
\begin{theorem}
	Let   $( \{A_n\}_{n},  \{\Psi_n\}_{n} )$  be a factorable weak  OVF in $ \mathcal{B}(\mathcal{H}, \mathcal{H}_0)$. 
	\begin{enumerate}[label=(\roman*)]
		\item Analysis operator 
		\begin{align*}
		\theta_A:\mathcal{H} \ni h \mapsto   \theta_A h\coloneqq\sum_{n=1}^\infty L_nA_n h \in \ell^2(\mathbb{N}) \otimes \mathcal{H}_0
		\end{align*}
		is a well-defined bounded  linear injective operator.
		\item Synthesis operator 
		\begin{align*}
		\theta_\Psi^*:\ell^2(\mathbb{N})\otimes \mathcal{H}_0 \ni z\mapsto\sum\limits_{n=1}^\infty \Psi_n^*L_n^*z \in \mathcal{H} 
		\end{align*}
		is a well-defined bounded  linear surjective operator.
		\item Frame operator 
		factors  as $S_{A,\Psi}=\theta_\Psi^*\theta_A.$
		\item $ P_{A,\Psi} \coloneqq \theta_A S_{A,\Psi}^{-1} \theta_\Psi^*:\ell^2(\mathbb{N})\otimes \mathcal{H}_0 \to \ell^2(\mathbb{N})\otimes \mathcal{H}_0$ is an  idempotent onto $ \theta_A(\mathcal{H})$.
	\end{enumerate}	
\end{theorem}
We next define the notions of Riesz and orthonormal factorable weak OVFs. 
\begin{definition}\label{RIESZOVF}
	A factorable weak OVF  $( \{A_n\}_{n},  \{\Psi_n\}_{n} )$  in $\mathcal{B}(\mathcal{H}, \mathcal{H}_0)$ is said to be a \textbf{Riesz OVF}   if $ P_{A,\Psi}= I_{\ell^2(\mathbb{N})}\otimes I_{\mathcal{H}_0}$. A Parseval and  Riesz OVF, i.e., $\theta_\Psi^*\theta_A=I_\mathcal{H} $ and  $\theta_A\theta_\Psi^*=I_{\ell^2(\mathbb{N})}\otimes I_{\mathcal{H}_0} $ is called as an \textbf{orthonormal OVF}. 
\end{definition}
\begin{proposition}\label{ORTHORESULT}
	A factorable weak OVF  $( \{A_n\}_{n},  \{\Psi_n\}_{n} )$ in $ \mathcal{B}(\mathcal{H}, \mathcal{H}_0)$ is an orthonormal OVF  if and only if it is a Parseval OVF and $ A_n\Psi_m^*=\delta_{n,m}I_{\mathcal{H}_0},\forall n,m \in \mathbb{N}$. 
\end{proposition}
\begin{proof}
	$(\Rightarrow)$	
	We have $\theta_A\theta_\Psi^*=I_{\ell^2(\mathbb{N})}\otimes I_{\mathcal{H}_0}.$ Hence 
	\begin{align*}
	e_m\otimes y&=\theta_A\theta_\Psi^*(e_m\otimes y)=\sum_{n=1}^\infty L_nA_n\left(\sum_{k=1}^\infty\Psi^*_kL_k^*(e_m\otimes y)\right)\\
	&=\sum_{n=1}^\infty L_nA_n\Psi^*_my=\sum_{n=1}^\infty(e_n\otimes A_n\Psi^*_my)\\
	&= e_m\otimes( A_m\Psi^*_m y)+\sum_{n=1, n\neq m}^\infty(e_n\otimes A_n\Psi^*_my),\forall m \in \mathbb{N}, \quad y \in\mathcal{H}_0 .
	\end{align*}
	 We then have  $A_n\Psi^*_my=\delta_{n,m}y,\forall y \in \mathcal{H}_0$.
	
	$(\Leftarrow)$ $ \theta_A\theta_\Psi^*=\sum_{n=1}^\infty L_nA_n(\sum_{k=1}^\infty\Psi_k^*L_k^*)=\sum_{n=1}^\infty L_nL_n^*=I_{\ell^2(\mathbb{N})}\otimes I_{\mathcal{H}_0}.$
	
\end{proof}
 We now derive  a dilation result for  factorable weak OVFs. First we need a lemma for this. 
\begin{lemma}\label{DILATIONLEMMA}
	Let $( \{A_n\}_{n},  \{\Psi_n\}_{n} )$ be  a factorable  weak OVF  in $ \mathcal{B}(\mathcal{H}, \mathcal{H}_0)$. Then the range of 	$\theta_A $ is closed.
\end{lemma}
\begin{proof}
	Let $ \{h_n\}_{n=1}^\infty$ in $\mathcal{H} $  be such that  $ \{\theta _Ah_n\}_{n=1}^\infty$ converges to $ y \in \mathcal{H}_0$. This gives $ S_{A,\Psi}h_n \rightarrow \theta_\Psi^*y$  as $ n \rightarrow \infty$ and this in turn gives $h_n \rightarrow S_{A,\Psi}^{-1} \theta_\Psi^*y $  as $ n \rightarrow \infty.$ An application of $ \theta_A$ gives $\theta_Ah_n \rightarrow \theta_AS_{A,\Psi}^{-1} \theta_\Psi^*y $ as $ n \rightarrow \infty.$ Therefore $ y=\theta_A(S_{A,\Psi}^{-1} \theta_\Psi^*y).$	
\end{proof}
\begin{theorem}\label{OPERATORDILATION}
	Let $( \{A_n\}_{n},  \{\Psi_n\}_{n} )$ be  a Parseval factorable weak OVF  in $ \mathcal{B}(\mathcal{H}, \mathcal{H}_0)$ such that $ \theta_A(\mathcal{H})=\theta_\Psi(\mathcal{H})$ and $ P_{A,\Psi}$ is projection. Then there exist a Hilbert space $ \mathcal{H}_1 $ which contains $ \mathcal{H}$ isometrically and  bounded linear operators $B_n,\Phi_n:\mathcal{H}_1\rightarrow \mathcal{H}_0, \forall n  $ such that $(\{B_n\}_{n} ,\{\Phi_n\}_{n})$ is an orthonormal OVF in $ \mathcal{B}(\mathcal{H}_1, \mathcal{H}_0)$ and $B_n|_{\mathcal{H}}=  A_n,\Phi_n|_{\mathcal{H}}=\Psi_n, \forall n \in \mathbb{N}$. 
\end{theorem}
\begin{proof}
	We first see  that  $P_{A,\Psi}$ is the  orthogonal projection from $ \ell^2(\mathbb{N})\otimes \mathcal{H}_0$ onto $\theta_A(\mathcal{H})=\theta_\Psi(\mathcal{H})$. Define $ \mathcal{H}_1\coloneqq\mathcal{H}\oplus \theta_A(\mathcal{H})^\perp$. From Lemma \ref{DILATIONLEMMA}, $\mathcal{H}_1$ becomes a Hilbert space. Then $\mathcal{H} \ni h \mapsto h\oplus 0 \in \mathcal{H}_1 $ is an isometry. Set $P_{A,\Psi}^\perp\coloneqq I_{\ell^2(\mathbb{N})\otimes \mathcal{H}_0}-P_{A,\Psi}$ and  define 
	\begin{align*}
	&B_n:\mathcal{H}_1\ni h\oplus g\mapsto A_nh+L_n^*P_{A,\Psi}^\perp g \in \mathcal{H}_0,  \\
	& \Phi_n:\mathcal{H}_1\ni h\oplus g\mapsto \Psi_nh+L_n^*P_{A,\Psi}^\perp g \in \mathcal{H}_0 , \quad \forall n \in \mathbb{N}.
	\end{align*}
	Then  clearly $B_n|_{\mathcal{H}}=  A_n,\Phi_n|_{\mathcal{H}}=\Psi_n, \forall n \in \mathbb{N}$. Now 
	\begin{align*}
	\theta_B(h\oplus g)=\sum_{n=1}^\infty L_nA_nh+\sum_{n=1}^\infty L_nL_n^*P_{A,\Psi}^\perp g=\theta_Ah+P_{A,\Psi}^\perp g, \quad \forall  h\oplus g \in \mathcal{H}_1.
	\end{align*}
	Similarly $\theta_\Phi(h\oplus g)=\theta_\Psi h+P_{A,\Psi}^\perp g, \forall h\oplus g\in \mathcal{H}_1 $.  Also 
	\begin{align*}
	\langle \theta_B^*z,h\oplus g \rangle&= \langle z,  \theta_B(h\oplus g) \rangle = \langle \theta_A^*z,  h\rangle+\langle P_{A,\Psi}^\perp z,  g\rangle \\
	&= \langle \theta_A^*z\oplus P_{A,\Psi}^\perp z, h\oplus g\rangle , \quad \forall z \in \ell^2(\mathbb{N})\otimes \mathcal{H}_0 , \forall  h\oplus g \in \mathcal{H}_1.
	\end{align*}
	Hence $\theta_B^*z=\theta_A^*z\oplus P_{A,\Psi}^\perp z, \forall z \in \ell^2(\mathbb{N})\otimes \mathcal{H}_0 $ and similarly $\theta_\Phi^*z=\theta_\Psi^*z\oplus P_{A,\Psi}^\perp z $,  $\forall z \in \ell^2(\mathbb{N})\otimes \mathcal{H}_0$.
	By  using $\theta_A(\mathcal{H})=\theta_\Psi(\mathcal{H}) $ and $\theta_\Psi^*P_{A,\Psi}^\perp=0=P_{A,\Psi}^\perp\theta_A ,$ we get  
	\begin{align*}
	S_{B,\Phi}(h\oplus g)&= \theta_\Phi^*(\theta_Ah+ P_{A,\Psi}^\perp g)=\theta_\Psi^*(\theta_Ah+P_{A,\Psi}^\perp g)\oplus P_{A,\Psi}^\perp(\theta_Ah+P_{A,\Psi}^\perp g)\\
	&=(S_{A,\Psi}h+0)\oplus(0+P_{A,\Psi}^\perp g)=S_{A,\Psi}h\oplus P_{A,\Psi}^\perp g\\
	&=I_\mathcal{H}h\oplus I_{\theta_A(\mathcal{H})^\perp}g, \quad \forall h\oplus g\in \mathcal{H}_1.
	\end{align*}
	Hence $(\{B_n\}_{n},\{\Phi_n\}_{n} )$ is  a Parseval weak OVF in $ \mathcal{B}(\mathcal{H}_1, \mathcal{H}_0)$. We further find 
	\begin{align*}
	P_{B,\Phi}z&=\theta_BS_{B,\Phi}^{-1}\theta_\Phi^*z=\theta_B\theta_\Phi^*z=\theta_B(\theta_\Psi^*z\oplus P_{A,\Psi}^\perp z)\\
	&=\theta_A(\theta_\Psi^*z)+ P_{A,\Psi}^\perp(P_{A,\Psi}^\perp z)=P_{A,\Psi} z+P_{A,\Psi}^\perp z\\
	&=P_{A,\Psi} z+((I_{\ell^2(\mathbb{N})}\otimes I_{\mathcal{H}_0})-P_{A,\Psi})z =(I_{\ell^2(\mathbb{N})}\otimes I_{\mathcal{H}_0} )z, \quad\forall z \in  \ell^2(\mathbb{N})\otimes \mathcal{H}_0.
	\end{align*}
	Therefore $(\{B_n\}_{n} ,\{\Phi_n\}_{n})$ is a Riesz weak OVF in $ \mathcal{B}(\mathcal{H}_1, \mathcal{H}_0)$. Thus $(\{B_n\}_{n} ,\{\Phi_n\}_{n})$ is  an orthonormal weak OVF in $ \mathcal{B}(\mathcal{H}_1, \mathcal{H}_0)$.
\end{proof}
\begin{theorem}\label{THAFSCHAROVF}
	A pair  $( \{A_n\}_{n},  \{\Psi_n\}_{n} )$ is a factorable weak OVF  in $ \mathcal{B}(\mathcal{H}, \mathcal{H}_0)$
	if and only if 
	\begin{align*}
	A_n=L_n^* U, \quad \Psi_n=L_n^*V, \quad \forall n \in \mathbb{N},
	\end{align*}  
	where $U, V:\mathcal{H} \rightarrow \ell^2(\mathbb{N}) \otimes \mathcal{H}_0$ are bounded linear operators such that $V^*U$ is bounded invertible.
\end{theorem}
\begin{proof}
	$(\Leftarrow)$ Clearly $\theta_A$ and $\theta_\Psi$ are well-defined bounded linear operators. Let $h\in \mathcal{H}$. Then using Equation (\ref{LEQUATION}), we have 
	\begin{align}\label{ORIGINALEQA}
	S_{A, \Psi}h= \sum_{n=1}^\infty
	(L_n^*V)^*L_n^*Uh=V^*\left(\sum_{n=1}^\infty L_nL_n^*\right)Uh=V^*Uh.
	\end{align} 
	Hence $S_{A, \Psi}$ is bounded invertible. \\
	$(\Rightarrow)$ Define $U\coloneqq \sum_{n=1}^\infty L_nA_n$, $V\coloneqq \sum_{n=1}^\infty L_n\Psi_n$. Then
	\begin{align*}
	&L_n^* U=L_n^* \left(\sum_{k=1}^\infty L_kA_k\right)=\sum_{k=1}^\infty L_n^*L_kA_k=A_n,\\
	&L_n^* V=L_n^* \left(\sum_{k=1}^\infty L_k\Psi_k\right)=\sum_{k=1}^\infty L_n^*L_k\Psi_k=\Psi_n, \quad \forall n \in \mathbb{N}
	\end{align*}
	and 
	\begin{align*}
	V^*U=\left(\sum_{n=1}^\infty \Psi_n^*L_n^*\right)\left(\sum_{k=1}^\infty L_kA_k\right) =  \sum_{n=1}^\infty \Psi_n^*A_n=S_{A, \Psi}
	\end{align*}
	which is bounded invertible.
\end{proof}
Using Theorem \ref{THAFSCHAROVF} we can characterize Riesz and orthonormal factorable weak OVFs.
\begin{corollary}\label{FIRSTCOROLLARY}
	A pair  $( \{A_n\}_{n},  \{\Psi_n\}_{n} )$ is a Riesz factorable weak OVF  in $ \mathcal{B}(\mathcal{H}, \mathcal{H}_0)$
	if and only if 
	\begin{align*}
	A_n=L_n^* U, \quad \Psi_n=L_n^*V, \quad \forall n \in \mathbb{N},
	\end{align*}  
	where $U, V:\mathcal{H} \rightarrow \ell^2(\mathbb{N}) \otimes \mathcal{H}_0$ are bounded linear operators such that $V^*U$ is bounded invertible and $ U(V^*U)^{-1}V^* =I_{\ell^2(\mathbb{N}) \otimes \mathcal{H}_0}$.	
\end{corollary}
\begin{proof}
	$(\Leftarrow)$ $P_{A,\Psi}= U(V^*U)^{-1}V^* =I_{\ell^2(\mathbb{N}) \otimes \mathcal{H}_0}$.
	
	$(\Rightarrow)$ Let $U$ and $V$ be as in Theorem \ref{THAFSCHAROVF}. Then 
	$U(V^*U)^{-1}V^*=P_{A,\Psi}=I_{\ell^2(\mathbb{N}) \otimes \mathcal{H}_0}$.
\end{proof}
\begin{corollary}
	A pair  $( \{A_n\}_{n},  \{\Psi_n\}_{n} )$ is an orthonormal  factorable weak OVF  in $ \mathcal{B}(\mathcal{H}, \mathcal{H}_0)$
	if and only if 
	\begin{align*}
	A_n=L_n^* U, \quad \Psi_n=L_n^*V, \quad \forall n \in \mathbb{N},
	\end{align*}  
	where $U, V:\mathcal{H} \rightarrow \ell^2(\mathbb{N}) \otimes \mathcal{H}_0$ are bounded linear operators such that $V^*U$ is bounded invertible and $ V^*U=I_\mathcal{H}$, $I_{\ell^2(\mathbb{N}) \otimes \mathcal{H}_0}= UV^*$.	
\end{corollary}
\begin{proof}
	We use Corollary \ref{FIRSTCOROLLARY}. 
	
	$ (\Leftarrow)$ $S_{A,\Psi}=V^*U=I_\mathcal{H}, P_{A,\Psi}=\theta_AS_{A,\Psi}^{-1}\theta_\Psi^*=\theta_A\theta_\Psi^*= \theta_FUV^*\theta_F^*=I_{\ell^2(\mathbb{N})}\otimes I_{\mathcal{H}_0}.$
	
	$(\Rightarrow)$ $V^*U=S_{A,\Psi}=I_\mathcal{H},$ and by using Proposition \ref{ORTHORESULT},
	\begin{align*}
	UV^*&= \left(\sum_{n=1}^\infty L_n^*A_n\right)\left( \sum_{k=1}^\infty\Psi_k^*L_k\right) =
	\sum_{n=1}^\infty L_nL_n^*=I_{\ell^2(\mathbb{N}) \otimes \mathcal{H}_0}. 
	\end{align*}
\end{proof}
 \begin{theorem}\label{SECONDCHAROV}
	Let $ \{F_n\}_{n}$ be an  orthonormal basis in $ \mathcal{B}(\mathcal{H},\mathcal{H}_0).$ Then 	a pair  $( \{A_n\}_{n}, $ $ \{\Psi_n\}_{n} )$ is a factorable weak OVF  in $ \mathcal{B}(\mathcal{H}, \mathcal{H}_0)$
	if and only if 
	\begin{align*}
	A_n=F_n U, \quad \Psi_n=F_nV, \quad \forall n \in \mathbb{N},
	\end{align*}  
	where $ U,V :\mathcal{H}\to \mathcal{H} $ are bounded linear operators such that  $ V^*U$ is bounded  invertible.	
\end{theorem}
\begin{proof}
	$(\Leftarrow)$ $ \sum_{n=1}^\infty L_n(F_nU)= (\sum_{n=1}^\infty L_nF_n)U,$ $ \sum_{n=1}^\infty L_n(F_nV)= (\sum_{n=1}^\infty L_nF_n)V.$ These show analysis operators for $ (\{F_nU\}_{n},\{F_nV\}_{n})$ are well-defined bounded linear operators  and the equality $$\sum_{n=1}^\infty(F_nV)^*(F_nU)=V^*U$$
	shows that it is a factorable weak  OVF.
	
	$(\Rightarrow)$ Let $( \{A_n\}_{n},  \{\Psi_n\}_{n} )$ be a factorable weak OVF. Note that the series $  \sum_{n=1}^\infty F_n^*A_n$ and $  \sum_{n=1}^\infty F_n^*\Psi_n$ converge. In fact,  for each $h \in \mathcal{H}$, 
	\begin{align*}
	\left\|\sum_{n=1}^mF_n ^*A_n h\right\|^2&=\left\langle\sum_{n=1}^mF_n^*A_nh, \sum_{k=1}^mF_k^*A_kh\right\rangle\\
	&= \sum_{n=1}^m\left\langle A_nh, F_n\left(\sum_{k=1}^mF_k^*A_kh\right) \right\rangle=\sum_{n=1}^m\|A_nh\|^2.
	\end{align*}
	which  converges to $\| \theta_Ah\|^2=\|\sum_{n=1}^\infty L_nA_nh\|^2=\sum_{n=1}^\infty\|A_nh\|^2$. Define $U\coloneqq \sum_{n=1}^\infty F_n^*A_n$ and $V\coloneqq \sum_{n=1}^\infty F_n^*\Psi_n$. Then $F_nU=A_n, F_nV=\Psi_n ,  \forall n \in \mathbb{N} $ and 
	\begin{align*}
	V^*U=\left(\sum_{n=1}^\infty\Psi_n^*F_n\right)\left(\sum_{k=1}^\infty F_k^*A_k\right)=\sum_{n=1}^\infty\Psi_n^*A_n=S_{A,\Psi}
	\end{align*}
	which is bounded invertible.	
\end{proof}	 
\begin{corollary}
	Let $ \{F_n\}_{n}$ be an  orthonormal basis in $ \mathcal{B}(\mathcal{H},\mathcal{H}_0).$ Then 	a pair  $( \{A_n\}_{n},  \{\Psi_n\}_{n} )$ is  		
	\begin{enumerate}[label=(\roman*)]
		\item a Riesz factorable weak OVF  in $ \mathcal{B}(\mathcal{H}, \mathcal{H}_0)$
		if and only if 
		\begin{align*}
		A_n=F_n U, \quad \Psi_n=F_nV, \quad \forall n \in \mathbb{N},
		\end{align*}  
		where $ U,V :\mathcal{H}\to \mathcal{H} $ are bounded linear operators such that  $ V^*U$ is bounded  invertible and $ U(V^*U)^{-1}V^* =I_{\mathcal{H}}$.	
		\item an orthonormal  factorable weak OVF  in $ \mathcal{B}(\mathcal{H}, \mathcal{H}_0)$
		if and only if 
		\begin{align*}
		A_n=F_n U, \quad \Psi_n=F_nV, \quad \forall n \in \mathbb{N},
		\end{align*}  
		where $ U,V :\mathcal{H}\to \mathcal{H} $ are bounded linear operators such that  $ V^*U$ is bounded  invertible and $ V^*U=I_\mathcal{H}= UV^*$. 	
	\end{enumerate}
\end{corollary}
\begin{proof}
	\begin{enumerate}[label=(\roman*)]
		\item $ (\Leftarrow)$ 
		\begin{align*}
		P_{A,\Psi}&=\theta_AS_{A,\Psi}^{-1}\theta_\Psi^*=\left(\sum_{n=1}^\infty L_nF_nU\right)(V^*U)^{-1}\left(\sum_{k=1}^\infty V^*F^*_kL_k^*\right)\\
		&=\theta_FU(V^*U)^{-1}V^*\theta_F^*=\theta_FI_\mathcal{H}\theta_F^* =\sum_{n=1}^\infty L_nF_n\left(\sum_{k=1}^\infty F_k^*L^*_k\right)\\
		&=\sum_{n=1}^\infty L_nL_n^*=I_{\ell^2(\mathbb{N})}\otimes I_{\mathcal{H}_0}.
		\end{align*} 
		$ (\Rightarrow)$ Let $U$ and $V$ be as in Theorem \ref{SECONDCHAROV}. Then 
		\begin{align*}
		U(V^*U)^{-1}V^*&=\left(\sum_{k=1}^\infty F_k^*A_k\right)S_{A,\Psi}^{-1}\left(\sum_{n=1}^\infty \Psi_n^*F_n\right)\\
		&=\left(\sum_{r=1}^\infty F_r^*L_r^*\right)\left(\sum_{k=1}^\infty L_kA_k\right)S_{A,\Psi}^{-1}\left(\sum_{n=1}^\infty \Psi_n^*L_n^*\right)\left(\sum_{m=1}^\infty L_mF_m\right)\\
		&=\left(\sum_{r=1}^\infty F_r^*L_r^*\right)\theta_AS_{A,\Psi}^{-1}\theta_\Psi^*\left(\sum_{m=1}^\infty L_mF_m\right)\\
		&=\left(\sum_{r=1}^\infty F_r^*L_r^*\right)P_{A,\Psi}\left(\sum_{m=1}^\infty L_mF_m\right) \\
		&=\left(\sum_{r=1}^\infty F_r^*L_r^*\right)(I_{\ell^2(\mathbb{N})}\otimes I_{\mathcal{H}_0})\left(\sum_{m=1}^\infty L_mF_m\right) \\
		&=\left(\sum_{r=1}^\infty F_r^*L_r^*\right)\left(\sum_{m=1}^\infty L_mF_m\right) 
		=\sum_{r=1}^\infty F_r^*F_r=I_{\mathcal{H}}.
		\end{align*}
		\item We use (i). 
		
		$ (\Leftarrow)$ $S_{A,\Psi}=V^*U=I_\mathcal{H}, P_{A,\Psi}=\theta_AS_{A,\Psi}^{-1}\theta_\Psi^*=\theta_A\theta_\Psi^*= \theta_FUV^*\theta_F^*=\theta_FI_\mathcal{H}\theta_F^*=I_{\ell^2(\mathbb{N})}\otimes I_{\mathcal{H}_0}.$
		
		$(\Rightarrow)$ $V^*U=S_{A,\Psi}=I_\mathcal{H}$ and using Proposition \ref{ORTHORESULT}, 
		\begin{align*}
			UV^*= \left(\sum_{n=1}^\infty F_n^*A_n\right)\left( \sum_{k=1}^\infty\Psi_k^*F_k\right) =\sum_{n=1}^\infty F_n^*F_n=I_\mathcal{H}.
		\end{align*}
	\end{enumerate}	
\end{proof}
We next derive another characterization which is free from natural numbers.
\begin{theorem}\label{OPERATORCHARACTERIZATIONHILBERT2}
	Let $\{A_n\}_{n},\{\Psi_n\}_{n}$ be in $ \mathcal{B}(\mathcal{H}, \mathcal{H}_0).$ Then  $( \{A_n\}_{n},  \{\Psi_n\}_{n} )$  is a factorable weak OVF  
	\begin{enumerate}[label=(\roman*)]
		\item   if and only if $$U:\ell^2(\mathbb{N})\otimes \mathcal{H}_0 \ni y\mapsto\sum\limits_{n=1}^\infty A_n^*L_n^*y \in \mathcal{H}, ~\text{and} ~ V:\ell^2(\mathbb{N})\otimes \mathcal{H}_0 \ni z\mapsto\sum\limits_{n=1}^\infty \Psi_n^*L^*_nz \in \mathcal{H} $$ 
		are well-defined bounded linear operators   such that  $  VU^*$ is bounded invertible.  
		\item    if and only if $$U:\ell^2(\mathbb{N})\otimes \mathcal{H}_0 \ni y\mapsto\sum\limits_{n=1}^\infty A_n^*L_n^*y \in \mathcal{H}, ~\text{and} ~ S: \mathcal{H} \ni g\mapsto \sum\limits_{n=1}^\infty L_n\Psi_ng \in \ell^2(\mathbb{N})\otimes \mathcal{H}_0 $$ 
		are well-defined bounded linear operators  such that  $  S^*U^*$  is bounded invertible.
		\item  if and only if  $$R:   \mathcal{H} \ni h\mapsto \sum\limits_{n=1}^\infty L_nA_nh \in \ell^2(\mathbb{N})\otimes \mathcal{H}_0, ~\text{and} ~ V: \ell^2(\mathbb{N})\otimes \mathcal{H}_0 \ni z\mapsto\sum\limits_{n=1}^\infty \Psi_n^*L_n^*z \in \mathcal{H} $$ 
		are well-defined bounded linear operators   such that  $  VR$ is bounded invertible.
		\item  if and only if  $$ R:   \mathcal{H} \ni h\mapsto \sum\limits_{n=1}^\infty L_nA_nh \in \ell^2(\mathbb{N})\otimes \mathcal{H}_0, ~\text{and} ~  S:   \mathcal{H} \ni g\mapsto \sum\limits_{n=1}^\infty L_n\Psi_ng \in \ell^2(\mathbb{N})\otimes \mathcal{H}_0 $$
		are well-defined bounded linear operators  such that  $ S^*R $ is bounded invertible.  
	\end{enumerate}
	
\end{theorem} 
\begin{proof}
	We prove (i) and others are similar. 
	
	 $(\Rightarrow)$ Now $U=\theta_A^*$, $V=\theta_\Psi^*$ and hence $VU^*=\theta_\Psi^*\theta_A=S_{A,\Psi}$.
	
	$(\Leftarrow)$ Now $\theta_A=U^*$, $\theta_\Psi=V^*$ and hence $S_{A,\Psi}=\theta_\Psi^*\theta_A=VU^*$.
\end{proof}
Now we try to characterize all dual OVFs. 
\begin{lemma}\label{FIRSTLEMMA}
	Let $( \{A_n\}_{n},  \{\Psi_n\}_{n} )$ be a factorable weak OVF in    $\mathcal{B}(\mathcal{H}, \mathcal{H}_0)$. Then a factorable weak  OVF  $ (\{B_n\}_{n} , \{\Phi_n\}_{n} )$  in $\mathcal{B}(\mathcal{H}, \mathcal{H}_0)$	is a dual for $( \{A_n\}_{n},  \{\Psi_n\}_{n} )$ if and only if 
	\begin{align*}
	B_n=L_n^* U, \quad \Phi_n=L_n^*V^*, \quad \forall n \in \mathbb{N}
	\end{align*}
	where $U:\mathcal{H} \rightarrow \ell^2(\mathbb{N}) \otimes \mathcal{H}_0$  is a bounded right-inverse of $\theta_\Psi^* $ and $V: \ell^2(\mathbb{N}) \otimes \mathcal{H}_0\to \mathcal{H}$ is a bounded left-inverse of $\theta_A $ such that $VU$ is bounded invertible.
\end{lemma}
\begin{proof}
	$(\Leftarrow)$  ``If" part of proof of Theorem \ref{THAFSCHAROVF}, says that $ (\{B_n\}_{n} , \{\Phi_n\}_{n} )$ is a factorable weak OVF in  $\mathcal{B}(\mathcal{H}, \mathcal{H}_0)$. We now  check for the duality of $ (\{B_n\}_{n} , \{\Phi_n\}_{n} )$. Consider $\theta_\Phi^*\theta_A=V^* \theta_A=I_\mathcal{H} $, $ \theta_\Psi^*\theta_B=\theta_\Psi^* U =I_\mathcal{H}$.\\
	$(\Rightarrow)$ Let $ (\{B_n\}_{n} , \{\Phi_n\}_{n} )$  be a dual factorable weak OVF for  $( \{A_n\}_{n},  \{\Psi_n\}_{n} )$.  Then $\theta_\Psi^*\theta_B =I_\mathcal{H}= \theta_\Phi^*\theta_A $. Define $ U\coloneqq\theta_B, V\coloneqq\theta_\Phi^*.$ Then $U:\mathcal{H} \rightarrow \ell^2(\mathbb{N}) \otimes \mathcal{H}_0$ is a bounded  right-inverse of $\theta_\Psi^* $ and $V: \ell^2(\mathbb{N}) \otimes \mathcal{H}_0\to \mathcal{H}$ is a left inverse of $\theta_A $ such that $VU=\theta_\Phi^*\theta_B=S_{B,\Phi}$ is bounded invertible. We now see 
	\begin{align*}
	L_n^* U=L_n^*\left(\sum\limits_{k=1}^\infty L_kB_k\right)=B_n, \quad L_n^*V^*=L_n^*\left(\sum\limits_{k=1}^\infty L_k\Phi_k\right)=\Phi_n, \quad \forall n \in \mathbb{N}.
	\end{align*}
\end{proof}
\begin{lemma}\label{SECONDLEMMA}
	Let $( \{A_n\}_{n},  \{\Psi_n\}_{n} )$ be a  factorable weak OVF in    $\mathcal{B}(\mathcal{H}, \mathcal{H}_0)$. Then  
	\begin{enumerate}[label=(\roman*)]
		\item $R:\mathcal{H} \to \ell^2(\mathbb{N})\otimes\mathcal{H}_0 $ is a bounded right-inverse of $ \theta_\Psi^*$ if and only if 
		\begin{align*}
		R=\theta_AS_{A,\Psi}^{-1}+(I_{\ell^2(\mathbb{N})\otimes\mathcal{H}_0}-\theta_AS_{A,\Psi}^{-1}\theta_\Psi^*)U,
		\end{align*}
		where $U :\mathcal{H} \to \ell^2(\mathbb{N})\otimes\mathcal{H}_0$ is a bounded linear operator.
		\item $L:\ell^2(\mathbb{N})\otimes\mathcal{H}_0\rightarrow \mathcal{H} $ is a bounded left-inverse of $ \theta_A$ if and only if 
		\begin{align*}
		L=S_{A,\Psi}^{-1}\theta_\Psi^*+V(I_{\ell^2(\mathbb{N})\otimes\mathcal{H}_0}-\theta_A S_{A,\Psi}^{-1}\theta_\Psi^*),
		\end{align*}  
		where $V:\ell^2(\mathbb{N})\otimes\mathcal{H}_0\to\mathcal{H}$ is a bounded linear operator.
	\end{enumerate}		
\end{lemma}
\begin{proof}
	\begin{enumerate}[label=(\roman*)]
		\item $(\Leftarrow)$ Let $U :\mathcal{H} \to \ell^2(\mathbb{N})\otimes\mathcal{H}_0$ be a bounded linear operator. Then 
		\begin{align*}
		\theta_\Psi^*(\theta_AS_{A,\Psi}^{-1}+(I_{\ell^2(\mathbb{N})\otimes\mathcal{H}_0}-\theta_AS_{A,\Psi}^{-1}\theta_\Psi^*)U)=I_\mathcal{H}+\theta_\Psi^*U-\theta_\Psi^*U=I_\mathcal{H}.
		\end{align*}  Therefore  $\theta_AS_{A,\Psi}^{-1}+(I_{\ell^2(\mathbb{N})\otimes\mathcal{H}_0}-\theta_AS_{A,\Psi}^{-1}\theta_\Psi^*)U$ is a bounded right-inverse of $\theta_\Psi^*$.
		
		$(\Rightarrow)$ Let $R:\mathcal{H} \to \ell^2(\mathbb{N})\otimes\mathcal{H}_0 $ be a bounded right-inverse of $ \theta_\Psi^*$. Define $U\coloneqq R$. Then $	\theta_AS_{A,\Psi}^{-1}+(I_{\ell^2(\mathbb{N})\otimes\mathcal{H}_0}-\theta_AS_{A,\Psi}^{-1}\theta_\Psi^*)U=	\theta_AS_{A,\Psi}^{-1}+(I_{\ell^2(\mathbb{N})\otimes\mathcal{H}_0}-\theta_AS_{A,\Psi}^{-1}\theta_\Psi^*)R=\theta_AS_{A,\Psi}^{-1}+R-\theta_AS_{A,\Psi}^{-1}=R$.
		\item $(\Leftarrow)$ Let $V: \ell^2(\mathbb{N})\otimes\mathcal{H}_0\rightarrow \mathcal{H}$ be a bounded linear operator. Then
		\begin{align*}
		(S_{A,\Psi}^{-1}\theta_\Psi^*+V(I_{\ell^2(\mathbb{N})\otimes\mathcal{H}_0}-\theta_A S_{A,\Psi}^{-1}\theta_\Psi^*))\theta_A=I_\mathcal{H}+V\theta_A-V\theta_A I_\mathcal{H}=I_\mathcal{H}.
		\end{align*}  Therefore  $S_{A,\Psi}^{-1}\theta_\Psi^*+V(I_{\ell^2(\mathbb{N})\otimes\mathcal{H}_0}-\theta_A S_{A,\Psi}^{-1}\theta_\Psi^*)$ is a bounded left-inverse of $\theta_A$.
		
		$(\Rightarrow)$ Let $ L:\ell^2(\mathbb{N})\otimes\mathcal{H}_0\rightarrow \mathcal{H}$ be a bounded left-inverse of $ \theta_A$. Define $V\coloneqq L$. Then $S_{A,\Psi}^{-1}\theta_\Psi^*+V(I_{\ell^2(\mathbb{N})\otimes\mathcal{H}_0}-\theta_A S_{A,\Psi}^{-1}\theta_\Psi^*) =S_{A,\Psi}^{-1}\theta_\Psi^*+L(I_{\ell^2(\mathbb{N})\otimes\mathcal{H}_0}-\theta_A S_{A,\Psi}^{-1}\theta_\Psi^*)=S_{A,\Psi}^{-1}\theta_\Psi^*+L-I_{\mathcal{H}}S_{A,\Psi}^{-1}\theta_\Psi^*= L$. 	
	\end{enumerate}	
\end{proof}
\begin{theorem}
	Let $( \{A_n\}_{n},  \{\Psi_n\}_{n} )$ be a  factorable  weak OVF in $\mathcal{B}(\mathcal{H}, \mathcal{H}_0)$. Then a  factorable weak OVF  $ (\{B_n\}_{n} , \{\Phi_n\}_{n} )$   in $\mathcal{B}(\mathcal{H}, \mathcal{H}_0)$ is a dual  for $( \{A_n\}_{n},  \{\Psi_n\}_{n} )$ if and only if
	\begin{align*}
	&B_n=A_nS_{A,\Psi}^{-1}+L_n^*U-A_nS_{A,\Psi}^{-1}\theta_\Psi^*U,\\
	&\Phi_n=\Psi_n(S_{A,\Psi}^{-1})^*+L_n^*V^*-\Psi_n (S_{A,\Psi}^{-1})^*\theta_A^*V^*, \quad \forall n \in \mathbb{N}
	\end{align*}
	such that the operator 
	\begin{align*}
	S_{A, \Psi}^{-1}+VU-V\theta_AS_{A, \Psi}^{-1}\theta_\Psi^* U
	\end{align*}
	is bounded invertible, where   $U :\mathcal{H} \to \ell^2(\mathbb{N})\otimes\mathcal{H}_0$  and $V:\ell^2(\mathbb{N})\otimes\mathcal{H}_0\to\mathcal{H}$ are bounded linear operators.	
\end{theorem}
\begin{proof}
	Lemmas \ref{FIRSTLEMMA} and  \ref{SECONDLEMMA}
	give the characterization of dual weak OVF as 
	\begin{align*}
	&B_n=L_n^*(\theta_AS_{A,\Psi}^{-1}+(I_{\ell^2(\mathbb{N})\otimes\mathcal{H}_0}-\theta_AS_{A,\Psi}^{-1}\theta_\Psi^*)U)\\
	& \quad=A_nS_{A,\Psi}^{-1}+L_n^*U-A_nS_{A,\Psi}^{-1}\theta_\Psi^*U,\\
	&\Phi_n=L_n^*(\theta_\Psi(S_{A,\Psi}^{-1})^*+(I_{\ell^2(\mathbb{N})\otimes\mathcal{H}_0}-\theta_\Psi (S_{A,\Psi}^{-1})^*\theta_A^*)V^*)\\
	&\quad=\Psi_n(S_{A,\Psi}^{-1})^*+L_n^*V^*-\Psi_n (S_{A,\Psi}^{-1})^*\theta_A^*V^*, \quad \forall n \in \mathbb{N}
	\end{align*}
	such that the operator 
	$$(S_{A,\Psi}^{-1}\theta_\Psi^*+V(I_{\ell^2(\mathbb{N})\otimes\mathcal{H}_0}-\theta_A S_{A,\Psi}^{-1}\theta_\Psi^*))(\theta_AS_{A,\Psi}^{-1}+(I_{\ell^2(\mathbb{N})\otimes\mathcal{H}_0}-\theta_AS_{A,\Psi}^{-1}\theta_\Psi^*)U) $$
	is bounded invertible, where $U :\mathcal{H} \to \ell^2(\mathbb{N})\otimes\mathcal{H}_0$ and $V:\ell^2(\mathbb{N})\otimes\mathcal{H}_0\to\mathcal{H}$  are bounded linear operators. We expand and  get 
	\begin{align*}
	&(S_{A,\Psi}^{-1}\theta_\Psi^*+V(I_{\ell^2(\mathbb{N})\otimes\mathcal{H}_0}-\theta_A S_{A,\Psi}^{-1}\theta_\Psi^*))(\theta_AS_{A,\Psi}^{-1}+(I_{\ell^2(\mathbb{N})\otimes\mathcal{H}_0}-\theta_AS_{A,\Psi}^{-1}\theta_\Psi^*)U)\\
	&=S_{A, \Psi}^{-1}+VU-V\theta_AS_{A, \Psi}^{-1}\theta_\Psi^* U.
	\end{align*}
\end{proof}
 We now define the orthogonality for weak OVFs.
\begin{definition}
	A weak  OVF  $ (\{B_n\}_{n} , \{\Phi_n\}_{n} )$  in $\mathcal{B}(\mathcal{H}, \mathcal{H}_0)$ is said to be \textbf{orthogonal} to a weak  OVF  $  ( \{A_n\}_{n},  \{\Psi_n\}_{n} ) $ in $\mathcal{B}(\mathcal{H}, \mathcal{H}_0)$ if 
	\begin{align*}
	\sum_{n=1}^\infty\Psi_n^*B_n= \sum_{n=1}^\infty\Phi^*_nA_n=0.
	\end{align*}
\end{definition}  
Remarkable property of orthogonal frames is that we can interpolate as well as  we can take direct sum of them to get new frames. These are illustrated in the following two results.
\begin{proposition}
	Let $  ( \{A_n\}_{n},  \{\Psi_n\}_{n} ) $ and $ (\{B_n\}_{n} , \{\Phi_n\}_{n} )$ be  two Parseval  OVFs in   $\mathcal{B}(\mathcal{H}, \mathcal{H}_0)$ which are  orthogonal. If $C,D,E,F \in \mathcal{B}(\mathcal{H})$ are such that $ C^*E+D^*F=I_\mathcal{H}$, then  
	\begin{align*}
	(\{A_nC+B_nD\}_{n}, \{\Psi_nE+\Phi_nF\}_{n})
	\end{align*}
	is a  Parseval weak  OVF in  $\mathcal{B}(\mathcal{H}, \mathcal{H}_0)$. In particular,  if scalars $ c,d,e,f$ satisfy $\bar{c}e+\bar{d}f =1$, then $ (\{cA_n+dB_n\}_{n}, \{e\Psi_n+f\Phi_n\}_{n}) $ is   a Parseval weak  OVF.
\end{proposition} 
\begin{proof}
	We use the definition of frame operator and get 
	\begin{align*}
	S_{AC+BD,\Psi E+\Phi F} &=\sum_{n=1}^\infty(\Psi_nE+\Phi_nF)^*(A_nC+B_nD)\\
	&=E^*S_{A,\Psi}C+E^*\left(\sum_{n=1}^\infty\Psi_n^*B_n\right)D+F^*\left(\sum_{n=1}^\infty\Phi_n^*A_n\right)C+F^*S_{B,\Phi}D\\
	&=E^*I_\mathcal{H}C+E^*0D+F^*0C+F^*I_\mathcal{H}D=I_\mathcal{H}.
	\end{align*}
\end{proof}
\begin{proposition}
	If $  ( \{A_n\}_{n},  \{\Psi_n\}_{n} ) $  and $ (\{B_n\}_{n} , \{\Phi_n\}_{n} )$ are   orthogonal weak OVFs in $ \mathcal{B}(\mathcal{H}, \mathcal{H}_0)$, then  $(\{A_n\oplus B_n\}_{n},\{\Psi_n\oplus \Phi_n\}_{n})$ is a  weak  OVF in $ \mathcal{B}(\mathcal{H}\oplus \mathcal{H}, \mathcal{H}_0).$    Further, if both $  ( \{A_n\}_{n},  \{\Psi_n\}_{n} ) $  and $ (\{B_n\}_{n} , \{\Phi_n\}_{n} )$ are  Parseval, then $(\{A_n\oplus B_n\}_{n},\{\Psi_n\oplus \Phi_n\}_{n})$ is Parseval.
\end{proposition}
\begin{proof}
	Let $ h \oplus g \in \mathcal{H}\oplus \mathcal{H}$. Then 
	\begin{align*}
	S_{A\oplus B, \Psi\oplus \Phi}(h\oplus g)&=\sum_{n=1}^\infty(\Psi_n\oplus \Phi_n)^*(A_n\oplus B_n)(h\oplus g)=\sum_{n=1}^\infty(\Psi_n\oplus \Phi_n)^*(A_nh+ B_ng)\\
	&=\sum_{n=1}^\infty(\Psi_n^*(A_nh+B_ng)\oplus \Phi_n^*(A_nh+B_ng))
	\\
	&=\left(\sum_{n=1}^\infty\Psi_n^*A_nh+\sum_{n=1}^\infty\Psi_n^*B_ng\right)\oplus \left(\sum_{n=1}^\infty\Phi_n^*A_nh+\sum_{n=1}^\infty\Phi_n^*B_ng\right)\\
	&=(S_{A,\Psi}h+0)\oplus(0+S_{B,\Phi}g) =(S_{A,\Psi}\oplus S_{B,\Phi})(h\oplus g).
	\end{align*}	
\end{proof}  
\section{EQUIVALENCE OF WEAK OPERATOR-VALUED FRAMES}\label{SIMILARITYCOMPOSITIONANDTENSORPRODUCT}
Definition \ref{SIMILARITYOVFKAFTALLARSONZHANG} introduced similarity for OVFs. Here is the  similar notion for factorable weak OVFs.
\begin{definition}
	A factorable weak OVF   $( \{B_n\}_{n},  \{\Phi_n\}_{n} ) $  in $ \mathcal{B}(\mathcal{H}, \mathcal{H}_0)$    is said to be \textbf{similar}  to a factorable  weak OVF  $  ( \{A_n\}_{n},  \{\Psi_n\}_{n} ) $ in $ \mathcal{B}(\mathcal{H}, \mathcal{H}_0)$  if there exist bounded  invertible  $ R_{A,B}, R_{\Psi, \Phi} \in \mathcal{B}(\mathcal{H})$   such that 
	\begin{align*}
	B_n=A_nR_{A,B} , \quad \Phi_n=\Psi_nR_{\Psi, \Phi}, \quad \forall n \in \mathbb{N}.
	\end{align*} 
\end{definition}
Since  $ R_{A,B} $ and $ R_{\Psi, \Phi}$ are bounded invertible, it easily follows that the notion similarity is symmetric. We further have that the relation ``similarity" is an equivalence relation on the set 
\begin{align*}
\left\{( \{A_n\}_{n},  \{\Psi_n\}_{n} ):( \{A_n\}_{n},  \{\Psi_n\}_{n} ) \text{  is a factorable weak OVF}\right\}.
\end{align*} 
Similar frames have nice property that knowing analysis, synthesis and frame operators of one give that of another.
\begin{lemma}\label{SIM}
	Let $  ( \{A_n\}_{n},  \{\Psi_n\}_{n} ) $ and  $  ( \{B_n\}_{n},  \{\Phi_n\}_{n} ) $ be similar factorable weak OVFs  and   $B_n=A_nR_{A,B} ,\Phi_n=\Psi_nR_{\Psi, \Phi},  \forall n \in \mathbb{N}$, for some invertible $ R_{A,B} ,R_{\Psi, \Phi} \in \mathcal{B}(\mathcal{H}).$ Then 
	\begin{enumerate}[label=(\roman*)]
		\item $ \theta_B=\theta_A R_{A,B}, \theta_\Phi=\theta_\Psi R_{\Psi,\Phi}$.
		\item $S_{B,\Phi}=R_{\Psi,\Phi}^*S_{A, \Psi}R_{A,B}$.
		\item $P_{B,\Phi}=P_{A, \Psi}.$
	\end{enumerate}
\end{lemma}
\begin{proof}
	$ \theta_B=\sum_{n=1}^\infty L_nB_n=\sum_{n=1}^\infty L_nA_nR_{A,B}=\theta_AR_{A,B} $. Similarly $ \theta_\Phi=\theta_\Psi R_{\Psi,\Phi}$. Now using operators $\theta_B$ and $\theta_\Phi$ we get $S_{B,\Phi}=\sum_{n=1}^\infty \Phi_n^*B_n=\sum_{n=1}^\infty(\Psi_nR_{\Psi,\Phi})^*(A_nR_{A,B})=R_{\Psi, \Phi}^*\left (\sum_{n=1}^\infty\Psi_n^*A_n\right )R_{A,B}=R_{\Psi,\Phi}^*S_{A, \Psi}R_{A,B}$. We now use (i) and (ii) to get 
	\begin{align*}
	P_{B,\Phi}=\theta_BS_{B,\Phi}^{-1}\theta_\Phi^*=(\theta_AR_{A,B})(R_{\Psi,\Phi}^*S_{A, \Psi}R_{A,B})^{-1}(\theta_\Psi R_{\Psi,\Phi})^*=P_{A,\Psi}.
	\end{align*}
\end{proof}
We now classify similarity using operators.
\begin{theorem}\label{RIGHTSIMILARITY}
	For two factorable weak OVFs  $  ( \{A_n\}_{n},  \{\Psi_n\}_{n} ) $ and  $  ( \{B_n\}_{n},  \{\Phi_n\}_{n} ) $, the following are equivalent.
	\begin{enumerate}[label=(\roman*)]
		\item $B_n=A_nR_{A,B} , \Phi_n=\Psi_nR_{\Psi, \Phi} ,  \forall n \in \mathbb{N},$ for some invertible  $ R_{A,B} ,R_{\Psi, \Phi} \in \mathcal{B}(\mathcal{H}). $
		\item $\theta_B=\theta_AR_{A,B} , \theta_\Phi=\theta_\Psi R_{\Psi, \Phi} $ for some invertible  $ R_{A,B} ,R_{\Psi, \Phi} \in \mathcal{B}(\mathcal{H}). $
		\item $P_{B,\Phi}=P_{A,\Psi}.$
	\end{enumerate}
	If one of the above conditions is satisfied, then  invertible operators in  $ \operatorname{(i)}$ and  $ \operatorname{(ii)}$ are unique and are given by $R_{A,B}=S_{A,\Psi}^{-1}\theta_\Psi^*\theta_B$,   $R_{\Psi, \Phi}=(S_{A,\Psi}^{-1})^*\theta_A^*\theta_\Phi.$ In the case that $  ( \{A_n\}_{n},  \{\Psi_n\}_{n} ) $ is Parseval, then $  ( \{B_n\}_{n},  \{\Phi_n\}_{n} ) $ is  Parseval if and only if $R_{\Psi, \Phi}^*R_{A,B}=I_\mathcal{H} $   if and only if $R_{A,B}R_{\Psi, \Phi}^*=I_\mathcal{H} $.
\end{theorem}
\begin{proof}
	The implications (i) $\Rightarrow$ (ii) $\Rightarrow$ (iii) follow from Lemma \ref{SIM}.  Assume (ii) holds. We show (i) holds. Using Equation (\ref{LEQUATION}), $ B_n=L_n^*\theta_B=L_n^*\theta_AR_{A,B}'=A_nR_{A,B}'$; the same procedure gives $ \Phi_n$ also.   Assume (iii). We note the following $ \theta_B=P_{B,\Phi}\theta_B$ and $ \theta_\Phi=P_{B,\Phi}^*\theta_\Phi.$ Using these, $ \theta_B=P_{A,\Psi}\theta_B=\theta_A(S_{A,\Psi}^{-1}\theta_\Psi^*\theta_B)$ and $ \theta_\Phi=P_{A,\Psi}^*\theta_\Phi=(\theta_AS_{A,\Psi}^{-1}\theta_\Psi^*)^*\theta_\Phi=\theta_\Psi((S_{A,\Psi}^{-1})^*\theta_A^*\theta_\Phi).$ We now try to show that both $S_{A,\Psi}^{-1}\theta_\Psi^*\theta_B$  and $(S_{A,\Psi}^{-1})^*\theta_A^*\theta_\Phi$ are invertible. This is achieved via, 
	\begin{align*}
	(S_{A,\Psi}^{-1}\theta_\Psi^*\theta_B)(S_{B,\Phi}^{-1}\theta_\Phi^*\theta_A)=S_{A,\Psi}^{-1}\theta_\Psi^*P_{B,\Phi}\theta_A= S_{A,\Psi}^{-1}\theta_\Psi^*P_{A,\Psi}\theta_A= S_{A,\Psi}^{-1}\theta_\Psi^*\theta_A=I_\mathcal{H},\\
	( S_{B,\Phi}^{-1}\theta_\Phi^*\theta_A)(S_{A,\Psi}^{-1}\theta_\Psi^*\theta_B)= S_{B,\Phi}^{-1}\theta_\Phi^*P_{A,\Psi}\theta_B=S_{B,\Phi}^{-1}\theta_\Phi^*P_{B,\Phi}\theta_B=S_{B,\Phi}^{-1}\theta_\Phi^*\theta_B=I_\mathcal{H}
	\end{align*}
	and 
	\begin{align*}
	((S_{A,\Psi}^{-1})^*\theta_A^*\theta_\Phi)((S_{B,\Phi}^{-1})^*\theta_B^*\theta_\Psi)&=(S_{A,\Psi}^{-1})^*\theta_A^*P_{B,\Phi}^*\theta_\Psi
	=(S_{A,\Psi}^{-1})^*\theta_A^*P_{A,\Psi}^*\theta_\Psi\\
	&=(S_{A,\Psi}^{-1})^*\theta_A^*\theta_\Psi=I_\mathcal{H},\\
	((S_{B,\Phi}^{-1})^*\theta_B^*\theta_\Psi)((S_{A,\Psi}^{-1})^*\theta_A^*\theta_\Phi)&=(S_{B,\Phi}^{-1})^*\theta_B^*P_{A,\Psi}^* \theta_\Phi=(S_{B,\Phi}^{-1})^*\theta_B^*P_{B,\Phi}^* \theta_\Phi\\
	&=(S_{B,\Phi}^{-1})^*\theta_B^* \theta_\Phi= I_\mathcal{H}.
	\end{align*} 
	Let $ R_{A,B}, R_{\Psi,\Phi} \in \mathcal{B}(\mathcal{H}) $ be invertible. From the previous arguments, $ R_{A,B}$ and $R_{\Psi,\Phi} $ satisfy (i) if and only if  they satisfy (ii). Let $B_n=A_nR_{A,B} , \Phi_n=\Psi_nR_{\Psi, \Phi} ,  \forall n \in \mathbb{N}.$ Using (ii), $\theta_B=\theta_AR_{A,B} , \theta_\Phi=\theta_\Psi R_{\Psi, \Phi}$  $\implies$  $\theta_\Psi^*\theta_B=\theta_\Psi^*\theta_AR_{A,B}=S_{A,\Psi}R_{A,B} , \theta_A^*\theta_\Phi=\theta_A^*\theta_\Psi R_{\Psi, \Phi}=S_{A,\Psi}^*R_{\Psi, \Phi}$. These imply the formula for $R_{A,B}$ and $ R_{\Psi, \Phi}.$ For the last,  we recall $ S_{B,\Phi}=R_{\Psi,\Phi}^*S_{A, \Psi}R_{A,B}$.  
\end{proof}
\begin{corollary}
	For any given factorable weak OVF $  ( \{A_n\}_{n},  \{\Psi_n\}_{n} ) $, the canonical dual of $  ( \{A_n\}_{n},  \{\Psi_n\}_{n} ) $ is the only dual factorable weak OVF  that is similar to $  ( \{A_n\}_{n}, $ $ \{\Psi_n\}_{n} ) $.
\end{corollary}
\begin{proof}
	Let  $ (\{B_n\}_{n} , \{\Phi_n\}_{n} )$ be a factorable  weak OVF which is both dual and similar for  $  ( \{A_n\}_{n},  \{\Psi_n\}_{n} ) $.  Then  we have $ \theta_B^*\theta_\Psi=I_\mathcal{H}=\theta_\Phi^*\theta_A$ and  there exist invertible $ R_{A,B},R_{\Psi,\Phi}\in \mathcal{B}(\mathcal{H})$ such that  $B_n=A_nR_{A,B} , \Phi_n=\Psi_nR_{\Psi, \Phi} ,  \forall n \in \mathbb{N} $. Theorem \ref{RIGHTSIMILARITY} gives  $R_{A,B}=S_{A,\Psi}^{-1}\theta_\Psi^*\theta_B, R_{\Psi, \Phi}=S_{A,\Psi}^{-1}\theta_A^*\theta_\Phi.$ But then $R_{A,B}=S_{A,\Psi}^{-1}I_\mathcal{H}=S_{A,\Psi}^{-1}$,  $ R_{\Psi, \Phi}=(S_{A,\Psi}^{-1})^*I_\mathcal{H}=(S_{A,\Psi}^{-1})^*.$ Therefore  $ (\{B_n\}_{n} , \{\Phi_n\}_{n} )$ is the  canonical  dual for  $  ( \{A_n\}_{n}, $ $ \{\Psi_n\}_{n} ) $.
\end{proof}
\begin{corollary}
	Two similar factorable weak OVF   cannot be orthogonal.
\end{corollary}
\begin{proof}
	Let a factorable weak OVF $  ( \{B_n\}_{n},  \{\Phi_n\}_{n} ) $ be similar to $  ( \{A_n\}_{n},  \{\Psi_n\}_{n} ) $. Choose invertible $ R_{A,B},R_{\Psi,\Phi}\in \mathcal{B}(\mathcal{H})$ such that  $B_n=A_nR_{A,B} , \Phi_n=\Psi_nR_{\Psi, \Phi} ,  \forall n \in \mathbb{N} $. Using 	Theorem \ref{RIGHTSIMILARITY} and the invertibility of $R_{A,B}^* $ and $S_{A,\Psi}^* $, we get 
	\begin{align*}
	\theta_B^*\theta_\Psi=(\theta_AR_{A,B})^*\theta_\Psi=R_{A,B}^*\theta_A^*\theta_\Psi=R_{A,B}^*S_{A,\Psi}^*\neq 0.
	\end{align*}
\end{proof}
For every factorable weak OVF $  ( \{A_n\}_{n},  \{\Psi_n\}_{n} ) $, each  of `OVFs'  $( \{A_nS_{A, \Psi}^{-1}\}_{n}, \{\Psi_n\}_{n})$    and  $ (\{A_n \}_{n}, \{\Psi_n(S_{A,\Psi}^{-1})^*\}_{n})$ is a Parseval OVF which is similar to  $  ( \{A_n\}_{n},  \{\Psi_n\}_{n} ) $.  Thus every OVF is similar to  Parseval OVFs.

\section{WEAK OPERATOR-VALUED FRAMES GENERATED BY GROUPS AND GROUP LIKE UNITARY SYSTEMS} \label{FRAMESANDDISCRETEGROUPREPRESENTATIONS}
In this section $G$ denotes discrete group and $\pi$ denotes unitary representation of $G$. Identity element of $G$ is denoted by $e$.
\begin{definition}
	Let $ \pi$ be a unitary representation of a discrete 
	group $ G$ on  a Hilbert space $ \mathcal{H}.$ An operator $ A$ in $ \mathcal{B}(\mathcal{H}, \mathcal{H}_0)$ is called a \textbf{factorable   operator frame generator} (resp. a  Parseval frame generator) w.r.t. an operator $ \Psi$ in $ \mathcal{B}(\mathcal{H}, \mathcal{H}_0)$ if $(\{A_g\coloneqq A \pi_{g^{-1}}\}_{g\in G}, \{\Psi_g\coloneqq \Psi \pi_{g^{-1}}\}_{g\in G})$ is a factorable weak OVF (resp.  Parseval)  in $ \mathcal{B}(\mathcal{H}, \mathcal{H}_0)$. In this case, we write $ (A,\Psi)$ is an operator  frame generator for $\pi$.
\end{definition}
\begin{proposition}\label{REPRESENATIONLEMMA}
	Let $ (A,\Psi)$ and $ (B,\Phi)$ be   operator frame generators    in $\mathcal{B}(\mathcal{H},  \mathcal{H}_0)$ for a unitary representation $ \pi$ of  $G$ on $ \mathcal{H}.$ Then
	\begin{enumerate}[label=(\roman*)]
		\item $ \theta_A\pi_g=(\lambda_g\otimes I_{\mathcal{H}_0})\theta_A,  \theta_\Psi \pi_g=(\lambda_g\otimes I_{\mathcal{H}_0})\theta_\Psi,  \forall g \in G.$
		\item $ \theta_A^*\theta_B,   \theta_\Psi^*\theta_\Phi,\theta_A^*\theta_\Phi$ are in the commutant $ \pi(G)'$ of $ \pi(G)''.$ Further, $ S_{A,\Psi} \in \pi(G)'$. 
		\item $ \theta_AT\theta_\Psi^*, \theta_AT\theta_B^*, \theta_\Psi T\theta_\Phi^* \in \mathscr{R}(G)\otimes \mathcal{B}(\mathcal{H}_0), \forall T \in \pi(G)'.$ In particular, $ P_{A, \Psi} \in \mathscr{R}(G)$ $\otimes \mathcal{B}(\mathcal{H}_0). $
	\end{enumerate}
\end{proposition}
\begin{proof} Let $ g,p,q \in G $ and $ h \in \mathcal{H}_0.$
	\begin{enumerate}[label=(\roman*)]
		\item  From the definition of $ \lambda_g $ and $ \chi_q$,  we get $ \lambda_g\chi_q=\chi_{gq}.$ Therefore $ L_{gq}h=\chi_{gq}\otimes h= \lambda_g\chi_q\otimes h= (\lambda_g\otimes I_{\mathcal{H}_0})(\chi_q\otimes h)=(\lambda_g\otimes I_{\mathcal{H}_0})L_qh.$  Using this, 
		\begin{align*}
		\theta_A\pi_g&=\sum\limits_{p\in G} L_pA_p\pi_g=\sum\limits_{p\in G} L_pA\pi_{p^{-1}}\pi_g=\sum\limits_{p\in G} L_pA\pi_{{p^{-1}}g}\\
		&=\sum\limits_{q\in G} L_{gq}A\pi_{q^{-1}}=\sum\limits_{q\in G}(\lambda_g\otimes I_{\mathcal{H}_0}) L_{q}A\pi_{q^{-1}}=(\lambda_g\otimes I_{\mathcal{H}_0})\theta_A. 
		\end{align*}
		Similarly $ \theta_\Psi \pi_g=(\lambda_g\otimes I_{\mathcal{H}_0})\theta_\Psi.$
		\item $ \theta_A^*\theta_B\pi_g=\theta_A^* (\lambda_g\otimes I_{\mathcal{H}_0})\theta_B=((\lambda_{g^{-1}}\otimes I_{\mathcal{H}_0})\theta_A)^*\theta_B=(\theta_A\pi_{g^{-1}})^*\theta_B=\pi_g\theta_A^*\theta_B.$ In the same way, $ \theta_\Psi^*\theta_\Phi, \theta_A^*\theta_\Phi\in \pi(G)'.$ By taking $ B=A$ and $ \Phi=\Psi$ we get  $ S_{A,\Psi} \in \pi(G)'.$ 
		\item Let  $ T \in \pi(G)'.$ Then 
	\begin{align*}
	\theta_AT\theta_\Psi^*(\lambda_g\otimes I_{\mathcal{H}_0})&= \theta_AT((\lambda_{g^{-1}}\otimes I_{\mathcal{H}_0})\theta_\Psi)^*=\theta_AT\pi_g\theta_\Psi^*\\
	&=\theta_A\pi_gT\theta_\Psi^*=(\lambda_g\otimes I_{\mathcal{H}_0})\theta_AT\theta_\Psi^*.
	\end{align*}
		From the construction of $ \mathscr{L}(G),$ we now get $\theta_AT\theta_\Psi^* \in (\mathscr{L}(G)\otimes \{I_{\mathcal{H}_0}\})'=\mathscr{L}(G)'\otimes \{I_{\mathcal{H}_0}\}'=\mathscr{R}(G)\otimes \mathcal{B}(\mathcal{H}_0).$ Similarly $\theta_AT\theta_B^*, \theta_\Psi S\theta_\Phi^* \in \mathscr{R}(G)\otimes \mathcal{B}(\mathcal{H}_0), \forall  S \in \pi (G)'.$ For the choice  $ T=S_{A,\Psi}^{-1}$ we get $ P_{A, \Psi} \in \mathscr{R}(G)\otimes \mathcal{B}(\mathcal{H}_0). $
	\end{enumerate}
\end{proof}
\begin{theorem}\label{gc1}
	Let $ G$ be a discrete group and $( \{A_g\}_{g\in G}, $ $ \{\Psi_g\}_{g\in G})$ be a Parseval  factorable weak OVF  in $ \mathcal{B}(\mathcal{H},\mathcal{H}_0).$ Then there is a  unitary representation $ \pi$  of $ G$ on  $ \mathcal{H}$  for which 
	$$ A_g=A_e\pi_{g^{-1}}, \quad\Psi_g=\Psi_e\pi_{g^{-1}}, \quad\forall  g \in G$$
	if and only if 
	$$A_{gp}A_{gq}^*=A_pA_q^* ,\quad  A_{gp}\Psi_{gq}^*=A_p\Psi_q^*,\quad  \Psi_{gp}\Psi_{gq}^*=\Psi_p\Psi_q^*, \quad  \forall g,p,q \in G.$$
\end{theorem} 
\begin{proof}
	$(\Rightarrow)$
	$$A_{gp}\Psi_{gq}^*= A_e \pi_{(gp)^{-1}}(\Psi_e\pi_{(gq)^{-1}})^*=A_e\pi_{p^{-1}}\pi_{g^{-1}}\pi_g\pi_q\Psi_e^*=A_p\Psi_q^*, \quad\forall g,p,q \in G.$$
	Similarly we get other two equalities.
	
	$(\Leftarrow)$ Using  assumptions, we use the following three equalities in the proof, among them  we derive the second, remainings are similar.
	For all $ g \in G,$
	\begin{align*}
	&	(\lambda_g\otimes I_{\mathcal{H}_0})\theta_A\theta_A^*=\theta_A\theta_A^*(\lambda_g\otimes I_{\mathcal{H}_0}), ~ (\lambda_g\otimes I_{\mathcal{H}_0})\theta_A\theta_\Psi^*=\theta_A\theta_\Psi^*(\lambda_g\otimes I_{\mathcal{H}_0}),\\
	&	(\lambda_g\otimes I_{\mathcal{H}_0})\theta_\Psi\theta_\Psi^*=\theta_\Psi\theta_\Psi^*(\lambda_g\otimes I_{\mathcal{H}_0}).
	\end{align*}
	Noticing $ \lambda_g$ is unitary, we get  $(\lambda_g\otimes I_{\mathcal{H}_0})^{-1}=(\lambda_g\otimes I_{\mathcal{H}_0})^*$; also we observed in the proof of Proposition \ref{REPRESENATIONLEMMA} that  $(\lambda_g\otimes I_{\mathcal{H}_0})L_q=L_{gq}.$ So
	\begin{align*}
	(\lambda_g\otimes I_{\mathcal{H}_0})\theta_A\theta_\Psi^*(\lambda_g\otimes I_{\mathcal{H}_0})^*&=\left(\sum\limits_{p\in G}(\lambda_g\otimes I_{\mathcal{H}_0})L_pA_p\right)\left(\sum\limits_{q\in G}(\lambda_g\otimes I_{\mathcal{H}_0})L_q\Psi_q\right)^*\\
	&=\sum\limits_{p\in G} L_{gp}\left(\sum\limits_{q\in G}A_p\Psi_q^*L_{gq}^*\right)
	=\sum\limits_{r\in G} L_r\left(\sum\limits_{s\in G}A_{g^{-1}r}\Psi_{g^{-1}s}^*L_s^*\right)\\
	& =\sum\limits_{r\in G} L_r\left(\sum\limits_{s\in G}A_r\Psi_s^*L_s^*\right)=\theta_A\theta_\Psi^*.
	\end{align*}
	Define $ \pi : G \ni g  \mapsto \pi_g\coloneqq \theta_\Psi^*(\lambda_g\otimes I_{\mathcal{H}_0})\theta_A  \in \mathcal{B}(\mathcal{H}).$ By using the  Parsevalness, 
	\begin{align*}
	\pi_g\pi_h&=\theta_\Psi^*(\lambda_g\otimes I_{\mathcal{H}_0})\theta_A \theta_\Psi^*(\lambda_h\otimes I_{\mathcal{H}_0})\theta_A =\theta_\Psi^*\theta_A \theta_\Psi^*(\lambda_g\otimes I_{\mathcal{H}_0}) (\lambda_h\otimes I_{\mathcal{H}_0})\theta_A \\
	&= \theta_\Psi^*(\lambda_{gh}\otimes I_{\mathcal{H}_0})\theta_A =\pi_{gh}, \quad \forall g, h \in G
	\end{align*}
	and 
	\begin{align*}
	\pi_g\pi_g^*&=\theta_\Psi^*(\lambda_g\otimes I_{\mathcal{H}_0})\theta_A\theta_A^*(\lambda_{g^{-1}}\otimes I_{\mathcal{H}_0})\theta_\Psi\\
	&=\theta_\Psi^*\theta_A\theta_A^*(\lambda_g\otimes I_{\mathcal{H}_0})(\lambda_{g^{-1}}\otimes I_{\mathcal{H}_0})\theta_\Psi=I_\mathcal{H}, \\ \pi_g^*\pi_g&=\theta_A^*(\lambda_{g^{-1}}\otimes I_{\mathcal{H}_0})\theta_\Psi\theta_\Psi^*(\lambda_{g}\otimes I_{\mathcal{H}_0})\theta_A\\
	&=\theta_A^*(\lambda_{g^{-1}}\otimes I_{\mathcal{H}_0})(\lambda_{g}\otimes I_{\mathcal{H}_0})\theta_\Psi\theta_\Psi^*\theta_A=I_\mathcal{H},   \quad \forall  g \in G.
	\end{align*}
	Since $ G $ has the discrete topology, this proves $ \pi$ is a unitary representation. It remains to prove  $ A_g=A_e\pi_{g^{-1}}, \Psi_g=\Psi_e\pi_{g^{-1}}  $ for all $ g \in G$. Indeed,
	\begin{align*}
	A_e\pi_{g^{-1}}&= L_e^*\theta_A\theta_\Psi^*(\lambda_{g^{-1}}\otimes I_{\mathcal{H}_0})\theta_A=L_e^*(\lambda_{g^{-1}}\otimes I_{\mathcal{H}_0})\theta_A\theta_\Psi^*\theta_A\\
	&=((\lambda_g\otimes I_{\mathcal{H}_0})L_e)^*\theta_A=L_{ge}^*\theta_A=A_g,
	\end{align*}
and
\begin{align*}
\Psi_e\pi_{g^{-1}}&=L_e^*\theta_\Psi \theta_\Psi^*(\lambda_{g^{-1}}\otimes I_{\mathcal{H}_0})\theta_A=L_e^*(\lambda_{g^{-1}}\otimes I_{\mathcal{H}_0})\theta_\Psi\theta_\Psi^*\theta_A\\
&=((\lambda_g\otimes I_{\mathcal{H}_0})L_e)^*\theta_\Psi=L_{ge}^*\theta_\Psi=\Psi_g.
\end{align*}
\end{proof}
In the direct part of Theorem \ref{gc1},   we can remove the word  `Parseval' since it has not been used in the proof;  same is true in the following corollary.
\begin{corollary}
	Let $ G$ be a discrete group and $( \{A_g\}_{g\in G},  \{\Psi_g\}_{g\in G})$ be a factorable weak OVF  in $ \mathcal{B}(\mathcal{H},\mathcal{H}_0).$ Then there is a  unitary representation $ \pi$  of $ G$ on  $ \mathcal{H}$  for which
	\begin{enumerate}[label=(\roman*)]
		\item  $ A_g=A_eS_{A,\Psi}^{-1}\pi_{g^{-1}}S_{A,\Psi}, \Psi_g=\Psi_e\pi_{g^{-1}}  $ for all $ g \in G$  if and only if
		\begin{align*}
	&	A_{gp}S_{A,\Psi}^{-1}(S_{A,\Psi}^{-1})^*A_{gq}^*=A_pS_{A,\Psi}^{-1}(S_{A,\Psi}^{-1})^*A_q^* ,\quad A_{gp}S_{A,\Psi}^{-1}\Psi_{gq}^*=A_pS_{A,\Psi}^{-1}\Psi_q^*,\\ 
		& \Psi_{gp}\Psi_{gq}^*=\Psi_p\Psi_q^*, \quad  \forall  g,p,q \in G.
		\end{align*} 
		\item  $ A_g=A_e\pi_{g^{-1}}, \Psi_g=\Psi_e(S_{A,\Psi}^{-1})^*\pi_{g^{-1}}S_{A,\Psi}  $ for all $ g \in G$  if and only if 
		\begin{align*}
		&A_{gp}A_{gq}^*=A_pA_q^* ,\quad A_{gp}S_{A,\Psi}^{-1}\Psi_{gq}^*=A_pS_{A,\Psi}^{-1}\Psi_q^*,\\ 
		& \Psi_{gp}(S_{A,\Psi}^{-1})^*S_{A,\Psi}^{-1}\Psi_{gq}^*=\Psi_p(S_{A,\Psi}^{-1})^*S_{A,\Psi}^{-1}\Psi_q^*, \quad \forall  g,p,q \in G.
		\end{align*}
\end{enumerate}	
\end{corollary}
\begin{proof}
\begin{enumerate}[label=(\roman*)]
		\item 	We apply Theorem  \ref{gc1} to the factorable Parseval OVF $(\{A_gS_{A,\Psi}^{-1}\}_{g\in G}, $ $ \{\Psi_g\}_{g\in G})$ to get: there is a  unitary representation $ \pi$  of $ G$ on  $ \mathcal{H}$  for which $ A_gS_{A,\Psi}^{-1}=(A_eS_{A,\Psi}^{-1})\pi_{g^{-1}}, \Psi_g=\Psi_e\pi_{g^{-1}}  $ for all $ g \in G$  if and only if 
		\begin{align*}
		&(A_{gp}S_{A,\Psi}^{-1})(A_{gq}S_{A,\Psi}^{-1})^*=(A_pS_{A,\Psi}^{-1})(A_qS_{A,\Psi}^{-1})^*, \quad (A_{gp}S_{A,\Psi}^{-1})\Psi_{gq}^*= (A_pS_{A,\Psi}^{-1})\Psi_q^*,\\
		&\Psi_{gp}\Psi_{gq}^*=\Psi_p\Psi_q^*, \quad \forall  g,p,q \in G.
		\end{align*}
		\item 	We apply Theorem  \ref{gc1} to the factorable Parseval OVF $( \{A_g\}_{g\in G}$, $  \{\Psi_g(S_{A,\Psi}^{-1})^*\}_{g\in G})$ to get: there is a  unitary representation $ \pi$  of $ G$ on  $ \mathcal{H}$  for which $ A_g=A_e\pi_{g^{-1}}$, $  \Psi_gS_{A,\Psi}^{-1}=(\Psi_e(S_{A,\Psi}^{-1})^*)\pi_{g^{-1}}  $ for all $ g \in G$  if and only if 
		\begin{align*}
		&A_{gp}A_{gq}^*=A_pA_q^* , \quad A_{gp}(\Psi_{gq}(S_{A,\Psi}^{-1})^*)^*=A_p(\Psi_q(S_{A,\Psi}^{-1})^*)^*,\\
		&(\Psi_{gp}(S_{A,\Psi}^{-1})^*)(\Psi_{gq}(S_{A,\Psi}^{-1})^*)^*=(\Psi_p(S_{A,\Psi}^{-1})^*)(\Psi_q(S_{A,\Psi}^{-1})^*)^* , \quad \forall  g,p,q \in G.
		\end{align*}
	\end{enumerate}		
\end{proof}
We next address the situation of factorable weak OVF whenever it is indexed by group-like unitary systems. Group-like unitary systems arose from the study of Weyl-Heisenberg frames. This was first formally defined by  \cite{GABARDO}. In the sequel, by $\mathbb{T}$, we mean the standard unit circle group centered at the origin equipped with usual multiplication.
\begin{definition}(\cite{GABARDO})
	A collection   $ \mathcal{U}\subseteq \mathcal{B}(\mathcal{H})$  containing $I_\mathcal{H}$ is called as a \textbf{unitary system}.  If the group generated by  unitary system $ \mathcal{U}$, denoted by $ \operatorname{group}(\mathcal{U})$ is such that 
	\begin{enumerate}[label=(\roman*)]
		\item $\operatorname{group}(\mathcal{U}) \subseteq \mathbb{T}\mathcal{U}\coloneqq \{\alpha U : \alpha \in \mathbb{T}, U\in \mathcal{U}  \}$, and 
		\item $\mathcal{U}$ is linearly independent, i.e.,  $\mathbb{T}U\ne\mathbb{T}V $ whenever $ U, V \in \mathcal{U}$ are such that $ U\ne V,$ 
	\end{enumerate}
	then $\mathcal{U}$ is called as a \textbf{group-like unitary system}.
\end{definition}
Let $ \mathcal{U}$ be a group-like unitary system. As in (\cite{GABARDOHANGROUPLIKE}), we  define mappings 
\begin{align*}
f:\operatorname{group}(\mathcal{U})\rightarrow \mathbb{T} \quad \text{ and } \quad \sigma:\operatorname{group}(\mathcal{U})\rightarrow \mathcal{U}.
\end{align*}
in the following way. For each  $ U \in  \operatorname{group}(\mathcal{U}) $ there are unique $\alpha\in \mathbb{T}, V \in \mathcal{U} $  such that $ U=\alpha V$. Define $ f(U)=\alpha$ and $\sigma(U)=V $. These  $ f, \sigma $ are well-defined and satisfy 
\begin{align*}
U=f(U)\sigma(U), \quad \forall U \in \operatorname{group}(\mathcal{U}).
\end{align*}
These mappings are called as \textbf{corresponding mappings} associated to $ \mathcal{U}$. We can picturize these maps as follows. 
\begin{center}
	\[
	\begin{tikzcd}
	\operatorname{group}(\mathcal{U}) \arrow[d,"\sigma"] \arrow[dr,"f"]\subseteq\mathbb{T}\mathcal{U}\\
	\mathcal{U}  & \mathbb{T}\arrow[l,] \\
	\end{tikzcd}
	\]
\end{center}
Next result gives certain fundamental properties of corresponding mappings associated with group-like unitary systems.
\begin{proposition}(\cite{GABARDOHANGROUPLIKE})\label{PER} 
	For a group-like unitary system $\mathcal{U}$ and $ f, \sigma $ as above,
	\begin{enumerate}[label=(\roman*)]
		\item $ f(U\sigma(VW))f(VW)=f(\sigma(UV)W)f(UV), \forall U,V,W \in \operatorname{group}(\mathcal{U}).$
		\item $ \sigma(U\sigma(VW))=\sigma(\sigma(UV)W), \forall U,V,W \in \operatorname{group} (\mathcal{U}).$
		\item $ \sigma(U)=U$ and $ f(U)=1$ for all $ U \in \mathcal{U}.$
		\item If $  V, W \in \operatorname{group} (\mathcal{U}),$ then
		\begin{align*}
		\mathcal{U}&=\{\sigma(UV) : U \in \mathcal{U}\}=\{\sigma(VU^{-1}) : U \in \mathcal{U}\}\\
		&=\{\sigma(VU^{-1}W) : U \in \mathcal{U}\}=\{\sigma(V^{-1}U) : U \in \mathcal{U}\}.
		\end{align*}
		\item For fixed  $  V, W \in \mathcal{U}$, the following mappings are  injective from  $ \mathcal{U} $ to itself:
		\begin{align*}
		U\mapsto \sigma(VU) \quad (\text{resp.} ~ \sigma(UV), \sigma(UV^{-1}), \sigma(V^{-1}U),\\
		 \sigma(VU^{-1}), \sigma(U^{-1}V), \sigma(VU^{-1}W)).
		\end{align*}
\end{enumerate}
\end{proposition}
Since $\operatorname{group} (\mathcal{U}) $ is a group, we note that, in (iv) of Proposition \ref{PER},  we can replace $V$ by $V^{-1}$. Hence, whenever  $V \in \operatorname{group} (\mathcal{U})$, we have $\sum_{U \in \mathcal{U}}x_U=\sum_{U \in \mathcal{U}}x_{\sigma(VU)}$.
\begin{definition}(\cite{GABARDOHANGROUPLIKE})
	A \textbf{unitary representation} $ \pi$ of a group-like unitary system $   \mathcal{U}$ on $ \mathcal{H}$ is an injective mapping from $  \mathcal{U}$ into the set of unitary operators on $ \mathcal{H}$ such that 
	$$\pi(U)\pi(V)=f(UV)\pi(\sigma(UV)) , \quad {\pi(U)}^{-1}=f(U^{-1})\pi(\sigma(U^{-1})), ~ \forall U,V \in \mathcal{U}, $$
	where $ f$ and $ \sigma $  are the corresponding mappings associated with $  \mathcal{U}.$ 
\end{definition}
Since $\pi $ is injective, once we have a unitary representation of a group-like unitary system $   \mathcal{U}$ on $\mathcal{H}$, then $ \pi(\mathcal{U})$ is also a group-like unitary system.

Let $ \mathcal{U}$ be a  group-like unitary system and  $ \{\chi_U\}_{U\in \mathcal{U}}$ be the  standard orthonormal  basis for $\ell^2(\mathcal{U}) $.  We define $\lambda $ on  $ \mathcal{U}$  by $ \lambda_U\chi_V=f(UV)\chi_{\sigma(UV)}, \forall   U,V \in \mathcal{U}.$ Then $ \lambda $ is a unitary  representation which we call as  left  regular representation of $ \mathcal{U}$. Similarly, we define right regular representation of $ \mathcal{U}$ by $ \rho_U\chi_V=f(VU^{-1})\chi_{\sigma(VU^{-1})}, \forall U,V \in \mathcal{U}$ (\cite{GABARDOHANGROUPLIKE}).  
Like frame generators for groups, we now define the frame generator for group-like unitary systems.
\begin{definition}
	Let $ \mathcal{U}$  be a group-like unitary system.  An operator $ A$ in $ \mathcal{B}(\mathcal{H}, \mathcal{H}_0)$ is called an  \textbf{operator frame generator} (resp. a  Parseval  frame generator) w.r.t. $ \Psi$ in $ \mathcal{B}(\mathcal{H}, \mathcal{H}_0)$  if $(\{A_U\coloneqq A\pi(U)^{-1}\}_{U\in \mathcal{U}},\{\Psi_U\coloneqq \Psi\pi(U)^{-1}\}_{U\in \mathcal{U}})$ is a factorable weak OVF  (resp. a Parseval) in $ \mathcal{B}(\mathcal{H}, \mathcal{H}_0)$.  We write $ (A,\Psi)$ is an operator frame generator for $\pi$.
\end{definition}
\begin{theorem}\label{CHARACTERIZATIONGROUPLIKE}
	Let $ \mathcal{U}$ be a  group-like unitary system, $ I$ be the identity of $ \mathcal{U}$ and $(\{A_U\}_{U\in \mathcal{U}},\{\Psi_U\}_{U\in \mathcal{U}})$ be a factorable Parseval weak OVF  in $ \mathcal{B}(\mathcal{H},\mathcal{H}_0)$ with $ \theta_A^*$ injective. Then there is a unitary representation $ \pi$  of $ \mathcal{U}$ on  $ \mathcal{H}$  for which 
	$$ A_U=A_I\pi(U)^{-1}, \quad\Psi_U=\Psi_I\pi(U)^{-1}, \quad\forall  U \in \mathcal{U}$$
	if and only if 
	\begin{align*}
	A_{\sigma(UV)}A_{\sigma(UW)}^*&=f(UV)\overline{f(UW)} A_VA_W^* ,\\
	A_{\sigma(UV)}\Psi_{\sigma(UW)}^*&=f(UV)\overline{f(UW)} A_V\Psi_W^*,\\ \Psi_{\sigma(UV)}\Psi_{\sigma(UW)}^*&=f(UV)\overline{f(UW)} \Psi_V\Psi_W^*, \quad \forall  U,V,W \in \mathcal{U}.
	\end{align*}
\end{theorem} 
\begin{proof}
	$(\Rightarrow)$ For all $U,V,W \in \mathcal{U}$, we have
	\begin{align*}
	A_{\sigma(UV)}A_{\sigma(UW)}^*&= A_I\pi(\sigma(UV))^{-1}( A_I\pi(\sigma(UW))^{-1} )^*\\
	&=A_I(\overline{f(UV)}\pi(U)\pi(V))^{-1} \overline{f(UW)}\pi(U)\pi(W)A^*_I\\
	&=f(UV)\overline{f(UW)} A_I\pi(V)^{-1}(A_I\pi(W)^{-1})^*\\
	&=f(UV)\overline{f(UW)} A_VA_W^*. 
	\end{align*}
	Others can be shown similarly. 
	
	$(\Leftarrow)$  We have to construct unitary representation which satisfies the stated conditions. Following observation plays an important role in this part. Let $ h\in \mathcal{H}.$ Then 
	\begin{align*}
	L_{\sigma(UV)}h&=\chi_{\sigma(UV)}\otimes h=\overline{f(UV)}\lambda_U\chi_V\otimes h=\overline{f(UV)}(\lambda_U\chi_V\otimes h)\\
	&=\overline{f(UV)}(\lambda_U\otimes I_{\mathcal{H}_0})(\chi_V\otimes h)=\overline{f(UV)}(\lambda_U\otimes I_{\mathcal{H}_0})L_V h.
	\end{align*}
	As in the proof of Theorem \ref{gc1},  we argue the following, for which now we prove the first.
	For all $ U \in \mathcal{U},$
	\begin{align*}
	&(\lambda_U\otimes I_{\mathcal{H}_0})\theta_A\theta_A^*=\theta_A\theta_A^*(\lambda_U\otimes I_{\mathcal{H}_0}), \quad (\lambda_U\otimes I_{\mathcal{H}_0})\theta_A\theta_\Psi^*=\theta_A\theta_\Psi^*(\lambda_U\otimes I_{\mathcal{H}_0}),\\
	&(\lambda_U\otimes I_{\mathcal{H}_0})\theta_\Psi\theta_\Psi^*=\theta_\Psi\theta_\Psi^*(\lambda_U\otimes I_{\mathcal{H}_0}).
	\end{align*}
	Consider 
	\begin{align*}
	(\lambda_U\otimes I_{\mathcal{H}_0})\theta_A\theta_A^*(\lambda_U\otimes I_{\mathcal{H}_0})^*&=\left(\sum\limits_{V\in \mathcal{U}}(\lambda_U\otimes I_{\mathcal{H}_0})L_VA_V\right)\left(\sum\limits_{W\in \mathcal{U}}(\lambda_U\otimes I_{\mathcal{H}_0})L_WA_W\right)^*\\ &=\left(\sum\limits_{V\in \mathcal{U}} f(UV)L_{\sigma(UV)}A_V\right)\left(\sum\limits_{W\in \mathcal{U}}f(UW)L_{\sigma(UW)}A_W\right)^*\\
	&=\sum\limits_{V\in \mathcal{U}} L_{\sigma(UV)}\left(\sum\limits_{W\in \mathcal{U}}f(UV)\overline{f(UW)}A_VA_W^*L_{\sigma(UW)}^*\right)\\
	&= \sum\limits_{V\in \mathcal{U}} L_{\sigma(UV)}\left(\sum\limits_{W\in \mathcal{U}}A_{\sigma(UV)}A_{\sigma(UW)}^*L_{\sigma(UW)}^*\right)\\
	&=\left(\sum\limits_{V\in \mathcal{U}} L_{\sigma(UV)}A_{\sigma(UV)}\right)\left(\sum\limits_{W\in \mathcal{U}}L_{\sigma(UW)}A_{\sigma(UW)}\right)^*\\
	&=\theta_A\theta_A^*
	\end{align*}
	where last part of Proposition \ref{PER} is used in the last equality.
	
	Define $ \pi : \mathcal{U} \ni U  \mapsto \pi(U)\coloneqq \theta_\Psi^*(\lambda_U\otimes I_{\mathcal{H}_0})\theta_A  \in \mathcal{B}(\mathcal{H}).$  Then 
	\begin{align*}
	 \pi(U)\pi(V)&=\theta_\Psi^*(\lambda_U\otimes I_{\mathcal{H}_0})\theta_A \theta_\Psi^*(\lambda_V\otimes I_{\mathcal{H}_0})\theta_A \\
	 &=\theta_\Psi^*\theta_A \theta_\Psi^*(\lambda_U\otimes I_{\mathcal{H}_0}) (\lambda_V\otimes I_{\mathcal{H}_0})\theta_A \\
	 &= \theta_\Psi^*(\lambda_U\lambda_V\otimes I_{\mathcal{H}_0})\theta_A \\
	 &=\theta_\Psi^*(f(UV)\lambda_{\sigma(UV)}\otimes I_{\mathcal{H}_0})\theta_A  \\
	 &=f(UV) \theta_\Psi^*(\lambda_{\sigma(UV)}\otimes I_{\mathcal{H}_0})\theta_A \\
	 &=f(UV)\pi({\sigma(UV)}), \quad \forall U, V \in \mathcal{U}
	\end{align*}
	and 
	\begin{align*}
	\pi(U)\pi(U)^*&=\theta_\Psi^*(\lambda_U\otimes I_{\mathcal{H}_0})\theta_A\theta_A^*(\lambda_U^*\otimes I_{\mathcal{H}_0})\theta_\Psi
	\\
	&=\theta_\Psi^*\theta_A\theta_A^*(\lambda_U\otimes I_{\mathcal{H}_0})(\lambda_U^*\otimes I_{\mathcal{H}_0})\theta_\Psi
	=I_\mathcal{H},\\
	  \pi(U)^*\pi(U)&=\theta_A^*(\lambda_U^*\otimes I_{\mathcal{H}_0})\theta_\Psi\theta_\Psi^*(\lambda_{U}\otimes I_{\mathcal{H}_0})\theta_A\\
	  &=\theta_A^*(\lambda_U^*\otimes I_{\mathcal{H}_0})(\lambda_U\otimes I_{\mathcal{H}_0})\theta_\Psi\theta_\Psi^*\theta_A=I_\mathcal{H}, \quad \forall U \in \mathcal{U}.
	\end{align*}
  Further, 
	\begin{align*}
	\pi(U)f(U^{-1})\pi(\sigma(U^{-1}))&=\theta_\Psi^*(\lambda_U\otimes I_{\mathcal{H}_0})\theta_Af(U^{-1})\theta_\Psi^*(\lambda_{\sigma(U^{-1})}\otimes I_{\mathcal{H}_0})\theta_A \\
	&=f(U^{-1})\theta_\Psi^*\theta_A\theta_\Psi^*(\lambda_U\otimes I_{\mathcal{H}_0})(\lambda_{\sigma(U^{-1})}\otimes I_{\mathcal{H}_0})\theta_A \\
	&=f(U^{-1})\theta_\Psi^*(\lambda_U\otimes I_{\mathcal{H}_0})(\lambda_{\sigma(U^{-1})}\otimes I_{\mathcal{H}_0})\theta_A\\
	&=f(U^{-1})\theta_\Psi^*(\lambda_U\lambda_{\sigma(U^{-1})}\otimes I_{\mathcal{H}_0})\theta_A\\
	&=f(U^{-1})\theta_\Psi^*(f(U\sigma(U^{-1}))\lambda_{\sigma(U\sigma (U^{-1}))}\otimes I_{\mathcal{H}_0})\theta_A\\
	&=\theta_\Psi^*(f(U\sigma(U^{-1}I))f(U^{-1}I)\lambda_{\sigma(U\sigma (U^{-1}I))}\otimes I_{\mathcal{H}_0})\theta_A\\
	&=\theta_\Psi^*(f(\sigma(UU^{-1})I)f(UU^{-1})\lambda_{\sigma({\sigma(UU^{-1})I})}\otimes I_{\mathcal{H}_0})\theta_A\\
	&=\theta_\Psi^*(\lambda_I\otimes I_{\mathcal{H}_0})\theta_A=I_\mathcal{H}
	\end{align*}
	$\Rightarrow {\pi(U)}^{-1}=f(U^{-1})\pi(\sigma(U^{-1}))$ for all $ U \in \mathcal{U}$. We shall now use $ \theta_A^*$ is injective   to show $ \pi$ is injective and thereby to  get $ \pi$ is a unitary representation. Let $ \pi(U)=\pi(V).$ Then 
	\begin{align*}
	&\theta_\Psi^*(\lambda_U\otimes I_{\mathcal{H}_0})\theta_A =\theta_\Psi^*(\lambda_V\otimes I_{\mathcal{H}_0})\theta_A  \Rightarrow \theta_\Psi^*(\lambda_U\otimes I_{\mathcal{H}_0})\theta_A \theta_A^*=\theta_\Psi^*(\lambda_V\otimes I_{\mathcal{H}_0})\theta_A  \theta_A^* \\
	&\Rightarrow \theta_\Psi^*\theta_A  \theta_A^*(\lambda_U\otimes I_{\mathcal{H}_0}) =\theta_\Psi^*\theta_A  \theta_A^*(\lambda_V\otimes I_{\mathcal{H}_0}) \Rightarrow \lambda_U\otimes I_{\mathcal{H}_0}=\lambda_V\otimes I_{\mathcal{H}_0}.
	\end{align*}
	We show $ U$ and $ V$ are identical at elementary tensors. For $ h \in \ell^2(\mathcal{U}),  y \in \mathcal{H}_0,  $ we get, $(\lambda_U\otimes I_{\mathcal{H}_0})(h\otimes y)=(\lambda_V\otimes I_{\mathcal{H}_0})(h\otimes y)\Rightarrow \lambda_Uh\otimes y=\lambda_Vh\otimes y \Rightarrow (\lambda_U-\lambda_V)h\otimes y=0 \Rightarrow 0= \langle  (\lambda_U-\lambda_V)h\otimes y, (\lambda_U-\lambda_V)h\otimes y\rangle= \|(\lambda_U-\lambda_V)h\|^2 \|y\|^2 .$ We may assume $y\neq0$ (if  $y=0$, then  $h\otimes y=0$). But then $ (\lambda_U-\lambda_V)(h)=0,$ and $ \lambda$ is a unitary representation (it is injective) gives $ U=V.$ We now show   $ A_U=A_I\pi(U)^{-1} $ and $ \Psi_U=\Psi_I\pi(U)^{-1}  $ for all $ U \in \mathcal{U}$ in the following:
	\begin{align*}
	A_I\pi(U)^{-1}&= L_I^*\theta_A(\theta_\Psi^*(\lambda_U\otimes I_{\mathcal{H}_0})\theta_A)^*
	= L_I^*(\theta_\Psi^*(\lambda_U\otimes I_{\mathcal{H}_0})\theta_A\theta_A^*)^*\\
	&=L_I^*(\theta_\Psi^*\theta_A\theta_A^*(\lambda_U\otimes I_{\mathcal{H}_0}))^*
	= L_I^*(\theta_A^*(\lambda_U\otimes I_{\mathcal{H}_0}))^*\\
	&=(\theta_A^*(\lambda_U\otimes I_{\mathcal{H}_0})L_I)^*= (\theta_A^*\overline{f(UI)}(\lambda_U\otimes I_{\mathcal{H}_0})L_I)^*\\
	&= (\theta_A^*L_{\sigma({UI})})^*=L_U^*\theta_A=A_U
	\end{align*}
	and
	\begin{align*}
	\Psi_I\pi(U)^{-1}&= L_I^*\theta_\Psi(\theta_\Psi^*(\lambda_U\otimes I_{\mathcal{H}_0})\theta_A)^*
	= L_I^*(\theta_\Psi^*(\lambda_U\otimes I_{\mathcal{H}_0})\theta_A\theta_\Psi^*)^*\\
	&=L_I^*(\theta_\Psi^*\theta_A\theta_\Psi^*(\lambda_U\otimes I_{\mathcal{H}_0}))^*
	= L_I^*(\theta_\Psi^*(\lambda_U\otimes I_{\mathcal{H}_0}))^*\\
	&=(\theta_\Psi^*(\lambda_U\otimes I_{\mathcal{H}_0})L_I)^*= (\theta_\Psi^*\overline{f(UI)}(\lambda_U\otimes I_{\mathcal{H}_0})L_I)^*\\
	&= (\theta_\Psi^*L_{\sigma({UI})})^*=L_U^*\theta_\Psi=\Psi_U.
	\end{align*}
\end{proof}
Note that neither  Parsevalness of the frame nor $ \theta_A^*$  is injective was used  in the direct part of Theorem \ref{CHARACTERIZATIONGROUPLIKE}. Since $ \theta_A$ acts between Hilbert spaces, we know that $ \overline{\theta_A(\mathcal{H})}=\operatorname{Ker}(\theta_A^*)^\perp$ and $ \operatorname{Ker}(\theta_A^*)=\theta_A(\mathcal{H})^\perp.$ From Lemma \ref{DILATIONLEMMA}, the range of $\theta_A$  is closed. Therefore $ \theta_A(\mathcal{H})=\operatorname{Ker}(\theta_A^*)^\perp.$ Thus the condition $ \theta_A^*$ is injective in the Theorem \ref{CHARACTERIZATIONGROUPLIKE} can be replaced by $ \theta_A$ is onto. 
\begin{corollary}
	Let $ \mathcal{U}$ be a  group-like unitary system,  $ I$ be the identity of $ \mathcal{U}$ and $(\{A_U\}_{U\in \mathcal{U}},\{\Psi_U\}_{U\in \mathcal{U}})$ be a factorable weak OVF   in $ \mathcal{B}(\mathcal{H},\mathcal{H}_0)$  with $ \theta_A^*$   is injective. Then there is a unitary representation $ \pi$  of  $ \mathcal{U}$ on  $ \mathcal{H}$  for which
	\begin{enumerate}[label=(\roman*)]
		\item    $ A_U=A_IS^{-1}_{A,\Psi}\pi(U)^{-1}S_{A, \Psi}, \Psi_U=\Psi_I\pi(U)^{-1}  $ for all $ U \in \mathcal{U}$  if and only if 
		\begin{align*}
		&A_{\sigma(UV)}S^{-1}_{A, \Psi}(S_{A,\Psi}^{-1})^*A_{\sigma(UW)}^*=f(UV)\overline{f(UW)} A_VS^{-1}_{A, \Psi}(S_{A,\Psi}^{-1})^*A_W^*,\\
		&	A_{\sigma(UV)}S_{A, \Psi}^{-1}\Psi_{\sigma(UW)}^*=f(UV)\overline{f(UW)} A_VS^{-1}_{A, \Psi}\Psi_W^*,\\
		&\Psi_{\sigma(UV)}\Psi_{\sigma(UW)}^*=f(UV)\overline{f(UW)} \Psi_V\Psi_W^*, \quad \forall U,V,W \in \mathcal{U}.
		\end{align*}
		\item   $ A_U=A_I\pi(U)^{-1}, \Psi_U =\Psi_I(S_{A,\Psi}^{-1})^*\pi(U)^{-1}S_{A, \Psi}  $ for all $ U \in \mathcal{U}$  if and only if 
		\begin{align*}
		&A_{\sigma(UV)}A_{\sigma(UW)}^*=f(UV)\overline{f(UW)} A_VA_W^*,\\
		&A_{\sigma(UV)}S^{-1}_{A, \Psi}\Psi_{\sigma(UW)}^*=f(UV)\overline{f(UW)}A_VS^{-1}_{A, \Psi}\Psi_W^*,\\
		&\Psi_{\sigma(UV)}(S_{A,\Psi}^{-1})^*S^{-1}_{A, \Psi}\Psi_{\sigma(UW)} ^*
		=f(UV)\overline{f(UW)} \Psi_V(S_{A,\Psi}^{-1})^*S^{-1}_{A, \Psi}\Psi_W^*, \quad \forall U,V,W \in \mathcal{U}.
		\end{align*}
	\end{enumerate}
\end{corollary}
\begin{proof}
\begin{enumerate}[label=(\roman*)]
		\item 	We  apply Theorem \ref{CHARACTERIZATIONGROUPLIKE} to the factorable  Parseval OVF  $(\{A_US_{A,\Psi}^{-1}\}_{U\in \mathcal{U}} ,$ $ \{\Psi_U\}_{U\in \mathcal{U}})$. There is a unitary representation $ \pi$  of  $ \mathcal{U}$ on  $ \mathcal{H}$  for which  $ A_US_{A,\Psi}^{-1}=(A_IS^{-1}_{A,\Psi})\pi(U)^{-1}, \Psi_U=\Psi_I\pi(U)^{-1}  $ for all $ U \in \mathcal{U}$  if and only if 
		\begin{align*}
		&(A_{\sigma(UV)}S^{-1}_{A, \Psi})(A_{\sigma(UW)}S_{A,\Psi}^{-1})^*=f(UV)\overline{f(UW)}( A_VS^{-1}_{A, \Psi})(A_WS_{A,\Psi}^{-1})^*,\\
		&(A_{\sigma(UV)}S_{A, \Psi}^{-1}) \Psi_{\sigma(UW)}^*=
		f(UV)\overline{f(UW)}( A_VS^{-1}_{A, \Psi})\Psi_W^*,\\
		&\Psi_{\sigma(UV)}\Psi_{\sigma(UW)}^*=f(UV)\overline{f(UW)} \Psi_V\Psi_W^*, \quad \forall U,V,W \in \mathcal{U}.
		\end{align*}
		\item 	We  apply Theorem \ref{CHARACTERIZATIONGROUPLIKE} to the factorable  Parseval OVF   $(\{A_U\}_{U\in \mathcal{U}} , \{\Psi_U(S_{A,\Psi}^{-1})^*\}_{U\in \mathcal{U}})$. There is a unitary representation $ \pi$  of  $ \mathcal{U}$ on  $ \mathcal{H}$  for which  $ A_U=A_I\pi(U)^{-1}, \Psi_U(S_{A,\Psi}^{-1})^*=(\Psi_I(S_{A,\Psi}^{-1})^*)\pi(U)^{-1}  $ for all $ U \in \mathcal{U}$  if and only if 
		\begin{align*}
		&A_{\sigma(UV)}A_{\sigma(UW)}^*=
		f(UV)\overline{f(UW)} A_VA_W^*,\\
		&A_{\sigma(UV)}(\Psi_{\sigma(UW)}(S_{A,\Psi}^{-1})^*)^*=f(UV)\overline{f(UW)}A_V(\Psi_W(S_{A,\Psi}^{-1})^*)^*,\\
		&(\Psi_{\sigma(UV)}(S_{A,\Psi}^{-1})^*)(\Psi_{\sigma(UW)}(S_{A,\Psi}^{-1})^*)^*=f(UV)\overline{f(UW)} (\Psi_V(S_{A,\Psi}^{-1})^*)(\Psi_W(S_{A,\Psi}^{-1})^*)^*,\\
		& \quad \quad \quad \quad \quad \quad \quad \quad \quad \quad \quad \quad \quad \quad \quad \quad\forall U,V,W \in \mathcal{U}.
		\end{align*}
	\end{enumerate}
\end{proof}

\section{PERTURBATIONS OF WEAK OPERATOR-VALUED FRAMES}\label{PERTURBATIONS}
In this section we derive stability results for factorable weak operator-valued frames. 
\begin{theorem}\label{PERTURBATION RESULT 1}
	Let $  ( \{A_n\}_{n},  \{\Psi_n\}_{n} ) $  be  a factorable weak OVF in $ \mathcal{B}(\mathcal{H}, \mathcal{H}_0)$. Suppose  $\{B_n\}_{n} $ in $ \mathcal{B}(\mathcal{H}, \mathcal{H}_0)$ is such that  there exist $\alpha, \beta, \gamma \geq 0  $ with $ \max\{\alpha+\gamma\|\theta_\Psi (S_{A,\Psi}^*)^{-1}\|, \beta\}$ $<1$ and for all $m=1,2, \dots, $
	\begin{align}\label{p3}
	\left\|\sum\limits_{n=1}^m(A_n^*-B_n^*)L_n^*y\right\|&\leq \alpha\left\|\sum\limits_{n=1}^mA_n^*L_n^*y\right\|+\beta\left\|\sum\limits_{n=1}^mB_n^*L_n^*y\right\|+\gamma \left(\sum\limits_{n=1}^m\|L_n^*y\|^2\right)^\frac{1}{2},\nonumber\\
	&\quad \forall y \in \ell^2(\mathbb{N})\otimes \mathcal{H}_0.
	\end{align} 
	Then  $  ( \{B_n\}_{n},  \{\Psi_n\}_{n} ) $ is a factorable weak OVF  with bounds 
	\begin{align*}
	\frac{1-(\alpha+\gamma\|\theta_\Psi (S_{A,\Psi}^*)^{-1}\|)}{(1+\beta)\|(S_{A,\Psi}^*)^{-1}\|} \quad \text{ and } \quad \frac{\|\theta_\Psi\|((1+\alpha)\|\theta_A\|+\gamma)}{1-\beta}.
	\end{align*}
\end{theorem}
\begin{proof}
	For $m=1,2,\dots, $ and for every $ y$ in $ \ell^2(\mathbb{N})\otimes \mathcal{H}_0$, 
	\begin{align*}
	\left\| \sum\limits_{n=1}^mB_n^*L_n^*y\right\|&\leq \left\| \sum\limits_{n=1}^m(A_n^*-B_n^*)L_n^*y\right\|+\left\| \sum\limits_{n=1}^mA_n^*L_n^*y\right\|\\
	&\leq(1+\alpha)\left\| \sum\limits_{n=1}^mA_n^*L_n^*y\right\|+\beta\left\| \sum\limits_{n=1}^mB_n^*L_n^*y\right\|+\gamma\left( \sum\limits_{n=1}^m\|L_n^*y\|^2\right)^\frac{1}{2}
	\end{align*}
	which implies 
	\begin{equation}\label{p1}
	\left\| \sum\limits_{n=1}^mB_n^*L_n^*y\right\|\leq\frac{1+\alpha}{1-\beta}\left\| \sum\limits_{n=1}^mA_n^*L_n^*y\right\|+\frac{\gamma}{1-\beta}\left( \sum\limits_{n=1}^m\|L_n^*y\|^2\right)^\frac{1}{2}, \quad \forall y \in \ell^2(\mathbb{N})\otimes \mathcal{H}_0.
	\end{equation}
	Since  
	$$ \langle y,y\rangle =\langle (I_{\ell^2(\mathbb{N})}\otimes I_{\mathcal{H}_0})y,y\rangle=\left\langle\sum\limits_{n=1}^\infty L_nL_n^* y,y\right\rangle=\sum\limits_{n=1}^\infty\|L_n^* y\|^2 , \quad \forall y \in \ell^2(\mathbb{N})\otimes \mathcal{H}_0,$$
	Inequality   (\ref{p1})  shows that $\sum_{n=1}^\infty B_n^*L_n^*y $ exists for all  $   y \in \ell^2(\mathbb{N})\otimes \mathcal{H}_0.$ 
	From the continuity of norm, Inequality (\ref{p1}) gives
	\begin{align}\label{p2}
	\left\| \sum\limits_{n=1}^\infty B_n^*L_n^*y\right\|&\leq\frac{1+\alpha}{1-\beta}\left\| \sum\limits_{n=1}^\infty A_n^*L_n^*y\right\|+\frac{\gamma}{1-\beta}\left( \sum\limits_{n=1}^\infty\|L_n^*y\|^2\right)^\frac{1}{2}\nonumber \\
	&=\frac{1+\alpha}{1-\beta}\left\| \theta_A^*y\right\|+\frac{\gamma}{1-\beta}\|y\| , \quad \forall y \in \ell^2(\mathbb{N})\otimes \mathcal{H}_0
	\end{align}
	and this gives $ \sum_{n=1}^\infty B_n^*L_n^* $   is bounded; therefore its adjoint exists, which is $ \theta_B$; Inequality (\ref{p2}) now produces $\|\theta_B^*y\|\leq \frac{1+\alpha}{1-\beta}\left\| \theta_A^*y\right\|+\frac{\gamma}{1-\beta}\|y\| , \forall y \in \ell^2(\mathbb{N})\otimes \mathcal{H}_0 $ and from this $\|\theta_B\|=\|\theta_B^*\|\leq \frac{1+\alpha}{1-\beta}\left\| \theta_A^*\right\|+\frac{\gamma}{1-\beta} =\frac{1+\alpha}{1-\beta}\left\| \theta_A\right\|+\frac{\gamma}{1-\beta}.$
	Thus  we derived $ S_{B, \Psi}$ is a  bounded linear operator. 
	Continuity of the norm, existence of frame  operators together with Inequality (\ref{p3}) give
	$$ \|\theta_A^*y-\theta_B^*y\|\leq \alpha\|\theta_A^*y\|+\beta\|\theta_B^*y\|+\gamma\|y\|, \quad \forall y \in \ell^2(\mathbb{N})\otimes \mathcal{H}_0$$
		which implies
		\begin{align*}
	 \|\theta_A^*(\theta_\Psi (S_{A,\Psi}^*)^{-1} h)-\theta_B^*(\theta_\Psi (S_{A,\Psi}^*)^{-1}h)\|&\leq \alpha\|\theta_A^*(\theta_\Psi (S_{A,\Psi}^*)^{-1} h)\|\\
	 &\quad +\beta\|\theta_B^*(\theta_\Psi (S_{A,\Psi}^*)^{-1} h)\| +\gamma\|\theta_\Psi (S_{A,\Psi}^*)^{-1} h\|, \\
	 &\quad \forall h \in  \mathcal{H}.	
		\end{align*}
	But $ \theta_A^*\theta_\Psi (S_{A,\Psi}^*)^{-1}=I_\mathcal{H}$ and $\theta_B^*\theta_\Psi (S_{A,\Psi}^*)^{-1}= S_{B,\Psi}^* (S_{A,\Psi}^*)^{-1}.$ Therefore 
	\begin{align*}
	\| h- S_{B,\Psi}^*(S_{A,\Psi}^*)^{-1}h\|
	&\leq \alpha\| h\|+\beta\|S_{B,\Psi}^* (S_{A,\Psi}^*)^{-1} h\|+\gamma\|\theta_\Psi (S_{A,\Psi}^*)^{-1} h\|\\
	&\leq(\alpha+\gamma\|\theta_\Psi (S_{A,\Psi}^*)^{-1}\|)\|h\|+\beta\|S_{B,\Psi}^* (S_{A,\Psi}^*)^{-1} h\|, \quad \forall h \in  \mathcal{H}.
	\end{align*}
	Since $ \max\{\alpha+\gamma\|\theta_\Psi (S_{A,\Psi}^*)^{-1}\|, \beta\}<1$, Theorem \ref{cc1} tells that  $S_{B,\Psi}^* (S_{A,\Psi}^*)^{-1} $ is invertible and 
	$\|(S_{B,\Psi}^* (S_{A,\Psi}^*)^{-1})^{-1}\| \leq \frac{1+\beta}{1-(\alpha+\gamma\|\theta_\Psi (S_{A,\Psi}^*)^{-1}\|)}.$ From these, we get 
	$$(S_{B,\Psi}^* (S_{A,\Psi}^*)^{-1})S_{A,\Psi}^*=S_{B,\Psi}^* $$ is invertible and 
	\begin{align*}
	\| S_{B,\Psi}^{-1}\|\leq\|(S_{A,\Psi}^*)^{-1}\|\| S_{A,\Psi}^*S_{B,\Psi}^{-1}\| \leq \frac{\|(S_{A,\Psi}^*)^{-1}\|(1+\beta)}{1-(\alpha+\gamma\|\theta_\Psi (S_{A,\Psi}^*)^{-1}\|)}.
	\end{align*}
	Therefore $  ( \{B_n\}_{n},  \{\Psi_n\}_{n} ) $ is a factorable weak  OVF.  Observing that 
	\begin{align*}
	\|S_{B,\Psi}\|\leq \|\theta_\Psi\|\|\theta_B\|\leq \frac{\|\theta_\Psi\|((1+\alpha)\|\theta_A\|+\gamma)}{1-\beta}
	\end{align*}
	and  $ \|S_{B,\Psi}^{-1}\|^{-1}$ and $\|S_{B,\Psi}\| $ are optimal lower and upper frame bounds for $  ( \{B_n\}_{n},  \{\Psi_n\}_{n} ) $,  we get the frame bounds stated in the theorem.
\end{proof}
\begin{corollary}
	Let $  ( \{A_n\}_{n},  \{\Psi_n\}_{n} ) $  be  a factorable weak OVF  in $ \mathcal{B}(\mathcal{H}, \mathcal{H}_0)$. Suppose  $\{B_n\}_{n} $ in $ \mathcal{B}(\mathcal{H}, \mathcal{H}_0)$ is such that 
	$$ r \coloneqq \sum_{n=1}^\infty\|A_n-B_n\|^2 <\frac{1}{\|\theta_\Psi (S_{A,\Psi}^*)^{-1}\|^2}.$$
	Then $  ( \{B_n\}_{n},  \{\Psi_n\}_{n} ) $ is   a factorable weak OVF  with bounds 
	\begin{align*}
	\frac{1-\sqrt{r}\|\theta_\Psi (S_{A,\Psi}^*)^{-1}\|}{\|(S_{A,\Psi}^*)^{-1}\|} \quad \text{ and }\quad {\|\theta_\Psi\|(\|\theta_A\|+\sqrt{r})}. 
	\end{align*}
\end{corollary}
\begin{proof}
	We apply Theorem \ref{PERTURBATION RESULT 1} by taking $ \alpha =0, \beta=0, \gamma=\sqrt{r}$. Then $ \max\{\alpha+\gamma\|\theta_\Psi (S_{A,\Psi}^*)^{-1}\|, \beta\}<1$ and  for all $m=1,2, \dots, $
	\begin{align*}
	\left\|\sum\limits_{n=1}^m(A_n^*-B_n^*)L_n^*y\right\|&\leq \left(\sum\limits_{n=1}^m\|A_n^*-B_n^*\|^2 \right)^\frac{1}{2}\left(\sum\limits_{n=1}^m\|L_n^*y\|^2\right)^\frac{1}{2}\\
	&\leq\gamma\left(\sum\limits_{n=1}^m\|L_n^*y\|^2\right)^\frac{1}{2}, \quad\forall y \in \ell^2(\mathbb{N})\otimes \mathcal{H}_0.
	\end{align*}
\end{proof}
We next derive another stability result with different condition.
\begin{theorem}\label{OVFQUADRATICPERTURBATION}
	Let $  ( \{A_n\}_{n},  \{\Psi_n\}_{n} ) $  be a factorable weak OVF  in $ \mathcal{B}(\mathcal{H}, \mathcal{H}_0)$. Suppose  $\{B_n\}_{n} $ in $ \mathcal{B}(\mathcal{H}, \mathcal{H}_0)$ is such that  $   \sum_{n=1}^\infty\|A_n-B_n\|^2$ converges, and 
	$\sum_{n=1}^\infty\|A_n-B_n\|\|\Psi_n(S_{A,\Psi}^*)^{-1}\|<1.$
	Then  $  ( \{B_n\}_{n},  \{\Psi_n\}_{n} ) $ is a factorable weak OVF  with bounds 
	\begin{align*}
	\frac{1-\sum_{n=1}^\infty\|A_n-B_n\|\|\Psi_n(S_{A,\Psi}^*)^{-1}\|}{\|(S_{A,\Psi}^*)^{-1}\|}\quad \text{ and } \quad \|\theta_\Psi\|\left(\left(\sum_{n=1}^\infty\|A_n-B_n\|^2\right)^{1/2}+\|\theta_A\|\right) .
	\end{align*}  
\end{theorem}
\begin{proof}
	Let $ \alpha =\sum_{n=1}^\infty\|A_n-B_n\|^2 $ and $\beta =\sum_{n=1}^\infty\|A_n-B_n\|\|\Psi_n(S_{A,\Psi}^*)^{-1}\|$. For  $m=1,2,\dots $ and for every $ y$ in $ \ell^2(\mathbb{N})\otimes \mathcal{H}_0$, 
	\begin{align*}
	\left\| \sum\limits_{n=1}^mB_n^*L_n^*y\right\|&\leq \left\| \sum\limits_{n=1}^m(A_n^*-B_n^*)L_n^*y\right\|+\left\| \sum\limits_{n=1}^mA_n^*L_n^*y\right\|\\
	&\leq \sum\limits_{n=1}^m\|A_n-B_n\|\|L_n^*y\|+\left\| \sum\limits_{n=1}^mA_n^*L_n^*y\right\| \\
	&\leq \left( \sum\limits_{n=1}^m\|A_n-B_n\|^2\right)^\frac{1}{2}\left( \sum\limits_{n=1}^m\|L_n^*y\|^2\right)^\frac{1}{2}+\left\| \sum\limits_{n=1}^mA_n^*L_n^*y\right\|\\
	&\leq \alpha^\frac{1}{2} \left( \sum\limits_{n=1}^m\|L_n^*y\|^2\right)^\frac{1}{2}+\left\| \sum\limits_{n=1}^mA_n^*L_n^*y\right\|\\
	&=\alpha^\frac{1}{2} \left\langle  \sum\limits_{n=1}^mL_nL_n^*y, y\right\rangle ^\frac{1}{2}+\left\| \sum\limits_{n=1}^mA_n^*L_n^*y\right\|,
	\end{align*}
	which converges to $\sqrt{\alpha}\|y\|+\|\theta_A^*y\|$. Hence 
	$\theta_B$ exists and $\|\theta_B\|\leq \sqrt{\alpha}+\|\theta_A\|$. Therefore  $S_{B,\Psi}=\theta_\Psi^*\theta_B=\sum_{n=1}^\infty\Psi^*_nB_n$ exists.
	Now 
	\begin{align*}
	\|I_\mathcal{H}-S_{B,\Psi}(S_{A,\Psi}^*)^{-1}\|&=\left\|\sum_{n=1}^\infty A_n^*\Psi_n (S_{A,\Psi}^*)^{-1}-\sum_{n=1}^\infty B_n^*\Psi_n (S_{A,\Psi}^*)^{-1}\right\|\\
	&=\left\|\sum_{n=1}^\infty(A_n^*-B_n^*)\Psi_n (S_{A,\Psi}^*)^{-1}\right\|\\
	&\leq \sum_{n=1}^\infty\|A_n-B_n\|\|\Psi_n (S_{A,\Psi}^*)^{-1}\| =\beta<1.
	\end{align*}
	Therefore $S_{B,\Psi}(S_{A,\Psi}^*)^{-1}$ is invertible and $ \|(S_{B,\Psi}(S_{A,\Psi}^*)^{-1})^{-1}\|\leq 1/(1-\beta)$. Calculation of frame bounds is similar to proof of Theorem \ref{PERTURBATION RESULT 1}.
\end{proof}

{\onehalfspacing \chapter{CONCLUSION AND FUTURE WORK}\label{chap8} }
In Chapter \ref{chap2} we initiated the study of frames for metric spaces. Since metric spaces are more general objects and have less structure than Banach spaces, study of frames for metric spaces goes in a different way than that of frames for Hilbert as well as for Banach spaces. Arens-Eells space is used as a tool which allows to use the functional analysis technique  to Lipschitz functions. However, this works good only when the co domain of Lipschitz functions is a Banach space. Most of the results in Chapter \ref{chap2} are concentrated whenever the codomain of Lipschitz function is a Banach space. In future we are interested to work on  frames for arbitrary metric spaces.

In Chapter \ref{chap3} we defined multipliers for metric spaces. We obtained some fundamental properties of multipliers. One of the future work is to explore further on multipliers in metric spaces.

In Chapter \ref{chap4} we studied a special class of approximate Schauder frames. We characterized a class of approximate Schauder frames and its duals. It is planned to obtain the description of frames and its duals for Banach spaces.

In Chapter \ref{chap5} we initiated the study of  the series $\sum_{n=1}^{\infty}\Psi_n^*A_n$. We mainly obtained results whenever this series is factored as the product of two bounded linear operators. In the future we are planning to study the series without factorability condition. We are also interested in studying path-connectedness of weak OVFs and try to get a result similar to Theorem \ref{KAFTALPATHCONNECTED}.

\leavevmode\newpage
\leavevmode\newpage
\addcontentsline{toc}{chapter}{APPENDIX A: DILATIONS OF LINEAR MAPS ON VECTOR SPACES}
\par~
\begin{center}
	\textbf{{\fontsize{16}{1em}\selectfont APPENDIX A: DILATIONS OF LINEAR MAPS ON VECTOR SPACES}} \\
\end{center}

{\onehalfspacing \section{DILATIONS OF FUNCTIONS ON SETS}
	One of the most useful results in the study of isometries on Hilbert spaces is the Wold decomposition. It describes the structure of an isometry. It uses the notion of a shift. 
	\begin{definition}(cf. \cite{NAGY})\label{SHIFTDEFINITION}
		Let $\mathcal{H}$ be a Hilbert space. An operator  $T:\mathcal{H}\to \mathcal{H}$ is called  a \textbf{shift} if   $\cap _{n=0}^\infty T^n(\mathcal{H})=\{0\}$.	
	\end{definition}
	\begin{theorem}(cf. \cite{NAGY, WOLD})
		(\textbf{Wold decomposition}) Let $T$ be an isometry on a Hilbert space $\mathcal{H}$. Then  $\mathcal{H}$ decomposes uniquely as $\mathcal{H}=\mathcal{H}_u\oplus \mathcal{H}_s$, where $\mathcal{H}_u$ and $\mathcal{H}_s$ are   $T$-reducing subspaces of $\mathcal{H}$, $T_{|\mathcal{H}_u}:\mathcal{H}_u\to \mathcal{H}_u$ is a  unitary and $T_{|\mathcal{H}_s}:\mathcal{H}_s \to \mathcal{H}_s$ is a shift.
	\end{theorem}	
	Using functional calculus and Weierstrass polynomial approximation theorem, Halmos in 1950 proved an important result that every contraction on a Hilbert space can be lifted to unitary. 
	\begin{theorem}(\cite{HALMOSORIGINAL}) \label{HALMOSDILATION}(\textbf{Halmos dilation})
		Let $\mathcal{H}$ be a Hilbert space and $T:\mathcal{H}\to \mathcal{H}$ be a contraction. Then the operator 
		\begin{align*}
		U\coloneqq \begin{pmatrix}
		T & \sqrt{I-TT^*}   \\
		\sqrt{I-T^*T} & -T^*   \\
		\end{pmatrix}
		\end{align*}is unitary on 	$\mathcal{H}\oplus \mathcal{H}$. In other words,
		\begin{align*}
		T=P_\mathcal{H}U_{|\mathcal{H}},
		\end{align*}
		where $P_\mathcal{H}:\mathcal{H}\oplus \mathcal{H}\to \mathcal{H}\oplus \mathcal{H}$ is the orthogonal projection onto $\mathcal{H}$.
	\end{theorem}
	Three years later, Sz. Nagy extended the result of Halmos which reads as follows.
	\begin{theorem}(\cite{NAGYPAPER})\label{NAGYTHEOREM} (\textbf{Sz. Nagy dilation})
		Let $\mathcal{H}$ be a Hilbert space and $T:\mathcal{H}\to \mathcal{H}$ be a contraction. Then there exists a Hilbert space $\mathcal{K}$	which contains $\mathcal{H}$ isometrically and a unitary $U:\mathcal{K}\to \mathcal{K}$ such that 
		\begin{align*}
		T^n=P_\mathcal{H}U_{|\mathcal{H}},^n \quad \forall n=1, 2,\dots, 
		\end{align*}
		where $P_\mathcal{H}:\mathcal{K}\to \mathcal{K}$  is the orthogonal projection onto $\mathcal{H}$.
	\end{theorem}
	Unitary operator $U$ in Theorem \ref{NAGYTHEOREM} is known as dilation operator and the space $\mathcal{K}$ is called as dilation space. If 
	\begin{align*}
	\mathcal{K}=\overline{\text{span}}\{U^nh:, n \in \mathbb{Z}_+, h \in \mathcal{H}\},
	\end{align*}
	then $(\mathcal{K},U)$ is said to be a minimal dilation.  It is known that in Theorem \ref{NAGYTHEOREM}, the space $\mathcal{K}$ can be taken as a minimal space.
	It was \cite{SCHAFFER} who gave a proof of Sz. Nagy dilation theorem using infinite matrices. In the following theorem, $\oplus_{n=-\infty}^{\infty} \mathcal{H}$ is the Hilbert space defined by 
	\begin{align*}
	\oplus_{n=-\infty}^{\infty} \mathcal{H}\coloneqq \left\{ \{h_n\}_{n=-\infty}^\infty, h_n \in \mathcal{H}, \forall n \in \mathbb{Z}, \sum_{n=-\infty}^{\infty}\|h_n\|^2<\infty\right\}
	\end{align*}
	with respect to  the inner product 
	\begin{align*}
	\langle \{h_n\}_{n=-\infty}^\infty, \{g_n\}_{n=-\infty}^\infty\rangle \coloneqq \sum_{n=-\infty}^{\infty}\langle h_n, g_n \rangle, \quad \forall \{h_n\}_{n=-\infty}^\infty, \{g_n\}_{n=-\infty}^\infty \in \oplus_{n=-\infty}^{\infty} \mathcal{H}.
	\end{align*}
	\begin{theorem}(\cite{SCHAFFER})\label{SCHAFFERTHEOREM}
		Let $\mathcal{H}$ be a Hilbert space and $T:\mathcal{H}\to \mathcal{H}$ be a contraction.  Let $U\coloneqq [u_{n,m}]_{-\infty < n,m< \infty}$ be the \textbf{Schaffer operator} defined on 
		$\oplus_{n=-\infty}^{\infty} \mathcal{H}$ given by  the infinite matrix defined as follows:
		\begin{align*}
		&u_{0,0}\coloneqq T, \quad u_{0,1}\coloneqq 	\sqrt{I-TT^*}, \quad u_{-1,0}\coloneqq 	\sqrt{I-T^*T},\\
		&  u_{-1,1}\coloneqq -T^*~,  \quad u_{n,n+1}\coloneqq I, ~\forall n \in \mathbb{Z}, n\neq 0,1, \quad u_{n,m} \coloneqq 0, \quad  \text{ otherwise},
		\end{align*}
		i.e., 
		\begin{align*}
		U=\begin{pmatrix}
		&\vdots &\vdots & \vdots & \vdots & \vdots & \\
		\cdots & 0 & I& 0 & 0&  0& \cdots & \\
		\cdots & 0 & 0& \sqrt{I-T^*T}& -T^*  & 0&\cdots  & \\
		\cdots & 0&0&\boxed{T}&\sqrt{I-TT^*}& 0&\cdots&\\
		\cdots & 0&0&0&0& I&\cdots &\\
		\cdots & 0&0&0&0& 0&\cdots &\\
		& \vdots &\vdots &\vdots &\vdots  & \vdots & \\
		\end{pmatrix}_{\infty\times \infty}
		\end{align*}
		where $T$ is in the $(0,0)$  position (which is in the box), is invertible  on 	$\oplus_{n=-\infty}^{\infty} \mathcal{H}$ and 
		\begin{align*}
		T^n=P_\mathcal{H}U_{|\mathcal{H}}^n,\quad \forall n\in \mathbb{N},
		\end{align*}
		where $P_\mathcal{H}:\oplus_{n=-\infty}^{\infty} \mathcal{H}\to \oplus_{n=-\infty}^{\infty} \mathcal{H}$ is the orthogonal   projection onto $\mathcal{H}$.
	\end{theorem}
	After a year of work of Sz. Nagy, it was Egervary  who observed that Halmos dilation of contraction can be extended finitely so that power of dilation will be dilation of power of contraction. 
	\begin{theorem}(\cite{EGERVARY}) \label{EGERVARY}(\textbf{N-dilation})
		Let $\mathcal{H}$ be a Hilbert space and $T:\mathcal{H}\to \mathcal{H}$ be a contraction. Let $N$ be a natural number. Then the operator 
		\begin{align*}
		U\coloneqq \begin{pmatrix}
		T & 0& 0 & \cdots &0 & \sqrt{I-TT^*}   \\
		\sqrt{I-T^*T} & 0& 0 & \cdots &0& -T^*   \\
		0&I&0&\cdots &0& 0\\
		0&0&I&\cdots &0 & 0\\
		\vdots &\vdots &\vdots & & \vdots &\vdots \\
		0&0&0&\cdots &0 & 0\\
		0&0&0&\cdots &I & 0\\
		\end{pmatrix}_{(N+1)\times (N+1)}
		\end{align*}is unitary on 	$\oplus_{k=1}^{N+1} \mathcal{H}$ and 
		\begin{align*}
		T^k=P_\mathcal{H}U_{|\mathcal{H}}^k,\quad \forall k=1, \dots, N,
		\end{align*}
		where $P_\mathcal{H}:\oplus_{k=1}^{N+1} \mathcal{H}\to \oplus_{k=1}^{N+1} \mathcal{H}$ is the orthogonal projection onto $\mathcal{H}$. 		
	\end{theorem}
	A very useful result which can be derived using Theorem \ref{EGERVARY} is the von Neumann's inequality. It was derived by \cite{VONNEUMANNINEQUALITYPAPER}  using the theory of analytic functions.
	\begin{theorem}(\textbf{von Neumann inequality}) (\cite{VONNEUMANNINEQUALITYPAPER, ORRGUIDED, RAINONE})
		Let $\mathcal{H}$ be a Hilbert space and $T:\mathcal{H}\to \mathcal{H}$ be a contraction. Then for every polynomial $p\in \mathbb{C}[z]$,
		\begin{align*}
		\|p(T)\|\leq \sup_{|z|=1}|p(z)|.
		\end{align*}	
	\end{theorem}
	Sz. Nagy's dilation theorem leads to the study  of dilating more than one operator which are commuting. After a decade of work of Sz. Nagy, Ando derived the following result.
	\begin{theorem}(\cite{ANDO})\label{ANDOTHEOREM} (\textbf{Ando dilation})
		Let $\mathcal{H}$ be a Hilbert space and $T_1, T_2:\mathcal{H}\to \mathcal{H}$ be commuting contractions.	Then there exist a Hilbert space $\mathcal{K}$	which contains $\mathcal{H}$ isometrically and a pair of commuting unitaries   $U_1, U_2:\mathcal{K}\to \mathcal{K}$ such that 
		\begin{align*}
		T_1^nT_2^m=P_\mathcal{H}U_1^n{U_2}_\mathcal{H}^m, \quad \forall n,m=1, 2, \dots, 
		\end{align*}
		where $P_\mathcal{H}:\mathcal{K}\to \mathcal{K}$  is the orthogonal projection onto $\mathcal{H}$.
	\end{theorem}
	An easy consequence of Ando dilation  is generalization of von Neumann inequality.
	\begin{theorem}(\textbf{Ando-von Neumann inequality}) (cf. \cite{BHATTACHARYYA, ANDO})
		Let $\mathcal{H}$ be a Hilbert space and $T_1,T_2:\mathcal{H}\to \mathcal{H}$ be  commuting contractions. Then for every polynomial $p\in \mathbb{C}[z, w]$,
		\begin{align*}
		\|p(T_1, T_2)\|\leq \sup_{|z|=|w|=1}|p(z,w)|.
		\end{align*}	
	\end{theorem}
	It is known that Ando dilation theorem can not be extended for more than two commuting contractions (cf. \cite{BHATTACHARYYA, PARROTT, VAROPOULOS, CRABBDAVIE, DRURY}). We next consider inter-twining lifting theorem. This says that any operator which intertwins contractions can be lifted so that the lifted operator intertwins dilation operator.
	\begin{theorem}(\cite{NAGYLIFTING})\label{ILT} (\textbf{Inter-twining lifting theorem}) Let $T_1:\mathcal{H}_1\to \mathcal{H}_1$, $T_2:\mathcal{H}_2\to \mathcal{H}_2$  be contractions, where $\mathcal{H}_1$, $\mathcal{H}_2$ are Hilbert spaces. Let $V_1:\mathcal{K}_1\to \mathcal{K}_1$, $V_2:\mathcal{K}_2\to \mathcal{K}_2$ be minimal isometric dilations of $T_1,T_2$, respectively. Assume that $S:\mathcal{H}_2\to \mathcal{H}_1$ is a bounded linear operator such that $T_1S=ST_2$. Then there exists a bounded linear operator $R:\mathcal{K}_2\to \mathcal{K}_1$ such that 
		\begin{align*}
		V_1R=RV_2, \quad P_{\mathcal{H}_1}R_{\mathcal{H}_2^\perp}=0, \quad P_{\mathcal{H}_1}R_{\mathcal{H}_2}=S, \quad  \|R\|=\|S\|.
		\end{align*}	
		Conversely if $R:\mathcal{K}_2\to \mathcal{K}_1$ is a bounded linear operator such that $V_1R=RV_2$ and $P_{\mathcal{H}_1}R_{\mathcal{H}_2^\perp}=0$, then $S\coloneqq P_{\mathcal{H}_1}R_{\mathcal{H}_2}$ satisfies $T_1S=ST_2$. 
	\end{theorem}
	Next theorem gives a characterization which gives a condition that a given operator in a larger space becomes a dilation of compression of it to a smaller space.
	\begin{theorem}(\cite{SARASON})  \label{SL}(\textbf{Sarason's lemma})
		Let $\mathcal{H}$ be a closed subspace of 	a Hilbert space $\mathcal{K}$ and $V:\mathcal{K}\to \mathcal{K}$ be a bounded linear operator. Define $T\coloneqq P_\mathcal{H} V_{|\mathcal{H}}$. Then $T^n= P_\mathcal{H} V^n_{|\mathcal{H}}$, for all $n\in \mathbb{N}$  if and only if there are closed subspaces $\mathcal{M}\subseteq \mathcal{N}\subseteq\mathcal{K}$ both are invariant for $V$  such that 
		\begin{align*}
		\mathcal{H}=\mathcal{N}\ominus\mathcal{M},
		\end{align*}
		where $\mathcal{N}\ominus\mathcal{M}$ denotes the orthogonal complement of $\mathcal{M}$ in $\mathcal{N}$.
	\end{theorem}
	Following  Theorems \ref{HALMOSDILATION},  \ref{NAGYTHEOREM},  \ref{EGERVARY}, \ref{ANDOTHEOREM}, \ref{ILT} and \ref{SL},  extension of contractions on Hilbert spaces became an active area of research, known as  dilation theory   (\cite{NAGY, LEVYSHALIT, ARVESON, ORRGUIDED, AMBROZIE, AGLERCARTHY, PAULSENCOMPLETELY, PISIERSIMILARITY}).  This study of contractions motivated  the study of  contractions and other classes of operators not only on Hilbert spaces, but also on Banach spaces (\cite{FACKLER, STROESCU, AKCOGLU}).

	Recently,  Bhat, De, and Rakshit abstracted the key ingredients in Halmos and Sz. Nagy dilation theorem and set up a set theoretic version of dilation theory. Following is the fundamental observation which lead Bhat, De, and Rakshit to set up a  set theoretic notion of dilation theory. 
	\begin{enumerate}[label=(\roman*)]
		\item \textbf{There is an embedding $i$ of the given space in a larger space}. 
		\item  \textbf{There is a nice map in the larger space}.
		\item \textbf{There is an idempotent from the larger space onto the given space}.
	\end{enumerate}
	These observations can be picturized using the following commutative diagram.
	\begin{center}
		\[
		\begin{tikzcd}
		&\mathcal{K}  \arrow[r,"U^n"]& \mathcal{K} \arrow[d,"P_{\mathcal{H}}"]\\
		&\mathcal{H} \arrow[u,"i"] \arrow[r,"T^n"] & \mathcal{H}
		\end{tikzcd}
		\]
	\end{center}
	\begin{definition}(\cite{BHATDERAKSHITH})\label{BHATSET}
		Let $\mathscr{A}$ be a (non empty)	set and $f:\mathscr{A}\to \mathscr{A}$ be a map. An \textbf{injective power dilation} of $f$ is a quadruple $(\mathscr{B}, i, g,p)$, where $\mathscr{B}$ is a set, $i:\mathscr{A}\to \mathscr{B}$,  $v:\mathscr{B}\to \mathscr{B}$ are injective maps,  $p:\mathscr{B}\to \mathscr{B}$ is an idempotent map such that $p(\mathscr{B})=i(\mathscr{A})$ and 
		\begin{align}\label{SETDILATIONEQUATION}
		i(f^n(a))=p(g^n(i(a))), \quad \forall a \in \mathscr{A}, \forall n \in \mathbb{Z}_+.
		\end{align}
		A dilation $(\mathscr{B}, i, g,p)$ of $f$ is said to be \textbf{minimal} if 
		\begin{align*}
		\mathscr{B}=\bigcup\limits_{n=0}^\infty g^n(i(\mathscr{A})).
		\end{align*}
	\end{definition}
	Equation \ref{SETDILATIONEQUATION} says that the following diagram commutes for all $n$. 
	\begin{center}
		\[
		\begin{tikzcd}
		\mathscr{B}	 \arrow[r,"g^n"]&\mathscr{B}  \arrow[r,"p"]& \mathscr{B}\\
		&\mathscr{A} \arrow[ul,"i"] \arrow[r,"f^n"] & \mathscr{A}\arrow[u,"i"]
		\end{tikzcd}
		\]
	\end{center}
	Bhat, De, and Rakshit succeeded in obtaining fundamental theorems of dilations. We now recall these results.
	\begin{definition}(\cite{BHATDERAKSHITH}) (\textbf{Set shifts})
		Let $\mathscr{A}$ be a 	set. A map $f:\mathscr{A}\to \mathscr{A}$ is said to be \textbf{shift} if $\cap_{n=0}^\infty f^n(\mathscr{A})=\emptyset$.
	\end{definition}
	\begin{theorem}(\cite{BHATDERAKSHITH}) (\textbf{Wold decomposition for sets})
		Let  $f:\mathscr{A}\to \mathscr{A}$	be an injective map. Then $\mathscr{A}$ decomposes uniquely as $\mathscr{A}=\mathscr{A}_b \sqcup \mathscr{A}_s$, where $\mathscr{A}_b$, $\mathscr{A}_s$ are invariant for $f$,  $f_{|\mathscr{A}_b}$ is a bijection and $f_{|\mathscr{A}_s}$ is a shift.
	\end{theorem}
	\begin{theorem}(\cite{BHATDERAKSHITH}) (\textbf{Halmos dilation for sets})
		Let  $f:\mathscr{A}\to \mathscr{A}$	be a map. Define $\mathscr{B}\coloneqq\mathscr{A}\times \{0,1\}$,
		\begin{align*}
		&i:\mathscr{A}\ni a \mapsto  (a,0)\in \mathscr{B}\\
		&g:\mathscr{B}\ni (a,m)\mapsto (a,1-m)\in \mathscr{B}
		\end{align*}
		and $p:\mathscr{B}\to \mathscr{B}$ by 
		\begin{align*}
		p(a,m)\coloneqq \bigg\{\begin{array}{ll}
		(a,0) & \text{ if } m=0\\
		(f(a),0) & \text{ if } m=1. \\
		\end{array} 
		\end{align*}
		Then $i$ is injective, $g$ is bijective, $p$ is idempotent and 	
		\begin{align*}
		i(f(a))= p(g(i(a))),\quad \forall a \in \mathscr{A}.
		\end{align*}
	\end{theorem}
	\begin{theorem}(\cite{BHATDERAKSHITH})\label{SZNAGYSET} (\textbf{Sz. Nagy dilation  for sets})
		Every map $f:\mathscr{A}\to \mathscr{A}$ admits a minimal injective power dilation.	
	\end{theorem}
	In \cite{BHATDERAKSHITH}, a particular type of minimal injective dilation, called as \textbf{standard dilation} was defined. This dilation is defined as follows. Let $f:\mathscr{A}\to \mathscr{A}$ be a map. Define
	\begin{align*}
	&\mathscr{B} \coloneqq \mathscr{A}\times \mathbb{Z}_+,\\
	&i(a)\coloneqq (a,0), \quad \forall a \in \mathscr{A},\\
	&g(a,m)\coloneqq (a,m+1), \quad \forall (a,m) \in \mathscr{B},\\
	&p(a,m)\coloneqq (f^m(a),0), \quad \forall (a,m) \in \mathscr{B}.
	\end{align*}
	Then $(\mathscr{B}, i, g,p)$ is a minimal dilation of $f$.
	\begin{theorem}(\cite{BHATDERAKSHITH}) (\textbf{Inter-twining lifting theorem for sets})
		Let $f_1:\mathscr{A}_1\to \mathscr{A}_1$, $f_2:\mathscr{A}_2\to \mathscr{A}_2$	be maps and $(\mathscr{B}_1, i_1, g_1,p_1)$, $(\mathscr{B}_2, i_2, g_2,p_2)$ be their standard dilations, respectively. Suppose $s:\mathscr{A}_2\to \mathscr{A}_1$ is a function such that $sf_2=f_1s$. Then there exists a map $r:\mathscr{B}_2\to \mathscr{B}_1$ such that 
		\begin{align*}
		rg_2=g_1r,\quad rp_2=p_1r, \quad ri_2=i_1s.
		\end{align*}
		Conversely if $r:\mathscr{B}_2\to \mathscr{B}_1$is a map such that $rg_2=g_1r$, $ rp_2=p_1r$, then there exists a map  $s:\mathscr{A}_2\to \mathscr{A}_1$ such that $ri_2=i_1s$ and $sf_2=f_1s$.
	\end{theorem}
	\begin{theorem}(\cite{BHATDERAKSHITH}) (\textbf{Ando dilation for sets})
		Let $\mathbb{J}$ be an index set, $\{f_j\}_{j\in \mathbb{J}}$ be a family of commuting functions on $\mathscr{A}$. Then there exists a quadruple $(\mathscr{B}, i, \{g_j\}_{j\in \mathbb{J}},p)$, where $\mathscr{B}$ is a set, $i:\mathscr{A}\to \mathscr{B}$ is an injective map, $\{v_j\}_{j\in \mathbb{J}}$ be a family of commuting functions on $\mathscr{B}$, $p:\mathscr{B}\to \mathscr{B}$ is idempotent such that 
		\begin{align*}
		i(f_{j_1}f_{j_2}\cdots f_{j_k}(a))=p(g_{j_1}f_{g_2}\cdots f_{g_k}(i(a))),\quad \forall j_1, \dots, j_k \in \mathbb{J}, \forall a \in \mathscr{A}.
		\end{align*}
	\end{theorem}
	\begin{theorem}(\cite{BHATDERAKSHITH})\label{SARASONLEMMASET} (\textbf{Sarason's lemma for sets})
		Let $g:\mathscr{B}\to \mathscr{B}$ be an injective map and let $\mathscr{A}\subseteq\mathscr{B}$. Suppose $f:\mathscr{A}\to \mathscr{A}$ is a map such that $f(a)=g(a)$ for all $a \in \mathscr{A}$ with $g(a) \in \mathscr{A}$. Suppose $\mathscr{A}=\mathscr{A}_2\setminus \mathscr{A}_1$, where $\mathscr{A}_1,$ and $\mathscr{A}_2$ are invariant under $g$. Then there exists a map $p:\mathscr{B}\to \mathscr{B}$ such that $p^2=p$, $p(\mathscr{B})=\mathscr{A}$ and 
		\begin{align*}
		pg^n(a)=f^n(a), \quad \forall n \in \mathbb{N}, \forall a \in \mathscr{A}.
		\end{align*}
	\end{theorem}

\section{WOLD DECOMPOSITION, HALMOS DILATION AND N-DILATION FOR VECTOR SPACES}
In this appendix we consider vector spaces (need not be finite dimensional) over arbitrary fields. We note that the Definition \ref{SHIFTDEFINITION} of shift of an operator on a Hilbert space  does not use the Hilbert space structure. Thus it can be formulated for vector spaces without modifications.
\begin{definition}
	Let $\mathcal{V}$ be a  vector space and $T:\mathcal{V}\to  \mathcal{V}$ be a linear map. The map $T$ is said to be a \textbf{shift} if $\cap _{n=0}^\infty T^n(\mathcal{V})=\{0\} $.	
\end{definition}
\begin{theorem}(\textbf{Wold decomposition for vector spaces})
	Let $T$ be an injective linear map on a vector  space $\mathcal{V}$. Then  $\mathcal{V}$ decomposes  as $\mathcal{V}=\mathcal{V}_b\oplus \mathcal{V}_s$, where $\mathcal{V}_b$ is a  $T$-invariant subspace of $\mathcal{V}$, $T_{|\mathcal{V}_b}:\mathcal{V}_b\to \mathcal{V}_b$ is a bijection and $T_{|\mathcal{V}_s}:\mathcal{V}_s \to \mathcal{V}$ is a shift.	
\end{theorem}
\begin{proof}
	Define $\mathcal{V}_b\coloneqq \cap _{n=0}^\infty T^n(\mathcal{V})$ and let $\mathcal{V}_s$ be a vector space complement of $\mathcal{V}_b$ in $\mathcal{V}$. We clearly have 	$\mathcal{V}=\mathcal{V}_b\oplus \mathcal{V}_s$. Now $T(\mathcal{V}_b)=T (\cap _{n=0}^\infty T^n(\mathcal{V}))\subseteq \cap _{n=0}^\infty T^n(\mathcal{V})=\mathcal{V}_b$. Thus  $\mathcal{V}_b$ is a  $T$-invariant subspace of $\mathcal{V}$. We now try to show  that $T_{|\mathcal{V}_b}$ is a bijection. Since $T$ is already injective, it suffices to show that $T_{|\mathcal{V}_b}$ is surjective. Let $y\in \mathcal{V}_b$. Then there exists a sequence $\{x_n\}_{n=1}^\infty$ in $\mathcal{V}$ such that $y=Tx_1=T^2x_2=T^3x_3=\cdots .$ Since $T$ is injective, we then have $x_1=Tx_2=T^2x_2=\cdots $. Therefore  $y=Tx_1$ and $x_1\in \mathcal{V}_b$. Thus  $T_{|\mathcal{V}_b}$ is surjective. We are now left with proving that $T_{|\mathcal{V}_s}$ is a shift. Let $y\in \cap _{n=0}^\infty (T_{|\mathcal{V}_s})^n(\mathcal{V}_s)\subseteq (\cap _{n=0}^\infty T^n(\mathcal{V}))\cap \mathcal{V}_s= \mathcal{V}_b\cap \mathcal{V}_s$.  Hence $y=0$ which completes the proof. 
\end{proof}
Since vector space complements are not unique, we do not have uniqueness in Wold decomposition for vector spaces. We next derive Halmos dilation for vector spaces.
\begin{theorem}(\textbf{Halmos dilation for vector spaces})\label{HALMOSVECTORSPACE}
	Let $\mathcal{V}$ be a vector space  and 	$T: \mathcal{V} \to \mathcal{V}$ be a linear map. Then the operator 
	\begin{align*}
	U\coloneqq \begin{pmatrix}
	T & I   \\
	I & 0  \\
	\end{pmatrix}
	\end{align*}
	is invertible  on 	$\mathcal{V}\oplus \mathcal{V}$. In other words,
	\begin{align*}
	T=P_\mathcal{V}U_{|\mathcal{V}},
	\end{align*}
	where $P_\mathcal{V}:\mathcal{V}\oplus \mathcal{V}\to \mathcal{V}\oplus \mathcal{V}$ is the first coordinate  projection onto $\mathcal{V}$.
\end{theorem}
\begin{proof}
	It suffices to produce inverse map for $U$. A direct calculation says that 	
	\begin{align*}
	V\coloneqq \begin{pmatrix}
	0 & I   \\
	I & -T  \\
	\end{pmatrix}
	\end{align*}
	is the inverse of $U$.
\end{proof}
In the sequel, any invertible operator of the form 
\begin{align*}
\begin{pmatrix}
T & B   \\
C & D  \\
\end{pmatrix},
\end{align*}
where  $B,C,D:\mathcal{V} \to \mathcal{V}$ are linear operators, will be called as a \textbf{Halmos dilation} of $T$.
Now we observe that Halmos dilation for vector spaces  is not unique. Using the theory of \textbf{block matrices} (\cite{LUSHIOU}) we can produce a variety of Halmos dilations for a given operator. Following are some classes of Halmos dilations.
\begin{enumerate}[label=(\roman*)]
	\item If $T: \mathcal{V} \to \mathcal{V}$ is an invertible linear map and the linear operators $B,C,D:\mathcal{V} \to \mathcal{V}$ are such that $D-CT^{-1}B$ is invertible, then 
	the operator 
	\begin{align*}
	U\coloneqq \begin{pmatrix}
	T & B  \\
	C & D  \\
	\end{pmatrix}
	\text{ is a Halmos dilation of $T$ on $\mathcal{V}\oplus \mathcal{V}$ whose inverse is }
	\end{align*}
	\begin{align*}
	\begin{pmatrix}
	T^{-1} +T^{-1}B(D-CT^{-1}B)^{-1}& -T^{-1}B(D-CT^{-1}B)^{-1}   \\
	-(D-CT^{-1}B)^{-1}CT^{-1} & (D-CT^{-1}B)^{-1}  \\
	\end{pmatrix}.
	\end{align*}
	\item $D: \mathcal{V} \to \mathcal{V}$ is an invertible linear map and the linear operators $B,C:\mathcal{V} \to \mathcal{V}$ are such that $T-BD^{-1}C$ is invertible, then 
	the operator 
	\begin{align*}
	\begin{pmatrix}
	T & B  \\
	C & D  \\
	\end{pmatrix}
	\text{ is a Halmos dilation of $T$ on $\mathcal{V}\oplus \mathcal{V}$ whose inverse is }
	\end{align*}
	\begin{align*}
	\begin{pmatrix}
	(T-BD^{-1}C)^{-1} & -(T-BD^{-1}C)^{-1}BD^{-1}   \\
	-D^{-1}C(T-BD^{-1}C)^{-1} & D^{-1}+D^{-1}C(T-BD^{-1}C)^{-1}BD^{-1}  \\
	\end{pmatrix}.
	\end{align*}
	\item $B: \mathcal{V} \to \mathcal{V}$ is an invertible linear map and the linear operators $C,D:\mathcal{V} \to \mathcal{V}$ are such that $C-DB^{-1}T$ is invertible, then 
	the operator 
	\begin{align*}
	\begin{pmatrix}
	T & B  \\
	C & D  \\
	\end{pmatrix}
	\text{ is a Halmos dilation of $T$ on $\mathcal{V}\oplus \mathcal{V}$ whose inverse is }
	\end{align*}
	\begin{align*}
	\begin{pmatrix}
	-(C-DB^{-1}T)^{-1}DB^{-1} &  (C-DB^{-1}T)^{-1}  \\
	B^{-1}+B^{-1}T(C-DB^{-1}T)^{-1}DB^{-1} & -B^{-1}T(C-DB^{-1}T)^{-1} \\
	\end{pmatrix}.
	\end{align*}	
	\item $C: \mathcal{V} \to \mathcal{V}$ is an invertible linear map and the linear operators $B,D:\mathcal{V} \to \mathcal{V}$ are such that $B-TC^{-1}D$ is invertible, then 
	the operator 
	\begin{align*}
	\begin{pmatrix}
	T & B  \\
	C & D  \\
	\end{pmatrix}
	\text{ is a Halmos dilation of $T$ on $\mathcal{V}\oplus \mathcal{V}$ whose inverse is }
	\end{align*}
	\begin{align*}
	\begin{pmatrix}
	-C^{-1}D	(B-TC^{-1}D)^{-1} &  C^{-1}+C^{-1}D(B-TC^{-1}D)^{-1}TC^{-1}  \\
	(B-TC^{-1}D)^{-1} &  -(B-TC^{-1}D)^{-1}TC^{-1}\\
	\end{pmatrix}.
	\end{align*}
\end{enumerate}
Recently,  \cite{BHATMUKHERJEE} proved that there is certain kind of uniqueness of Halmos dilation for strict contractions in Hilbert spaces, as shown below.
\begin{theorem}(\cite{BHATMUKHERJEE})\label{BM}
	Let $\mathcal{H}$ be a finite dimensional Hilbert space and $T:\mathcal{H}\to \mathcal{H}$ be a strict contraction (i.e., $\|T\|<1$).	Then Halmos dilation of $T$ on $\mathcal{H}\oplus \mathcal{H}$ is unitarily equivalent to 
	\begin{align*}
	\begin{pmatrix}
	T & -\sqrt{I-TT^*}W   \\
	\sqrt{I-T^*T} & T^*W  \\
	\end{pmatrix}, \quad \text{ for some unitary operator } W:\mathcal{H}\to \mathcal{H}.
	\end{align*}
\end{theorem}
We next derive a negative result to Theorem \ref{BM} for Halmos dilation in vector spaces.
\begin{theorem}\label{HALMOSTHEOREMVS}
	Let $\mathcal{V}$ be a finite dimensional vector space and $T:\mathcal{V} \to \mathcal{V}$ be a linear operator with nonzero trace. Then there are Halmos dilations of $T$ which are not similar. 	
\end{theorem}
\begin{proof}
	Note that 	\begin{align*}
	\begin{pmatrix}
	T & T-I   \\
	T+I & T \\
	\end{pmatrix}
	\end{align*}
	is an invertible operator and hence is a Halmos dilation of $T$. It is now enough to show that the matrices 
	\begin{align*}
	\begin{pmatrix}
	T & T-I   \\
	T+I & T \\
	\end{pmatrix} \quad \text{ and } \quad \begin{pmatrix}
	T & I   \\
	I & 0 \\
	\end{pmatrix} 
	\end{align*} 
	are not similar. Since $\mathcal{V}$ is finite dimensional, we can use the property of trace map to conclude that these matrices are not similar.
\end{proof}
Theorem \ref{HALMOSVECTORSPACE} can be generalized which gives vector space version of Theorem \ref{EGERVARY}.
\begin{theorem}(\textbf{N-dilation for vector spaces})\label{NDILATIONVECTOR}
	Let $\mathcal{V}$ be a vector space  and 	$T: \mathcal{V} \to \mathcal{V}$ be a linear map. Let $N$ be a natural number. Then the operator 
	\begin{align*}
	U\coloneqq \begin{pmatrix}
	T & 0& 0 & \cdots &0 & I   \\
	I & 0& 0 & \cdots &0& 0   \\
	0&I&0&\cdots &0& 0\\
	0&0&I&\cdots &0 & 0\\
	\vdots &\vdots &\vdots & & \vdots &\vdots \\
	0&0&0&\cdots &0 & 0\\
	0&0&0&\cdots &I & 0\\
	\end{pmatrix}_{(N+1)\times (N+1)}
	\end{align*}is invertible  on 	$\oplus_{k=1}^{N+1} \mathcal{V}$ and 
	\begin{align}\label{FINITEDILATIONEQUATION}
	T^k=P_\mathcal{V}U_{|\mathcal{V}}^k,\quad \forall k=1, \dots, N,
	\end{align}
	where $P_\mathcal{V}:\oplus_{k=1}^{N+1} \mathcal{V}\to \oplus_{k=1}^{N+1} \mathcal{V}$ is the first coordinate  projection onto $\mathcal{V}$.
\end{theorem}
\begin{proof}
	A direct calculation of power of $U$ gives Equation (\ref{FINITEDILATIONEQUATION}). To complete the proof, now we need show that $U$ is invertible. Define
	\begin{align*}
	V\coloneqq \begin{pmatrix}
	0 & I& 0& 0 &  \cdots &0 & 0   \\
	0 & 0& I& 0 &  \cdots &0& 0   \\
	0&0&0& I&\cdots &0& 0\\
	0&0&0& 0&\cdots &0 & 0\\
	\vdots &\vdots  &\vdots & \vdots& & \vdots &\vdots \\
	0&0&0& 0&\cdots &0 & I\\
	I&-T&0& 0&\cdots &0 & 0\\
	\end{pmatrix}_{(N+1)\times (N+1)}.
	\end{align*}
	Then $UV=VU=I$. Thus $V$ is the inverse of $U$.	
\end{proof}
Note that the Equation (\ref{FINITEDILATIONEQUATION}) holds only upto $N$ and not for $N+1$ and higher natural numbers. We next derive vector space version of Theorem \ref{SCHAFFERTHEOREM}.  In the following theorem, $\oplus_{n=-\infty}^{\infty} \mathcal{V}$ is the vector space defined by 
\begin{align*}
\oplus_{n=-\infty}^{\infty} \mathcal{V}\coloneqq \left\{ \{x_n\}_{n=-\infty}^\infty, x_n \in \mathcal{V}, \forall n \in \mathbb{Z}, x_n\neq 0 
\text{ only for finitely many } n' \text{s}\right\}
\end{align*}
with respect to  natural operations.
\begin{theorem}\label{SCHAFFERVECTOR}(\textbf{Sz. Nagy dilation  for vector spaces})
	Let $\mathcal{V}$ be a vector space  and 	$T: \mathcal{V} \to \mathcal{V}$ be a linear map.  Let $U\coloneqq[u_{n,m}]_{-\infty \leq n,m\leq \infty}$ be the operator defined on 
	$\oplus_{n=-\infty}^{\infty} \mathcal{V}$ given by  the infinite matrix defined as follows:
	\begin{align*}
	u_{0,0}\coloneqq T, \quad u_{n,n+1}\coloneqq I, \quad \forall n \in \mathbb{Z},  \quad u_{n,m}\coloneqq 0 \quad  \text{ otherwise},
	\end{align*}
	i.e., 
	\begin{align*}
	U=\begin{pmatrix}
	&\vdots &\vdots & \vdots & \vdots & \vdots & \\
	\cdots & 0 & I& 0 & 0&  0& \cdots & \\
	\cdots & 0 & 0& I & 0& 0&\cdots  & \\
	\cdots & 0&0&\underline{T}&I& 0&\cdots&\\
	\cdots & 0&0&0&0& I&\cdots &\\
	\cdots & 0&0&0&0& 0&\cdots &\\
	& \vdots &\vdots &\vdots &\vdots  & \vdots & \\
	\end{pmatrix}_{\infty\times \infty}
	\end{align*}
	where $T$ is in the $(0,0)$  position (which is underlined), which is invertible  on 	$\oplus_{n=-\infty}^{\infty} \mathcal{V}$ and 
	\begin{align}\label{INFINITEDILATIONEQUATION}
	T^n=P_\mathcal{V}U_\mathcal{V}^n,\quad \forall n\in \mathbb{N},
	\end{align}
	where $P_\mathcal{V}:\oplus_{n=-\infty}^{\infty} \mathcal{V}\to \oplus_{n=-\infty}^{\infty} \mathcal{V}$ is the first coordinate  projection onto $\mathcal{V}$.
\end{theorem}
\begin{proof}
	We  get Equation (\ref{INFINITEDILATIONEQUATION}) by calculation of powers of operator $U$. The matrix   $V\coloneqq [v_{n,m}]_{-\infty < n,m< \infty}$ defined by  
	\begin{align*}
	v_{0,0}\coloneqq 0, \quad v_{1,-1}\coloneqq -T, \quad v_{n,n-1}\coloneqq I, \quad \forall n \in \mathbb{Z},  \quad v_{n,m}\coloneqq 0 \quad  \text{ otherwise},
	\end{align*}
	i.e., 
	\begin{align*}
	V=\begin{pmatrix}
	&\vdots &\vdots & \vdots & \vdots & \vdots & \\
	\cdots & I & 0& 0 & 0&  0& \cdots & \\
	\cdots & 0 & I& \underline{0} & 0& 0&\cdots  & \\
	\cdots & 0&-T&I&0& 0&\cdots&\\
	\cdots & 0&0&0&I& 0&\cdots &\\
	\cdots & 0&0&0&0& I&\cdots &\\
	& \vdots &\vdots &\vdots &\vdots  & \vdots & \\
	\end{pmatrix}_{\infty\times \infty}
	\end{align*}
	where $0$ is in the $(0,0)$  position (which is underlined), satisfies $UV=VU=I$ and hence $U$ is invertible which completes the proof.
\end{proof}

\section{MINIMAL DILATION, INTERTWINING LIFTING THEOREM AND VARIANT OF ANDO DILATION  FOR VECTOR SPACES}
 An important observation associated with Theorems \ref{HALMOSTHEOREMVS}, \ref{NDILATIONVECTOR} and \ref{SCHAFFERVECTOR} is that the dilation is not optimal, i.e., even if the given operator is invertible, then also $U$ is not same as $T$. To overcome this, next we move on with the definition of dilation given by Bhat, De, and Rakshit (\cite{BHATDERAKSHITH}). Set theoretic definition of dilation, given in Definition \ref{BHATSET} motivated  Bhat, De, and Rakshit, to introduce the dilation of linear maps on vector spaces.
\begin{definition}(\cite{BHATDERAKSHITH})
Let $\mathcal{V}$ be a vector space  and 	$T: \mathcal{V} \to \mathcal{V}$ be a linear map. A \textbf{linear injective dilation} of $T$ is a quadruple $(\mathcal{W}, I, U,P)$, where $\mathcal{W}$ is a vector space,   and 	$I: \mathcal{V} \to \mathcal{W}$ is an  injective  linear map, $U: \mathcal{W} \to \mathcal{W}$ is an injective  linear map, $P: \mathcal{W} \to \mathcal{W}$ is an idempotent linear map such that $P(\mathcal{W})=I(\mathcal{W})
$ and 
\begin{align*}
\text{(\textbf{Dilation equation})} \quad IT^nx=PU^nIx, \quad \forall n\in \mathbb{Z}_+,  \forall x \in  \mathcal{V}.
\end{align*}
A dilation $(\mathcal{W}, I, U,P)$ of $T$ is said to be \textbf{minimal} if 
\begin{align*}
\mathcal{W}=\operatorname{span}\{U^nIx:  n\in \mathbb{Z}_+,   x \in  \mathcal{V}\}.
\end{align*}	
\end{definition}
An easier way to remember the dilation equation is the following commutative diagram. 
\begin{center}
	\[
	\begin{tikzcd}
	\mathcal{W}	 \arrow[r,"U^n"]&\mathcal{W}  \arrow[r,"P"]& \mathcal{W}\\
	&\mathcal{V} \arrow[ul,"I"] \arrow[r,"T^n"] & \mathcal{V}\arrow[u,"I"]
	\end{tikzcd}
	\]
\end{center}
 Following result is the vector space version of Theorem \ref{SZNAGYSET}.
\begin{theorem}(\cite{BHATDERAKSHITH})\label{STANDARDDILATION} (\textbf{Minimal Sz. Nagy dilation  for sets})
	Every linear map $T: \mathcal{V} \to \mathcal{V}$ admits  minimal injective linear dilation.
\end{theorem}
\begin{proof}
	We reproduce the proof given by \cite{BHATDERAKSHITH} for the sake of future use. Define 
	\begin{align*}
	\mathcal{W}\coloneqq \left\{(x_n)_{n=0}^\infty :x_n \in \mathcal{V}, \forall n \in  \mathbb{Z}_+, x_n\neq 0 \text{ only for finitely many } n'\text{s}\right\}.
	\end{align*} 
	Clearly $\mathcal{W}$ is a vector space w.r.t. natural operations. Now define 
	\begin{align*}
	&	I:\mathcal{V} \ni x \mapsto (x, 0, \dots ) \in  \mathcal{W},\\
	&	U: \mathcal{W} \ni (x_n)_{n=0}^\infty \mapsto (0, x_0, \dots) \in \mathcal{W},\\
	&	P:\mathcal{W} \ni (x_n)_{n=0}^\infty \mapsto \sum_{n=0}^{\infty}IT^nx_n\in \mathcal{W}.
	\end{align*}
	Then $(\mathcal{W}, I, U,P)$ is a minimal injective linear dilation of $T$.
\end{proof}
We call the dilation $(\mathcal{W}, I, U,P)$ constructed in Theorem \ref{STANDARDDILATION} as the \textbf{standard dilation} of $T$. We next consider inter-twining lifting theorem.
\begin{theorem}(\textbf{Inter-twining lifting theorem for vector spaces}) Let $\mathcal{V}_1$, $\mathcal{V}_2$ be vector spaces,   $T_1: \mathcal{V}_1 \to \mathcal{V}_1$, $T_2: \mathcal{V}_2 \to \mathcal{V}_2$ be linear maps. Let $(\mathcal{W}_1, I_1, U_1,P_1)$, $(\mathcal{W}_2, I_2, U_2,P_2)$ be standard dilations of   $T_1$, $T_2$, respectively. If  $S: \mathcal{V}_2 \to \mathcal{V}_1$ is a linear map such that $T_1S=ST_2$, then there exists a linear map  $R: \mathcal{W}_2 \to \mathcal{W}_1$ such that 
	\begin{align}\label{INTERIN}
	U_1R=RU_2, \quad RP_2=P_1R, \quad RI_2=I_1S.
	\end{align}
	Conversely if $R: \mathcal{W}_2 \to \mathcal{W}_1$ is a linear map such that $U_1R=RU_2,  RP_2=P_1R$, then there exists a linear map   $S: \mathcal{V}_2 \to \mathcal{V}_1$ such that 
	\begin{align}\label{INTERSECOND}
	RI_2=I_1S, \quad T_1S=ST_2.
	\end{align}
\end{theorem}
\begin{proof}
	Define $R:\mathcal{W}_2 \ni (x_n)_{n=0}^\infty \mapsto (Sx_n)_{n=0}^\infty \in \mathcal{W}_1 $. We now verify three equalities in Equation (\ref{INTERIN}). Let $ (x_n)_{n=0}^\infty \in \mathcal{W}_2$. Then 
	\begin{align*}
	&	U_1R(x_n)_{n=0}^\infty=	U_1(Sx_n)_{n=0}^\infty=(0, S x_0, Sx_1, \dots), \\ 
	&RU_2(x_n)_{n=0}^\infty=R(0, x_0, x_1,\dots)=(0, S x_0, Sx_1, \dots), 
	\end{align*}
	\vspace{-1cm}
	\begin{align*}	
	RP_2(x_n)_{n=0}^\infty&=R\left(\sum_{n=0}^{\infty}I_2T_2^nx_n\right)=\sum_{n=0}^{\infty}RI_2T_2^nx_n\\
	&=\sum_{n=0}^{\infty}R(T_2^nx_n, 0, 0, \dots)=\sum_{n=0}^{\infty}(ST_2^nx_n, 0, 0, \dots), 
	\end{align*}
	\vspace{-1cm}
	\begin{align*}
	P_1R(x_n)_{n=0}^\infty&= P_1(Sx_n)_{n=0}^\infty=\sum_{n=0}^{\infty}I_1T_1^nSx_n\\
	&=\sum_{n=0}^{\infty}I_1ST_2^nx_n=\sum_{n=0}^{\infty}(ST_2^nx_n, 0, 0, \dots), 
	\end{align*}
	\vspace{-1cm}
	\begin{align*}
	& RI_2x=R(x, 0, 0, \dots)=(Sx, 0, 0, \dots), \quad I_1Sx=(Sx, 0, 0, \dots).
	\end{align*}
	We now consider the converse part. For this, first we have to define linear map $S$. Let $y \in \mathcal{V}_2$. Now $RP_2(y, 0, \dots)=P_1R(y, 0, \dots)\in I_1(\mathcal{V}_1)$ and  $I_1$ is injective implies that there exists a unique $x \in \mathcal{V}_2$ such that $RP_2(y, 0, \dots)=P_1R(y, 0, \dots)=I_1(x)$. We now define $Sy\coloneqq x.$ Then $S$ is well-defined and linear. Let $y \in \mathcal{V}_2$ and $x \in \mathcal{V}_2$ be such that $Sy=x$. Then $I_1Sy=RP_2(y, 0, \dots)=RI_2y$. Thus we verified first equality in  (\ref{INTERSECOND}). We are left with verification of second equality. We now calculate 
	\begin{align}\label{CONVERSE-1}
	RP_2U_2(x, 0, \dots)=RP_2(0, x, 0, \dots)=RI_2T_2x
	\end{align}
	and 
	\begin{align}\label{CONVERSEZERO}
	P_1U_1R(x, 0, \dots)&=P_1RU_2(x, 0, \dots)=P_1R(0,x,0, \dots)\\
	&=RP_2(0,x,0, \dots)=RI_2T_2x, \quad \forall x \in \mathcal{V}_2.
	\end{align}
	Given conditions produce
	\begin{align}\label{CONVERSE}
	RP_2U_2=P_1RU_2=P_1U_1R.
	\end{align}
	Equation (\ref{CONVERSE}) says that (\ref{CONVERSE-1}) and (\ref{CONVERSEZERO}) are equal which completes the proof.
\end{proof}
Following is a variant of Ando dilation for vector spaces.
\begin{theorem}(\textbf{Ando like dilation for vector spaces})
	Let $\mathcal{V}$ be a vector space  and 	$T, S: \mathcal{V} \to \mathcal{V}$ be commuting  linear maps.	Then there are dilations $(\mathcal{W}, I, U_1,P)$ and $(\mathcal{W}, I, U_2,P)$ of $T,S$ respectively, such that 
	\begin{align*}
	\begin{pmatrix}
	0_c& U
	\end{pmatrix}
	=\begin{pmatrix}
	0_r\\
	V
	\end{pmatrix}
	\end{align*}
	and 
	\begin{align*}
	\quad IT^nS^mx=PU^nV^mIx, \quad \forall n,m\in \mathbb{Z}_+,  \forall x \in  \mathcal{V},
	\end{align*}
	where $0_c$ denotes the infinite column matrix of zero vectors and $0_r$ denotes the infinite row matrix of zero vectors.
\end{theorem}
\begin{proof}
	We extend the construction in the proof of Theorem \ref{STANDARDDILATION}.	Define 
	\begin{align*}
	\mathcal{W}\coloneqq \bigg\{
	\begin{pmatrix}
	x_{0,0} & x_{0,1} & x_{0,2}&\cdots \\
	x_{1,0} & x_{1,1} & x_{1,2}&\cdots\\
	x_{2,0} & x_{2,1} & x_{2,2}&\cdots\\
	\vdots &\vdots &\vdots &\ddots 
	\end{pmatrix}_{\infty \times \infty }
	:x_{n,m} \in \mathcal{V}, \forall n,m \in  \mathbb{Z}_+, x_{n,m}\neq 0 \\
	\text{ only for finitely many } (n,m)'\text{s}\bigg\}.
	\end{align*}
	Then $\mathcal{W}$ becomes a vector space with respect to natural operations.  We now define the following four linear maps:
	\begin{align*}
	&	I:	\mathcal{V}\ni x \mapsto 	\begin{pmatrix}
	x & 0 & 0&\cdots \\
	0 & 0 & 0&\cdots\\
	0 & 0 & 0&\cdots\\
	\vdots &\vdots &\vdots &\ddots 
	\end{pmatrix} 
	\in \mathcal{W}\\
	&	U: \mathcal{W} \ni \begin{pmatrix}
	x_{0,0} & x_{0,1} & x_{0,2}&\cdots \\
	x_{1,0} & x_{1,1} & x_{1,2}&\cdots\\
	x_{2,0} & x_{2,1} & x_{2,2}&\cdots\\
	\vdots &\vdots &\vdots &\ddots 
	\end{pmatrix} \mapsto \begin{pmatrix}
	0&0&0\\
	x_{0,0} & x_{0,1} & x_{0,2}&\cdots \\
	x_{1,0} & x_{1,1} & x_{1,2}&\cdots\\
	\vdots &\vdots &\vdots &\ddots 
	\end{pmatrix} \in \mathcal{W}\\
	&	V: \mathcal{W} \ni \begin{pmatrix}
	x_{0,0} & x_{0,1} & x_{0,2}&\cdots \\
	x_{1,0} & x_{1,1} & x_{1,2}&\cdots\\
	x_{2,0} & x_{2,1} & x_{2,2}&\cdots\\
	\vdots &\vdots &\vdots &\ddots 
	\end{pmatrix} \mapsto \begin{pmatrix}
	0&x_{0,0} & x_{0,1} &\cdots \\
	0&	x_{1,0} & x_{1,1}&\cdots\\
	0& 	x_{2,0} & x_{2,1} &\cdots \\
	\vdots &	\vdots &\vdots  &\ddots 
	\end{pmatrix} \in \mathcal{W}\\
	&	P :\mathcal{W} \ni \begin{pmatrix} x_{0,0} & x_{0,1} & x_{0,2}&\cdots \\
	x_{1,0} & x_{1,1} & x_{1,2}&\cdots\\
	x_{2,0} & x_{2,1} & x_{2,2}&\cdots\\
	\vdots &\vdots &\vdots &\ddots 
	\end{pmatrix} \mapsto \sum_{m=0}^{\infty}\sum_{n=0}^{\infty}IT^nS^mx_{n,m} \in  \mathcal{W}.
	\end{align*}
	We then have 
	\begin{align*}
	\begin{pmatrix}
	0_c& U
	\end{pmatrix}=
	\begin{pmatrix}
	0&	0 & 0 & 0&\cdots \\
	0&	x_{0,0} & x_{0,1} & x_{0,2}&\cdots \\
	0&	x_{1,0} & x_{1,1} & x_{1,2}&\cdots\\
	0&	x_{2,0} & x_{2,1} & x_{2,2}&\cdots\\
	0&	\vdots &\vdots &\vdots &\ddots 
	\end{pmatrix}
	=\begin{pmatrix}
	0_r\\
	V
	\end{pmatrix}.
	\end{align*}
	Now $PU^nIx=IT^nx, $ $ 	PV^nIx=IS^mx$, $\forall x \in \mathcal{V}$, $\forall n,m\in \mathbb{Z}_+$. Hence  $(\mathcal{W}, I, U_1,P)$ and $(\mathcal{W}, I, U_2,P)$ are dilations of $T,S$, respectively. A calculation now shows that $ IT^nS^mx=PU^nV^mIx,  \forall n,m\in \mathbb{Z}_+,  \forall x \in  \mathcal{V}$.
\end{proof}
\textbf{Conclusion and future work : } In this appendix  we derived some basic results on dilation of linear maps. Since vector spaces are more general than Hilbert spaces and tools of Hilbert spaces will not work in vector space, we are interested to explore algebraic aspects of dilation theory.

\leavevmode\newpage
\addcontentsline{toc}{chapter}{APPENDIX B: COMMUTATORS CLOSE TO THE IDENTITY}
\par~
\begin{center}
	\textbf{{\fontsize{16}{1em}\selectfont APPENDIX B: COMMUTATORS CLOSE TO THE IDENTITY}} \\
\end{center}

\section{C*-ALGEBRAS}
Israel M. Gelfand  defined abstractly the notion of a complete  algebra (\cite{GELFAND}). These are Banach spaces in which we can multiply the elements and the multiplication enjoys continuity.
\begin{definition}(cf. \cite{ZHUBOOK})
	A Banach space $ \mathcal{A}$  over $ \mathbb{C} $ is said to be a unital \textbf{Banach algebra} if it is a unital algebra  and the multiplication satisfies the following:
	\begin{enumerate} [label=(\roman*)]
	\item $ \|xy\|\leq \|x\|\|y\|$,  $ \forall x, y \in\mathcal{A}.  $
	\item $\|e\|=1$, where $e$ is the multiplicative identity of $ \mathcal{A}$.
	\end{enumerate}
\end{definition}

\begin{example}(cf. \cite{ZHUBOOK})
	\begin{enumerate}[label=(\roman*)]
	\item If $K$ is a compact Hausdorff space, then the space $\mathcal{C}(K)$ of all complex-valued continuous functions on $K$ is a commutative unital Banach algebra w.r.t. sup-norm and pointwise multiplication.
	\item If  $\mathcal{X} $ is a Banach space, then    the collection $\mathcal{B}(\mathcal{X})$ of all bounded linear operators on $\mathcal{X}$  is a noncommutative unital Banach algebra w.r.t. operator-norm and operator composition.
	\end{enumerate}
\end{example}
\begin{proposition}(cf. \cite{ALLAN})
	Every unital Banach algebra $\mathcal{A}$ can be isometrically embedded in $\mathcal{B}(\mathcal{A})$.
\end{proposition}
One of the most important notion associated with the study of Banach algebras is the notion of spectrum.
\begin{definition}(cf. \cite{ZHUBOOK})
	Let $ \mathcal{A}$  be a unital Banach algebra with the identity $ e$. \textbf{Spectrum} of an element $ x $ in $ \mathcal{A}$ is the set of all complex numbers 
	$ \lambda$ such that $ \lambda e-x$ is not invertible. 
\end{definition}
\begin{theorem}(cf. \cite{ZHUBOOK})
	Spectrum of every  element of a unital Banach algebra is a nonempty compact subset of $\mathbb{C}$.
\end{theorem}
Following is the first fundamental theorem in the study of Banach algebras which characterizes Banach algebras using the information of spectrum.
\begin{theorem}(cf. \cite{ZHUBOOK}) (\textbf{Gelfand-Mazur theorem})
	If every nonzero element of a Banach algebra is invertible, then it is isometrically isomorphic to $\mathbb{C}$.
\end{theorem}
A subclass of Banach algebras known as C*-algebras allows to do most of the things which hold good for complex numbers. Notion of C*-algebras, for first time,  appeared in the work of  \cite{GELFANDNEUMARK}.
\begin{definition}(cf. \cite{ZHUBOOK})
	A unital Banach algebra $ \mathcal{A}$ is called a unital \textbf{C*-algebra} if there exists a map $*:\mathcal{A}\ni x \mapsto x^* \in \mathcal{A}$ such that following conditions hold.
	\begin{enumerate}[label=(\roman*)]
	\item $ ((x)^*)^*=x, \forall x \in\mathcal{A}. $
	\item $ (x+y)^*=x^*+y^*, \forall x, y \in\mathcal{A}. $
	\item $(\alpha x)^*=\overline{\alpha}x^*$, $\forall \alpha \in \mathbb{K}, \forall x \in\mathcal{A}$.
	\item $ (xy)^*=y^*x^*, \forall x, y \in\mathcal{A}. $
	\item $\|x^*x\|=\|x\|^2, \forall x \in\mathcal{A}. $
	\end{enumerate}
	A map $*:\mathcal{A}\ni x \mapsto x^* \in \mathcal{A}$ satisfying (i)-(iii) is called as  \textbf{involution}.
\end{definition}
Segal called the term C*-algebra; the letter `C' stands for uniformly closed. C*-algebras are also known as Gelfand-Naimark algebras (cf. \cite{PIETSCH}).
\begin{example}(cf. \cite{ZHUBOOK})
	\begin{enumerate}[label=(\roman*)]
		\item  If $K$ is a compact Hausdorff space, then $\mathcal{C}(K)$ is a commutative unital C*-algebra w.r.t. involution $f^*(x)\coloneqq\overline{f(x)}, \forall x \in K$.
		\item If $\mathcal{H}$ is a Hilbert space, then $\mathcal{B}(\mathcal{H})$ is a noncommutative unital  C*-algebra  w.r.t. operator adjoint. 
		\item If $\mathcal{H}$ is a Hilbert space, the the space $\mathcal{K}(\mathcal{H})$ of compact operators is a noncommutative C*-subalgebra of $\mathcal{B}(\mathcal{H})$. If $\mathcal{H}$ is infinite dimensional, then this algebra is non unital.
	\end{enumerate}
\end{example}
Following two results characterize unital C*-algebras.
\begin{theorem}(cf. \cite{ZHUBOOK}) (\textbf{Gelfand-Naimark theorem})
	If  $ \mathcal{A}$ is a commutative unital C*-algebra, then  $ \mathcal{A}$ is isometrically $*$-isomorphic to 
	$\mathcal{C}(K)$ for some compact Hausdorff space $K$.
\end{theorem}
\begin{theorem}(cf. \cite{ZHUBOOK}) (\textbf{Gelfand-Naimark-Segal theorem})\label{GELFANDNAIMARKSEGAL}
		Let $ \mathcal{A}$ be a  unital C*-algebra. Then  there exists a  Hilbert space $ \mathcal{H}$ such that  $ \mathcal{A}$  is isometrically $ *$-isomorphic 
		to a C*-subalgebra of $ \mathcal{B}(\mathcal{H})$.
	\end{theorem}
	\begin{example}(cf. \cite{KADISONRINGROSE})
		Consider the unital C*-algebra $ \mathcal{C}[0, 1]$. The map 
		\begin{align*}
	 &\pi :  \mathcal{C}[0, 1] \ni f \mapsto  \pi (f) \in  \mathcal{B}(\mathcal{L}^2[0, 1]);\\
	 & \pi (f):\mathcal{L}^2[0, 1]\ni g \mapsto (\pi (f))(g)	\coloneqq fg \in \mathcal{L}^2[0, 1]
		\end{align*}
	is an  isometric $ *$-isomorphism to a C*-subalgebra of $\mathcal{B}(\mathcal{L}^2[0, 1])$.
	\end{example}

{\onehalfspacing \section{COMMUTATORS CLOSE TO THE IDENTITY IN $\mathcal{B}(\mathcal{H})$}
	Let $n\in \mathbb{N}$ and $M_n(\mathbb{K})$ be the ring of $n$ by $n$ matrices over $\mathbb{K}$. Using the property of trace map
	we easily get that there does not exist $D, X   \in M_n(\mathbb{K})$ such that $DX-XD=1_{M_n(\mathbb{K})}$ (\cite{HALMOS}). This argument will not  work for bounded linear operators 
	on infinite dimensional Hilbert space since the map trace is not defined on the algebra $\mathcal{B}(\mathcal{H})$ of all bounded linear operators on an infinite dimensional Hilbert space $\mathcal{H}$ (it is defined for a proper subalgebra of $\mathcal{B}(\mathcal{H})$  known as the trace class operators (\cite{SCHATTEN})). Operators of the form $DX-XD$ are called as \textbf{commutator} of $D$ and $X$ and are denoted by $[D,X]$. An operator $T\in \mathcal{B}(\mathcal{H})$ is said to be a commutator if $T=[D,X]$, for some $D, X \in \mathcal{B}(\mathcal{H})$.\\
	Using 
	the property of spectrum of bounded linear operator,  Winter in 1947  proved that the following result.
	\begin{theorem}(\cite{WINTNER})\label{WINTNERTHEOREM}
		Let $\mathcal{H}$ be an infinite dimensional Hilbert space. Then there does not exist $D, X \in \mathcal{B}(\mathcal{H})$ such that  
		\begin{align}\label{COMMUTATORIDENTITY}
		[D,X]=1_{\mathcal{B}(\mathcal{H})}.
		\end{align}
	\end{theorem}
	After two years,  \cite{WIELANDT}  gave a 
	simple proof for the failure of  Equation \ref{COMMUTATORIDENTITY}. We note that the boundedness of operators is crucial in Theorem \ref{WINTNERTHEOREM}. Following example shows that Theorem \ref{WINTNERTHEOREM} fails for unbounded operators
	\begin{example}(\cite{HALMOS})
		Let $\mathcal{H}\coloneqq \mathcal{L}^2(\mathbb{R})$ and define 	
		\begin{align*}
		(Df)(x)\coloneqq \frac{d}{dx}f, \quad (Xf)(x)\coloneqq xf(x)
		\end{align*}
		Then $[D,X]=1_{\mathcal{B}(\mathcal{H})}$.
	\end{example}
	Theorem \ref{WINTNERTHEOREM} leads to the question that which operators on infinite dimensional Hilbert spaces can be written as commutators of operators? First partial answer was given by Halmos. 
	\begin{theorem}(cf. \cite{PUTNAM})
		Let $\mathcal{H}$ be an infinite dimensional Hilbert space. If $C\in  \mathcal{B}(\mathcal{H})$ is compact, then $I_{\mathcal{B}(\mathcal{H})}+C$ is not a commutator.
	\end{theorem}
	\cite{BROWNPEARCY} characterized the set of bounded operators which can be written as commutators. 
	\begin{theorem}(\cite{BROWNPEARCY})
		Let $\mathcal{H}$ be an infinite dimensional separable Hilbert space. Then an operator in $  \mathcal{B}(\mathcal{H})$ is a commutator if and only if it is not of the form $\lambda I_{\mathcal{B}(\mathcal{H})}+C$, where $\lambda$ is a nonzero scalar and $C$ is a compact operator. 
	\end{theorem}
	Following the paper of \cite{BROWNPEARCY}  there is a series of papers devoted to the  study of commutators
	on   sequence spaces, $\mathcal{L}^p$-spaces,  Banach spaces, C*-algebras, von Neumann algebras,   Banach *-algebras etc (\cite{SCHNEEBERGER, LAUSTSEN, DYKEMAFIGIELGARYWODZICKI, MARCOUX2006, STASINSKI, DESOVJOHNSON2010, KADISONLIUTHOM, YOOD, DESOVJOHNSONSCHECHTMAN, DYKEMASKRIPKA, DOSEV2009, KAFTALNGZHANG, MARCOU2010}).

	 It was \cite{POPA} who started  a quantitative study of commutators close to the identity operator.  He gave the 
	following quantitative bound given by Popa  for the product of norm of operators whenever the commutator is close to the identity.
	\begin{theorem}(\cite{POPA})\label{POPAFIRST}
		Let $\mathcal{H}$ be an infinite dimensional  Hilbert space. Let 	$D,X \in \mathcal{B}(\mathcal{H})$ be such that 
		\begin{align*}
		\|[D,X]-1_{\mathcal{B}(\mathcal{H})}\|\leq \varepsilon
		\end{align*} 
		for some $\varepsilon>0$. Then  
		\begin{align*}
		\|D\|\|X\|\geq\frac{1}{2}\log\frac{1}{\varepsilon}.
		\end{align*}
	\end{theorem}
	Now the problem in  Theorem \ref{POPAFIRST},   is the existence of $D,X \in \mathcal{B}(\mathcal{H})$ such that the commutator $[D,X]$ is close to the identity operator. This was again obtained by Popa which is stated in the following result. Given real $r $ and positive $s$, by $r=O(s)$ we mean that there is positive $\gamma$ such that $|r|\leq \gamma s$.
	\begin{theorem}(\cite{POPA, TAO})\label{POPASECOND}
		Let $\mathcal{H}$ be an infinite dimensional  Hilbert space.
		Then 	for each $0<\varepsilon\leq 1$,  there exist $D,X \in \mathcal{B}(\mathcal{H})$ with 
		\begin{align*}
		\|[D,X]-1_{\mathcal{B}(\mathcal{H})}\|\leq \varepsilon
		\end{align*} 
		and
		\begin{align*}
		\|D\|\|X\|=O(\varepsilon^{-2}).
		\end{align*}
	\end{theorem}
	Terence Tao improved Theorem \ref{POPASECOND} and obtained the following theorem.
	\begin{theorem}(\cite{TAO})\label{TAOTHEOREM}
		Let $\mathcal{H}$ be an infinite dimensional  Hilbert space.
		Then 	for each $0<\varepsilon\leq 1/2$,  there exist $D,X \in \mathcal{B}(\mathcal{H})$ with 
		\begin{align*}
		\|[D,X]-1_{\mathcal{B}(\mathcal{H})}\|\leq \varepsilon
		\end{align*} 
		such that 
		\begin{align*}
		\|D\|\|X\|=O\left(\log^5\frac{1}{\varepsilon}\right).
		\end{align*}
	\end{theorem}
	In (\cite{POPA}) there is another result about commutators. Let $\mathcal{K}(\mathcal{H})$ be the ideal of compact operators in $\mathcal{B}(\mathcal{H})$ and define 
	\begin{align*}
	\mathbb{C}+\mathcal{K}(\mathcal{H})\coloneqq \{\lambda. 1_{\mathcal{B}(\mathcal{H})}+T: \lambda \in \mathbb{C}, T \in \mathcal{K}(\mathcal{H})\}.
	\end{align*}
	\begin{theorem}(\cite{POPA})\label{POPACOMPACT}
		If 	$K \in \mathcal{B}(\mathcal{H})$ is such that 
		\begin{align*}
		\|A\|=O(1), \quad \|A\|=O(\text{dist}(A, \mathbb{C}+\mathcal{K}(\mathcal{H}))^\frac{2}{3}),
		\end{align*}
		then  there exist $D,X \in \mathcal{B}(\mathcal{H})$ with 
		\begin{align*}
		\|D\|\|X\|=O(1)\quad \text{ such that } \quad A=[D,X].
		\end{align*}
	\end{theorem}

\section{COMMUTATORS CLOSE TO THE IDENTITY IN \\
	UNITAL C*-ALGEBRAS}
We   
recall fundamentals of matrices over unital C*-algebras as given in (\cite{MURPHY}).  

Let  $\mathcal{A}$ be a unital C*-algebra. For $n\in \mathbb{N}$, $M_n(\mathcal{A})$ is defined as the set of all $n$ by $n$ matrices over $\mathcal{A}$. It is clearly an algebra with respect to natural matrix operations. We define the involution of an element  $A\coloneqq [a_{i,j}]_{1\leq i,j \leq n}\in M_n(\mathcal{A})$ as $A^*\coloneqq [a_{j,i}^*]_{1\leq i,j \leq n}$. Then $M_n(\mathcal{A})$ is a *-algebra.  From the Gelfand-Naimark-Segal theorem (Theorem \ref{GELFANDNAIMARKSEGAL}) there exists unique universal representation $(\mathcal{H}, \pi) $, where $\mathcal{H}$ is a Hilbert space, $\pi:M_n(\mathcal{A})\to M_n(\mathcal{B}(\mathcal{H}))$ is an isometric *-homomorphism. This gives a norm on $M_n(\mathcal{A})$ defined as 
\begin{align*}
\|A\|\coloneqq \|\pi(A)\|, \quad \forall A \in M_n(\mathcal{A}).
\end{align*}
This norm makes $M_n(\mathcal{A})$ as a C*-algebra.\\
In the sequel, $\mathcal{A}$ is a unital C*-algebra. We first derive a lemma followed by a corollary for unital C*-algebras. Proof of the lemma is a direct algebraic calculation.
\begin{lemma}(\textbf{Commutator calculation})\label{COMMUTATORLEMMA}
	Let $u,v, b_1, \dots, b_n \in \mathcal{A}	$ and $\delta>0$. Let 
	\begin{align*}
	X\coloneqq  \begin{pmatrix}
	0 & 0 & 0& \cdots & 0 &\delta b_1\\
	1_\mathcal{A} & 0 & 0& \cdots& 0 &\delta b_2\\
	0 & 1_\mathcal{A} & 0& \cdots& 0 &\delta b_3\\
	\vdots &\vdots  & \vdots& \cdots& \vdots &\vdots\\
	0&0&0&\cdots& 0& \delta b_{n-1}\\
	0&0&0&\cdots& 1_\mathcal{A}& \delta b_n
	\end{pmatrix} \in M_n(\mathcal{A})
	\end{align*}
	and 
	\begin{align*}
	D\coloneqq  \begin{pmatrix}
	\frac{v}{\delta} &1_\mathcal{A} & 0& \cdots& 0 & \delta b_1u\\
	\frac{u}{\delta} & \frac{v}{\delta} & 2.1_\mathcal{A}& \cdots& 0 & \delta b_2u\\
	0 & \frac{u}{\delta} & \frac{v}{\delta}& \cdots& 0 & \delta b_3u\\
	\vdots &\vdots  & \vdots& \cdots& \vdots &\vdots\\
	0&0&0&\cdots& \frac{v}{\delta}&  (n-1)1_\mathcal{A}+\delta b_{n-1}u\\
	0&0&0&\cdots& \frac{u}{\delta}&  \frac{v}{\delta}+\delta b_nu
	\end{pmatrix} \in M_n(\mathcal{A}).
	\end{align*}
	Then 
	\begin{align*}
	[D,X]=1_{M_n(\mathcal{A})}+\begin{pmatrix}
	0 & 0 & 0& \cdots& 0 & [v,b_1]+0+\delta b_2+\delta b_1[u,b_n]\\
	0 & 0 & 0& \cdots& 0 &[v,b_2]+[u,b_1]+2\delta b_3+\delta b_2[u,b_n]\\
	0 & 0 & 0& \cdots& 0 &[v,b_3]+[u,b_2]+3\delta b_4+\delta b_3[u,b_n]\\
	\vdots &\vdots  & \vdots& \cdots& \vdots &\vdots\\
	0&0&0&\cdots& 0& [v,b_{n-1}]+[u,b_{n-2}]+(n-1)\delta b_{n}+\delta b_{n-1}[u,b_n]\\
	0&0&0&\cdots& 0& [v,b_n]+[u,b_{n-1}]+0+\delta b_n[u,b_n]-n.1_\mathcal{A}
	\end{pmatrix}.
	\end{align*}
\end{lemma}
\begin{corollary}\label{COROLLARYLEMMA}
	Let $u,v, b_1, \dots, b_n \in \mathcal{A}	$. Assume that for some $\delta>0$, we have equations 
	\begin{align}\label{CORASSUMPTION1}
	[v,b_i]+[u,b_{i-1}]+i\delta b_{i+1}+\delta b_i[u,b_n]=0, \quad \forall i =2, \dots, n-1
	\end{align} 
	and
	\begin{align}\label{CORASSUMPTION2}
	[v,b_n]+[u,b_{n-1}]+\delta b_n[u,b_n]=n\cdot 1_{{M_n(\mathcal{A})}}.
	\end{align}
	Then for any $\mu>0$, there exist matrices $D_\mu, X_\mu \in M_n(\mathcal{A})$ such that 
	\begin{align*}
	&\|	D_\mu\| \leq \frac{\|u\|}{\mu^2\delta}+\frac{\|v\|}{\mu\delta}+(n-1)+\delta \sum_{i=1}^{n}\mu^{n-i-1}\|b_i\|\|u\|,\\
	& \|	X_\mu\| \leq 1+ \delta \sum_{i=1}^{n}\mu^{n-i+1}\|b_i\| \quad \text{ and }\\
	&\|[D_\mu,X_\mu]-1_{M_n(\mathcal{A})}\|\leq\mu^{n-1}\|[v,b_1]+\delta b_2+\delta b_1 [u,b_n]\|.
	\end{align*}
\end{corollary}
\begin{proof}
	Let $D$ and $X$ be as in Lemma \ref{COMMUTATORLEMMA}. Define 
	\begin{align*}
	S_\mu &\coloneqq \begin{pmatrix}
	\mu^{n-1}& 0 & 0& \cdots& 0 &0\\
	0 & \mu^{n-2} & 0& \cdots& 0 &0\\
	0 &0 & \mu^{n-3}  &\cdots&  0&0\\
	\vdots &\vdots  & \vdots& \cdots& \vdots &\vdots\\
	0&0&0&\cdots& \mu& 0\\
	0&0&0&\cdots& 0& 1
	\end{pmatrix} \in M_n(\mathbb{K}), \\
	D_\mu &\coloneqq \frac{1}{\mu}S_\mu DS_\mu^{-1}, \quad X_\mu \coloneqq \mu S_\mu XS_\mu^{-1}.
	\end{align*}
	Then 
	\begin{align*}
	\|D_\mu \|&=\left\| \begin{pmatrix}
	\frac{v}{\mu\delta} &1_\mathcal{A} & 0& \cdots& 0 & \mu^{n-2}\delta b_1u\\
	\frac{u}{\mu^2\delta} & \frac{v}{\mu\delta} & 2.1_\mathcal{A}& \cdots& 0 & \mu^{n-3}\delta b_2u\\
	0 & \frac{u}{\mu^2\delta} & \frac{v}{\mu\delta}& \cdots& 0 & \mu^{n-4}\delta b_3u\\
	\vdots &\vdots  & \vdots& \cdots& \vdots &\vdots\\
	0&0&0&\cdots &\frac{v}{\mu\delta}&  (n-1)1_\mathcal{A}+\delta b_{n-1}u\\
	0&0&0&\cdots& \frac{u}{\mu^2\delta}&  \frac{v}{\mu\delta}+\mu^{-1}\delta b_nu
	\end{pmatrix}\right\|\\
	&\leq \left\|\frac{u}{\mu^2\delta}\right\|+\left\|\frac{v}{\mu\delta}\right\|+\|(n-1)1_\mathcal{A}\|+\left\| \begin{pmatrix}
	0 &0 & 0& \cdots& 0 & \mu^{n-2}\delta b_1u\\
	0 & 0& 0& \cdots& 0 & \mu^{n-3}\delta b_2u\\
	0 & 0 & 0& \cdots& 0 & \mu^{n-4}\delta b_3u\\
	\vdots &\vdots  & \vdots& \cdots& \vdots &\vdots\\
	0&0&0&\cdots& 0&  \delta b_{n-1}u\\
	0&0&0&\cdots& 0&  \mu^{-1}\delta b_nu
	\end{pmatrix}\right\|\\
	&\leq \frac{\|u\|}{\mu^2\delta}+\frac{\|v\|}{\mu\delta}+(n-1)+\delta \sum_{i=1}^{n}\mu^{n-i-1}\|b_i\|\|u\|
	\end{align*}
	and 
	\begin{align*}
	\|X_\mu\| &=\left\|\begin{pmatrix}
	0 &0 & 0& \cdots& 0 & \mu^{n}\delta b_1\\
	1_\mathcal{A} & 0& 0& \cdots& 0 & \mu^{n-1}\delta b_2\\
	0 & 1_\mathcal{A} & 0& \cdots& 0 & \mu^{n-2}\delta b_3\\
	\vdots &\vdots  & \vdots& \cdots& \vdots &\vdots\\
	0&0&0&\cdots& 0&  \mu^2 \delta b_{n-1}\\
	0&0&0&\cdots& 1_\mathcal{A}&  \mu\delta b_n
	\end{pmatrix}\right\|\\
	&\leq \|1_\mathcal{A}\|+\left\|\begin{pmatrix}
	0 &0 & 0& \cdots& 0 & \mu^{n}\delta b_1\\
	0 & 0& 0& \cdots& 0 & \mu^{n-1}\delta b_2\\
	0 & 0 & 0& \cdots& 0 & \mu^{n-2}\delta b_3\\
	\vdots &\vdots&   \vdots& \cdots& \vdots &\vdots\\
	0&0&0&\cdots& 0&  \mu^2 \delta b_{n-1}\\
	0&0&0&\cdots& 0&  \mu\delta b_n
	\end{pmatrix}\right\|\\
	&\leq 1+ \delta \sum_{i=1}^{n}\mu^{n-i+1}\|b_i\|.
	\end{align*}
	Now using (\ref{CORASSUMPTION1}) and (\ref{CORASSUMPTION2})	 we get 
	\begin{align*}
	\|[D_\mu,X_\mu]-1_{M_n(\mathcal{A})}\|&=\left\|\begin{pmatrix}
	0 &0 & 0& \cdots& 0 & \mu^{n-1}([v,b_1]+\delta b_2+\delta b_1 [v,b_n])\\
	0 & 0& 0& \cdots& 0 & 0\\
	0 & 0 & 0& \cdots& 0 & 0\\
	\vdots &\vdots  & \vdots& \cdots& \vdots &\vdots\\
	0&0&0&\cdots& 0&  0\\
	0&0&0&\cdots& 0&  0
	\end{pmatrix}\right\|\\
	&\leq\mu^{n-1}\|[v,b_1]+\delta b_2+\delta b_1 [u,b_n]\|.
	\end{align*}
\end{proof}
Let $\mathcal{A}$ be a unital C*-algebra. Assume that  there are isometries $u, v \in \mathcal{A}$ such that
\begin{align}\label{IMPORTANTEQUATION}
u^*u=v^*v=uu^*+vv^*=1_\mathcal{A} \quad \text{and} \quad  u^*v=v^*u=0.
\end{align}
Examples of such unital C*-algebras are $\mathcal{B}(\mathcal{H})$ (where $\mathcal{H}$ is an infinite dimensional Hilbert space) as well as any unital  C*-algebra which contains the Cuntz algebra $\mathcal{O}_2$ (\cite{CUNTZ}). Note that whenever a unital C*-algebra admits a trace map there are no isometries satisfying Equation (\ref{IMPORTANTEQUATION}). In particular, any finite dimensional unital C*-algebra does not have such elements. It is also clear that no commutative unital C*-algebra can have isometries satisfying Equation (\ref{IMPORTANTEQUATION}).

It is shown in \cite{TAO} that whenever  $\mathcal{H}$ is  an infinite dimensional Hilbert space, then the Banach algebras $\mathcal{B}(\mathcal{H})$ and $M_2(\mathcal{B}(\mathcal{H})) $ are isometrically isomorphic. We now do these results for C*-algebras whenever they have isometries satisfying Equation (\ref{IMPORTANTEQUATION}). To do so we first need a result from the theory of C*-algebras.
\begin{theorem}(cf. \cite{TAKESAKI, PEDERSEN})\label{INJECTIOVEHOMOISISO}
	\begin{enumerate}[label=(\roman*)]
		\item Every *-homomorphism between C*-algebras is norm decreasing.
		\item If a *-homomorphism between C*-algebras is   injective, then it is isometric.
	\end{enumerate}
\end{theorem}
\begin{theorem}\label{ALGEBRAMATRIX}
	Let  $\mathcal{A}$ be a unital C*-algebra. If  there are isometries $u, v \in \mathcal{A}$ such that Equation (\ref{IMPORTANTEQUATION}) holds, then the map 
	\begin{align}\label{FIRSTMAP}
	\phi:\mathcal{A} \ni x \mapsto 
	\begin{pmatrix}
	u^*xu & u^*xv \\
	v^*xu & v^*xv
	\end{pmatrix} \in M_2(\mathcal{A})
	\end{align}
	is a C*-algebra  isomorphism with the inverse map 
	\begin{align}\label{SECONDMAP}
	\psi:M_2(\mathcal{A})\ni   \begin{pmatrix}
	a & b \\
	c & d
	\end{pmatrix} \mapsto uau^*+ubv^*+vcu^*+vdv^* \in \mathcal{A}.
	\end{align}
\end{theorem}
\begin{proof}
	Using 	Equation (\ref{IMPORTANTEQUATION}), a direct computation gives 
	\begin{align*}
	\phi\psi\begin{pmatrix}
	a & b \\
	c & d
	\end{pmatrix}&=\phi(uau^*+ubv^*+vcu^*+vdv^*)\\
	&= \begin{pmatrix}
	u^*(uau^*+ubv^*+vcu^*+vdv^*)u & u^*(uau^*+ubv^*+vcu^*+vdv^*)v \\
	v^*(uau^*+ubv^*+vcu^*+vdv^*)u & v^*(uau^*+ubv^*+vcu^*+vdv^*)v
	\end{pmatrix}\\
	&=\begin{pmatrix}
	1_\mathcal{A} a1_\mathcal{A}+1_\mathcal{A}b0+0c1_\mathcal{A}+0d0 & 1_\mathcal{A}a0+1_\mathcal{A}b1_\mathcal{A}+0c0+0d1_\mathcal{A} \\
	0a1_\mathcal{A}+0b0+1_\mathcal{A}c1_\mathcal{A}+1_\mathcal{A}d0& 0a0+0b1_\mathcal{A}+1_\mathcal{A}c0+1_\mathcal{A}d1_\mathcal{A}
	\end{pmatrix}\\
	&=\begin{pmatrix}
	a & b \\
	c & d
	\end{pmatrix}, \quad \forall \begin{pmatrix}
	a & b \\
	c & d
	\end{pmatrix} \in M_2(\mathcal{A})
	\end{align*}
	and 
	\begin{align*}
	\psi\phi x&=\psi \begin{pmatrix}
	u^*xu & u^*xv \\
	v^*xu & v^*xv
	\end{pmatrix}\\
	&=u(u^*xu)u^*+u(u^*xv)v^*+v(v^*xu)u^*+v(v^*xv)v^*\\
	&=uu^*x(uu^*+vv^*)+vv^*x(uu^*+vv^*)\\
	&=uu^*x1_\mathcal{A}+vv^*x1_\mathcal{A}=(uu^*+vv^*)x=x, \quad \forall x \in \mathcal{A}.
	\end{align*}
	Further, 
	\begin{align*}
	(\phi(x))^*&= \begin{pmatrix}
	u^*xu & u^*xv \\
	v^*xu & v^*xv
	\end{pmatrix}^*=\begin{pmatrix}
	u^*x^*u & (v^*xu)^* \\
	(u^*xv)^* & v^*x^*v
	\end{pmatrix}\\
&=\begin{pmatrix}
	u^*x^*u & u^*x^*v \\
	v^*x^*u & v^*x^*v
	\end{pmatrix}=\phi(x^*), \quad \forall x \in \mathcal{A}.
	\end{align*}
	Hence $\phi$ is a *-isomorphism. Using Theorem \ref{INJECTIOVEHOMOISISO},  to show $\phi$ is a C*-algebra isomorphism (i.e., isometric isomorphism), it suffices to show that $\phi$ is injective. Let $ x \in \mathcal{A}$ be such that $\phi x=0$. Then 
	\begin{align*}
	u^*xu=u^*xv=0, \quad v^*xv=v^*xu=0.
	\end{align*}
	Using the first equation we get $uu^*xuu^*=uu^*xvv^*=0$ which implies  $uu^*x=uu^*x(uu^*+vv^*)=0$. Similarly using the second equation we get $vv^*x=0$. Therefore $x=(uu^*+vv^*)x=0$. Hence $\phi$ is injective which completes the proof.
\end{proof}
Along with the lines of Theorem \ref{ALGEBRAMATRIX} we can easily derive the following result.
\begin{theorem}\label{ALLC}
	Let  $\mathcal{A}$ be a unital C*-algebra and $n \in \mathbb{N}$. If  there are isometries $u, v \in \mathcal{A}$ such that Equation (\ref{IMPORTANTEQUATION}) holds, then the map 
	\begin{align*}
	\phi:M_n(\mathcal{A}) \ni X \mapsto 
	\begin{pmatrix}
	u^*Xu & u^*Xv \\
	v^*Xu & v^*Xv
	\end{pmatrix} \in M_{2n}(\mathcal{A})
	\end{align*}
	is a C*-algebra  isomorphism with the inverse map 
	\begin{align*}
	\psi:M_{2n}(\mathcal{A})\ni   \begin{pmatrix}
	A & B \\
	C & D
	\end{pmatrix} \mapsto uAu^*+uBv^*+vCu^*+vDv^* \in M_n(\mathcal{A}), 
	\end{align*}
	where if $X\coloneqq [x_{i,j}]_{i,j}$ is a matrix, and $a,b\in \mathcal{A}$, by $aXb$ we mean the matrix $[ax_{i,j}b]_{i,j}$. In particular, the C*-algebras $\mathcal{A}, M_{2}(\mathcal{A}), M_{4}(\mathcal{A}), \dots, M_{2n}(\mathcal{A}), \dots $ are all *-isometrically isomorphic.
\end{theorem}
In the rest of this chapter, we assume that unital C*-algebra  $\mathcal{A}$ has isometries $u,v$ satisfying Equation (\ref{IMPORTANTEQUATION}). In the next result we use the following notation. Given a vector $x \in \mathcal{A}^n$, $x_i$ means its $i^{\text{th}}$ coordinate.
\begin{proposition}\label{RIGHTPROPOSI}
	Let $n\geq2$ and $T:\mathcal{A}^n\to \mathcal{A}^{n-1}$ be the bounded linear operator defined by 
	\begin{align*}
	T(b_i)_{i=1}^n\coloneqq ([v,b_i]+[u,b_{i-1}])_{i=2}^n, \quad \forall (b_i)_{i=1}^n \in \mathcal{A}^n.
	\end{align*}	
	Then there exists a bounded linear right-inverse  $R:\mathcal{A}^{n-1}\to \mathcal{A}^{n}$ for $T$ such that 
	\begin{align*}
	\|Rb\|&= \sup_{1\leq i \leq n}\|(Rb)_i\|\leq 8 \sqrt{2}n^2\sup_{2\leq i \leq n}\|b_i\|\\
	&\leq 8 \sqrt{2}n^2\sup_{1\leq i \leq n}\|b_i\|=8 \sqrt{2}n^2\|b\|, \quad \forall b \in \mathcal{A}^{n}.
	\end{align*}
\end{proposition}
\begin{proof}
	Define 
	\begin{align*}
	L:\mathcal{A}^{n-1} \ni (x_i)_{i=2}^n\mapsto \left(-\frac{1}{2}x_iv^*-\frac{1}{2}x_{i+1}u^*\right)_{i=1}^n\in \mathcal{A}^{n}, \quad \text{ where } x_1\coloneqq 0, x_{n+1}\coloneqq 0
	\end{align*}
	and 
	\begin{align*}
	E:\mathcal{A}^{n-1}\ni (x_i)_{i=2}^n\mapsto \left(\frac{1}{2}(vx_iv^*+vx_{i+1}u^*+ux_{i-1}v^*+ux_iu^*)\right)_{i=2}^n \in \mathcal{A}^{n-1}.
	\end{align*}
	Then 
	\begin{align*}
	&TL(x_i)_{i=2}^n=T\left(-\frac{1}{2}x_iv^*-\frac{1}{2}x_{i+1}u^*\right)_{i=1}^n=-\frac{1}{2}(T(x_iv^*)_{i=2}^n+T(x_{i+1}u^*)_{i=2}^n)\\
	&=-\frac{1}{2}(([v, x_iv^*]+[u, x_{i-1}v^*])_{i=2}^n+([v,x_{i+1}u^*]+[u,x_{i}u^*])_{i=2}^n)\\
	&=-\frac{1}{2}(vx_iv^*-x_iv^*v+ux_{i-1}v^*-x_{i-1}v^*u+vx_{i+1}u^*-x_{i+1}u^*v+ux_{i}u^*-x_{i}u^*u)_{i=2}^n\\
	&=-\frac{1}{2}(vx_iv^*-x_i+ux_{i-1}v^*-x_{i-1}v^*u+vx_{i+1}u^*+ux_{i}u^*-x_{i})_{i=2}^n\\
	&=(x_i)_{i=2}^n-\frac{1}{2}(vx_iv^*+vx_{i+1}u^*+ux_{i-1}v^*+ux_iu^*)_{i=1}^n \\
	&=(1-E)(x_i)_{i=2}^n, \quad \forall (x_i)_{i=2}^n \in \mathcal{A}^{n-1}, \quad \text{ where } 1(x_i)_{i=2}^n\coloneqq (x_i)_{i=2}^n.
	\end{align*}
	We next try to show that the operator $1-E$ is bounded invertible with the help of Neumann series. First step is to change the  norm on $\mathcal{A}^{n}$ to an equivalent norm so that  invertibility property will not  affect in both norms. Define a new norm on  $\mathcal{A}^{n-1}$ by 
	\begin{align*}
	\|(x_i)_{i=2}^n\|'\coloneqq\sup _{2\leq i \leq n}\left(2-\frac{i^2}{n^2}\right)^\frac{-1}{2}\|x_i\|.
	\end{align*}
	Let $x=(x_i)_{i=2}^n\in	\mathcal{A}^{n-1}$ be such that $	\|(x_i)_{i=2}^n\|'\leq1$. Then
	\begin{align*}
	\left(2-\frac{i^2}{n^2}\right)^\frac{-1}{2}\|x_i\|\leq \sup _{2\leq i \leq n}\left(2-\frac{i^2}{n^2}\right)^\frac{-1}{2}\|x_i\|\leq 1, \quad \forall 2\leq i \leq n.
	\end{align*} 
	Hence $\|x_i\|\leq  \left(2-\frac{i^2}{n^2}\right)^\frac{1}{2}$ for all $2\leq i \leq n$. Using Theorem \ref{ALGEBRAMATRIX} we now get 
	\begin{align*}
	\|(Ex)_i\|&=\frac{1}{2}\|vx_iv^*+vx_{i+1}u^*+ux_{i-1}v^*+ux_iu^*\|\\
	&=\frac{1}{2} \left\|\begin{pmatrix}
	x_{i} & x_{i+1} \\
	x_{i-1} & x_{i}
	\end{pmatrix}\right\|\leq \frac{1}{2}\left\|\begin{pmatrix}
	\|x_{i}\| & \|x_{i+1}\| \\
	\|x_{i-1}\| & \|x_{i}\|
	\end{pmatrix}\right\|\\
	&\leq \frac{1}{2} \left(\|x_{i}\|^2+\|x_{i+1}\|^2+\|x_{i-1}\|^2+\|x_{i}\|^2\right)^\frac{1}{2}\\
	&\leq \frac{1}{2} \left(\left(2-\frac{i^2}{n^2}\right)+\left(2-\frac{(i+1)^2}{n^2}\right)+\left(2-\frac{(i-1)^2}{n^2}\right)+\left(2-\frac{i^2}{n^2}\right)\right)^\frac{1}{2}\\
	&=  \left(2-\frac{i^2}{n^2}-\frac{1}{2n^2}\right)^\frac{1}{2}\leq \left(1-\frac{1}{8n^2}\right)^\frac{1}{2}\left(2-\frac{i^2}{n^2}\right)^\frac{1}{2}\\
	&\leq \left(1-\frac{1}{8n^2}\right)\left(2-\frac{i^2}{n^2}\right)^\frac{1}{2}, \quad \forall 2 \leq i \leq n.
	\end{align*}
	Hence 
	\begin{align*}
	\|Ex\|'=\sup _{2\leq i \leq n}\left(2-\frac{i^2}{n^2}\right)^\frac{-1}{2}\|(Ex)_i\|\leq \left(1-\frac{1}{8n^2}\right)\|x\|', \quad \forall x \in \mathcal{A}^{n-1}.
	\end{align*}
	Since $1-\frac{1}{8n^2}<1$, $1-E$ is invertible and $\|(1-E)^{-1}x\|'\leq 8n^2\|x\|'$. Now going back to the original norm, we get 
	\begin{align*}
	\frac{1}{\sqrt{2}}\|((1-E)^{-1}x)_i\|&\leq \sup _{2\leq i \leq n}\left(2-\frac{i^2}{n^2}\right)^\frac{-1}{2}\|((1-E)^{-1}x)_i\|\\
	&=\|(1-E)^{-1}x\|'\leq 8n^2\|x\|'\\
	&= 8n^2\sup _{2\leq i \leq n}\left(2-\frac{i^2}{n^2}\right)^\frac{-1}{2}\|x_i\|\\
	&\leq 8n^2\sup _{2\leq i \leq n}\|x_i\|=8n^2\|x\|, \quad \forall x \in \mathcal{A}^{n-1}.
	\end{align*}
	Define $R\coloneqq L(1-E)^{-1}$. Then $TR=TL(1-E)^{-1}=(1-E)(1-E)^{-1}=1$ and 
	\begin{align*}
	\|Rb\|&=\sup _{1\leq i \leq n}\|(Rb)_i\|=\|L(1-E)^{-1}b\|\leq \|L\|\|(1-E)^{-1}b\|\leq \|(1-E)^{-1}b\|\\
	&=\sup _{2\leq i \leq n}\|((1-E)^{-1}b)_i\|\leq 8\sqrt{2} n^2\|b\|= 8\sqrt{2} n^2\sup _{2\leq i \leq n}\|b_i\|,\quad \forall b \in \mathcal{A}^{n-1}.
	\end{align*}
\end{proof}
As given in \cite{TAO} we try to shift from the systems of equations (\ref{CORASSUMPTION1}) and  (\ref{CORASSUMPTION2}) to the solution of single equation.  Let $n\geq 2$. Define $a\coloneqq (0, \dots, n)\in \mathcal{A}^n$, 
\begin{align*}
F:\mathcal{A}^n \ni (b_i)_{i=1}^n \mapsto (-2b_3, \dots, -(n-1)b_n,0)\in \mathcal{A}^{n-1}
\end{align*}
and 
\begin{align*}
G:\mathcal{A}^n\times \mathcal{A}^n \ni ((b_i)_{i=1}^n, (c_i)_{i=1}^n)\mapsto (-b_2[u,c_n], \dots, -b_n[u,c_n])\in \mathcal{A}^{n-1}.
\end{align*}
We then have $\|F\|\leq n-1$ and $\|G\|\leq 2$. 
\begin{proposition}
	Systems  (\ref{CORASSUMPTION1}) and  (\ref{CORASSUMPTION2}) have a solution $b$ if and only if 
	\begin{align}\label{SOLUTIONEQUATION}
	Tb=a+\delta F(b)+\delta G(b,b).
	\end{align}
\end{proposition}
\begin{proof}
	Systems  (\ref{CORASSUMPTION1}) and  (\ref{CORASSUMPTION2}) have a solution $b$ if and only if 	
	\begin{align*}
	[v,b_i]+[u,b_{i-1}]=-i\delta b_{i+1}-\delta b_i[u,b_n], \quad \forall i =2, \dots, n-1
	\end{align*} 
	and
	\begin{align*}
	[v,b_n]+[u,b_{n-1}]=-\delta b_n[u,b_n]+n\cdot 1_{{M_n(\mathcal{A})}}
	\end{align*}
	if and only if 
	\begin{align*}
	&([v,b_i]+[u,b_{i-1}])_{i=2}^n=\\
	&\quad (0, \dots, n)+\delta (-2b_3, \dots, -(n-1)b_n,0)+\delta (-b_2[u,b_n], \dots, -b_n[u,b_n])
	\end{align*}
	if and only if 
	\begin{align*}
	Tb=a+\delta F(b)+\delta G(b,b).
	\end{align*}
\end{proof}
The above proposition reduces the work of solving systems (\ref{CORASSUMPTION1}) and  (\ref{CORASSUMPTION2}) to a single operator equation. To solve (\ref{SOLUTIONEQUATION}) we need an abstract lemma from \cite{TAO}.
\begin{lemma}(\cite{TAO})\label{TAOTHEOREMABST}
	Let $\mathcal{X}$, $\mathcal{Y}$ be Banach spaces,  $T,F:\mathcal{X}\to \mathcal{Y}$ be bounded linear operators, and  let 
	$G:\mathcal{X}\times\mathcal{X} \to \mathcal{Y}$ be a bounded bilinear operator with bound $r>0$ and let $ a \in\mathcal{Y}$. Suppose that
	$T$ has a bounded linear right inverse $R:\mathcal{Y}\to \mathcal{X}$. If $\delta>0$ is such that 
	\begin{align}\label{LEMMACONDITION}
	\delta(2\|F\|\|R\|+4r\|R\|^2\|a\|)<1,
	\end{align}
	then there exists  $ b \in\mathcal{X}$ with $\|b\|\leq 2 \|R\|\|a\|$ that solves the equation 
	\begin{align*}
	Tb=a+\delta F(b)+\delta G(b,b).
	\end{align*}
\end{lemma}
\begin{theorem}
	For each $n\geq2$, there exists a solution $b$ to  Equation  (\ref{SOLUTIONEQUATION}) such that $\|b\|\leq 16 \sqrt{2}n^3$. 	
\end{theorem}
\begin{proof}
	We apply Lemma 	\ref{TAOTHEOREMABST} for 
	\begin{align*}
	\delta \coloneqq \frac{1}{2000n^5}.
	\end{align*}
	Then using Proposition \ref{RIGHTPROPOSI}, we get 
	\begin{align*}
	\delta(2\|F\|\|R\|+4r\|R\|^2\|a\|)&\leq \frac{1}{2000n^5}(2(n-1)8 \sqrt{2}n^2+4.2.128.n^4.n)\\
	&\leq \frac{1}{2000n^5}(16\sqrt{2}n^3+1024n^5)<1.
	\end{align*}
	Lemma \ref{TAOTHEOREMABST} now says that  there exists a $b$ which satisfies (\ref{SOLUTIONEQUATION}).
\end{proof} 
\begin{theorem}\label{BBERF}
	For each $n\geq2$, let $b$  be an element satisfying   Equation  (\ref{SOLUTIONEQUATION}) and $\|b\|\leq 16 \sqrt{2}n^3$. Then for $\mu=\frac{1}{2}$, $D_\mu, X_\mu \in M_n(\mathcal{A})$ such that  
	\begin{align*}
	\|	D_\mu\|=O(n^5), \quad \|	X_\mu\|=O(1), \quad \|[D_\mu,X_\mu]-1_{M_n(\mathcal{A})}\|=O(n^32^{-n}).
	\end{align*}
\end{theorem}
\begin{proof}
	Let $D_\mu, X_\mu \in M_n(\mathcal{A})$ be as in  Corollary \ref{COROLLARYLEMMA}. We then have  
	\begin{align*}
	&\|	D_\mu\| \leq  4. 2000n^5\|u\|+ 2. 2000n^5\|v\|+(n-1)+\frac{1}{2000n^5} \sum_{i=1}^{n}\frac{1}{2^{n-i-1}}16 \sqrt{2}n^3\|u\|\\
	&\quad \quad =O(n^5),\\
	& \|	X_\mu\| \leq 1+ \frac{1}{2000n^5} \sum_{i=1}^{n}\frac{1}{2^{n-i-1}}16 \sqrt{2}n^3  =O(1),\\
	&\|[D_\mu,X_\mu]-1_{M_n(\mathcal{A})}\|\leq 2\mu^{n-1}(\|v\|\|b_1\|+\delta \|b_2\|+\delta \|b_1 \|\|u\|\|b_n\|).\\
	&\quad \leq 2\frac{1}{2^{n-1}}(\|u\|16 \sqrt{2}n^3+\frac{1}{2000n^5} 16 \sqrt{2}n^3+\frac{1}{2000n^5} 16 \sqrt{2}n^3 \|u\|16 \sqrt{2}n^3)\\
	&\quad \leq 2\frac{n^3}{2^{n-1}}(\|v\|16 \sqrt{2}+\frac{1}{2000n^5} 16 \sqrt{2}+\frac{1}{2000n^5} 16 \sqrt{2} n^3\|u\|16 \sqrt{2})=O(n^32^{-n}).
	\end{align*}
\end{proof}
\begin{theorem}\label{BEFORE}
	Let $0<\varepsilon\leq 1/2$. Then  there exist an  even integer $n$  and   $D,X \in M_n(\mathcal{A})$ with 
	\begin{align*}
	\|[D,X]-1_{M_n(\mathcal{A})}\|\leq \varepsilon
	\end{align*} 
	such that 
	\begin{align*}
	\|D\|\|X\|=O\left(\log^5\frac{1}{\varepsilon}\right).
	\end{align*}
\end{theorem}
\begin{proof}
	Let $D_\mu, X_\mu \in M_n(\mathcal{A})$ be as in  Corollary \ref{COROLLARYLEMMA}. Theorem \ref{BBERF} says that there are  $\alpha, \beta, \gamma>0$ be such that  	
	\begin{align*}
	\|	D_\mu\|\leq \alpha n^5, \quad \|	X_\mu\|\leq \beta, \quad \|[D_\mu,X_\mu]-1_{M_n(\mathcal{A})}\|\leq \gamma n^32^{-n}.
	\end{align*}
	Since  $2^n>n^4$ all but finitely many $n$'s,  $ \gamma n^32^{-n}<\varepsilon$ all but finitely many $n$'s. 
	We now choose  real $c$ such that $n=c \log\frac{1}{\varepsilon}$ is even $ \gamma n^32^{-n}<\varepsilon$. We then have $\|	D_\mu\|=O(\log^5(\frac{1}{\varepsilon}))$ and $ \|[D_\mu,X_\mu]-1_{M_n(\mathcal{A})}\|\leq \varepsilon$.
\end{proof}
Theorem \ref{BEFORE} and Theorem \ref{ALLC} easily give the following.
\begin{theorem}\label{LASTTHEOREM}
	Let $\mathcal{A}$ be a unital C*-algebra. Suppose  there are isometries $u, v \in \mathcal{A}$ such that Equation (\ref{IMPORTANTEQUATION}) holds. Then 
	for each $0<\varepsilon\leq 1/2$,  there exist $d,x \in \mathcal{A}$ with 
	\begin{align*}
	\|[d,x]-1_{\mathcal{A}}\|\leq \varepsilon
	\end{align*} 
	such that 
	\begin{align*}
	\|d\|\|x\|=O\left(\log^5\frac{1}{\varepsilon}\right).
	\end{align*}	
\end{theorem} 
\begin{remark}
Let $\mathcal{A}$ be a finite dimensional unital C*-algebra. From the structure theory (\cite{DAVIDSON})  we have 
\begin{align*}
\mathcal{A}\cong M_{n_1}(\mathbb{C})\oplus \cdots \oplus  M_{n_r}(\mathbb{C}),
\end{align*}
for unique (upto permutation) natural numbers $n_1, \dots, n_r$. This result says that normalized trace map (a trace map $\text{Tr} $ such that $\text{Tr}(1_\mathcal{A})=1$) exists on $\mathcal{A}$. Using this we make the following two observations. 
\begin{enumerate}[label=(\roman*)]
	\item $\mathcal{A}$ can not have isometries satisfying Equation (\ref{IMPORTANTEQUATION}). Suppose that there are such isometries. Then 
	\begin{align*}
	1=\text{Tr}(uu^*+vv^*)=\text{Tr}(uu^*)+\text{Tr}(vv^*)=\text{Tr}(u^*u)+\text{Tr}(v^*v)=2
	\end{align*}
	which is impossible.
	\item In (\cite{TAO}), Tao observed that if $\mathcal{H}$ is a finite dimensional Hilbert space, then there are no  $D,X \in \mathcal{B}(\mathcal{H})$ satisfying $ \|[D,X]-1_{\mathcal{B}(\mathcal{H})}\|<1$. We elaborate this for any finite dimensional unital C*-algebra $\mathcal{A}$, namely, there do not exist $d,x \in \mathcal{A}$ 
	satisfying $\|[d,x]-1_{\mathcal{A}}\|<1$.  In other words, Theorem \ref{LASTTHEOREM} fails for  every finite dimensional unital C*-algebra. Let $d,x \in \mathcal{A}$ be arbitrary. From the structure theory,  we identify  that $d$ as $D$ and $x$ as $X$ for some matrices $D, X \in M_n(\mathbb{C})$ and for some $n$. Using the commutativity of trace we then have $\text{Tr}([D,X])=0$. Let $\lambda_1, \dots, \lambda_n$ be eigenvalues of $[D,X]$. Then $\sum_{j=1}^{n}\lambda_j=\text{Tr}([D,X])=0$. This gives 
	\begin{align*}
	n=\left|\sum_{j=1}^{n}(\lambda_j-1)\right|\leq \sum_{j=1}^{n}|\lambda_j-1|.
	\end{align*}
	Previous inequality says that there is atleast one $j$ such that $|\lambda_j-1|\geq 1$. We next see that all the eigenvalues of 
	$[D,X]-1_{M_n(\mathbb{C})}$ are $\lambda_1-1, \dots, \lambda_n-1$. Using the property of operator norm we finally get 
	\begin{align*}
	\|[d,x]-1_{\mathcal{A}}\|=\|[D,X]-1_{M_n(\mathbb{C})}\|\geq \sup_{1\leq j \leq n}|\lambda_j-1|\geq 1.
	\end{align*}
\end{enumerate}  	
\end{remark}

\textbf{Conclusion and future work : }
In this appendix  we showed that the result of Tao's is valid in more general spaces. One of the future objectives is to improve the bounds in Theorem \ref{LASTTHEOREM} and Theorem \ref{POPACOMPACT}.

\leavevmode\newpage
\bibliographystyle{apalike}
\cleardoublepage

\phantomsection

\addcontentsline{toc}{chapter}{BIBIOGRAPHY}

\bibliography{reference.bib}
\leavevmode\newpage
\leavevmode\newpage
\addcontentsline{toc}{chapter}{LIST OF SYMBOLS AND ABBREVIATIONS}
\par~
\begin{center}
	\textbf{{\fontsize{16}{1em}\selectfont LIST OF SYMBOLS AND ABBREVIATIONS}} \\
	\end{center}
\begin{tabular}{lcl}
$\mathbb{N}$ &:& Set of natural numbers\\
$\mathbb{R}$ &:& Field of real numbers\\
$\mathbb{C}$ &:& Field of complex numbers\\
$\{\cdot\}_n$ &:& Collections/sequences indexed by $\mathbb{N}$\\
$\mathbb{M}$ &:& Subset  of natural numbers\\
$\mathbb{M}^c$ &:& Complement of $\mathbb{M}$\\
$\mathbb{K}$ &:& $\mathbb{R}$ or $\mathbb{C}$\\
$\alpha, \beta, \gamma$ &:& Elements of $\mathbb{K}$\\
 $\mathcal{H}$, $\mathcal{H}_0$, \dots &:& Separable Hilbert spaces\\
 $h$, $h_0$, \dots &:& Elements of  Hilbert spaces\\
 $\mathcal{H}\otimes \mathcal{H}_0$, \dots &:& Tensor product of $\mathcal{H}$ and  $\mathcal{H}_0$\\
 $\langle \cdot, \cdot \rangle $ &:& Inner product which is linear in first variable and conjugate \\
 & & linear in second variable\\
 $\mathcal{X}$, $\mathcal{Y}$, \dots &:& Separable Banach spaces\\
 $\|\cdot\|$ &:& Norm\\
 $\mathcal{X}^*$ &:& Dual of  Banach space $\mathcal{X}$ equipped with operator norm\\
 $I_\mathcal{X}$	 &:& Identity operator on $\mathcal{X}$\\ 
 $\mathcal{B}(\mathcal{X}, \mathcal{Y})$ &:&  Banach space of bounded linear operators from $\mathcal{X}$ to $\mathcal{Y}$\\ 
 & &equipped with operator-norm\\
cf.	 &:& Cross reference (reference may not be the first reference \\
& &where the notion/result arose)\\
 ASF &:& Approximate Schauder frame\\
	OVF &:& Operator-valued frame\\
	p  &:& A real number in $[1, \infty)$\\
	$\ell^p(\mathbb{N})$ &:& $\{\{a_n\}_n: a_n \in \mathbb{K}, \forall n \in \mathbb{N}, \sum_{n=1}^{\infty}|a_n|^p<\infty\}$\\
	$\ell^\infty(\mathbb{N})$ &:& $\{\{a_n\}_n: a_n \in \mathbb{K}, \forall n \in \mathbb{N}, \sup_{n\in \mathbb{N}}|a_n|<\infty\}$\\
	$\mathcal{L}^p(\mathbb{R})$ &:& $\{f: \mathbb{R} \in  \mathbb{C}, f  \text{ measurable },  \int_{\mathbb{R}}|f(x)|^p\,dx<\infty\}$\\
$\{e_n\}_n $&:& Standard Schauder basis for $\ell^p(\mathbb{N})$\\
	$\mathcal{A}$ &:& C*-algebra\\
		$\mathcal{M}$, $\mathcal{N}$ &:& Metric spaces\\
		$d(\cdot, \cdot)$ &:& Metric on a metric space\\
		\end{tabular}\\
\begin{tabular}{lcl}	
	 $G$&:& Locally compact group\\
	 $\mu_G$ &:& Left Haar measure on $G$\\
	 $\pi$&:& Unitary representation of $G$\\
	$\mathcal{U}$ &:& Group-like unitary system\\
	$\mathcal{X}_d$ &:& BK-space\\
	$\operatorname{Lip}(\cdot)$ &:& Lipschitz number\\
	$(\mathcal{M}, 0)$ &:& Pointed metric space\\
	$\|\cdot\|_{\operatorname{Lip}_0}$&:& Lipschitz norm\\
	$\operatorname{Lip}(\mathcal{M},
	\mathcal{X})$ &:& Space of Lipschitz functions from $\mathcal{M}$ to $\mathcal{X}$ equipped \\
	& & with Lipschitz number\\
	$\operatorname{Lip}_0(\mathcal{M},
	\mathcal{X})$ &:& Banach space of base point preserving Lipschitz functions \\
	& & from $\mathcal{M}$ to $\mathcal{X}$ equipped with Lipschitz norm \\
	$\operatorname{Lip}_0(\mathcal{X})$ & & $\operatorname{Lip}_0(\mathcal{X},
	\mathcal{X})$\\
$\mathcal{F}(\mathcal{M})$&: &	Lipschitz free Banach space of $\mathcal{M}$\\
$c_0(\mathbb{N})$ &:& $\{\{a_n\}_n: a_n \in \mathbb{K}, \forall n \in \mathbb{N}, \lim_{n\to\infty} a_n=0\}$\\
$c(\mathbb{N})$ &:& $\{\{a_n\}_n: a_n \in \mathbb{K}, \forall n \in \mathbb{N}, \{a_n\}_n \text{ converges in } \mathbb{K}\}$\\
$T^*$ &:& Adjoint of the operator $T$\\
$\theta_\tau$ &:& Analysis  operator for frame $\{\tau_n\}_n$ for  Hilbert space\\
$\theta_\tau^*$ &:& Synthesis   operator for frame $\{\tau_n\}_n$ for  Hilbert space\\
$S_\tau$ &:& Frame  operator for frame $\{\tau_n\}_n$ for  Hilbert space\\
$\theta_f$  &:& Analysis  operator for p-ASF  $(\{f_n\}_n, \{\tau_n\}_n)$ for  Banach space\\
$S_{f, \tau}$ &:&  Frame operator for p-ASF $(\{f_n\}_n, \{\tau_n\}_n)$ for  Banach space\\
$\theta_A$ &:& Analysis operator for OVF $\{A_n\}_n$\\
$\theta_A^*$ &:& Synthesis operator for OVF $\{A_n\}_n$\\
$S_A$ &:& Frame operator for OVF $\{A_n\}_n$\\
$S_{A, \Psi}$ &:&  Frame operator for weak OVF  $(\{A_n\}_n, \{\Psi_n\}_n)$\\
$M_{\lambda, f , \tau}$&:& Multiplier \\
$\tau\otimes f$&:& Map defined by $(\tau\otimes f)(x)\coloneqq f(x)\tau$\\
$f^{-1}$&:& Inverse of map $f$ \\
$\chi_{[0,1]}$&:& Characteristic function on $[0,1]$\\
dim &:& Dimension of a subspace of a vector space\\
$P^\perp$ &:& $I_\mathcal{H}-P$, whenever $P$ is a projection on $\mathcal{H}$\\
$[\cdot, \cdot]$ &:& Semi-inner product\\
$A^\dagger$ &:& Generalized adjoint of the operator $A$ in a semi-inner product space\\
$\mathcal{A}'$ &:& Commutant of a set in an algebra\\
	\end{tabular}\\
\begin{tabular}{lcl}
ker &:& Kernel of a linear operator\\
$W^\perp$ &:& Orthogonal complement of a subspace $W$ of a Hilbert space\\
$\mathbb{T}$ &:& Unit circle group centered at origin\\
group$(\cdot)$ &:& Group generated by a subset of a group \\
$\mathcal{B}(\mathcal{H})$ &:& $\mathcal{B}(\mathcal{H}, \mathcal{H})$\\
$\mathcal{H}\oplus \mathcal{H}_0$ &: & Direct sum of Hilbert spaces $\mathcal{H} $ and $\mathcal{H}_0$\\
$\overline{\text{span}} ~W$ &:& Closure of span of subset $W$ of a Banach space $\mathcal{X}$\\
 $M_n(\mathcal{A})$ &:& C*-algebra of $n$ by $n$ matrices over unital C*-algebra $\mathcal{A}$\\
 $[A, B]$ &:& $AB-BA$, commutator of $A$ and $B$\\
 $\mathcal{O}_2$ &:&  Cuntz algebra generated by two isometries\\
 $\mathcal{K}(\mathcal{H})$ &:& C*-algebra of compact operators in $\mathcal{B}(\mathcal{H})$\\
 dist$(x,Y)$ &:& Distance between an element $x$ and a subset $Y$ of a metric space\\
  $r=O(s)$ &:& Asymptotic notation\\
  $\mathcal{C}(K)$ &:& C*-algebra of  all complex-valued continuous functions \\
  & & on compact Hausdorff space $K$\\
  $\mathscr{A}$ &:& Non empty set \\
  $\mathcal{V}$ &:& Vector space \\
   $\mathbb{J}$ &:& Index set \\
   $\mathbb{Z}_+$ &:& $\{0\}\cup \mathbb{N}$ \\
  $ \delta_{n,m}$ &:&  Kronecker delta \\
\end{tabular}\\
\leavevmode\newpage
\leavevmode\newpage
\begin{center}
	{\onehalfspacing \section*{PUBLICATIONS}}
\end{center}

\addcontentsline{toc}{chapter}{PUBLICATIONS}


\begin{enumerate}
		\item K. Mahesh Krishna  and P. Sam Johnson. \textbf{Frames for metric spaces}. \textit{Results in Mathematics} 77, 49 (2022). 
		
			\vspace{-.1cm} 
			https://doi.org/10.1007/s00025-021-01583-3.
\item {K. Mahesh Krishna  and P. Sam Johnson.  \textbf{Towards characterizations of approximate Schauder frame and its duals for Banach spaces}. \textit{ Journal of Pseudo-Differential Operators and Applications}  12, 9 (2021).
	
			\vspace{-.1cm} 
	https://doi.org/10.1007/s11868-021-00379-x.} 
\item {K. Mahesh Krishna  and P. Sam Johnson.  \textbf{Dilation theorem for p-approximate Schauder frames for separable Banach spaces}. \textit{Palestine Journal of Mathematics} (accepted).
	
			\vspace{-.1cm} 
https://arXiv.org/2011.12188v1 [math.FA] 23 November 2020.}
\item {K. Mahesh Krishna and P. Sam Johnson.  \textbf{Expansion of weak reconstruction sequences to approximate Schauder
frames for Banach spaces}. \textit{Asian-European Journal of Mathematics}   2250060  (2022).

		\vspace{-.1cm} 
	https://doi.org/10.1142/S1793557122500607.}
\item 	{K. Mahesh Krishna  and P. Sam Johnson.  \textbf{New Identity on Parseval p-Approxim-ate Schauder Frames and
	Applications}. \textit{ Journal of Interdisciplinary Mathematics} 1-10 (2021).
	
			\vspace{-.1cm} 
	https://doi.org/10.1080/09720502.2021.1891698.}
\item {K. Mahesh Krishna  and P. Sam Johnson.  \textbf{Perturbation of p-approximate Schau-der frames for separable Banach spaces}. \textit{Poincare Journal of Analysis and Applications} (accepted). 
	
			\vspace{-.1cm}  https://arXiv.org/2012.030054v1  [math.FA] 5 December 2020.}
\item{K. Mahesh Krishna  and P. Sam Johnson.  \textbf{Factorable weak operator-valued frames}.
	 \textit{Annals of Functional Analysis}, 13, 11 (2022). 
	 
			\vspace{-.1cm} 
		https://doi.org/10.1007/s43034-021-00155-4.}
\item{K. Mahesh Krishna and   P. Sam Johnson.	 \textbf{Commutators close to the identity in unital C*-algebras}. \textit{Proceedings - Mathematical Sciences} (accepted).
	
	\vspace{-.1cm}  https://arXiv.org/2104.02035v1 [math.FA] 5 April 2021.}
\item{K. Mahesh Krishna  and P. Sam Johnson.  \textbf{Dilations of linear maps on vector spaces}. \textit{Operators and Matrices} (accepted).
	
	\vspace{-.1cm}  
	https://arXiv.org/2104.07544v1 [math.FA] 14 April 2021.}	
	\item {K. Mahesh Krishna and P. Sam Johnson.  \textbf{Multipliers for Lipschitz p-Bessel sequences in metric spaces}.
	
	\vspace{-.1cm} 
	https://arXiv.org/2007.03209v1 [math.FA] 7 July 2020.}
\end{enumerate}
	


\leavevmode\newpage


	
	
	
	

\newpage
\newpage 
\newpage 
\end{document}